\documentclass[11pt]{amsart}%

\usepackage{amsmath}%
\usepackage{amssymb}%
\usepackage{amscd}%
\usepackage{graphicx}
\usepackage{array}
\usepackage{color}%
\usepackage{mathrsfs}%
\usepackage[all]{xy}%
\usepackage{stmaryrd}
\usepackage{hyperref}
\usepackage{comment}
\usepackage[british,french]{babel}

\def\wh#1{\widehat{#1}}% 	wide hat
\def\wt#1{\widetilde{#1}}% 	wide tilde
\theoremstyle{plain}
    \newtheorem{theorem}{Theorem}[section]
    
    \newtheorem{proposition}[theorem]{Proposition}
    \newtheorem{lemma}[theorem]{Lemma}
    \newtheorem{corollary}[theorem]{Corollary}
      
      \newtheorem{proposition-definition}[theorem]{Proposition-Definition}
\theoremstyle{definition}
    \newtheorem{definition}[theorem]{Definition}
    \newtheorem{notation}[theorem]{Notation}
    
    \newtheorem{remark}[theorem]{Remark}
	\newtheorem{construction}[theorem]{Construction}

\def\Alphabet{A,B,C,D,E,F,G,H,I,J,K,L,M,N,O,P,Q,R,S,T,U,V,W,X,Y,Z}%  Capitalized Alphabet
\def\alphabet{a,b,c,d,e,f,g,h,i,j,k,l,m,n,o,p,q,r,s,t,u,v,w,x,y,z}%	lowercase alphabet
\def\endpiece{xxx}%									marks end of list
\def\makeAlphabet[#1]{\expandafter\makeA#1,xxx,}%		Ex. \makeAlphabet[A,B]
\def\makealphabet[#1]{\expandafter\makea#1,xxx,}%		Ex. \makealphabet[c,d]
\def\makeA#1,{\def\temp{#1}\ifx\temp\endpiece\else%
\mkbb{#1}\mkfrak{#1}\mkbf{#1}\mkcal{#1}\mkscr{#1} \mkbs{#1} \expandafter\makeA\fi}%
\def\makea#1,{\def\temp{#1}\ifx\temp\endpiece\else\mkfrak{#1}\mkbf{#1} \mkbs{#1} \expandafter\makea\fi}%
\def\mkbb#1{\expandafter\def\csname bb#1\endcsname{\mathbb{#1}}}%      Define bb
\def\mkfrak#1{\expandafter\def\csname fr#1\endcsname{\mathfrak{#1}}}%    Define frak
\def\mkbf#1{\expandafter\def\csname b#1\endcsname{\mathbf{#1}}}%           Define bold letters
\def\mkcal#1{\expandafter\def\csname c#1\endcsname{\mathcal{#1}}}%       Define calligraphy
\def\mkscr#1{\expandafter\def\csname s#1\endcsname{\mathscr{#1}}}%       Define script
\def\mkbs#1{\expandafter\def\csname bs#1\endcsname{{\boldsymbol{#1}}}}%       Define bold symbol
\def\makeop[#1]{\xmakeop#1,xxx,}%					Ex. \makeop[Hom,Spec]
\def\mkop#1{\expandafter\def\csname #1\endcsname{{\mathrm{#1}}}} % 
\def\xmakeop#1,{\def\temp{#1}\ifx\temp\endpiece\else\mkop{#1} \expandafter \xmakeop\fi}%
\def\makeup[#1]{\xmakeup#1,xxx,}%					Ex. \makeop[Hom,Spec]
\def\mkup#1{\expandafter\def\csname #1\endcsname{{\mathrm{#1}\,}}} % 
\def\xmakeup#1,{\def\temp{#1}\ifx\temp\endpiece\else\mkup{#1} \expandafter\xmakeup\fi}%
\makeAlphabet[\Alphabet]%				Define bb, frak, bf, cal for Capitalized Alphabet
\makealphabet[\alphabet]%  				Define frak and bf for uncapitalized alphabet
\makeop[Alt,End,Ext,Hom,Sym,Tor,dR,rig,HK,Li,Cone,res,Res,id, pol,Gr,MHS,Re,Vec,Rep,Rex,bil,GL,cris,et,conv,GK,PD,Gd]
\makeup[Spec,Spwf,Sp,Im,Tot,Spf,hocolim,holim,Ker,Coker]
  \makeatletter
    
    \@addtoreset{equation}{section}
  \makeatother

\begin{document}
%---the title--------------------------------------------------------------------------------------
\title{Rigid analytic reconstruction of Hyodo--Kato theory}
\author[Ertl]{Veronika Ertl}
\address{Université Caen Normandie, Laboratoire de Mathématiues Nicolas Oresme, 6 boulevard Maréchal Juin, 14032 Caen CEDEX, France}
\email{veronika.ertl@unicaen.fr}
\author[Yamada]{Kazuki Yamada}
\address{Gakushuin University, Department of Mathematics, 1-5-1 Mejiro, Toshima-ku, Tokyo, Japan}
\email{k.yamada@gakushuin.ac.jp}
\date{\today}

\selectlanguage{british}
\begin{abstract}
	We give a new and very intuitive construction of Hyodo--Kato cohomology and the Hyodo--Kato map, based on logarithmic rigid cohomology.
	We show that it is independent of the choice of a uniformiser and study its dependence on  the  choice of a branch of the $p$-adic logarithm.
	Moreover, we show the compatibility with the classical construction of Hyodo--Kato cohomology and the Hyodo--Kato map.\\

\noindent
\textsc{R\'esum\'e.}
On met en place une nouvelle construction tr\`es intuitive de cohomologie de Hyodo--Kato et de morphisme de Hyodo--Kato,
fond\'ee sur la cohomologie rigide logarithemique.
On d\'emontre qu'elle est ind\'ependante du choix d'une uniformisante et on \'etudie sa d\'ependance du choix d'une branche 
du logarithme $p$-adique.
En outre, on montre la compatibilit\'e \`a la construction classique de la cohomologie et du morphisme de Hyodo--Kato.\\

\end{abstract}

\selectlanguage{british}

\thanks{This research was partially supported by the JSPS grant KAKENHI 18H05233, 22K13899, and 23KJ0332, by the DFG grant SFB 1085 ``Higher Invariants'',  
by the KLL 2018--2019 PhD.\,Program Research Grant,
and by the Research Institute for Mathematical Sciences, an International Joint Usage/Research Center located in Kyoto University.}

\maketitle

\setcounter{tocdepth}{1}
\tableofcontents

%%%%%%%%%%%%%%%%%%%%%%%%%%%%%%

%%%%%%%%%%%%%%%%
%
\section*{Introduction}\label{Sec: Intro}
%
%%%%%%%%%%%%%%%%

Let $V$ be a complete discrete valuation ring of mixed characteristic $(0,p)$ with fraction field $K$ and perfect residue field $k$.
Let $F$ be the fraction field of the ring of Witt vectors $W(k)$.
Hyodo--Kato theory was established by Hyodo and Kato \cite{HK} as an answer to the question of Jannsen \cite{Ja} concerning a cohomological construction of $(\varphi,N)$-module structures on the de Rham cohomology groups of proper schemes over $V$ of semistable reduction.
More precisely, they defined the crystalline Hyodo--Kato cohomology $R\Gamma_\HK^\cris(X)$ whose rational cohomology groups are finite dimensional $F$-vector spaces, endowed with a Frobenius-linear automorphism $\varphi$, an $F$-linear endomorphism $N$ satisfying $N\varphi= p \varphi N$, and a homomorphism $\Psi_\pi^\cris\colon H^{\cris,i}_\HK(X)\rightarrow H^i_\dR(X_K)$ which is an isomorphism after tensoring with $K$.
The latter depends on the choice of a uniformiser $\pi$ of $V$, and a change of uniformiser is encoded in a transition function involving the exponential of the monodromy.

Hyodo--Kato cohomology is a key object in the formulation of the semistable conjecture of Fontaine and Jannsen \cite{Fo} which restores the Galois action on the \'etale cohomology from the additional structures on the Hyodo--Kato and de Rham cohomology.
Therefore it plays an important role in several areas of arithmetic geometry, including the research of special values of $L$-functions.
The semistable conjecture is now a theorem due to different proofs by Tsuji \cite{Ts}, Faltings \cite{Fa2}, Niziol \cite{Ni}, and Beilinson \cite{Bei}. 
In particular, Beilinson represented  crystalline Hyodo--Kato cohomology by a complex with nilpotent monodromy.
What is more, his construction of the  Hyodo--Kato map is independent of the choice of a uniformiser.
This enabled Nekov\'a\v{r} and Nizio\l{} to extend syntomic cohomology to $K$-varieties.
However Beilinson's Hyodo--Kato theory doesn't lend itself to explicit computations, as it is based on very abstract (crystalline) considerations.

Whereas the original construction used crystalline methods, Gro\ss{}e-Kl\"onne proposed a rigid analytic version of the Hyodo--Kato map \cite{GK3}, based on log rigid cohomology instead of log crystalline cohomology.
It provides a description of the Hyodo--Kato map in terms of $p$-adic differential forms on certain dagger spaces. 
However it depends on the choice of a uniformiser of $V$ and passes through several zig-zags of quasi-isomorphisms whose intermediate objects are quite complicated.
For example, even if $X$ is affine, we have to pass through the cohomology of simplicial log schemes with boundary which are no longer affine.
Therefore his Hyodo--Kato map is also difficult to compute explicitly.

Our motivation is to establish a rigid Hyodo--Kato theory which would lend itself to explicit computation and which would be independent of the choice of a uniformiser.
In this paper, we  first give a new definition of  rigid Hyodo--Kato cohomology $R\Gamma_\HK^\rig(\cX)$ using a construction due to Kim and Hain \cite{KH}.
A difference to \cite{KH} is that we use a complex $\omega^\bullet_{\cZ/W^\varnothing,\bbQ}[u]$ on a dagger space over the open unit disk, instead of a complex on the fiber at zero.
This small difference allows us to define a rigid Hyodo--Kato map $\Psi_{\pi,q}\colon R\Gamma_\HK(\cX)\rightarrow R\Gamma_\dR(\cX)$ in a very simple and natural way.
Thus it can be computed explicitly in terms of \v{C}ech cocycles.
The authors expect that this feature will be instrumental in extending the construction of the Hyodo--Kato morphism to more general  variants of cohomology theories,
such as cohomology with coefficients in log overconvergent $F$-isocrystals.

Furthermore, our rigid Hyodo--Kato map slightly generalises the original Hyodo--Kato map, in the sense that the definition of $\Psi_{\pi,q}$ involves not only the choice of a uniformiser $\pi$ of $V$ but also of a non-zero element $q$ of the maximal ideal of $V$.
We will see that it is in fact independent of the choice of $\pi$, and depends only on the branch of the $p$-adic logarithm defined by $q$.
This is consistent with the fact that the  functor $D_{\mathrm{st}}$ associating filtered $(\varphi,N)$-modules to $p$-adic Galois representations  depends on the choice of a branch of the $p$-adic logarithm.
We will also see that our construction of the rigid Hyodo--Kato map for the choice $q=\pi$ is compatible with  Hyodo--Kato's original and with  Gro\ss e-Kl\"{o}nne's construction for the choice $\pi$.

%%%
\subsubsection*{Overview of the paper}
%%%

We start with technical preparations in \S \ref{Sec: Weak formal schemes}, where we introduce weak formal schemes 
which are not necessarily adic over the base.
In \S \ref{subsection: weak formal schemes} we first discuss notions and properties of pseudo weakly complete finitely generated (pseudo-wcfg) algebras over
a base ring $R$ with a distinguished ideal $I$,
and use them as building blocks to define weak formal schemes. 
We point out that  weak completeness is, unlike completeness, a relative notion.
While some notions concerning weak formal schemes, that are not necessarily $p$-adic, 
depend a priori on the choice of  an ideal of definition and its generators, 
it turns out that they are independent of such a choice.
In the case that the base is a complete discrete valuation ring of mixed characteristic $(0,p)$,
we define the dagger space associated to a weak formal scheme.
Crucially for the construction of rigid cohomology, we discuss differential forms of dagger spaces.

To be able to make use of these weak formal schemes in the construction of our (log) rigid cohomology we need rather strong lifting properties.
To facilitate this, we discuss issues related to smoothness and \'{e}taleness in \S \ref{subsection: smoothness}.
In particular we introduce strong smoothness and strong \'{e}taleness in the context of weak formal schemes and study necessary and sufficient conditions for these properties.

In \S \ref{subsection: weak formal log schemes} we handle in particular weak formal log schemes 
and extend  important concepts from the usual to the logarithmic case.
Notably, we discuss the exactification of an immersion of weak formal log schemes. 

The latter simplifies considerably the definition of log rigid cohomology that we recall in general in \S \ref{Sec: Log rig coh}: 
In order to define the log rigid cohomology of a log scheme in positive characteristic $p$, we consider weak formal local lifts to characteristic $0$ and show how to glue them.
The existence of exactifications allows us to only consider homeomorphic exact closed immersions at this point, instead of general closed immersions.
A point that plays an important role in the constructions of the subsequent section is that using Godement resolution, we may find an explicit complex representing log rigid cohomology. 

Based on this, we introduce in \S \ref{Sec: Rig HK coh} more specifically rigid Hyodo--Kato cohomology with Frobenius and monodromy operators.
The core of this construction is a Kim--Hain complex on a dagger space over the open unit disk -- 
a non-$p$-adic weak formal scheme over a complete discrete valuation ring of mixed characteristic $(0,p)$.
We also relate it to  other log rigid cohomology theories.
To establish these relations, it turned out to be useful 
to understand the Kim--Hain complexes in terms of (total complexes of) certain double complexes,
which we discuss in a technical part at the beginning of this section.

The use of Kim--Hain complexes over the open unit disk allows us to define the rigid Hyodo--Kato map in \S \ref{Sec: Rig HK map}
in a very direct way.
We study  dependence of the rigid Hyodo--Kato map on the choice of a uniformiser and a branch of the $p$-adic logarithm, 
and show that it becomes an isomorphism after tensoring with $K$.

To illustrate our results as well as the advantage of the use of not necessarily $p$-adic weak formal schemes, 
we compute in \S \ref{Sec: Tate curve} the rigid Hyodo--Kato cohomology and the rigid Hyodo--Kato map of a Tate curve.
In \S \ref{Sec: comparison}, we compare our rigid Hyodo-Kato map and the classical crystalline Hyodo--Kato map  in suitable situations.

%%%
\subsubsection*{Acknowledgements}
%%%

We would like to thank Kenichi Bannai for helpful comments and for creating a pleasant working atmosphere which allowed us to enjoy many productive discussions.
It is a pleasure to thank Yoshinosuke Hirakawa, David Loeffler, Yukiyoshi Nakkajima, Wies\l{}awa Nizio\l{} and Sarah Livia Zerbes for stimulating discussions and helpful comments related to the topic of this article. 
We are indebted to a referee of \cite{EY2} for pointing out to us an error in the literature that previously played a role in certain computations of the present paper. We also would like to thank Christopher Lazda for explaining this problem to us in detail.
We warmly thank the referee of this article and the referee of \cite{Ya} for many suggestions to improve this article.

%%%%%%%%%%%%%%%%
%
\section{Weak formal schemes}\label{Sec: Weak formal schemes}
%
%%%%%%%%%%%%%%%%

In this section we slightly generalise the notion of weak formal schemes, originally defined by Meredith \cite{Me}.
While in his sense all weak formal schemes are adic over the base ring, we don't make this assumption.
This takes advantage of the favourable properties of weak formal schemes,  but in addition simplifies many of the technical arguments.
A key point is that it allows us to canonically define exactifications of immersions of weak formal log schemes.

%%%%%%%%
\subsection{Basic notions concerning weak formal schemes}\label{subsection: weak formal schemes}
%%%%%%

Let $R$ be a noetherian ring with an ideal $I$.
For any $n\geq 0$, we denote  by $R_{[n]}$ the polynomial algebra $R[s_1,\ldots,s_n]$ and by $I_{[n]}\subset R_{[n]}$  the ideal generated by $I$ and $s_1,\ldots,s_n$.
Denote by $\widehat{R}$ the $I$-adic completion of $R$.

Recall that for an $R$-algebra $A$, the $I$-adic weak completion $A^\dagger$ of $A$ is the $R$-subalgebra of $\widehat{A} = \varprojlim_n A/I^nA$ which consists of elements $\xi$ which can be written as
	\[\xi = \sum_{i\geq 0}P_i(x_1,\ldots,x_r),\]
where $x_1,\ldots,x_r\in A$, $P_i(X_1,\ldots,X_r)\in I^i\cdot R[X_1,\ldots,X_r]$, and there exists a constant $c>0$ such that
	\[c(i + 1)\geq\mathrm{deg}P_i\]
for all $i\geq 0$.

An $R$-algebra $A$ is said to be \textit{weakly complete} if $A^\dagger = A$, and \textit{weakly complete finitely generated (wcfg)} if for some $k\in\mathbb{N}$ there exists a surjection $R[t_1,\ldots,t_k]^\dagger \rightarrow  A$ over $R$. 

According to a remark in \cite[paragraph after Thm.\ 1.5]{MW} weakly complete algebras over $R$ and $\widehat{R}$ are equivalent to each other.
In the same way, one can see that this is also true for weakly complete algebras which are in addition weakly finitely generated, as we show in the following lemma.

\begin{lemma}\label{lem: base completion}
	For any wcfg $R$-algebra $A$, the structure map uniquely extends to a homomorphism $\widehat{R} \rightarrow  A$ of $R$-algebras, and $A$ is $I\widehat{R}$-adically wcfg over $\widehat{R}$.
	Conversely, a wcfg algebra over $\widehat{R}$ is also wcfg over $R$.
\end{lemma}

\begin{proof}
	The first claim of the lemma is just \cite[Thm.~1.5]{MW}, because $R^\dagger = \widehat{R}$.
	
	Then it suffices to show the equality
	\[R[t_1,\ldots,t_k]^\dagger=\widehat{R}[t_1,\ldots,t_k]^\dagger\]
	as subsets of the $IR[t_1,\ldots,t_k]$-adic completion of $R[t_1,\ldots,t_k]$.
	Here the left hand side is the $I$-adic weak completion of $R[t_1,\ldots,t_k]$ over $R$ and the right hand side is the $I\widehat{R}$-adic weak completion of $\widehat{R}[t_1,\ldots,t_k]$ over $\widehat{R}$.
	
	The inclusion $R[t_1,\ldots,t_k]^\dagger\subset\widehat{R}[t_1,\ldots,t_k]^\dagger$ is clear by definition.
	We prove the opposite inclusion.
	Note that any element of $\widehat{R}[t_1,\ldots,t_k]^\dagger$ is expressed as an infinite sum
	\[\xi=\sum_{i\geq 0}P_i(t_1,\ldots,t_k)\]
	where $P_i\in I^i\cdot \widehat{R}[t_1,\ldots,t_k]$ and there exists a constant $c>0$ such that
	\[\deg P_i\leq c(i+1)\]
	for all $i$. (In other words, $\{t_1,\ldots,t_k\}$ is a set of weak generators of $\widehat{R}[t_1,\ldots,t_k]^\dagger$ in the sense of \cite[Def.\ 2.1]{MW}.)
	Taking $I$-adic expansions of coefficients, each $P_i$ can be written as
		\[P_i(t_1,\ldots,t_k) = \sum_{j\geq i}Q_{i,j}(t_1,\ldots,t_k),\]
	where $Q_{i,j}(t_1,\ldots,t_k)\in I^j\cdot R[t_1,\ldots,t_k]$ and $\mathrm{deg}P_i  =  \max_{j\geq i} \mathrm{deg}Q_{i,j}$. 
	If we set
		\[P'_j(t_1,\ldots,t_k): = \sum_{i\leq j}Q_{i,j}(t_1,\ldots,t_k)\]
	for any $j\geq 0$, then we have
		\[\xi = \sum_{j\geq 0}P'_j(t_1,\ldots,t_k),\ P'_j\in I^j\cdot R[t_1,\ldots,t_k],\]
	and
		\[\mathrm{deg}P'_j\leq\max_{i\leq j}\{\mathrm{deg} Q_{i,j}\}\leq c(j + 1).\]
	Thus $\xi$ belongs to $R[t_1,\ldots,t_k]^\dagger$.
\end{proof}

\begin{remark}
	Let $A$ be an $R$-algebra and fix elements $x_1,\ldots,x_r\in A$, $\xi\in\widehat{A}$.
	Assume that $\xi$ can be written in the form
		\[\xi = \sum_{\alpha\in\bbN^r}a_\alpha x_1^{\alpha_1}\cdots x_r^{\alpha_r}\ (a_\alpha\in R)\]
	and there exists a constant $d>0$ such that
		\[d(\mathrm{ord}_I(a_\alpha) + 1)\geq\lvert\alpha\rvert\]
	for all $\alpha\in\bbN^r$.
	Then $\xi$ can also  be written as
		\[\xi = \sum_{i\geq 0}P_i(x_1,\ldots,x_r),\]
	where $P_i(X_1,\ldots,X_r)\in I^i\cdot R[X_1,\ldots,X_r]$ and there exists a constant $c>0$ such that
	\[c(i + 1)\geq\mathrm{deg}P_i\]
	for all $i\geq 0$.
	The converse holds if $R$ is $I$-adically complete, but not in general.
\end{remark}

For $R$ and $\widehat{R}$, we always consider the $I$-adic topology. 
In particular, if we say $A$ is a topological $R$-algebra, 
then the homomorphism $R\rightarrow A$ is continuous.

Recall that a topological ring $A$ is called \textit{pre-adic} if there exists an ideal $J\subset A$ such that $\{J^n\}_{n\geq 0}$ is a fundamental system of open neighbourhoods of $0$.
In this case, $J$ is called an \textit{ideal of definition} of $A$.

\begin{definition}\label{def: pwcfg}
	Let $A$ be a topological $R$-algebra which is a pre-adic ring. We say that $A$ is
	{\it pseudo-weakly complete finitely generated (pseudo-wcfg)} with respect to $(R,I)$ if there exists an ideal of definition $J$ of $A$ and a finite generating system $f_1,\ldots,f_n\in A$ of $J$ such that $A$ is $I_{[n]}$-adically wcfg over $R_{[n]}$ with respect to the map
		\[R_{[n]} = R[s_1,\ldots,s_n] \rightarrow  A,\ s_i\mapsto f_i.\]
	This means that there exists an integer $k\geq 0$ and an $R_{[n]}$-linear surjection 
		\[\rho\colon R_{[n]}[t_1,\ldots,t_k]^{\dagger} \rightarrow  A,\]
	whose domain denotes the $I_{[n]}$-adic weak completion of $R_{[n]}[t_1,\ldots,t_k]$.
	In this case, we call $\rho$ a representation of $A$.
	A morphism of pseudo-wcfg algebras with respect to $(R,I)$ 
is a morphism of topological $R$-algebras.
Note that any pseudo-wcfg algebra is noetherian, just as a wcfg algebra is noetherian by \cite[Thm.\,2.1]{MW}.
\end{definition}

\begin{lemma}\label{lem: independence}
	Let $A$ be a topological $R$-algebra with an ideal of definition.
	Choose elements $f_1,\ldots,f_n\in A$ which generates an ideal of definition of $A$.
	Then the $I_{[n]}$-adic weak completion of $A$ with respect to the map
		\[R_{[n]} = R[s_1,\ldots,s_n] \rightarrow  A,\ s_i\mapsto f_i\]
	is, up to canonical $R$-linear isomorphisms, independent of the choice of an ideal of definition and its generators.
\end{lemma}

\begin{proof}
	Assume that ideals $J = (f_1,\ldots,f_n)$ and $J' = (g_1,\ldots,g_m)$ define the same topology on $A$.
	Let $B$ and $B'$ be the weak completions of $A$ obtained from $f_1,\ldots,f_n$ and $g_1,\ldots,g_m$ respectively.
	We denote by
		\begin{align*}
		\theta\colon R_{[n]} \rightarrow  A,\ s_i\mapsto f_i  & & \text{and} & & \theta'\colon R_{[m]} \rightarrow  A,\ s_j\mapsto g_j
		\end{align*}
	the structure maps over $R_{[n]}$ and $R_{[m]}$ respectively.
	It is enough to show that $B$ and $B'$ coincide with each other as $A$-subalgebras of $\widehat A$, the $J$-adic completion of $A$.
	
	We claim that $B'$ is $I_{[n]}$-adically weakly complete over $R_{[n]}$ with respect to the composition $R_{[n]} \xrightarrow{\theta} A \rightarrow  B'$.
	Thus consider a series
		\begin{equation}\label{eq: Q1}
		\xi = \sum_{q\geq 0}P_q(x_1,\ldots,x_r)\in\widehat{A},
		\end{equation}
		where $x_1,\ldots,x_r\in B'$, $P_q(X_1,\ldots,X_r)\in I_{[n]}^{q}\cdot R_{[n]}[X_1,\ldots,X_r]$, and there exists a constant $c>0$ such that
		\begin{equation}\label{eq: MW1}
		c(q+1)\geq\mathrm{deg}P_q
		\end{equation}
	for any $q\geq 0$.
	This is an element in the $I_{[n]}$-adic weak completion of $B'$ with respect to $R_{[n]} \xrightarrow{\theta} A \rightarrow  B'$.
	We will show that $\xi$ lies in fact in $B'$.
	Since $J$ and $J'$ define the same topology, there exists an integer $k\geq 0$ with $J^k\subset J'$.
	Let $F_1,\ldots,F_N\in A$ be the monic monomials in $f_1,\ldots,f_n$ of degree $k$, i.e. each $F_\nu$ is of the form $F_\nu = f_1^{k_1}\cdots f_n^{k_n}$ with $k_1 + \cdots k_n = k$.
	Each $F_\nu$ has a presentation
		\begin{equation}\label{eq: Q2}
		F_\nu = \sum_{j = 1}^my_{\nu,j}g_j
		\end{equation}
	in terms of the generators $g_1,\ldots,g_m$ of $J'$ with some $y_{\nu,j}\in A$.
	For each $q\geq 0$, $P_q$ can be written as
		\begin{equation}\label{eq: Q3}
		P_q(X_1,\ldots,X_r) = \sum_{\substack{\alpha\in\bbN^r\\ \lvert\alpha\rvert\leq\mathrm{deg}P_q}}h_{q,\alpha}X_1^{\alpha_1}\cdots X_r^{\alpha_r}
		\end{equation}
	with $h_{q,\alpha}\in I_{[n]}^q\subset R_{[n]}$.
	Moreover, the image of $h_{q,\alpha}$ in $A$ under $\theta$ is a series
		\begin{equation}\label{eq: Q4}
		\theta(h_{q,\alpha}) = \sum_{\substack{\beta\in\bbN^n\\ \lvert\beta\rvert<k}} \sum_{\substack{\gamma\in\bbN^N\\ \lvert\gamma\rvert\leq M_{q,\alpha,\beta}}}a_{q,\alpha,\beta,\gamma}f_1^{\beta_1}\cdots f_n^{\beta_n}F_1^{\gamma_1}\cdots F_N^{\gamma_N}
		\end{equation}
	for some integers $M_{q,\alpha,\beta}\geq 0$, where $a_{q,\alpha,\beta,\gamma}\in R$ and
		\begin{equation}\label{eq: indep}
		\mathrm{ord}_I(a_{q,\alpha,\beta,\gamma}) + \lvert\beta\rvert + k\lvert\gamma\rvert\geq q.
		\end{equation}
	Thus we obtain
		\[\xi = \sum_{\ell\geq 0}\sum_{\lambda\in\Lambda_\ell}Q_\lambda \left((x_\mu)_{1\leq \mu\leq r},(y_{\nu,j})_{\substack{1\leq \nu\leq N\\ 1\leq j\leq m}},(f_i)_{1\leq i\leq n}\right),\]
	where $\Lambda_\ell$ is the set of all systems $ \left(q,\alpha,\beta,\gamma,(\ell_{\nu,j})_{\substack{1\leq \nu\leq N\\ 1\leq j\leq \gamma_\nu}}\right)$ such that
		\begin{align*}
		&q\geq 0,\\
		&\alpha\in\bbN^r, & & \lvert\alpha\rvert\leq\mathrm{deg}P_q,\\
		&\beta\in\bbN^n, & & \lvert\beta\rvert<k,\\
		&\gamma\in\bbN^N, & & \lvert\gamma\rvert\leq M_{q,\alpha,\beta},\\
		&1\leq\ell_{\nu,j}\leq m, & & \mathrm{ord}_I(a_{q,\alpha,\beta,\gamma}) + \lvert\gamma\rvert = \ell,
		\end{align*}
	and the polynomial $Q_\lambda$ is defined by
		\begin{eqnarray*}
		Q_\lambda \left((X_\mu)_{1\leq \mu\leq r},(Y_{\nu,j})_{\substack{1\leq \nu\leq N\\ 1\leq j\leq m}},(Z_i)_{1\leq i\leq n}\right)&: = &a_{q,\alpha,\beta,\gamma}X_1^{\alpha_1}\cdots X_r^{\alpha_r}Z_1^{\beta_1}\cdots Z_n^{\beta_n}\prod_{\substack{1\leq \nu\leq N\\ 1\leq j\leq\gamma_\nu}}s_{\ell_{\nu,j}}Y_{\nu,\ell_{\nu,j}}\\
		&\in&I_{[m]}^\ell\cdot R_{[m]}[(X_\mu)_{1\leq \mu\leq r},(Y_{\nu,j})_{\substack{1\leq \nu\leq N
		\\ 1\leq j\leq m}},(Z_i)_{1\leq i\leq n}]
		\end{eqnarray*}
	
	Note that for any $\lambda\in\Lambda_\ell$, \eqref{eq: indep} implies that $q$ must satisfy $q\leq k(\ell + 1)$, and hence that $\Lambda_\ell$ is a finite set.
	Since we have 
		\begin{eqnarray*}
		\mathrm{deg}Q_\lambda& = &\lvert\alpha\rvert + \lvert\beta\rvert + \lvert\gamma\rvert\leq \mathrm{deg}P_q + k + \lvert\gamma\rvert\leq c(q + 1) + k + \lvert\gamma\rvert\\
		&\leq&c(\mathrm{ord}_I(a_{q,\alpha,\beta,\gamma}) + \lvert\beta\rvert + k\lvert\gamma\rvert + 1) + k + \lvert\gamma\rvert\\
		&\leq&c(k\ell + k + 1) + k + \ell,
		\end{eqnarray*}
	there exists $d>0$ such that $\mathrm{deg}Q_\lambda\leq d(\ell + 1)$.
	Since $B'$ is $I_{[m]}$-adically weakly complete over $R_{[m]}$, we see that $\xi$ lies in $B'$.
	This proves the claim.
\end{proof}

\begin{corollary}\label{cor: ideal of definition}
	Let $A$ be a topological $R$-algebra.
	Then $A$ is pseudo-wcfg with respect to $(R,I)$ if and only if the condition in Definition \ref{def: pwcfg} holds for any ideal of definition $J$ of $A$ and any finite generating system $f_1,\ldots,f_n$ of $J$.
\end{corollary}

\begin{proof}
	We use the same notation as in the proof of Lemma \ref{lem: independence}. 
	By the lemma, the $I_{[n]}$-adic and $I_{[m]}$-adic weak completeness of $A$ are equivalent to each other.
	Moreover the computation in the proof of the lemma also shows that if $(x_\mu)_{1\leq \mu\leq r}$) is a weak generating system of $A$ over $R_{[n]}$, then $(x_\mu)_{1\leq \mu\leq r}$, $(y_{\nu,j})_{\substack{1\leq \nu\leq N\\ 1\leq j\leq m}}$, and $(f_i)_{1\leq i\leq n}$ together form a weak generating system of $A$ over $R_{[m]}$.
\end{proof}

Now we discuss the weak formal spectrum of a pseudo-wcfg algebra.
From now on we suppose that $R$ is a noetherian local ring and $I$ is the maximal ideal of $R$.
Note that $R_{[n]}$ is not a local ring, but its $I_{[n]}$-adic completion $\widehat{R_{[n]}}=\widehat{R}\llbracket s_1,\ldots,s_n\rrbracket$ is a local ring with maximal ideal $\widehat{I_{[n]}}=I_{[n]}\widehat{R_{[n]}}$.
By Lem.\,\ref{lem: base completion} a wcfg algebra over $(R_{[n]},I_{[n]})$ is automatically wcfg over $(\widehat{R_{[n]}},\widehat{I_{[n]}})$.
For this reason the assumption that $R$ is local will not cause any problems in the arguments below.

\begin{definition}
Let $A$ be a pseudo-wcfg algebra with respect to $(R,I)$ 
with a representation $R_{[n]}[t_1,\ldots,t_k]^\dagger \rightarrow  A$.
	For $f\in A$, we denote by $A_f^\dagger$ the $I_{[n]}$-adic weak completion of $A_f$.
	Then by Lemma \ref{lem: independence} $A_f^\dagger$ is independent of the choice of representation. 
	
	We define $\Spwf A$ to be the topologically ringed space whose underlying topological space is $\Spec A/J$ for some ideal of definition $J$ of $A$ and the structure sheaf $\cO_{\Spwf A}$ is defined by
		\[\Gamma(\Spec (A/J)_{\overline{f}},\cO_{\Spwf A}): = A_f^\dagger,\]
	for $f\in A$ and $\overline{f}$ its image in $A/J$.
	
	Note that the topological space $\Spec A/J$ is independent of the choice of $J$.
	Moreover the structure sheaf $\cO_{\Spwf A}$ and hence the topologically ringed space $\Spwf A$ are also well-defined.
	
	An \textit{affine weak formal scheme} with respect to $(R,I)$ is a topologically ringed space which is isomorphic to $\Spwf A$ for some pseudo-wcfg algebra $A$ with respect to $(R,I)$.
\end{definition}

\begin{remark}\label{rem: Spwf} $\quad$
\begin{enumerate}
\item The notion of a pseudo-wcfg algebra and its weak formal spectrum as developed in the present paper can be understood
in the set-up developed by Meredith \cite{Me}.
As we have seen, a pseudo-wcfg algebra $A$ over $R$ can be seen as a wcfg algebra over $\widehat{R_{[n]}}$ for some $n$,
and $\Spwf A$ can be seen as a weak formal scheme over $\widehat{R_{[n]}}$ in Meredith's sense. 
The difference to the classical definition is that we endow $A$ only with an $R$-algebra structure, not $\widehat{R_{[n]}}$-algebra structure, thereby making the definition canonical.
\item\label{item: open prime} As a general fact for pre-adic rings, $\Spwf A$ is canonically identified with the set of open prime ideals of $A$.
\end{enumerate}
\end{remark}

In \cite[\S 0]{Me}, it was stated that the stalks of $\cO_{\Spwf A}$ at any point is a local ring. The proof seems to be omitted, but one can prove it by the same argument as in the case for formal schemes.
More precisely we have the following.

\begin{lemma}\label{lem: local ring}
	Let $\cZ=\Spwf A$ be an affine weak formal scheme with respect to $(R,I)$.
	Let $x\in\cZ$ be a point corresponding to an open prime ideal $\frp\subset A$.
	Then the stalk $\cO_{\cZ,x}$ is a local ring with maximal ideal $\frp\cO_{\cZ,x}$, and it is flat over $A$.
	Moreover, if we denote by $\tau\colon A\rightarrow\cO_{\cZ,z}$ the natural homomorphism, then we have $\tau^{-1}(\frp\cO_{\cZ,x})=\frp$.
\end{lemma}

\begin{proof}
By definition we have $\cO_{\cZ,x}=\varinjlim_{f\in A\setminus\frp}A_f^\dagger$, 
which is flat over $A$ 
because the $A_f^\dagger$ are flat over $A$ \cite[\S2, Lem.\,2]{Me}.
We also have $\frp\cO_{\cZ,x}=\varinjlim_{f\in A\setminus\frp}\frp A_f^\dagger$.

Let $J\subset A$ be an ideal of definition and set 
$\overline{A}:=A/J$, $\overline{\frp}:=\frp \overline{A}$, 
and $Z:=\Spec \overline{A}$.
Let $q_f\colon A_f^\dagger\rightarrow A_f^\dagger/JA_f^\dagger=\overline{A}_{\overline{f}}$ 
and $q_x\colon \cO_{\cZ,x}\rightarrow\cO_{Z,x}$ 
be the natural surjections, 
where $\overline{f}$ denotes the image of $f$ in $\overline{A}$.

Since $JA_f^\dagger$ is contained in the Jacobson radical of $A_f^\dagger$ \cite[Thm.\,1.6]{Me}, 
an element $g\in A_f^\dagger$ is invertible if and only if $q_f(g)$ is invertible.
This implies that an element $h\in \cO_{\cZ,x}$ is invertible if and only if $q_x(h)$ is invertible or equivalently $q_x(h)\notin \overline{\frp}\cO_{Z,x}$ because $\overline{\frp}\cO_{Z,x}$ is the unique maximal ideal of $\cO_{Z,x}$. 
Thus $h\in \cO_{\cZ,x}$ being invertible is equivalent to $h\notin\frp\cO_{\cZ,x}$ because we have $\frp\cO_{\cZ,x}=q_x^{-1}(\overline{\frp}\cO_{Z,x})$.
This shows that $\frp\cO_{\cZ,x}$ is the unique maximal ideal of $\cO_{\cZ,x}$.

Finally, we have $\tau^{-1}(\frp\cO_{\cZ,x})=\tau^{-1}(q_x^{-1}(\overline{\frp}\cO_{Z,x}))=q^{-1}(\overline{\tau}^{-1}(\overline{\frp}\cO_{Z,x}))=q^{-1}(\overline{\frp})=\frp$, where $\overline{\tau}\colon\overline{A}\rightarrow\cO_{Z,x}$ and $q\colon A\rightarrow\overline{A}$ are the natural maps.
\end{proof}

The weak formal schemes considered in a classical context are adic over the base ring $R$.
The notion of pseudo-wcfg algebras developed above allows us to extend this notion,
so that it includes weak formal schemes which are not necessarily adic over $R$.
We recall that a \textit{locally topologically ringed space} is a topologically ringed  space (i.e.\ a topological space with a sheaf of topological rings) such that the stalk at each point is a local ring.
	 A morphism of locally topologically ringed spaces is a morphism of locally ringed spaces which induces continuous homomorphisms of topological rings.

\begin{definition}
\begin{enumerate}
\item
	A \textit{weak formal scheme} with respect to $(R,I)$ 
is a topologically ringed space over $\Spwf\widehat{R}$ which admits an open covering $\{\cU_i\}$ such that for each $i$ there exists an isomorphism $\cU_i\xrightarrow{\cong}\Spwf A_i$ over $\Spwf\widehat{R}$ for some pseudo-wcfg algebra $A_i$ with respect to $(R,I)$.
Note that weak formal schemes with respect to $(R,I)$ are locally topologically ringed spaces by Lemma \ref{lem: local ring}.
	A morphism of weak formal schemes with respect to $(R,I)$ is a morphism of locally topologically ringed spaces over $\Spwf\widehat{R}$.
\item
Let $\cZ=(\cZ,\cO_\cZ)$ be a weak formal scheme.
A coherent ideal sheaf $\cJ\subset\cO_\cZ$ is called an \textit{ideal of definition} if
for any affine open weak formal subscheme $\cU=\Spwf A\subset\cZ$ the ideal $\Gamma(\cU,\cJ)\subset A$ is an ideal of definition.
\item	To a weak formal scheme $\cZ$ 
	we may associate a sheaf $\wh\cO_\cZ$ by taking the completion with respect to an ideal of definition. 
	We call $\wh\cZ:=(\cZ,\wh\cO_\cZ)$ the \textit{completion} of $\cZ$.
\end{enumerate}	
\end{definition}

\begin{lemma}\label{lem: principal open}
	Let $\cZ$ be a weak formal scheme.
	For any element $f\in\Gamma(\cZ,\cO_\cZ)$,
	\[D(f):=\{z\in\cZ\mid \text{the image of $f$ in $\cO_{\cZ,z}$ is invertible}\}\]
	is an open subset of $\cZ$.
\end{lemma}

\begin{proof}
	This is just a special case of \cite[\href{https://stacks.math.columbia.edu/tag/01HZ}{Lem.\,01HZ}]{stacks}.
\end{proof}

\begin{proposition-definition}\label{prop-def: associated sheaf}
	Let $\cZ=\Spwf A$ be an affine weak formal scheme with respect to $(R,I)$ and $M$ a finite $A$-module.
	\begin{enumerate}
	\item\label{item: finite module sheaf} The presheaf $M^\sim$ on principal open subsets of $\cZ$ defined by $D(f)\mapsto M\otimes_AA_f^\dagger$ is a sheaf.
		We denote again by $M^\sim$ the sheaf on $\cZ$ induced by $M^\sim$.
		This gives an exact functor from the category of finite $A$-modules to the category of coherent sheaves on $\cZ$.
	\item\label{item: sub or quotient} Let $\cF$ be a coherent sheaf on $\cZ$ which is a subsheaf or a quotient of $M^\sim$.
	Then we have a canonical isomorphism $\Gamma(\cZ,\cF)^\sim\xrightarrow{\cong}\cF$.
	\item\label{item: p-adic coherent}If $R$ is a discrete valuation ring with maximal ideal $I$ and $A$ is wcfg over $R$, then we have a canonical isomorphism $\Gamma(\cZ,\cF)^\sim\xrightarrow{\cong}\cF$ for any coherent sheaf $\cF$ on $\cZ$.
	\end{enumerate}
\end{proposition-definition}

\begin{proof}
	\eqref{item: finite module sheaf} is proved in \cite[\S 2, Props.\,5 and 12]{Me}	and \eqref{item: p-adic coherent} is \cite[\S 3, Thm.\,3]{Me}.
	To prove \eqref{item: sub or quotient}, it suffices by \eqref{item: finite module sheaf} to deal with the case that $\cF$ is a subsheaf of $M^\sim$.
	
	We first note that, for any $f\in A$, $\Gamma(D(f),\cF)$ is a finite $A_f^\dagger$-module, since it is a submodule of $\Gamma(D(f),M^\sim)=M\otimes_AA_f^\dagger$ and $A_f^\dagger$ is noetherian.
	In particular $F:=\Gamma(\cZ,\cF)$ is a finite $A$-module.
	
	By \cite[\href{https://stacks.math.columbia.edu/tag/01BK}{Lem.\ 01BK}]{stacks} we can take a finite covering of $\cZ$ of the form $\{D(g_i)\}_i$ with $g_i\in A$ such that $F_i^\sim\rightarrow\cF|_{D(g_i)}$, where $F_i:=\Gamma(D(g_i),\cF)$, is an isomorphism for each $i$.
	Let $L_i:=(M\otimes_AA_{g_i}^\dagger)/F_i$.
	Then we have an isomorphism $L_i^\sim\xrightarrow{\cong}(M^\sim/\cF)|_{D(g_i)}$.
	For any $f\in A$, we have the following commutative diagram:
	\begin{equation}\label{eq: diag Li}
	\xymatrix{
	\Gamma(\cZ,M^\sim/\cF)\otimes_AA_f^\dagger\ar[r]\ar[d]&
	\bigoplus_i\Gamma(D(g_i),M^\sim/\cF))\otimes_AA_f^\dagger\ar[d]\ar@{=}[r]&
	\bigoplus_iL_i\otimes_{A_{g_i}^\dagger}(A_{g_i}^\dagger\otimes_AA_f^\dagger)\ar[d]\\
	\Gamma(D(f),M^\sim/\cF)\ar[r]&
	\bigoplus_i\Gamma(D(f)\cap D(g_i),M^\sim/\cF)\ar@{=}[r]&
	\bigoplus_i L_i\otimes_{A_{g_i}^\dagger}A_{fg_i}^\dagger.
	}\end{equation}
	Since $A_f^\dagger$ is flat over $A$, the upper horizontal map in the diagram \eqref{eq: diag Li} is injective.
	
	Let $J\subset A$ be an ideal of definition.
	We have equalities $\widehat{A_{g_i}^\dagger\otimes_AA_f^\dagger}=\widehat{A_{fg_i}^\dagger}=\widehat{A_{fg_i}}$ between the $J$-adic completions.
	Thus both of $L_i\otimes_{A_{g_i}^\dagger}(A_{g_i}^\dagger\otimes_AA_f^\dagger)$ and $L_i\otimes_{A_{g_i}^\dagger}A_{fg_i}^\dagger$ are submodules of $L_i\otimes_{A_{g_i}^\dagger}\widehat{A_{fg_i}}$, 
	and it follows that the right vertical map and hence the left vertical map of \eqref{eq: diag Li} are injective.
	
	We also have a commutative diagram
	\begin{equation}\label{eq: diag F tilde}
	\xymatrix{0\ar[r]&
	F\otimes_AA_f^\dagger\ar[r]\ar[d]&
	M\otimes_AA_f^\dagger\ar[r]\ar@{=}[d]&
	\Gamma(\cZ,M^\sim/\cF)\otimes_AA_f^\dagger\ar[d]\\
	0\ar[r]&
	\Gamma(D(f),\cF)\ar[r]&
	\Gamma(D(f),M^\sim)\ar[r]&
	\Gamma(D(f),M^\sim/\cF),
	}\end{equation}
	with exact rows.
	(We again use flatness of $A_f^\dagger$ over $A$.)
	As a consequence of the above argument the right vertical map of \eqref{eq: diag F tilde} is injective, hence the left vertical map is an isomorphism.
	This shows that $F^\sim\rightarrow \cF$ is an isomorphism, as desired.
\end{proof}

\begin{remark}
	For any weak formal scheme $\cZ$, there exists at least one ideal of definition.
		Indeed, take an open covering $\{\cU_i\}_i$ of $\cZ$ given by affine weak formal schemes $\cU_i=\Spwf A_i$.
		Then taking the largest ideal of definition on each $A_i$, they glue and give an ideal of definition $\cJ$ of $\cZ$.
		
		Even if $\cZ$ is affine, it is in general not known whether a coherent sheaf $\cF$ on $\cZ$ comes from a finite module.
		However, by Proposition-Definition \ref{prop-def: associated sheaf} any ideal of definition on an affine weak formal scheme comes from a finite module.
\end{remark}

\begin{proposition}\label{prop: conservative}
	The functor $\cZ\mapsto\widehat{\cZ}$ from the category of weak formal schemes with respect to $(R,I)$ to the category of formal schemes is faithful and conservative.
\end{proposition}

\begin{proof}
	It suffices to show that the functor $A\mapsto\widehat{A}$ from the category of pseudo-wcfg algebras with respect to $(R,I)$ to the category of adic rings is faithful and conservative.
	Let $f\colon A\rightarrow B$ be a morphism of pseudo-wcfg algebras and $\widehat{f}\colon\widehat{A}\rightarrow\widehat{B}$ its completion.
	Then we have $f=\widehat{f}|_A$, hence the faithuful-ness is clear.
	Suppose that $\widehat{f}$ is an isomorphism.
	Then $f$ is injective as $A$, since the composite $A\xrightarrow{f}B\hookrightarrow\widehat{B}$ coincides with $A\hookrightarrow\widehat{A}\xrightarrow{\widehat{f}}\widehat{B}$ which is injective.
	Take an ideal of definition $J$ of $A$.
	Then $J\widehat{B}$ gives an ideal of definition of $\widehat{B}$, and we have $A/J=\widehat{A}/J\widehat{A}\cong \widehat{B}/J\widehat{B}=B/JB$.
	Thus $f$ is surjective by \cite[Thm.\,3.2]{MW}.
\end{proof}

\begin{proposition}\label{prop: Hom compare}
	Let $\cZ$ and $\cZ'$ be weak formal schemes with respect to $(R,I)$ and suppose that $\cZ=\Spwf A$ is affine.
	Then the natural map
	\begin{equation}\label{eq: Hom compare}\Hom(\cZ',\cZ)\rightarrow\Hom(A,\Gamma(\cZ',\cO_{\cZ'})),\end{equation}
	where the $\Hom$ sets are taken in the category of weak formal schemes with respect to $(R,I)$ 
	and in the category of topological $R$-algebras, respectively.
	In particular, the association $A\mapsto\Spwf A$ defines a fully faithful functor from the category of pseudo-wcfg algebras with respect to $(R,I)$ to the category of locally topologically ringed spaces over $\Spwf \wh{R}$.
\end{proposition}

\begin{proof}
	This is proved by the same argument as in the case for schemes or formal schemes.
	Indeed, the inverse map is constructed as follows:
	
	Let $\psi\colon A\rightarrow\Gamma(\cZ',\cO_{\cZ'})$ be a continuous homomorphism over $R$.
	For $z\in\cZ'$, the maximal ideal $\frm_z$ of $\cO_{\cZ',z}$ is open by Lemma \ref{lem: local ring}.
	Thus the inverse image of $\frm_z$ under the composition
	\[A\xrightarrow[]{\psi}\Gamma(\cZ',\cO_{\cZ'})\rightarrow\cO_{\cZ',z}\]
	is an open prime ideal of $A$, which defines a point of $\cZ=\Spwf A$ via the identification of Remark \ref{rem: Spwf}\eqref{item: open prime}.
	This gives a map $\cZ'\rightarrow\cZ$, which we denote by $g_\psi$.
	For $f\in A$ we have
	\[g_\psi^{-1}(\Spwf A_f^\dagger)=D(\psi(f)),\]
	which is open in $\cZ'$ by Lemma \ref{lem: principal open}.
	Thus $g_\psi$ is continuous.
	
	To construct a morphism $\cO_\cZ\rightarrow g_{\psi,*}\cO_{\cZ'}$, it suffices to give continuous homomorphisms $A_f^\dagger\rightarrow\Gamma(D(\psi(f)),\cO_{\cZ'})$ for $f\in A$ in a compatible way.
	Let $\{\Spwf B_\lambda\}_{\lambda\in L}$ be an affine open covering of $D(\psi(f))$ and let $\{\Spwf C_{\lambda,\lambda',\mu}\}_{\mu\in M_{\lambda,\lambda'}}$ be an affine open covering of $\Spwf B_\lambda\cap\Spwf B_{\lambda'}$ for $\lambda,\lambda'\in L$.
	For $\lambda\in L$, denote by $\psi_\lambda$ the composition of $\psi$ and the restriction map $\Gamma(\cZ',\cO_{\cZ'})\rightarrow B_\lambda$.
	Since $\psi_\lambda$ is continuous over $R$, there is a commutative diagram
	\[\xymatrix{
	A\ar[r]^{\psi_\lambda}&B_\lambda\\
	R_{[n]}\ar[u]\ar[r]&R_{[n+k]}\ar[u]
	}\]
	where the lower horizontal map is a natural injection and the vertical maps make $A$ and $B_\lambda$ wcfg over $(R_{[n]},I_{[n]})$ and $(R_{[n+k]},I_{[n+k]})$, respectively.
	Then $B_\lambda$ is weakly complete over $(R_{[n]},I_{[n]})$.
	Since $f$ is invertible on $\Spwf B_\lambda$, $\psi_\lambda$ uniquely extends to a homomorphism $A_f\rightarrow B_\lambda$.
	Moreover, by taking the weak completion over $(R_{[n]},I_{[n]})$, this uniquely extends to $A_f^\dagger\rightarrow B_\lambda$.
	Similarly $A\rightarrow C_{\lambda,\lambda',\mu}$ uniquely extends to $A_f^\dagger\rightarrow C_{\lambda,\lambda',\mu}$, which is again continuous.
	By uniqueness the composition
	\[A_f^\dagger\rightarrow \prod_{\lambda}B_\lambda\rightarrow\prod_{\lambda,\lambda',\mu}C_{\lambda,\lambda',\mu},\]
	is zero.
	Thus we obtain a continuous homomorphism
	\[A_f^\dagger\rightarrow\Ker\Bigl(\prod_{\lambda}B_\lambda\rightarrow\prod_{\lambda,\lambda',\mu}C_{\lambda,\lambda',\mu}\Bigr)=\Gamma(D(\psi(f),\cO_{\cZ'}).\]
	This defines a morphism $\cO_\cZ\rightarrow g_{\psi,*}\cO_{\cZ'}$ which we denote by $g_\psi^\sharp$.
	
	By construction $(g_\psi,g_\psi^\sharp)$ is a morphism over $\Spwf\widehat{R}$. Moreover, similarly to the case for schemes, one can easily see that $g_\psi^\sharp$ induces local homomorphisms on stalks and that the association $\psi\mapsto(g_\psi,g_\psi^\sharp)$ is in fact the inverse of \eqref{eq: Hom compare}.
\end{proof}

\begin{definition}
	Let $f\colon\cZ'\rightarrow\cZ$ be a morphism of weak formal schemes.
\begin{enumerate}
	\item We say $f$ is a \textit{closed immersion} if it is a closed immersion of topological spaces, the map $\cO_\cZ\rightarrow f_\ast\cO_{\cZ'}$ is surjective, and moreover its kernel $\cN$ is locally generated by sections.
		In this case, let $\cC_{\cZ'/\cZ}=f^*(\cN/\cN^2)$ and call it the \textit{conormal sheaf} of $f$.
		We say $f$ is a \textit{first order thickening} if $\cN^2=0$.
	\item We say $f$ is an \textit{immersion} if it can be written as a composition of a closed immersion $\cZ'\hookrightarrow \cZ''$ and an open immersion $\cZ''\hookrightarrow \cZ$.
	\item We say $f$ is {\it adic} if there exists an ideal of definition of $\cZ$ such that its inverse image by $f$ is an ideal of definition of $\cZ'$.
		Note that in this case any ideal of definition of $\cZ$ satisfies this property.
\end{enumerate}
\end{definition}

\begin{remark}
	The above definition of a closed immersion conforms to \cite[\href{https://stacks.math.columbia.edu/tag/01C2}{Def.\,01C2}]{stacks}, which is slightly stronger than the standard definition.
	 Here we adapt it in order to ensure that the ideal $\cJ$ is coherent.
	 Indeed, one may easily show that the structure sheaf of a weak formal scheme is coherent by using the exactness of the functor in Proposition-Definition \ref{prop-def: associated sheaf}\eqref{item: finite module sheaf} and the fact that any pseudo-wcfg algebra is noetherian.
	 Since the condition that $\cJ$ is locally generated by sections implies that it is of finite type, we see that $\cJ$ is coherent by \cite[\href{https://stacks.math.columbia.edu/tag/01BY}{Lem.\,01BY}]{stacks}.
	
	 In the definition of a closed immersion, we do not make any assumption concerning the topology. 
	 However, it automatically follows that the topology on $\cO_\cY$ is given by the quotient topology of $\cO_\cZ$ by the following proposition.
\end{remark}

\begin{proposition}\label{prop: closed immersion} $\quad$
	\begin{enumerate}
	\item\label{item: closed affine} Let $j\colon\cY\hookrightarrow\cZ=\Spwf A$ be a closed immersion into an affine weak formal schemes and let $N:=\Gamma(\cZ,\Ker(\cO_\cZ\rightarrow j_\ast\cO_\cY))$.
		Then there is a unique isomorphism $\cY\xrightarrow{\cong}\Spwf A/N$, where $A/N$ is endowed with the quotient topology, such that the natural closed immersion $\Spwf A/N\hookrightarrow\cZ$ is compatible with $j$.
	\item\label{item: closed stable} Closed immersions are stable under base change.
	\end{enumerate}
\end{proposition}

\begin{proof}
	We first prove \eqref{item: closed affine}.
	By Proposition-Definition \ref{prop-def: associated sheaf} \eqref{item: sub or quotient} 
	we have $\Ker(\cO_\cZ\rightarrow j_\ast\cO_\cY)=N^\sim$ and $j_\ast\cO_\cY=(A/N)^\sim$.
	Set $\cY':=\{z\in\cZ\mid (A/N)^\sim_z\neq 0\}$ 
	and let $j'\colon\cY'\hookrightarrow\cZ$ be the natural embedding.
	$\cY'$ has the structure of a locally ringed space by considering $j'^{-1}(A/N)^\sim$ as the structure sheaf.
	Then by \cite[\href{https://stacks.math.columbia.edu/tag/01HO}{Lem.\,01HO}]{stacks}, there is a unique isomorphism $\rho\colon\cY\xrightarrow{\cong}\cY'$ of locally ringed spaces such that $j=j'\circ\rho$.
	Thus it suffices to show that we may identify $\cY'$ with $\Spwf A/N$ and that $\Gamma(\cU,\cO_{\Spwf A/N})\rightarrow\Gamma(\cU,\rho_\ast\cO_{\cY})$ is a homeomorphism for any open subset $\cU\subset\cY'$.
	
	For $z\in\cZ$, let $\frp\subset A$ the open prime ideal corresponding to $z$.
	Then $(A/N)^\sim_z\neq 0$ is equivalent 
	to $N\cO_{\cZ,z}\subset\frp\cO_{\cZ,z}$ 
	and hence to $N\subset\frp$ by Lemma \ref{lem: local ring}.
	Thus $\cY'$ is homeomorphic to the set of open prime ideals of $A/N$ with the Zariski topology, 
	that is $\Spwf A/N$. 
	Clearly $j'^{-1}(A/N)^\sim$ coincides with $\cO_{\Spwf A/N}$, 
	hence $\cY'$ is identified with $\Spwf A/N$.
	
	To show that $\cO_{\Spwf A/N}\rightarrow\rho_\ast\cO_{\cY}$ induces homeomorphisms between sections on open subsets, it suffices to show that the continuous isomorphism $(A/N)_f^\dagger\xrightarrow{\cong}\Gamma(\rho^{-1}(D(f)),\cO_\cY)$ is homeomorphic for any $f\in A/N$.
	Set $B:=(A/N)_f^\dagger$ and $C:=\Gamma(\rho^{-1}(D(f)),\cO_\cY)$.
	Let $J\subset B$ be an ideal of definition.
	As $B\rightarrow C$ is continuous, 
	the topology of $C$ is coarser than the $JC$-adic topology.
	If $c\in C$ is contained in an ideal of definition of $C$, 
	then there exists a homomorphism $R\llbracket s\rrbracket\rightarrow C$, $s\mapsto c$, 
	which induces $(R/I)\llbracket s\rrbracket\rightarrow C/JC$, $s\mapsto \overline{c}$.
	However, we have an isomorphism $B/J\rightarrow C/JC$ 
	and $B/J$ is finitely generated over the field $R/I$.
	Thus we conclude that $\overline{c}$ must be nilpotent.
	This shows that the quotient topology on $C/JC$ is discrete, 
	and the topology on $C$ is finer than the $JC$-adic topology.
	
	Finally, \eqref{item: closed stable} immediately follows from \eqref{item: closed affine}, 
	since the property of closed immersion is local.
\end{proof}

\begin{lemma}\label{lem: fibre coproducts}
	The category of pseudo-wcfg algebras with respect to $(R,I)$  
has fibre coproducts.
	We denote them as $A_1\otimes^\dagger_AA_2$.
\end{lemma}

\begin{proof}
	Let $A \rightarrow  A_i$ for $i = 1,2$ be continuous $R$-linear maps between pseudo-wcfg algebras with respect to $(R,I)$.  
	Take generating systems $f_{i,1},\ldots,f_{i,n_i}$ of ideals of definition of $A_i$ for $i = 1,2$.
	Then the $I_{[n_1 + n_2]}$-adic weak completion of $A_1\otimes_AA_2$ with respect to the map
		\[R_{[n_1 + n_2]} = R[s_1,\ldots,s_{n_1 + n_2}] \rightarrow  A_1\otimes_AA_2,\ s_k\mapsto\begin{cases}f_{1,k}\otimes 1&(1\leq k\leq n_1)\\ 1\otimes f_{2,k-n_1}&(n_1 + 1\leq k\leq n_1 + n_2),\end{cases}\]
	equipped with the $I_{[n_1+n_2]}$-adic topology, is independent of any choices, and gives the fibre coproduct.
\end{proof}

\begin{proposition}
	The category of weak formal schemes with respect to $(R,I)$ 
has fibre products.
\end{proposition}

\begin{proof}
	This follows from Proposition \ref{prop: Hom compare} and Lemma \ref{lem: fibre coproducts}.
\end{proof}

We may define the notion of continuous differentials as in \cite{MW}, which will be used for the study of smoothness of weak formal schemes and the definition of rigid cohomology.

\begin{definition}
	Let $A\rightarrow B$ be a morphism of pseudo-wcfg algebra with respect to $(R,I)$ 
and $J$ be an ideal of definition of $B$. 
	We define the complex of {\it continuous differentials} on $B$ over $A$ by
	\[ \Omega^\bullet_{B/A}: =  \underline{\Omega}^\bullet_{B/A}/\bigcap_{n\geq 0}J^n \underline{\Omega}^\bullet_{B/A},\]
	where $\underline{\Omega}^\bullet_{B/A}$ is the complex of K\"{a}hler differentials of $B$ over $A$.
	This is independent of the choice of an ideal of definition $J$.
\end{definition}

\begin{lemma}\label{lem: wf omega}
	Let $A\rightarrow B$ be a morphism of pseudo-wcfg algebra with respect to $(R,I)$. 
	\begin{enumerate}
	\item\label{item: omega finite} $\Omega^1_{B/A}$ is a finite $B$-module.
	\item\label{item: omega exterior} $\Omega^\bullet_{B/A}=\bigwedge_B\Omega^1_{B/A}$.
	\item\label{item: omega completion} If we let $\wh A$ and $\wh B$ be the completions of $A$ and $B$ with respect to their ideals of definition, then we have a canonical isomorphism
		$\Omega^1_{B/A}\otimes_B \wh B\xrightarrow{\cong}\Omega^1_{\wh B/\wh A}$.
	\item\label{item: omega local} For any $f\in B$ we have a canonical isomorphism
		$\Omega^1_{B/A}\otimes_BB_f^\dagger\xrightarrow{\cong}\Omega^1_{B_f^\dagger/A}$.
	\end{enumerate}
\end{lemma}

\begin{proof}
	\eqref{item: omega finite} and \eqref{item: omega exterior} follow by the same arguments to \cite[Thm.\ 4.5]{MW} and \cite[Thm.\ 1.1]{MW2}.
	\eqref{item: omega completion} is \cite[p.36, Rem.]{AJP2005}.
	Noting that $\widehat{B_f^\dagger}= \widehat{B_f}$, we have isomorphisms
	\begin{align*}
	\Omega^1_{B/A}\otimes_BB_f^\dagger\otimes_{B_f^\dagger}\wh{B_f^\dagger}&=\Omega^1_{B/A}\otimes_B\widehat{B_f}=\Omega^1_{B/A}\otimes_B\wh{B}\otimes_{\wh{B}}\wh{B_f}\\
	&\xrightarrow{\cong}\Omega^1_{\wh{B}/\wh{A}}\otimes_{\wh{B}}\wh{B_f}\xrightarrow{\cong}\Omega^1_{\wh{B_f}/\wh{A}}\xleftarrow{\cong}\Omega^1_{B_f^\dagger/A}\otimes_{B_f^\dagger}\wh{B_f^\dagger},
	\end{align*}
	where the third and fifth isomorphisms follow by \eqref{item: omega completion}, and the fourth isomorphism is \cite[Lem.\ 2.1.10 (2)]{AJP2005}.
	This implies \eqref{item: omega local}, since $\wh{B_f^\dagger}$ is faithfully flat over $B_f^\dagger$ by the following Lemma \ref{lem: faithfully flat}.
\end{proof}

\begin{lemma}\label{lem: faithfully flat}
	The completion of a pseudo-wcfg algebra $A$ with respect to an ideal of definition is faithfully flat over $A$.
\end{lemma}

\begin{proof}
	This follows immediately from \cite[Theorem 1.6]{MW} and \cite[Theorem 8.14]{Ma}.
\end{proof}

\begin{proposition-definition}\label{prop-def: Omega}
	For a morphism of weak formal schemes $f\colon\cZ'\rightarrow\cZ$ with respect to $(R,I)$, 
there exists a unique pair $(\Omega^1_{\cZ'/\cZ},d_{\cZ'/\cZ})$ of a coherent sheaf 
$\Omega^1_{\cZ'/\cZ}$ on $\cZ'$ and a $\cZ$-derivation $d_{\cZ'/\cZ}\colon\cO_{\cZ'}\rightarrow\Omega^1_{\cZ'/\cZ}$, 
such that for any affine open subsets $\Spwf A\subset \cZ$ and $\Spwf B\subset f^{-1}(\Spwf A)$ there is an isomorphism
	\[\Omega^1_{\cZ'/\cZ}|_{\Spwf B}\cong (\Omega^1_{B/A})^\sim\]
	which is compatible with derivations.
	We define $\Omega^\bullet_{\cZ'/\cZ}:=\bigwedge_{\cO_{\cZ'}}\Omega^1_{\cZ'/\cZ}$ and call this the complex of sheaves of \textit{continuous differentials} of $\cZ'$ over $\cZ$.
\end{proposition-definition}

\begin{proof}
	For any $x\in \cZ'$, take an affine open neighbourhood $\cV=\Spwf A\subset\cZ$ of $f(x)$, 
and set $\Omega^1_{f,x}:=\varinjlim_{x\in\Spwf B\subset f^{-1}(\cV)}\Omega^1_{B/A}$, where the limit runs through all affine neighbourhoods of $x$ in $f^{-1}(\cV)$.
	Note that $\Omega^1_{f,x}$ is independent of the choice of $\cV$.
	For any open subset $\cU\subset\cZ'$ we let $\Gamma(\cU,\Omega^1_{\cZ'/\cZ})$  be the set of all maps $s\colon\cU\rightarrow\coprod_{x\in\cU}\Omega^1_{f,x}$ such that for any $x\in\cU$ there exist affine open neighbourhoods $\Spwf A\subset\cZ$ of $f(x)$ and $\Spwf B\subset f^{-1}(\Spwf A)$ of $x$ and $\omega\in\Omega^1_{B/A}$ such that for any $y\in\Spwf B$ the image of $\omega$ under the canonical map $\Omega^1_{B/A}\rightarrow\Omega^1_{f,y}$ coincides with $s(y)$.
	For $\alpha\in\Gamma(\cU,\cO_{\cZ'})$ and $x\in\cU$, take  affine open neighbourhoods $\Spwf A\subset\cZ$ of $f(x)$ and $\Spwf B\subset f^{-1}(\Spwf A)$ of $x$, and define $(d_{\cZ'/\cZ}\alpha)_x$ to be the image of $\alpha|_{\Spwf B}$ under the maps $B\xrightarrow{d_{B/A}}\Omega^1_{B/A}\rightarrow\Omega^1_{f,x}$.
	We obtain a derivation $d_{\cZ'/\cZ}\colon\Gamma(\cU,\cO_{\cZ'})\rightarrow\Gamma(\cU,\Omega^1_{\cZ'/\cZ})$ given by $\alpha\mapsto(x\mapsto(d_{\cZ'/\cZ}\alpha)_x)$.
	By construction $\Omega^1_{\cZ'/\cZ}$ is a sheaf with a $\cZ$-derivation $d_{\cZ'/\cZ}$, and the stalk of $\Omega^1_{\cZ'/\cZ}$ at a point $x\in\cZ'$ is canonically isomorphic to $\Omega^1_{f,x}$.
	
	For affine open subsets $\Spwf A\subset \cZ$ and $\Spwf B\subset f^{-1}(\Spwf A)$, the natural homomorphism $\Omega^1_{B/A}\rightarrow\Gamma(\Spwf B,\Omega^1_{\cZ'/\cZ})$ induces a morphism of sheaves $(\Omega^1_{B/A})^\sim\rightarrow\Omega^1_{\cZ'/\cZ}|_{\Spwf B}$.
	This is an isomorphism as the induced morphisms on stalks are isomorphisms by Lemma \ref{lem: wf omega} \eqref{item: omega local}.
\end{proof}

\begin{proposition}\label{prop: Omega affine}
	Let $\cZ$ be a weak formal scheme with respect to $(R,I)$ and let $\cZ':=\cZ\times_{\Spwf \wh{R}}\Spwf \wh R\llbracket s_1,\ldots,s_n\rrbracket[t_1,\ldots,t_k]^\dagger$.
	Then $\Omega^1_{\cZ'/\cZ}$ is a free $\cO_{\cZ'}$-module generated by the global sections $ds_1,\ldots,ds_n,dt_1,\ldots,dt_k\in\Gamma(\cZ',\Omega^1_{\cZ'/\cZ})$.
\end{proposition}

\begin{proof}
	It is clear that there exists a homomorphism $\bigoplus_{i=1}^n\cO_{\cZ'}ds_i\oplus\bigoplus_{j=1}^k\cO_{\cZ'}dt_j\rightarrow\Omega^1_{\cZ'/\cZ}$.
	To show that this is an isomorphism, 
	we may suppose that $\cZ$ is affine.
	Then the assertion follows from a similar result for formal schemes \cite[Ex.\,2.1.7]{AJP2005} and Lemma \ref{lem: faithfully flat}.
\end{proof}

Next we give the definition of the weak completion of a weak formal scheme along a locally closed weak formal subscheme.
We first explain the local construction.

\begin{definition}\label{def: weak formal completion along closed}
Let $A$ be a pseudo-wcfg algebra with respect to $(R,I)$ 
with a representation $\rho\colon R_{[n]}[t_1,\ldots,t_k]^\dagger \rightarrow  A$.
Let $J\subset A$ be an ideal (not necessarily an ideal of definition) with a generating system $f_1,\ldots,f_m\in A$.
Define the {\it weak completion} of $A$ along $J$ as the $I_{[n + m]}$-adic weak completion of $A$ with respect to the map $R_{[n + m]} = R_{[n]}[s_{n + 1},\ldots,s_{n + m}] \rightarrow  A$ given by $s_i\mapsto\rho(s_i)$ for $1\leq i\leq n$ and $s_{n+j}\mapsto f_j$ for $1\leq j\leq m$, equipped with the $I_{[n+m]}$-adic topology.
In the following statement, we denote it by $A'$. 
\end{definition}

\begin{proposition}\label{prop: weak completion}
	Let $A$ and $J$ be as above.
	The weak completion $A'$ of $A$ along $J$ is pseudo-wcfg with respect to $(R,I)$ 
and satisfies the following universal property:
	Let $f\colon A\rightarrow B$ be a homomorphism of pseudo-wcfg algebras with respect to $(R,I)$ and suppose that $f$ is continuous when we consider the $J$-adic topology on $A$.
	Then $f$ uniquely factors through a continuous homomorphism $A'\rightarrow B$.
	Moreover, $A'$ depends only on the topology on $A$ defined by $J$.
\end{proposition}

\begin{proof}  
	We first prove that $A'$ is pseudo-wcfg with respect to $(R,I)$.
	The $I_{[n+m]}$-adic weak completion of the map
		\[R_{[n + m]}[t_1,\ldots,t_k] \rightarrow  A,\ t_j\mapsto\rho(t_j),\]
	is a map $\rho'\colon R_{[n + m]}[t_1,\ldots,t_k]^\dagger \rightarrow  A'$.
	We want to show that $\rho'$ is surjective.
	Any element $\xi\in A'$ has a presentation of the form
		\[\xi = \sum_{i\geq 0}P_i(x_1,\ldots,x_r),\]
	where $x_1,\ldots,x_r\in A$, $P_i(X_1,\ldots,X_r)\in I_{[n + m]}^i\cdot R_{[n + m]}[X_1,\ldots,X_r]$, and there exists a constant $c>0$ such that
		\[c(i + 1)\geq \mathrm{deg}P_i\]
	for any $j\geq 0$.
	Since $\rho$ is surjective, each $x_j$ can be written as
		\[x_j = \sum_{\ell\geq 0}\rho (Q_{j,\ell}(t_1,\ldots,t_k)),\]
	 where $Q_{j,\ell}(t_1,\ldots,t_k)\in I^\ell_n\cdot R_{[n]}[t_1,\ldots,t_k]$ and there exists a constant $d>0$ such that
		\[d(\ell + 1)\geq\mathrm{deg}Q_{j,\ell}\]
	for any $j = 1,\ldots,r$ and $\ell\geq 0$. 
	If we write 
		\begin{align*}
		P_i(X_1,\ldots,X_r) = \sum_{\substack{\alpha\in\bbN^r\\ \lvert\alpha\rvert\leq\mathrm{deg}P_i}}g_{i,\alpha}X_1^{\alpha_1}\cdots X_r^{\alpha_r}, & & g_{i,\alpha}\in I_{[n + m]}^i,\\
		Q_{j,\ell}(t_1,\ldots,t_k) = \sum_{\substack{\beta\in\bbN^k\\ \lvert\beta\rvert\leq\mathrm{deg}Q_{j,\ell}}}a_{j,\ell,\beta}t_1^{\beta_1}\cdots t_k^{\beta_k}, & &  a_{j,\ell,\beta}\in I_{[n]}^\ell,
		\end{align*}
	we obtain an expression
		\[\xi = \sum_{q\geq 0}\sum_{\lambda\in\Lambda_q}\rho (F_\lambda(t_1,\ldots,t_k)),\]
	where $\Lambda_q$ is the set of all systems $ \left(i,\alpha,(\ell_{j,\nu})_{\substack{1\leq j\leq r\\ 1\leq \nu\leq \alpha_j}}\right)$ such that
		\begin{align*}
		i\geq 0, & & \alpha\in\bbN^r, & & \lvert\alpha\rvert\leq\mathrm{deg}P_i,& &\ell_{j,\nu}\geq 0, & & i + \sum_{1\leq j\leq r}\sum_{1\leq \nu\leq \alpha_j}\ell_{j,\nu} = q,
		\end{align*}
	and the polynomial $F_\lambda$ is defined by
		\begin{eqnarray*}
		F_\lambda(t_1,\ldots,t_k): = g_{i,\alpha}\prod_{1\leq j\leq r} \prod_{1\leq\nu\leq\alpha_j} \sum_{\substack{\beta\in\bbN^k\\ \lvert\beta\rvert\leq\mathrm{deg}Q_{j,\ell_{j,\nu}}}}a_{j,\ell_{j,\nu},\beta}t_1^{\beta_1}\cdots t_k^{\beta_k}
		\in I_{[n + m]}^q\cdot R_{[n + m]}[t_1,\ldots,t_k].
		\end{eqnarray*}
	We note that $\Lambda_q$ is a finite set.
	Since we have
		\begin{eqnarray*}
		\mathrm{deg}F_\lambda&\leq&\sum_{1\leq j\leq r}\sum_{1\leq \nu\leq \alpha_j}\mathrm{deg}Q_{j,\ell_{j,\nu}}\leq\sum_{1\leq j\leq r}\sum_{1\leq \nu\leq \alpha_j}d(\ell_{j,\nu} + 1) = d(q-i + \lvert\alpha\rvert)\\
		&\leq&d \left(q-i + c(i + 1)\right)\leq d \left(q + c(q + 1)\right),
		\end{eqnarray*}
	there exists a constant $c'>0$ such that $\mathrm{deg}F_\lambda\leq c'(q + 1)$.
	This shows the surjectivity of $\rho'$.
	
	Next we prove the universal property.
	Let $f\colon A\rightarrow B$ be a homomorphism as in the statement.
	By the assumption on $f$, we can take an ideal of definition of $B$ which contains the image of $I_{[n+m]}$ under the composite $R_{[n+m]}\rightarrow A\rightarrow B$.
	Thus $B$ is $I_{[n+m]}$-adically weakly complete.
	By \cite[Thm.\,1.5]{MW}, $f$ uniquely extends to a continuous morphism $A'\rightarrow B$.
	
	Finally, the independence from the choice of an ideal of definition and its generating system immediately follows from the universal property.
\end{proof}

\begin{proposition-definition}\label{prop-def: weak completion}
	The natural inclusion from the category of homeomorphic closed immersions of weak formal schemes to the category of immersions of weak formal schemes has a right adjoint of the form $(Z\hookrightarrow\cZ)\mapsto (Z\hookrightarrow\cZ_Z)$.
	We call $\cZ_Z$ the \textit{weak completion} of $\cZ$ along $Z$.
\end{proposition-definition}

\begin{proof}
	We will show that, for any closed immersion $Z\hookrightarrow\cZ$ we may associate in a functorial way a factorization $Z\hookrightarrow\cZ_Z\rightarrow\cZ$ satisfying the following:
	\begin{enumerate}
	\item\label{item: weak completion 1} $Z\hookrightarrow\cZ_Z$ is a homeomorphic closed immersion,
	\item\label{item: weak completion 2} For any commutative diagram
		\[\xymatrix{
		Y\ar[d]\ar@{^(->}[rr]&&\cY\ar[d]\\
		Z\ar@{^(->}[r]&\cZ_Z\ar[r]&\cZ
		}\]
		where the upper horizontal morphism is a homeomorphic closed immersion, there exists a unique morphism $\cY\rightarrow\cZ_Z$ which keeps the diagram commutative.
	\end{enumerate}
	Let $i\colon Z\hookrightarrow\cZ$ be an immersion with a factorization into a closed immersion $Z\hookrightarrow\cZ'$ and an open immersion $\cZ'\rightarrow\cZ$.
	If $\cZ'|_Z$ satisfying the above conditions for $Z\hookrightarrow\cZ'$ exists, then it also satisfies the above conditions for $Z\hookrightarrow\cZ$, and hence is independent of the choice the factorization.
	Therefore we may assume that $i$ is a closed immersion.
	
	Moreover, if $\cZ_Z$ satisfying the above conditions exists, 
	then for any weak formal open subscheme $\cU\subset\cZ$, 
	the factorization $Z\times_\cZ\cU\hookrightarrow \cZ_Z\times_\cZ\cU\rightarrow\cU$   
	also satisfies the same conditions.
	Therefore, in order to give a construction of $\cZ_Z$, we may work locally.
		
	Thus let $\cZ=\Spwf A$ be affine and $i$ defined by an ideal $J\subset A$.
	Let $B$ be the weak completion of $A$ along $J$.
	We will show that the factorization $\Spwf A/J\hookrightarrow\Spwf B\rightarrow\Spwf A$ satisfies the above conditions.
	
	The condition \eqref{item: weak completion 1} is clear by construction.	
	To see \eqref{item: weak completion 2}, we may suppose that $Y=\Spwf D$ and $\cY=\Spwf C$ are affine.
	Let $M$ be an ideal of definition of $C$ and $J'$ the kernel of the surjection $C\rightarrow D$.
	Since $Y\hookrightarrow\cY$ is homeomorphic, $J'(C/M)$ is nilpotent.
	Thus we have $J'^k\subset M$ for some integer $k$, hence the weak completion of $C$ along $J'$ is $C$ itself.
	Since we have $JC\subset J'$, the map $A\rightarrow C$ uniquely extends to a map $B\rightarrow C$ by the functoriality of weak completion.
	This finishes the proof.
\end{proof}

From now on, let $V$ be a complete discrete valuation ring of characteristic $(0,p)$ with maximal ideal $\frm$ and fraction field $K$.
Next we provide the construction of dagger spaces associated to weak formal schemes with respect to $(V,\frm)$.
This will be crucial for the computation of (log) rigid cohomology in later sections. 
Indeed, Gro\ss{}e-Kl\"onne introduced dagger spaces as overconvergent analogue to rigid spaces \cite{GK1} which allowed him to interpret rigid cohomology in terms of their de Rham cohomology. 
A dagger space can be obtained as the generic fibre of a weak formal scheme \cite{LM}. 
More precisely, let $A$ be a wcfg algebra  with respect to $(V,\frm)$.
Then by \cite[\S4.2]{LM}, $A\otimes_V K$ is a $K$-dagger algebra in the sense of \cite[\S1.2]{GK1} and one may consider the associated dagger space $\Sp(A\otimes_V K)$ which is the set of maximal ideals endowed with a $G$-topology induced from the $G$-topology on the rigid space obtained by completing \cite[\S2.7]{GK1}.
In the following, we extend the construction of the associated dagger space to pseudo-wcfg algebras and weak formal schemes, following \cite[(0.2.6)]{Ber}.

\begin{lemma}\label{lem: shrink by J}
	Let $A$ be a pseudo-wcfg algebra with respect to $(V,\frm)$ and $J\subset A$ an ideal.
	\begin{enumerate}
	\item\label{item: universal} There exists a $p$-torsion free pseudo-wcfg algebra $A[\frac{J}p]^\dagger$ and a morphism $\iota\colon A\rightarrow A[\frac{J}p]^\dagger$ satisfying $\iota(J)\subset pA[\frac{J}p]^\dagger$, which are universal in the following sense:
	Any morphism $\tau\colon A\rightarrow B$ of pseudo-wcfg algebras with $B$ $p$-torsion free and $\tau(J)\subset pB$ factors uniquely through a morphism $A[\frac{J}p]^\dagger\rightarrow B$.
	\item\label{item: localization} Let $f\in A$ and $J_f:=JA_f^\dagger$.
		Then there exists a unique isomorphism $\left(A[\frac{J}p]^\dagger\right)_f^\dagger\cong A_f^\dagger[\frac{J_f}p]^\dagger$ of $A$-algebras.
	\item\label{item: wcfg} If $J$ is an ideal of definition of $A$, then $A[\frac{J}p]^\dagger$ is wcfg.
	\end{enumerate}
\end{lemma}

\begin{proof}
	Take elements $a_1,\ldots,a_n\in A$ generating $J$.
	Then
	\[A\left[\frac{J}p\right]^\dagger:=\left(A[x_1,\ldots,x_n]^\dagger/(px_1-a_1,\ldots, px_n-a_n)\right)/(\text{$p$-torsion})\]
	and the natural inclusion $\iota:A\rightarrow A[\frac Jp]^\dagger$ satisfy the universal property.
	Indeed, let $\tau\colon A\rightarrow B$ be a morphism as in the statement.
	By the assumption on $B$ and $\tau$, for any $i=1,\ldots,n$, there exists a unique element $b_i\in B$ such that $\iota(a_i)=pb_i$.
	Thus $\tau$ extends uniquely to a morphism $A[\frac Jp]^\dagger\rightarrow B$, which maps $x_i$ to $b_i$.
	This shows \eqref{item: universal}.
	
	For any $f\in A$, $\left(A[\frac{J}p]^\dagger\right)_f^\dagger$ is flat over $A[\frac Jp]^\dagger$ \cite[\S 2 Lem.\,2]{Me}, and hence $p$-torsion free.
	Note also that the natural map $A\rightarrow\left(A[\frac{J}p]^\dagger\right)_f^\dagger$ is universal with respect to morphisms $A[\frac Jp]^\dagger\rightarrow C$ with $f$ invertible in $C$.
	These properties together with the universality of $A_f^\dagger[\frac {J_f}p]^\dagger$ as in \eqref{item: universal} imply \eqref{item: localization}.
	
	When $J$ is an ideal of definition of $A$, then by the construction above we see that $A[\frac Jp]^\dagger/pA[\frac Jp]^\dagger$ is of finite type over $V/pV$, hence \eqref{item: wcfg} follows.
\end{proof}

\begin{proposition}
	Let $A$ be a pseudo-wcfg algebra  with respect to $(V,\frm)$ with an ideal of definition $J$.
	For any $k\geq 0$, the morphism $A[\frac{J^{k+1}}p]^\dagger\rightarrow A[\frac{J^k}p]^\dagger$ given by the universal property induces an open immersion
	\[\Sp\left(A\left[\frac{J^k}{p}\right]^\dagger\otimes_VK\right)\hookrightarrow \Sp\left(A\left[\frac{J^{k+1}}{p}\right]^\dagger\otimes_VK\right)\]
	of a Weierstraß domain.
	Moreover, the dagger space
		\[\bigcup_{k\geq 0}\Sp \left(A\left[\frac{J^k}{p}\right]^\dagger\otimes_VK\right)\]
	is independent of the choice of an ideal of definition $J$.
\end{proposition}

\begin{proof}
	Let $\widetilde{J}:=JA[\frac{J^{k+1}}p]^\dagger$.
	Then by the universal property as in Lemma \ref{lem: shrink by J} we see that $A[\frac{J^k}p]^\dagger=(A[\frac{J^{k+1}}p]^\dagger)[\frac{\widetilde{J}^k}p]^\dagger$.
	Thus, for elements $f_1,\ldots,f_m\in A[\frac{J^{k+1}}p]^\dagger$ generating $\widetilde{J}^k$, we have
	\[A\left[\frac{J^k}p\right]^\dagger\otimes_VK=\left(A\left[\frac{J^{k+1}}p\right]^\dagger\otimes_VK\right)\left[x_1,\ldots,x_m\right]^\dagger/\left(x_1-\frac{f_1}p,\ldots,x_n-\frac{f_m}p\right),\]
	which shows the first assertion.
	
	To show the second assertion, let $J'$ be another ideal of definition.
	Then there exists an integer $r\geq 0$ such that $J^r\subset J'$.
	Thus there are canonical morphisms $A[\frac{J^{rk}}p]^\dagger\rightarrow A[\frac{J'^k}p]^\dagger$ for all $k$ compatible with each other.
	They induce a morphism
	\[\bigcup_{k\geq 0}\Sp \left(A\left[\frac{J'^k}{p}\right]^\dagger\otimes_VK\right)\rightarrow\bigcup_{k\geq 0}\Sp \left(A\left[\frac{J^k}{p}\right]^\dagger\otimes_VK\right),\]
	and an opposite morphism is constructed similarly.
	Again by the universal property of Lemma \ref{lem: shrink by J}, we see that they are inverse to each other.
\end{proof}

For a wcfg algebra $B$ with respect to $(V,\frm)$, the specialization map $\mathrm{sp}\colon\Sp(B\otimes_VK)\rightarrow \Spwf B$ is defined as in the case of formal schemes.
Let $A$ be a pseudo-wcfg algebra $A$ with respect to $(V,\frm)$ with an ideal of definition $J$.
Noting that $A[\frac{J^k}p]^\dagger$ is wcfg by Lemma \ref{lem: shrink by J}\eqref{item: wcfg}, we see that the composites
\[\Sp\left(A\left[\frac{J^k}p\right]\otimes_VK\right)\xrightarrow{\mathrm{sp}}\Spwf A\left[\frac{J^k}p\right]^\dagger\rightarrow\Spwf A\]
for all $k$ together induce a well-defined morphism
\begin{equation}\label{eq: sp for affine}
\mathrm{sp}\colon \bigcup_{k\geq 0}\Sp \left(A\left[\frac{J^k}{p}\right]^\dagger\otimes_VK\right)\rightarrow\Spwf A.
\end{equation}
For any $f\in A$, we have
\[\mathrm{sp}^{-1}(\Spwf A_f^\dagger)=\bigcup_{k\geq 0}\Sp \left(A_f^\dagger\left[\frac{J^k_f}{p}\right]^\dagger\otimes_VK\right)\]
where we set $J_f=JA_f^\dagger$.

\begin{definition}\label{def: gen fib}
	Let $\cZ$ be a weak formal scheme with respect to $(V,\frm)$. 
	Take an ideal of definition $\cJ$ of $\cZ$ and an affine covering $\{\Spwf A_\lambda\}_{\lambda}$ of $\cZ$, and let $J_\lambda=\Gamma(\Spwf A_\lambda,\cJ)$.
	Then for each $k\geq 0$, the weak formal schemes $\Spwf A_\lambda[\frac{J_\lambda^k}{p}]^\dagger$ naturally glue due to Lemma \ref{lem: shrink by J}\eqref{item: localization}.
	We denote by $\cZ[\frac{\cJ^k}{p}]^\dagger$ the resulting weak formal scheme.
	Similarly, the dagger spaces $\bigcup_{k\geq 0}\Sp (A_\lambda[\frac{J_\lambda^k}{p}]^\dagger\otimes_RK)$ naturally glue.
	We call the resulting dagger space $\frZ$ the {\it dagger space associated to} $\cZ$.
	The maps \eqref{eq: sp for affine} for all $A_\lambda$ together induce a map $\mathrm{sp}\colon \frZ\rightarrow\cZ$, which we call the \textit{specialization map}.
\end{definition}

%%%%%%%%%%%%%%%%
\subsection{Smoothness and \'{e}taleness}
\label{subsection: smoothness}
%%%%%%%%%%%%%%%%

In this subsection, we go back to the general situation: We consider weak formal schemes with respect to a pair $(R,I)$, consisting of a noetherian ring and an ideal.

In order to discuss lifting properties of a morphism of weak formal schemes, we will often consider a commutative diagram of weak formal schemes
\begin{equation}\label{eq: test diagram}\xymatrix{
\cY'\ar[r]^-i\ar[d]^-{\theta'}&\cY\ar[d]^-\theta\\
\cZ'\ar[r]^-f&\cZ
}\end{equation}
where $i$ is a homeomorphic closed immersion defined by an ideal $\cN\subset\cO_\cY$.

As in the classical case, the following lemma will be useful for the study of smoothness of weak formal schemes.

\begin{lemma}\label{lem: pseudo torsor}
	Consider the commutative diagram \eqref{eq: test diagram} and suppose that $i$ is a first order thickening.
	Let $\cF$ be the sheaf of sets on $\cY$ defined by $\Gamma(\cU,\cF):=\{g\colon \cU\rightarrow\cZ'\mid f\circ g=\theta|_{\cU},\ g\circ i|_{\cU'}=\theta'|_{\cU'}\}$ for each open weak formal subscheme $\cU\subset\cY$,
	where $\cU'$ denotes $\cU\times_\cY\cY'$.
	Then the sheaf $\cH:=\cH om_{\cO_{\cY'}}(\theta'^*\Omega^1_{\cZ'/\cZ},\cC_{\cY'/\cY})$ naturally acts on $\cF$, and it turns $\cF$ into a pseudo-$\cH$-torsor.
	This means that the action of $\Gamma(\cU,\cH)$ is simply transitive if $\Gamma(\cU,\cF)$ is non-empty.
\end{lemma}

\begin{proof}
	The lemma is proved in a similar manner to the classical case.
	We first 
	give a description of the action.
	For sections $g\in\Gamma(\cU,\cF)$ and $\varphi\in\Gamma(\cU,\cH)$, the morphism $g^\varphi\colon\cU\rightarrow\cZ'$ is given by the continuous map $g$ between the underlying topological spaces and the homomorphism of topological rings $g^{-1}\cO_{\cZ'}\rightarrow\cO_\cY$ defined by $\alpha\mapsto g^*\alpha+\varphi(\theta'^*(d\alpha))$ where we identify a section of $\cC_{\cY'/\cY}$ with that of $\cN\subset\cO_\cY$.
	
	Conversely, suppose that two sections $g,h\in\Gamma(\cU,\cF)$ are given.
	To prove that there is a unique section $\varphi\in\Gamma(\cU,\cH)$ satisfying $g^\varphi=h$, 
	we may assume that $\cZ$, $\cZ'$, $\cY$, and $\cY'$ are affine and that  $\cU=\cY$.
	Then the diagram \eqref{eq: test diagram} corresponds to a commutative diagram of pseudo-wcfg algebras
	\[\xymatrix{
	D&C\ar[l]_-{i^\sharp}\\
	B\ar[u]_-{\theta'^\sharp}&A,\ar[l]_-{f^\sharp}\ar[u]_-{\theta^\sharp}
	}\]
	and $g$ and $h$ give morphisms $g^\sharp$ and $h^\sharp$ from $B$ to $C$. 
	
	Let $J\subset B$ and $J'\subset C$ be ideals of definition.
	Then $J'D$ is an ideal of definition of $D$ by Proposition \ref{prop: closed immersion} \eqref{item: closed affine}.
	Because $\theta'^\sharp$ is continuous we have $J^nD\subset J'D$ for some $n\geq 1$.
	Since $N:=\Ker i^\sharp$ is $J'$-adically separated, we see that $N$ regarded as a $B$-module via $\theta'^\sharp$   is $J$-adically separated.
	Thus the $A$-linear derivation $h^\sharp-g^\sharp\colon B\rightarrow N$ uniquely factors through a $B$-linear homomorphism 
	$\varphi\colon\Omega^1_{B/A}\rightarrow N$.
	Since we have a canonical isomorphism $\Hom_B(\Omega^1_{B/A},N)\cong \Hom_D(\Omega^1_{B/A}\otimes D,N)=\Gamma(\cU,\cH)$, $\varphi$ gives an element of $\Gamma(\cU,\cH)$.
	It is easy to see that this $\varphi$ satisfies $g^\varphi=h$.
\end{proof}

We define smoothness and \'{e}taleness of morphisms of weak formal schemes by the infinitesimal lifting property for first order thickenings.
However it would be convenient for discussions concerning rigid cohomology to require stronger lifting properties, as in the following definition:

\begin{definition}\label{def: non-log-lifting-properties}
	Let $f\colon \cZ' \rightarrow \cZ$ be a morphism of weak formal schemes with respect to $(R,I)$.
	\begin{enumerate}
	\item We say $f$ is {\it smooth} if for any commutative diagram as in \eqref{eq: test diagram} where $i$ is a first order thickening, there exists locally on $\cY$ a morphism $g\colon\cY\rightarrow\cZ'$ such that $g\circ i=\theta'$ and $f\circ g=\theta$.
	\item We say $f$ is {\it \'{e}tale} if for any commutative diagram as in \eqref{eq: test diagram} where $i$ is a first order thickening, there exists a unique morphism $g\colon\cY\rightarrow\cZ'$ such that $g\circ i=\theta'$ and $f\circ g=\theta$.
	\item We say $f$ is {\it strongly smooth} if it is smooth and for any commutative diagram as in \eqref{eq: test diagram} where $\cY$ is adic over $\Spwf R$, there exists locally on $\cY$ a morphism $g\colon\cY\rightarrow\cZ'$ such that $g\circ i=\theta'$ and $f\circ g=\theta$.
	\item We say $f$ is  {\it strongly \'{e}tale} if it is \'{e}tale and for any commutative diagram as in \eqref{eq: test diagram} where $\cY$ is adic over $\Spwf R$, there exists a unique morphism $g\colon\cY\rightarrow\cZ'$ such that $g\circ i=\theta'$ and $f\circ g=\theta$.
	\end{enumerate}
\end{definition}

\begin{remark}\label{rem: formal smoothness}
	Formal smoothness and formal \'{e}taleness of a morphism $f\colon\cZ'\rightarrow\cZ$ of locally noetherian formal schemes was defined by infinitesimal lifting properties for first order thickenings of schemes (\cite[Def.\ 2.1]{AJP}).
	If $f$ is formally smooth (resp.\ formally \'{e}tale), it automatically satisfies the respective infinitesimal lifting property for any homeomorphic closed immersions of locally noetherian formal schemes (\cite[Prop.\ 2.3]{AJP}, \cite[Prop.\ 3.4.1]{AJP2005}).
	Such an $f$ is called smooth (resp.\ \'{e}tale) if it is moreover of pseudo-finite type in the sense of \cite[Def.\ 1.5]{AJP}.
	
	Note that, the completion of any morphism of weak formal schemes is pseudo-finite type, because we consider only morphisms between weak formal schemes over the same base $(R,I)$.
\end{remark}

The statements in the following proposition are immediate.

\begin{proposition}\label{prop: easy properties}
	For morphisms of weak formal schemes with respect to $(R,I)$, we have the following:
	\begin{enumerate}
	\item\label{item: easy property 1} An open immersion is strongly \'{e}tale.
	\item\label{item: easy property 2} Composition of smooth (resp.\ \'{e}tale, strongly smooth, strongly \'{e}tale) morphisms is smooth (resp.\ \'{e}tale, strongly smooth, strongly \'{e}tale).
	\item\label{item: easy property 3} Smoothness, \'{e}taleness, strong smoothness, and strong \'{e}taleness are stable under base change.
	\item\label{item: easy property 4} If $f$ is \'{e}tale and $f\circ g$ is smooth (resp.\ \'{e}tale), then $g$ is smooth (resp.\ \'{e}tale).
	If $f$ is strongly \'{e}tale and $f\circ g$ is strongly smooth (resp.\ strongly \'{e}tale), then $g$ is strongly smooth (resp.\ strongly \'{e}tale).
	\end{enumerate}
\end{proposition}

\begin{proposition}\label{prop: fundamental properties of Omega} $\quad$
	\begin{enumerate}
	\item\label{item: Omega locally free} If $f\colon\cZ'\rightarrow\cZ$ is a smooth (resp.\ \'{e}tale) morphism of weak formal schemes, then $\Omega^1_{\cZ'/\cZ}$ is locally free (resp.\ zero).
	\item\label{item: Omega exact sequence} For any morphisms of weak formal schemes $f\colon\cZ''\rightarrow\cZ'$ and $g\colon\cZ'\rightarrow\cZ$, we have an exact sequence
		\begin{equation}\label{eq: fund seq}
		f^*\Omega^1_{\cZ'/\cZ}\rightarrow\Omega^1_{\cZ''/\cZ}\rightarrow\Omega^1_{\cZ''/\cZ'}\rightarrow 0.\end{equation}
		Moreover we have the following:
		\begin{enumerate}
		\item\label{item: split 1} If $f$ is smooth, then the first map of \eqref{eq: fund seq} is injective and locally split. If $f$ is \'{e}tale, then the first map of \eqref{eq: fund seq} is an isomorphism.
		\item\label{item: split 2} Suppose that  $g\circ f$ is smooth.
			If the first map of \eqref{eq: fund seq} is injective and locally split, then $f$ is smooth.
		If the first map of \eqref{eq: fund seq} is an isomorphism, then $f$ is \'{e}tale.
		\end{enumerate}
	\item\label{item: Omega affine case} Consider a commutative diagram as in \eqref{eq: test diagram} where $i$ is a first order thickening.
		Suppose that $f$ is smooth, $\cY,\cY'$ are affine, and $\theta'^*\Omega^1_{\cZ'/\cZ}$ is isomorphic to the sheaf associated to a finite module over $\Gamma(\cY',\cO_{\cY'})$.
		Then there exists a morphism $g\colon\cY\rightarrow\cZ'$ such that $g\circ i=\theta'$ and $f\circ g=\theta$.
		The third condition on $\theta'^*\Omega^1_{\cZ'/\cZ}$ is satisfied for example if $\cZ'$ and $\cZ$ are affine.
	\item\label{item: Omega product}
	Let $f\colon \cX'\rightarrow \cX$ and $g\colon \cZ\rightarrow\cX$ be morphisms of weak formal schemes, and denote the canonical projection $\cZ':=\cX'\times_\cX\cZ\rightarrow\cX'$ by $g'$.
	Then we have a canonical isomorphism $g'^\ast\Omega^1_{\cX'/\cX}\xrightarrow{\cong}\Omega^1_{\cZ'/\cZ}$.
	\end{enumerate}
\end{proposition}

\begin{proof}
	We note that, for a pseudo-wcfg algebra $A$ and a finite $A$-module $M$, the $A$-algebra $A\oplus M$ defined by $(a,m)(a',m'):=(aa',am'+a'm)$ is also pseudo-wcfg.
	Indeed, if we take elements $m_1,\ldots, m_r\in M$ generating $M$ over $A$, then the homomorphism
	\[A\otimes^\dagger R[X_1,\ldots,X_r]^\dagger\rightarrow A\oplus M;\, a\otimes X_1^{m_1}\cdots X_r^{m_r}\mapsto
	\begin{cases}(a,0)& m_1=\cdots m_r=0\\
	(0,am_i)&m_i=1, m_j=0 \,(\forall j\neq i)\\
	(0,0)&\text{otherwise}
	\end{cases}
	\]
	is surjective.
	The natural injection $A\rightarrow A\oplus M$ and the natural surjection $A\oplus M\rightarrow A$ are morphisms of pseudo-wcfg algebras.
	
	Therefore, for any coherent sheaf $\cE$ on a weak formal scheme $\cZ$, we may construct a weak formal scheme $\cZ\oplus\cE$ with a first order thickening 
	$\iota\colon\cZ\hookrightarrow\cZ\oplus\cE$ satisfying 
	$\cC_{\cZ\oplus\cE/\cZ}=\cE$ and a natural retraction $\rho\colon\cZ\oplus\cE\rightarrow\cZ$ by gluing the above construction.

	We first prove \eqref{item: Omega locally free}.
	For any point $z\in\cZ'$ we can take an open neighbourhood $\cU$ of $z$ and a surjection $\cO_\cU^{\oplus r}\rightarrow\Omega^1_{\cZ'/\cZ}|_\cU$.
	Consider a commutative diagram
	\[\xymatrix{\cU\ar[r]\ar[d]&\cU\oplus\Omega^1_{\cZ'/\cZ}|_\cU\ar[r]&\cU\oplus\cO_\cU^{\oplus r}\ar[d]\\
	\cZ'\ar[rr]^-f&&\cZ}\]
	where the right vertical morphism is given by composing the retraction $\rho$ with the inclusion $\cU\hookrightarrow\cZ'$ and $f$.
	Then by Lemma \ref{lem: pseudo torsor}, the smoothness $f$ implies that $\Hom_{\cO_{\cZ'}}(\Omega^1_{\cZ'/\cZ}|_\cU,\cO^{\oplus r}_\cU)\rightarrow\Hom_{\cO_{\cZ'}}(\Omega^1_{\cZ'/\cZ}|_\cU,\Omega^1_{\cZ'/\cZ}|_\cU)$ is surjective.
	Thus $\Omega^1_{\cZ'/\cZ}|_\cU$ is a direct summand of a finite free module, and hence locally free (\cite[\href{https://stacks.math.columbia.edu/tag/0BCI}{Lem.\,0BCI}]{stacks}).
	
	Next we prove \eqref{item: Omega exact sequence}. 
	For a coherent sheaf $\cE$ on $\cZ''$, consider the commutative diagram
	\begin{equation}\label{eq: diag chase}
	\xymatrix{\cZ''\ar@{=}[d]\ar[rr]^-\iota&&\cZ''\oplus\cE\ar[d]^-{g\circ f\circ \rho}\\
	\cZ''\ar[r]^-f&\cZ'\ar[r]^-g&\cZ.}\end{equation}
	Let $\cF_1$, $\cF_2$, and $\cF_3$ be the sheaves of sets on $\cZ''\oplus\cE$ whose sections on an open subset $\cU\subset\cZ\oplus\cE$ are defined by
	\begin{align*}
	\Gamma(\cU,\cF_1)&:=\{h\colon\cU\rightarrow\cZ'\mid g\circ h=(g\circ f\circ\rho)|_\cU,\, h\circ\iota|_\cU=f|_\cU\},\\
	\Gamma(\cU,\cF_2)&:=\{h\colon\cU\rightarrow\cZ''\mid g\circ f\circ  h=(g\circ f\circ\rho)|_\cU,\, h\circ\iota|_\cU=\id_\cU\},\\
	\Gamma(\cU,\cF_3)&:=\{h\colon\cU\rightarrow\cZ''\mid f\circ h=(f\circ\rho)|_\cU,\, h\circ\iota|_\cU=\id_\cU\}.
	\end{align*}
	Then by Lemma \ref{lem: pseudo torsor} we have isomorphisms
	\begin{align*}
	&\cH om_{\cO_{\cZ''}}(f^\ast\Omega^1_{\cZ'/\cZ},\cE)\xrightarrow{\cong}\cF_1;\, \varphi\mapsto (f\circ\rho)^\varphi,\\
	&\cH om_{\cO_{\cZ''}}(\Omega^1_{\cZ''/\cZ},\cE)\xrightarrow{\cong}\cF_2;\, \varphi\mapsto \rho^\varphi,\\
	&\cH om_{\cO_{\cZ''}}(\Omega^1_{\cZ''/\cZ'},\cE)\xrightarrow{\cong}\cF_3;\, \varphi\mapsto \rho^\varphi.	
	\end{align*}
	Using them, we see that the sequence
	\begin{equation}\label{eq: exact for H}
	0\rightarrow \cH om_{\cO_{\cZ''}}(\Omega^1_{\cZ''/\cZ'},\cE)\rightarrow \cH om_{\cO_{\cZ''}}(\Omega^1_{\cZ''/\cZ},\cE)\rightarrow \cH om_{\cO_{\cZ''}}(f^\ast\Omega^1_{\cZ'/\cZ},\cE)
	\end{equation}
	is exact for any coherent sheaf $\cE$, by a diagram chase for \eqref{eq: diag chase}.
	This implies the first assertion.
	
	When $f$ is smooth, one also sees that the third morphism of \eqref{eq: exact for H} is surjective.
	Applying this to the case $\cE=f^\ast\Omega^1_{\cZ'/\cZ}$, 
	we may locally lift the identity morphism on $f^\ast\Omega^1_{\cZ'/\cZ}$ to a morphism $\Omega^1_{\cZ''/\cZ}\rightarrow f^\ast\Omega^1_{\cZ'/\cZ}$.
	If $f$ is \'{e}tale, then the third map of \eqref{eq: exact for H} is an isomorphism, hence $f^\ast\Omega^1_{\cZ'/\cZ}\rightarrow\Omega^1_{\cZ''/\cZ}$ is also an isomorphism.
	This finishes the proof of (\ref{item: split 1}).
	
	To show (\ref{item: split 2}), we consider a commutative diagram
	\[\xymatrix{
	\cY'\ar[r]^-i\ar[d]^-{\theta'}&\cY\ar[d]^-\theta\ar[rd]^-{g\circ\theta}\\
	\cZ''\ar[r]^-f&\cZ'\ar[r]^-g&\cZ.}\]
	By assumption we may suppose that there exist a retraction $\Omega^1_{\cZ''/\cZ}\rightarrow f^\ast\Omega^1_{\cZ'/\cZ}$ and a morphism $h\colon\cY\rightarrow\cZ''$ with $g\circ f\circ h=g\circ\theta$ and $h\circ i=\theta'$.
	By Lemma \ref{lem: pseudo torsor} there exists $\varphi\in\Hom(\theta'^\ast f^\ast\Omega^1_{\cZ'/\cZ},\cC_{\cY'/\cY})$ such that $(f\circ h)^\varphi=\theta$.
	Let $\varphi':=\varphi\circ \theta'^\ast\gamma\in \Hom(\theta'^\ast\Omega^1_{\cZ''/\cZ},\cC_{\cY'/\cY})$.
	Then we have $f\circ h^{\varphi'}=\theta$, which shows that $f$ is smooth.
	The statement for the \'{e}taleness follows similarly.
	
	Next we show \eqref{item: Omega affine case}.
	Since $f$ is smooth, the stalks of the sheaf $\cF$ in Lemma \ref{lem: pseudo torsor} are non-empty, i.e.\ $\cF$ is an $\cH$-torsor.
	We write $\cY'=\Spwf B$, $\cY=\Spwf A$, and $N:=\Ker(A\rightarrow B)$.
	Then we have $\Ker(\cO_\cY\rightarrow i_\ast\cO_{\cY'})=N^\sim$ by Proposition-Definition \ref{prop-def: associated sheaf}\eqref{item: sub or quotient}.
	It follows that $\cC_{\cY'/\cY}=N^\sim$ where $N$ is regarded as a $B$-module. 
	Suppose that $\theta'^*\Omega^1_{\cZ'/\cZ}$ corresponds to a finite $B$-module $M$.
	In this case, $\cH$ is isomorphic to the sheaf associated to the finite $B$-module $\Hom_B(M,N)$.
	Therefore $H^1(\cY',\cH)$ vanishes by \cite[\S 2, Thm.\ 14]{Me}, hence $\cF$ is the trivial $\cH$-torsor by \cite[\href{https://stacks.math.columbia.edu/tag/02FQ}{Lem.\,02FQ}]{stacks}.
	This means that $\Gamma(\cY',\cF)$ is non-empty.
	
	We finally prove \eqref{item: Omega product}.
	We denote the natural projection $\cZ'\rightarrow\cZ$ by $f'$.
	For any coherent sheaf $\cE$ on $\cZ'$, consider the first order thickening $\iota\colon\cZ'\hookrightarrow\cZ'\oplus\cE$ with retraction $\rho\colon\cZ'\oplus\cE\rightarrow\cZ'$ introduced in the beginning of the proof. We define sheaves $\cG_1$ and $\cG_2$ on $\cZ'\oplus\cE$ by setting
	\begin{align*}
	&\Gamma(\cU,\cG_1):=\{h\colon \cU\rightarrow\cZ'\mid f'\circ h=(f'\circ\rho)|_\cU,\,h\circ\iota|_\cU=\id_\cU\},\\
	&\Gamma(\cU,\cG_2):= \{h\colon\cU\rightarrow\cX'\mid f\circ h=(g\circ f'\circ\rho)|_\cU,\,h\circ\iota|_\cU=g'|_\cU\}
	\end{align*}
	for any open subset $\cU\subset\cZ'\oplus\cE$.
	The universal property of the fiber product shows that the natural morphism $\cG_1\rightarrow\cG_2$ is an isomorphism.
	This with arguments similar to above shows that $\cH om_{\cO_{\cZ'}}(\omega^1_{\cZ'/\cZ},\cE)\rightarrow\cH om_{\cO_{\cZ'}}(g'^\ast\omega^1_{\cX'/\cX},\cE)$ is an isomorphism, and the assertion follows.
\end{proof}

\begin{proposition}\label{prop: completion of smooth}
	If $f\colon\cZ'\rightarrow\cZ$ is a smooth (resp.\ \'{e}tale) morphism of weak formal schemes with respect to $(R,I)$, its completion $\wh f\colon\wh\cZ'\rightarrow\wh\cZ$ is also smooth (resp.\ \'{e}tale).
\end{proposition}

\begin{proof}
	We will first prove the assertion for smoothness.
	Since the smoothness is a local condition, we may suppose $\cZ'=\Spwf B$ and $\cZ=\Spwf A$ are affine.
	According to Remark \ref{rem: formal smoothness}, it suffices to show $\wh f$ is formally smooth.
	Consider a commutative diagram of topological rings
	\begin{equation}\label{eq: completion diagram}
	\xymatrix{
	D&C\ar[l]^-{i^\sharp}\\
	\wh B\ar[u]^-{\theta'^\sharp}&\wh A\ar[l]^-{\wh{f}^\sharp}\ar[u]^-{\theta^\sharp}
	}\end{equation}
	where $C$ and $D$ are discrete and $i^\sharp$ is surjective with square zero kernel.
	Let $J$ be an ideal of definition of $A$.
	As $C$ is discrete, $\theta^\sharp$ factors through $A/J^k=\widehat{A}/J^k\widehat{A}$ for some $k$.
	Thus $\Im\theta^\sharp$ is finitely generated over $R$.
	Similarly $\Im\theta'^\sharp$ is finitely generated over $R$.
	Thus there exists a finitely generated $A/J^k$-subalgebra $C'\subset C$ such that $\Im \theta^\sharp\subset C'$ and $\Im\theta'^\sharp\subset D':=\Im(C'\xrightarrow{i^\sharp}D)$.
	Then we obtain a commutative diagram
	\[\xymatrix{
	D'&C'\ar[l]\\
	B\ar[u]&A\ar[l]\ar[u]
	}\]
	where the vertical maps in the diagram are given by composing $\theta^\sharp$ and $\theta'^\sharp$ with the canonical injections $A\rightarrow\widehat{A}$ and $B\rightarrow\widehat{B}$, respectively.
	Note that $C'$ and $D'$ are wcfg with respect to $(R,I)$, and $C'\rightarrow D'$ is surjective with square zero kernel.
	As $f$ is smooth there exists a morphism $B\rightarrow C'$ keeping the above diagram commutative by Proposition \ref{prop: fundamental properties of Omega} \eqref{item: Omega affine case}.
	Then the completion of the composite $B\rightarrow C'\rightarrow C$ induces a morphism $\wh B\rightarrow C$ which fit into the diagram \eqref{eq: completion diagram}, as desired.
	
	Note that, if a homomorphism $\wh B\rightarrow C$ fitting into \eqref{eq: completion diagram} is given, it factors through $C'$.
	If $\wh f^\sharp$ is \'{e}tale, we see that such homomorphism is unique since $B$ is dense in $\wh B$.
	Thus we obtain the assertion for \'{e}taleness. 
\end{proof}

\begin{proposition}\label{prop: Jacobian}
	Let $f\colon\cZ' \rightarrow \cZ$ be a morphism of weak formal schemes with respect to $(R,I)$ and $i\colon\cZ''\hookrightarrow\cZ'$ a closed immersion.
	Then there exists an exact sequence
		\begin{equation}\label{eq: Jacobian}
		\cC_{\cZ''/\cZ'} \rightarrow  i^* \Omega^1_{\cZ'/\cZ} \rightarrow  \Omega^1_{\cZ''/\cZ} \rightarrow  0.
		\end{equation}
	Moreover, the first map in \eqref{eq: Jacobian} is injective if the composite $f\circ i\colon\cZ'' \rightarrow \cZ$ is smooth.
\end{proposition}

\begin{proof}
	The construction of the morphisms in \eqref{eq: Jacobian} is obvious.
	By Lemma \ref{lem: wf omega}\eqref{item: omega completion} and Lemma \ref{lem: faithfully flat}, the exactness and the injectivity of the first map follow from similar results for formal schemes proved in \cite[Prop.\,2.3.4]{AJP2005} and \cite[Prop.\ 2.6.8]{LNS}.
\end{proof}

Next we prove a weak formal version of the strong fibration theorem \cite[Thm.~1.3.7]{Ber}.

\begin{proposition}[Strong fibration theorem]\label{prop: strong fibration}
	Let $\iota\colon Y\hookrightarrow \cZ$ and $\iota'\colon Y\hookrightarrow\cZ'$ be homeomorphic closed immersions of weak formal schemes with respect to $(R,I)$.
	Let $f\colon\cZ' \rightarrow \cZ$ be a smooth morphism with $f\circ\iota' = \iota$.
	Let $\cJ$ be the ideal of $\iota'\colon Y\hookrightarrow \cZ'$, and set $\tau: = (\mathrm{id},\iota')\colon Y\hookrightarrow Y': = Y\times_{\cZ}\cZ'$.
	If there exist sections $t_1,\ldots,t_d\in\Gamma(\cZ',\cJ)$ such that their images in $\Gamma(Y,\cC_{Y/Y'})$ form an $\cO_Y$-basis, then a morphism $g\colon \cZ' \rightarrow \cZ'': = \cZ\times_{\Spwf \widehat{R}}\Spwf \widehat{R}\llbracket s_1,\ldots,s_d\rrbracket$ over $\cZ$ defined by $s_i\mapsto t_i$ is an isomorphism.
\end{proposition}

\begin{proof}
	By Proposition \ref{prop: Jacobian} there is a natural isomorphism $\cC_{Y/Y'} \xrightarrow{\cong}\tau^* \Omega^1_{Y'/Y}\cong\iota'^* \Omega^1_{\cZ'/\cZ}$.
	Since the images $d\overline{t}_1,\ldots,d\overline{t}_d$ of $t_1,\ldots,t_d$ in $\iota'^* \Omega^1_{\cZ'/\cZ}$ form a basis, the same is true for  $dt_1,\ldots,dt_d$ in $ \Omega^1_{\cZ'/\cZ}$.
	Therefore $g$ is \'{e}tale by Proposition \ref{prop: fundamental properties of Omega} \eqref{item: Omega exact sequence}.
	Thus the claim follows from the following Lemma \ref{lem: adic etale}.
\end{proof}

\begin{lemma}\label{lem: adic etale} 
	Let $\iota\colon Y\hookrightarrow \cZ$ and $\iota'\colon Y\hookrightarrow\cZ'$ be homeomorphic 
closed immersions of weak formal schemes with respect to $(R,I)$.
	Let $f\colon\cZ' \rightarrow \cZ$ be an \'{e}tale morphism satisfying $f\circ\iota' = \iota$.
	Then $f$ is an isomorphism.
\end{lemma}

\begin{proof}
	By \cite[Cor.\ 3.4.2]{AJP2005}, $\wh f$ is an isomorphism.
	Thus $f$ is also an isomorphism by Proposition \ref{prop: conservative}.
\end{proof}

\begin{remark}
	Locally one can find sections $t_1,\ldots,t_d$ as in Proposition \ref{prop: strong fibration}, because $\cC_{Y/Y'}$ is locally free.
\end{remark}

Next we will give sufficient conditions for smoothness/\'{e}taleness and strong smoothness/strong \'{e}taleness.
We first give a few statements for weak formal schemes with respect to an arbitrary noetherian ring  $R$.
To discuss lifting properties, we often consider as a local description the following commutative diagram of pseudo-wcfg algebras
\begin{equation}\label{eq: test diagram affine}
\xymatrix{
	D&C\ar[l]^-{i^\sharp}\\
	B\ar[u]^-{\theta'^\sharp}&A\ar[l]^-{f^\sharp}\ar[u]^-{\theta^\sharp},
}\end{equation}
where $i^\sharp$ is a surjection with kernel $N$, such that $\Spwf D\hookrightarrow\Spwf C$ is a homeomorphism.

\begin{proposition}\label{prop: weak completion is etale}
	Let $j\colon Z\hookrightarrow\cZ$ be an immersion of weak formal schemes and $\cZ_Z$ the weak completion of $\cZ$ along $Z$.
	Consider a commutative diagram
	\[\xymatrix{
	\cY'\ar[r]^-i\ar[d]^-{\theta'}&\cY\ar[d]^-\theta\\
\cZ_Z\ar[r]^-f&\cZ
	}\]
	where $f$ is the canonical morphism induced by adjointness, and $i$ is a homeomorphic closed immersion.
	Then there exists a unique morphism $g\colon\cY\rightarrow\cZ_Z$ such that $g\circ i=\theta'$ and $f\circ g=\theta$.
	In particular, $f\colon\cZ_Z\rightarrow\cZ$ is strongly \'{e}tale. 
\end{proposition}

\begin{proof}
	Set $Y':=Z\times_{\cZ_Z}\cY'$ and let $i'\colon Y'\hookrightarrow\cY'$ be the canonical morphism, which is a homeomorphic closed immersion.
	Applying the universal property of $\cZ_Z$ to $i\circ i'\rightarrow j$, we obtain a unique morphism $g\colon\cY\rightarrow\cZ_Z$ such that the following diagram commutes:
	\[\xymatrix{
	Y'\ar[rr]^{i\circ i'}\ar[d]&&\cY\ar[d]\ar[ld]_-g\\
	Z\ar[r]&\cZ_Z\ar[r]^-f&\cZ.
	}\]
	Since $g\circ i$ and $\theta'$ both make the diagram
	\[\xymatrix{
	Y'\ar[r]\ar[d]&\cY'\ar[d]\ar[rd]&\\
	Z\ar[r]&\cZ_Z\ar[r]&\cZ
	}\]
	commutative, the uniqueness of such a morphism implies $g\circ i=\theta'$.
\end{proof}

\begin{corollary}\label{cor: weak completion preserve smoothness}
	Consider a commutative diagram of weak formal schemes
	\[\xymatrix{
	Z'\ar@{^(->}[r]\ar[d]&\cZ'\ar[d]^-f\\
	Z\ar@{^(->}[r]&\cZ
	}\]
	whose horizontal morphisms are immersions.
	If $f$ is smooth (resp.\ \'{e}tale, strongly smooth, strongly \'{e}tale), then the induced morphism $\cZ'_{Z'}\rightarrow\cZ_Z$ between the weak completions is also smooth (resp.\ \'{e}tale, strongly smooth, strongly \'{e}tale).
\end{corollary}

\begin{proof}
	This statement follows from Proposition \ref{prop: easy properties} \eqref{item: easy property 4} and Proposition \ref{prop: weak completion is etale}.
\end{proof}

\begin{lemma}\label{lem: disc smooth}
	The weak formal scheme $\Spwf \wh R\llbracket s_1,\ldots,s_m\rrbracket [t_1,\ldots,t_n]^\dagger$ is strongly smooth over $\Spwf \wh R$.
\end{lemma}

\begin{proof}
	It is easy to verify the lifting property.
\end{proof}

Next we consider the special case $(R,I)=(V,\frm)$, that is a complete discrete valuation ring of mixed characteristic and the maximal ideal.

\begin{proposition}\label{prop: p-adic strong smooth}
	Let $f\colon\cZ'\rightarrow\cZ$ be a morphism between weak formal schemes which are adic over $\Spwf V$.
	Then $f$ is strongly smooth (resp.\ strongly \'{e}tale) if $\wh f$ is smooth (resp.\ \'{e}tale).
\end{proposition}

\begin{proof}
	We will first prove the statement concerning strong smoothness.
	As a local description, consider the commutative diagram \eqref{eq: test diagram affine} such that $A$ and $B$ are wcfg with respect to $(V,\frm)$ and $\wh f^\sharp$ is smooth.
	It suffices to show the existence of $g^\sharp\colon B\rightarrow C$ with $i^\sharp\circ g^\sharp=\theta'^\sharp$ and $g^\sharp\circ f^\sharp=\theta^\sharp$ in the following cases:
	\begin{enumerate}
	\item\label{item: case 1} $C$ and $D$ are wcfg,
	\item\label{item: case 2} $N^2=0$ (where $N=\Ker i$).
	\end{enumerate}
	
	We first discuss the case \eqref{item: case 1}.
	As in \cite[Prop.\ 3.4.1]{AJP2005} there exists a morphism $h^\sharp\colon\wh B\rightarrow\wh C$ such that $\wh i^\sharp\circ h^\sharp=\wh\theta'^\sharp$ and $h^\sharp\circ\wh f^\sharp=\wh\theta^\sharp$.
	We will construct $g^\sharp\colon B\rightarrow C$ from $h^\sharp$ by the variant of Artin approximation theorem given in \cite[Prop.\ 2.4.1]{vdP} and \cite[Thm.\ 2]{Bos}.
	Choose representations
	\begin{align*}
	&A=V[X_1,\ldots,X_l]^\dagger/(f_1,\ldots,f_a),&&B=V[Y_1,\ldots,Y_m]^\dagger/(g_1,\ldots,g_b),\\
	&C=V[Z_1,\ldots,Z_n]^\dagger/(h_1,\ldots,h_c),&&D=V[Z_1,\ldots,Z_n]^\dagger/(q_1,\ldots,q_d)
	\end{align*}
	and series
	\begin{align*}
	&x_i=x_i(Y_1,\ldots,Y_m)\in V[Y_1,\ldots,Y_m]^\dagger&&\text{which lifts }f^\sharp(X_i)\in B,\\
	&x'_i=x'_i(Z_1,\ldots,Z_n)\in V[Z_1,\ldots,Z_n]^\dagger&&\text{which lifts }\theta^\sharp(X_i)\in C,\\
	&\wt y_j=\wt y_j(Z_1,\ldots,Z_n)\in V\langle Z_1,\ldots,Z_n\rangle&&\text{which lifts }h^\sharp(Y_j)\in\wh C,\\
	&y'_j=y'_j(Z_1,\ldots,Z_n)\in V[Z_1,\ldots,Z_n]^\dagger&&\text{which lifts }\theta'^\sharp(Y_j)\in D
	\end{align*}
	for each $1\leq i\leq l$ and $1\leq j\leq m$.
	Here the angle brackets in  ``$V\langle Z_1,\ldots,Z_n\rangle$'' denote the restricted power series ring.
	Then the well-definedness of $h^\sharp$ and the condiitons $\wh i^\sharp\circ h^\sharp=\wh\theta'^\sharp$, $h^\sharp\circ \wh f^\sharp=\wh\theta^\sharp$ imply equations
	\begin{align*}
	&g_k(\wt y_1,\ldots,\wt y_m)=\sum_{s=1}^c\wt\alpha_{k,s}h_s,&
	&\wt y_j-y'_j=\sum_{t=1}^d\wt\beta_{j,t}q_t,&
	&x'_i-x_i(\wt y_1,\ldots,\wt y_m)=\sum_{s=1}^c\wt\gamma_{i,s}h_s
	\end{align*}
	for some elements $\wt\alpha_{k s},\wt\beta_{j,t},\wt\gamma_{i,s}\in V\langle Z_1,\ldots,Z_n\rangle$.
	Applying Artin approximation to these equations, one can replace $\wt y_j,\wt\alpha_{\nu s},\wt\beta_{j,t},\wt\gamma_{i,s}$ by some elements $y_j,\alpha_{\nu s},\beta_{j,t},\gamma_{i,s}\in V[Z_1,\ldots,Z_n]^\dagger$.
	Then by associating $Y_j\mapsto y_j$ we obtain a well-defined homomorphism $g^\sharp\colon B\rightarrow C$ satisfying $i^\sharp\circ g^\sharp=\theta'^\sharp$ and $g^\sharp\circ f^\sharp=\theta^\sharp$ as desired.

	Next we discuss the case \eqref{item: case 2}.
	Take a surjection $f'^\sharp\colon A':=A\otimes^\dagger_VV[t_1,\ldots,t_n]^\dagger\rightarrow B$ over $A$.
	Choose elements $x_1,\ldots,x_n\in C$ which lift $\theta'^\sharp\circ f'^\sharp(t_1),\ldots,\theta'^\sharp\circ f'^\sharp(t_n)\in D$, and extend $\theta^\sharp$ to $\wt\theta^\sharp\colon A'\rightarrow C$ by $\wt\theta^\sharp(t_i):=x_i$.
	Then $C':=\Im\,\wt\theta^\sharp$ and $D':=\Im\,\theta'^\sharp$ are wcfg algebras, and $i$ induces a surjection $i'\colon C'\rightarrow D'$ whose kernel is clearly square zero.
	By the above argument for the case \eqref{item: case 1} there exists a morphism $g'^\sharp\colon B\rightarrow C'$ such that the following diagram commutes:
	\[\xymatrix{
	D'&C'\ar[l]_-{i'^\sharp}\\
	B\ar[ur]^-{g'^\sharp}\ar[u]^-{\theta'^\sharp}&A\ar[l]^-{f^\sharp}\ar[u]_-{\theta^\sharp}.
	}\]
	Composing $g'^\sharp$ with the natural inclusion $C'\hookrightarrow C$, we obtain $g^\sharp\colon B\rightarrow C$ satisfying the desired commutativity.
	
	Finally, the statement for strong \'{e}taleness follows from the fact that the functor $\cZ\mapsto\wh\cZ$ is faithful.
\end{proof}

\begin{proposition}\label{prop: lift strong smooth}
	Consider a commutative diagram \eqref{eq: test diagram} of weak formal schemes with respect to $(V,\frm)$, and let $\cJ$ be an ideal of definition of $\cY$.
	If $f$ is strongly smooth, then for any $k\in\bbN$ there exists locally on $\cY[\frac{\cJ^k}p]^\dagger$ a morphism $g_k\colon \cY[\frac{\cJ^k}p]^\dagger\rightarrow\cZ'$ such that the diagram
	\[\xymatrix{
	\cY'\times_\cY\cY[\frac{\cJ^k}p]^\dagger\ar@{^(->}[r]\ar[d]&\cY[\frac{\cJ^k}p]^\dagger\ar[d]\ar[ld]^-{g_k}\\
	\cZ'\ar[r]_-f&\cZ
	}\]
	commutes.
	If $f$ is strongly \'{e}tale, such a $g_k$ exists globally and uniquely.
\end{proposition}

\begin{proof}
	Immediate from the lifting property of $f$.
\end{proof}

\begin{corollary}\label{cor: dagger lifting property}
	Consider a commutative diagram \eqref{eq: test diagram} of weak formal schemes with respect to $(V,\frm)$, and let
	\[\xymatrix{
	\frY'\ar@{^(->}[r]^-{\wt i}\ar[d]^-{\wt\theta'}&\frY\ar[d]^-{\wt\theta}\\
	\frZ'\ar[r]^-{\wt f}&\frZ}\]
	be the induced diagram of the associated dagger spaces.
	If $f$ is strongly smooth, then there exists locally on $\frY$ a morphism of dagger spaces $g\colon \frY\rightarrow\frZ'$ such that $g\circ\wt i=\wt\theta'$ and $\wt f\circ g=\wt \theta$.
	If $f$ is strongly \'{e}tale, such a $g$ exists globally and uniquely.
\end{corollary}

\begin{proof}
	The local existence of $g$ immediately follows from Proposition \ref{prop: lift strong smooth}. 
	To show uniqueness, assume that $f$ is strongly \'{e}tale and that there are two morphisms $g^{(1)}$ and $g^{(2)}$ satisfying the property as in the statement.
	To show that $g^{(1)}=g^{(2)}$, it suffices to show that their restrictions to an affinoid subspace $\frU\subset\frY$ coincide with each other.
	Let $\cK$ and $\cK'$ be ideals of definition of $\cZ$ and $\cZ'$, respectively, such that $\cK'\subset f^\ast\cK$.
	For $k\geq 0$, let $\frZ_k$ and $\frZ'_k$ be the dagger spaces associated to $\cZ[\frac{\cK^k}p]^\dagger$ and $\cZ'[\frac{\cK'^k}p]^\dagger$, respectively, and let $f_k\colon\cZ'[\frac{\cK'^k}p]^\dagger\rightarrow\cZ[\frac{\cK^k}p]^\dagger$ be the natural morphism induced from $f$.
	Since an affinoid space is quasi-compact, $g^{(1)}|_\frU$ and $g^{(2)}|_\frU$ factor through $\frZ'_k$ for some $k$.
	Then by Raynaud's theorem for dagger spaces \cite{LM}, there exist a weak formal model $\cU$ of $\frU$ which is adic over $\Spwf V$ and morphisms $h^{(j)}\colon\cU\rightarrow \cZ'[\frac{\cJ'^k}p]^\dagger$, for $j=1,2$, which induce $g^{(j)}|_\frU$.
	Let $\iota\colon\cU'\hookrightarrow\cU$ be the closed immersion which induces $\frU':=\frU\times_\frY\frY'\hookrightarrow\frU$.
	Then we have $\Theta':=h^{(1)}\circ\iota=h^{(2)}\circ\iota$ because both sides induce the same morphism on dagger spaces.
	Similarly we have $\Theta:=f_k\circ h^{(1)}=f_k\circ h^{(2)}$.
	Now we have a commutative diagram
	\[\xymatrix{
	\cU'\ar@{^(->}[r]^-\iota\ar[d]_-{\Theta'}&\cU\ar[d]^-\Theta\ar[ld]_-{h^{(j)}}\\
	\cZ'\ar[r]_-f&\cZ
	}\]
	where we denote the composition of $h^{(j)}$ with the canonical morphism $\cZ'[\frac{\cJ'^k}p]^\dagger\rightarrow\cZ'$ by the same symbol, and apply the same rule to $\Theta$ and $\Theta'$.
	Then Proposition \ref{prop: lift strong smooth} implies $h^{(1)}=h^{(2)}$, and hence $g^{(1)}|_\frU=g^{(2)}|_\frU$ as desired.
\end{proof}

Let $\cZ$ be a weak formal scheme with respect to $(V,\frm)$ and $\frZ$ the dagger space associated to $\cZ$.
For any $\cO_\cZ$-module $\cF$, we associate an $\cO_\frZ$-module $\cF_\bbQ:=\mathrm{sp}^*\cF=\mathrm{sp}^{-1}\cF\otimes_{\mathrm{sp}^{-1}\cO_\cZ}\cO_\frZ$.
We note that, for a morphism $\cZ'\rightarrow\cZ$ of weak formal schemes, the derivation $\cO_{\cZ'}\rightarrow\Omega^1_{\cZ'/\cZ}$ is not $\cO_{\cZ'}$-linear.
Therefore we cannot directly apply $\mathrm{sp}^*$ to obtain the derivation on $\Omega^1_{\cZ'/\cZ,\bbQ}$.

However, as we will see in the following, it is isomorphic to the sheaf of differentials $\Omega^1_{\frZ'/\frZ}$ with a derivation $d_{\frZ'/\frZ}$ of the associated morphism between dagger spaces $\frZ'\rightarrow\frZ$, which is defined by gluing local constructions. 
The compatibility with restriction is implied from that for rigid spaces via completion, similarly to Lemma \ref{lem: wf omega}.
Consequently, there is a natural derivation on $\Omega^1_{\cZ'/\cZ,\bbQ}$ induced by that of $\Omega^1_{\frZ'/\frZ}$.

\begin{proposition}\label{prop: Omega on generic fiber}
	Let $f\colon \cZ'\rightarrow\cZ$ be a morphism of weak formal schemes with respect to $(V,\frm)$, $\cJ$ an ideal of definition of $\cZ'$, and $k$ a positive integer.
	Let $\frZ$, $\frZ'$, and $\frZ'_k$ be the dagger space associated to $\cZ$, $\cZ'$, and $\cZ'_k:=\cZ'[\frac{\cJ^k}p]^\dagger$, respectively.
	Then there exist canonical isomorphisms
	\begin{equation}\label{eq: compare Omega}
	\Omega^1_{\cZ'/\cZ,\bbQ}|_{\frZ'_k}\xrightarrow{\cong}\Omega^1_{\cZ'_k/\cZ}\otimes_{V}K\xrightarrow{\cong}\Omega^1_{\frZ'_k/\frZ}
	\end{equation}
	and hence
	\[\Omega^1_{\cZ'/\cZ,\bbQ}\xrightarrow{\cong}\Omega^1_{\frZ'/\frZ}.\]	
	Here the middle term of \eqref{eq: compare Omega} is defined locally as follows: 
	For any affine open subset $\cU=\Spwf C_\cU$ of $\cZ'_k$, the restriction $\Omega^1_{\cZ'_k/\cZ}|_\cU$ is isomorphic to the sheaf associated to a finite $C_\cU$-module $M_\cU$.
	The sheaves on $\Sp(C_\cU\otimes_VK)$ associated to $M_\cU\otimes_VK$ glue to a sheaf $\Omega^1_{\cZ'_k/\cZ}\otimes_{V}K$.
\end{proposition}

\begin{proof}
Working locally, assume that $\cZ=\Spwf A$ and $\cZ'=\Spwf B$ are affine.
Take ideals of definition $L\subset A$ and $J\subset B$ such that $LB\subset J$, and let $A_k:=A[\frac{L^k}p]^\dagger$ and $B_k:=B[\frac{J^k}p]^\dagger$.
Now the sheaves in \eqref{eq: compare Omega} correspond to the finite $B_k\otimes_VK$-modules $\Omega^1_{B/A}\otimes_B B_k\otimes_VK$, $\Omega^1_{B_k/A}\otimes_VK$, and $\Omega^1_{B_k\otimes_VK/A_k\otimes_VK}$, respectively.

Take generators $f_1,\ldots,f_n$ of $J$ and identify $B_k$ with
	\[B':=B[x_{\mathbf{k}}\mid \mathbf{k}=(k_i)_i\in\bbN^n,\ \lvert\mathbf{k}\rvert=k]^\dagger/(px_{\mathbf{k}}-f_1^{k_1}\cdots f_n^{k_n})\]
	modulo $p$-torsion.
	Then by the Propositions \ref{prop: Omega affine},  \ref{prop: fundamental properties of Omega} and  \ref{prop: Jacobian},
	$\Omega^1_{B_k/A}$ is generated by generators of $\Omega^1_{B/A}$ and symbols $dx_{\mathbf{k}}$ with relations $pdx_{\mathbf{k}}=d(f_1^{k_1}\cdots f_n^{k_n})$ and $d\alpha=0$ for $p$-torsion elements $\alpha\in B'$.
Thus we see that the natural map $\Omega^1_{B/A}\otimes_{B}B_k\otimes_VK\rightarrow\Omega^1_{B_k/A}\otimes_VK$ is an isomorphism, and this shows the first isomorphism in \eqref{eq: compare Omega}.
By a similar observation one can see that $\Omega^1_{A_k/A}\otimes_VK=0$.
Thus, again by Proposition \ref{prop: fundamental properties of Omega}, we have $\Omega^1_{B_k/A}\otimes_VK\rightarrow\Omega^1_{B_k/A_k}\otimes_VK$.
Composing this with $\Omega^1_{B_k/A_k}\otimes_VK\xrightarrow{\cong}\Omega^1_{B_k\otimes_VK/A_k\otimes_VK}$, we obtain the second isomorphism in \eqref{eq: compare Omega}.
\end{proof}

%%%%%%%%%%
\subsection{Weak formal log schemes}
\label{subsection: weak formal log schemes}
%%%%%%%%%%

The notion of log schemes in the sense of Fontaine, Illusie and Kato \cite{Ka} extends to weak formal log schemes in the obvious way,
and we freely use terminologies concerning the log structure on schemes or (weak) formal schemes in this sense.
Throughout, log structures of all (weak formal) log schemes are defined as  sheaves with respect to the Zariski topology.
Note that fine log schemes with respect to the Zariski topology correspond to fine log schemes with respect to \'{e}tale topology which Zariski locally admit charts via the equivalence of categories in \cite[Cor.~1.1.11]{Sh1}.
For a weak formal log scheme $\cZ$, we denote its log structure by $\cM_\cZ$.
We say $\cZ$ is \textit{fine} if it has locally a chart by a fine monoid.
We will always work in the category of fine weak formal log schemes.
Note that the category of fine weak formal log schemes with respect to a base $(R,I)$ has fibre products, which can be seen similarly to \cite[(2.8)]{Ka}.

\begin{definition}
	Let $f\colon\cZ'\rightarrow\cZ$ be a morphism of fine weak formal log schemes.
	\begin{enumerate}
	\item We say $f$ is \textit{strict} if $f^*\cM_\cZ\rightarrow\cM_{\cZ'}$ is an isomorphism.
	\item We say $f$ is a \textit{strict open immersion} if $f$ is strict and the underlying morphism of weak formal schemes is an open immersion.
	\item We say $f$ is a \textit{closed immersion} (resp.\ an \textit{exact closed immersion}) if the underlying morphism of weak formal schemes is a closed immersion and $f^*\cM_\cZ\rightarrow\cM_{\cZ'}$ is surjective (resp.\ an isomorphism).
		We say $f$ is a \textit{first order thickening} if it is an exact closed immersion defined by a square zero ideal.
	\item We say $f$ is an \textit{immersion} if it can be written as a composition of a closed immersion and a strict open immersion.
	\item We say $f$ is \textit{adic} if the underlying morphism of weak formal schemes is adic.
	\item Let $\omega^1_{\cZ'/\cZ}$ be the quotient of $\Omega^1_{\cZ'/\cZ}\oplus(\cO_{\cZ'}\otimes \cM_{\cZ'}^{\mathrm{gp}})$ divided by the $\cO_{\cZ'}$-submodule locally generated by sections of the form $(d\alpha(a),0)-(0,\alpha(a)\otimes a)$ for $a\in\cM_{\cZ'}$ and $(0,1\otimes a)$ for $a\in\Im(f^{-1}\cM_\cZ\rightarrow\cM_{\cZ'})$.
		Here $\alpha\colon\cM_{\cZ'}\rightarrow\cO_{\cZ'}$ denotes the structure morphism. It immediately follows from Proposition-Definition \ref{prop-def: Omega} that $\omega^1_{\cZ'/\cZ}$ is coherent.
	\item For a fine weak formal log scheme $\cZ=(\cZ,\cM_\cZ)$, we let $\cM_{\widehat{\cZ}}$ be the log structure on $\widehat{\cZ}$ associated with the pre-log structure $\cM_{\cZ}\rightarrow\cO_{\cZ}\rightarrow\cO_{\widehat{\cZ}}$.
		We call $(\widehat{\cZ},\cM_{\widehat{\cZ}})$ the \textit{completion} of $\cZ$, and often denote it again by $\widehat{\cZ}$.
	\end{enumerate}
\end{definition}

\begin{proposition}
The functor $\cZ\mapsto\wh\cZ$ from the category of fine weak formal log schemes with respect to $(R,I)$ to the category of fine formal log schemes is faithful and conservative.
\end{proposition}

\begin{proof}
	Recall that we already know this statement if we forget the log structures by Proposition \ref{prop: conservative}. 
	For a fine weak formal log scheme $\cZ$ and a point $z\in\cZ$, we have $\cM_{\wh{\cZ},z}=\cM_{\cZ,z}\oplus_{\cO_{\cZ,z}^\times}\cO_{\wh{\cZ},z}^\times$.
	Thus $\cM_{\cZ,z}\rightarrow\cM_{\wh\cZ,z}$ is injective, which implies that the functor $\cZ\mapsto\wh\cZ$ is faithful.
	Moreover we have $\cM_{\wh\cZ,z}/\cO_{\wh\cZ,z}^\times\cong \cM_{\cZ,z}/\cO_{\cZ,z}^\times$, which implies that the functor $\cZ\mapsto\wh\cZ$ is conservative.
\end{proof}

Since we have a natural definition of weak completion along locally closed weak formal subschemes (c.f.\,Definition \ref{def: weak formal completion along closed}), 
we may define the exactification of an immersion of fine weak formal log schemes in the same manner as in the case of formal log schemes treated in \cite{Sh2} by Shiho.

\begin{proposition-definition}\label{def: exactification}
	The natural inclusion from the category of homeomorphic exact closed immersions of fine weak formal 
log schemes with respect to $(R,I)$ to the category of immersions of fine weak formal log schemes with respect to $(R,I)$ has a right adjoint of the form $(Z\hookrightarrow\cZ)\mapsto(Z\hookrightarrow\cZ_Z^{\mathrm{ex}})$.
	We call this right adjoint functor the {\it exactification}.
\end{proposition-definition}
\begin{proof}
This is the weak formal analogon of  \cite[Prop.-Def.~2.10]{Sh2} and \cite[Cor.~2.11]{Sh2}.
The proof can be carried out mutatis mutandis by working in the category of fine weak formal log schemes instead of in the category of fine formal log schemes.
It is an adaption of \cite[Prop.\,4.10(1)]{Ka} to the weak formal setting. 

For the benefit of the reader we give a construction of $\cZ_Z^{\mathrm{ex}}$.
By a similar argument to the proof of Prop.-Def.\ \ref{prop-def: weak completion}, we may assume that $i:Z\hookrightarrow \cZ$ is a closed immersion with a chart $\left(P_Z\rightarrow \cM_Z, Q_{\cZ}\rightarrow \cM_{\cZ}, Q\xrightarrow{\alpha} P\right)$ such that the groupification $\alpha^{\mathrm{gp}}: Q^{\mathrm{gp}} \rightarrow P^{\mathrm{gp}}$ is surjective.
Here $P_Z$ and $Q_\cZ$ denote the constant sheaves associated to $P$ and $Q$, respectively.
By taking weak completion, we may also assume that $i$ is a homeomorphism.

Now we define a weak formal log scheme $\cZ_Z^{\mathrm{ex}}$ in the following way: 
set $Q':=\alpha^{\mathrm{gp},-1} (P)$ and let $\cZ'$ be the weak formal log scheme whose underlying scheme is $\cZ\times_{\Spwf \wh R[Q]^\dagger} \Spwf \wh R[Q']^\dagger$ with the log structure associated  to $Q'_{\cZ'} \rightarrow \cO_{\cZ'}$.
This produces a commutative diagram of weak formal log schemes
$$
\xymatrix{
Z \ar[dr]^{i'} \ar[rr] \ar@/_1pc/[ddr]_i && \Spwf\wh R[P]^\dagger\ar[d]\\
& \cZ' \ar[r] \ar[d]^{f} & \Spwf\wh R[Q']^\dagger \ar[d]^\alpha\\
& \cZ \ar[r] & \Spwf\wh R[Q]^\dagger
}
$$
where $i'$ is an exact closed immersion. 
Now we define $\cZ^{\mathrm{ex}}_Z$ to be the weak completion of $\cZ'$ along $Z$ with the pull-back log structure. 
Then $Z\hookrightarrow\cZ_Z^{\mathrm{ex}}$ is a homeomorphic exact closed immersion, and satisfies the universality by construction.
\end{proof}

As before, we will often consider a commutative diagram of fine weak formal log schemes
\begin{equation}\label{eq: test diagram log}\xymatrix{
\cY'\ar[r]^-i\ar[d]^-{\theta'}&\cY\ar[d]^-\theta\\
\cZ'\ar[r]^-f&\cZ
}\end{equation}
where $i$ is a homeomorphic exact closed immersion defined by an ideal $\cN\subset\cO_\cY$.

We may provide the following statements and definitions in a similar manner to the non-log case (see Definition \ref{def: non-log-lifting-properties}, Propositions \ref{prop: easy properties} and \ref{prop: fundamental properties of Omega}).

\begin{lemma}\label{lem: pseudo torsor log}
	Consider the commutative diagram \eqref{eq: test diagram log} and suppose that $i$ is a first order thickening.
	Let $\cF_{\log}$ be the sheaf of sets on $\cY$ defined by $\Gamma(\cU,\cF_{\log}):=\{g\colon \cU\rightarrow\cZ'\mid f\circ g=\theta|_{\cU},\ g\circ i|_{\cU'}=\theta'|_{\cU'}\}$ for each open weak formal subscheme $\cU\subset\cY$.
	Here $\cU'$ denotes $\cU\times_\cY\cY'$.
	Then the sheaf $\cH_{\log}:=\cH om_{\cO_{\cY'}}(\theta'^*\omega^1_{\cZ'/\cZ},\cC_{\cY'/\cY})$ naturally acts on $\cF_{\log}$, and it turns $\cF_{\log}$ a pseudo-$\cH_{\log}$-torsor.
	This means that the action of $\Gamma(\cU,\cH_{\log})$ is simply transitive if $\Gamma(\cU,\cF_{\log})$ is non-empty.
\end{lemma}

\begin{definition}
	Let $f\colon \cZ' \rightarrow \cZ$ be a morphism of fine weak formal log schemes with respect to $(R,I)$.
	\begin{enumerate}
	\item We say $f$ is {\it log smooth} if for any commutative diagram as in \eqref{eq: test diagram log} where $i$ is a first order thickening, there exists locally on $\cY$ a morphism $g\colon\cY\rightarrow\cZ'$ such that $g\circ i=\theta'$ and $f\circ g=\theta$.
	\item We say $f$ is {\it log \'{e}tale} if for any commutative diagram as in \eqref{eq: test diagram log} where $i$ is a first order thickening, there exists a unique morphism $g\colon\cY\rightarrow\cZ'$ such that $g\circ i=\theta'$ and $f\circ g=\theta$.
	\item We say $f$ is {\it strongly log smooth} if it is smooth and for any commutative diagram as in \eqref{eq: test diagram log} where $\cY$ is adic over $\Spwf R$, there exists locally on $\cY$ a morphism $g\colon\cY\rightarrow\cZ'$ such that $g\circ i=\theta'$ and $f\circ g=\theta$.
	\item We say $f$ is  {\it strongly log \'{e}tale} if it is \'{e}tale and for any commutative diagram as in \eqref{eq: test diagram log} where $\cY$ is adic over $\Spwf R$, there exists a unique morphism $g\colon\cY\rightarrow\cZ'$ such that $g\circ i=\theta'$ and $f\circ g=\theta$.
	\end{enumerate}
\end{definition}

\begin{proposition}\label{prop: easy properties log}
	For morphisms of fine weak formal log schemes with respect to $(R,I)$, we have the following:
	\begin{enumerate}
	\item\label{item: easy property log 1} A strict open immersion is strongly log \'{e}tale.
	\item\label{item: easy property log 2} Composition of log smooth (resp.\ log \'{e}tale, strongly log smooth, strongly log \'{e}tale) morphisms is log smooth (resp.\ log \'{e}tale, log strongly smooth, strongly log \'{e}tale).
	\item\label{item: easy property log 3} Log smoothness, log \'{e}taleness, strong log smoothness, and strong log \'{e}taleness are stable under base change.
	\item\label{item: easy property log 4} If $f$ is log \'{e}tale and $f\circ g$ is log smooth (resp.\ log \'{e}tale), then $g$ is log smooth (resp.\ log \'{e}tale).
	If $f$ is strongly log \'{e}tale and $f\circ g$ is strongly log smooth (resp.\ strongly log \'{e}tale), then $g$ is strongly log smooth (resp.\ strongly log \'{e}tale).
	\end{enumerate}
\end{proposition}

\begin{proposition}\label{prop: fundamental properties of omega} $\quad$
	\begin{enumerate}
	\item\label{item: omega locally free} If $f\colon\cZ'\rightarrow\cZ$ is a log smooth (resp.\ log \'{e}tale) morphism of fine weak formal log schemes, then $\omega^1_{\cZ'/\cZ}$ is locally free (resp.\ zero).
	\item\label{item: omega exact sequence} For any morphisms of fine weak formal log schemes $f\colon\cZ''\rightarrow\cZ'$ and $g\colon\cZ'\rightarrow\cZ$, we have an exact sequence
		\begin{equation}\label{eq: omega exact sequence}f^*\omega^1_{\cZ'/\cZ}\rightarrow\omega^1_{\cZ''/\cZ}\rightarrow\omega^1_{\cZ''/\cZ'}\rightarrow 0.\end{equation}
		Moreover we have the following:
		\begin{enumerate}
		\item  If $f$ is log smooth, then the first map of \eqref{eq: omega exact sequence} is injective and locally split.
			If $f$ is log \'{e}tale, then the first map of \eqref{eq: omega exact sequence} is an isomorphism.
		\item Suppose that $g\circ f$ is log smooth.
			If the first map of \eqref{eq: omega exact sequence} is injective and locally split, then $f$ is smooth.
			If the first map of \eqref{eq: omega exact sequence} is an isomorphism, then $f$ is \'{e}tale.
		\end{enumerate}
		\item\label{item: omega affine case} Consider a commutative diagram as in \eqref{eq: test diagram log} where $i$ is a first order thickening.
		Suppose that $f$ is log smooth, $\cY,\cY'$ are affine, and $\theta'^*\omega^1_{\cZ'/\cZ}$ is isomorphic to the sheaf associated to a finite module over $\Gamma(\cY',\cO_{\cY'})$.
		Then there exists a morphism $g\colon\cY\rightarrow\cZ'$ such that $g\circ i=\theta'$ and $f\circ g=\theta$.
		The third condition on $\theta'^*\omega^1_{\cZ'/\cZ}$ is satisfied for example if $\cZ'$ and $\cZ$ are affine weak formal log schemes with global chart.
	\item\label{item: omega product}
	Let $\cX'\rightarrow\cX$ and $g\colon\cZ\rightarrow\cX$ be morphisms of weak formal log schemes, and denote the canonical projection $\cZ':=\cX'\times_\cX\cZ\rightarrow\cX'$ by $g'$.
	Then we have a canonical isomorphism $g'^\ast\omega^1_{\cX'/\cX}\xrightarrow{\cong}\omega^1_{\cZ'/\cZ}$.
	\end{enumerate}
\end{proposition}

\begin{proposition}\label{prop: completion of log smooth}
	If $f\colon\cZ'\rightarrow\cZ$ is a log smooth (resp.\ log \'{e}tale) morphism of fine weak formal log schemes with respect to $(R,I)$, its completion $\wh f\colon\wh\cZ'\rightarrow\wh\cZ$ is also log smooth (resp.\ log \'{e}tale).
\end{proposition}

\begin{proof}
	It suffices to show the existence and uniqueness part of the infinitesimal lifting property for a commutative diagram
	\begin{equation}\label{eq: diag completion of log smooth}\xymatrix{
	\cY'\ar[r]^-i\ar[d]^-{\theta'}&\cY\ar[d]^-\theta\\
	\widehat{\cZ'}\ar[r]^-{\widehat{f}}&\wh\cZ}\end{equation}
	where $i$ is a first order thickening between fine log schemes.
	Assume first, that $f$ is log smooth.
	Let us show that there exists a morphism $\cY\rightarrow\widehat{\cZ'}$ that lifts $\theta$ in the sense that the two resulting triangles in the diagram commute. This will show that $\widehat{f}$ is also log smooth. 
	For a point $y\in\cY$, we denote again by $y$ the point of $\cY'$ corresponding to $y$.
	We take integers $a,b,c\geq 0$ and surjective homomorphisms
	\begin{align*}
	\alpha\colon\bbZ^a\rightarrow\cM^{\mathrm{gp}}_{\cZ,\theta(y)},&&\beta\colon\bbZ^b\rightarrow\cM^{\mathrm{gp}}_{\cZ',\theta'(y)},&&\gamma\colon\bbZ^c\rightarrow\cM^{\mathrm{gp}}_{\cY,y}.
	\end{align*}
	Let $\wt\beta\colon\bbZ^{a+b}\rightarrow\cM^{\mathrm{gp}}_{\cZ',\theta'(y)}$ be the map induced from $\beta$ and the composition $\bbZ^a\xrightarrow{\alpha}\cM^{\mathrm{gp}}_{\cZ,\theta(y)}\rightarrow\cM^{\mathrm{gp}}_{\cZ',\theta'(y)}$.
	We lift the composition $\bbZ^b\xrightarrow{\beta}\cM^{\mathrm{gp}}_{\cZ',\theta'(y)}\rightarrow\cM^{\mathrm{gp}}_{\wh\cZ',\theta'(y)}\rightarrow\cM^{\mathrm{gp}}_{\cY',y}$ to a homomorphism $\beta'\colon\bbZ^b\rightarrow\cM^{\mathrm{gp}}_{\cY,y}$ via the surjection $\cM^{\mathrm{gp}}_{\cY,y}\rightarrow\cM^{\mathrm{gp}}_{\cY',y}$.
	Let $\wt\gamma\colon \bbZ^{a+b+c}\rightarrow\cM^{\mathrm{gp}}_{\cY,y}$ be the map induced from $\gamma$, $\beta'$, and the composition $\bbZ^a\rightarrow\cM_{\cZ,\theta(y)}^{\mathrm{gp}}\rightarrow\cM^{\mathrm{gp}}_{\wh\cZ,\theta(y)}\rightarrow\cM^{\mathrm{gp}}_{\cY,y}$.
	Let
	\begin{align*}
	P:=\alpha^{-1}(\cM_{\cZ,\theta(y)}),&&Q:=\wt\beta^{-1}(\cM_{\cZ',\theta'(y)}),&&S:=\wt\gamma^{-1}(\cM_{\cY,y}).
	\end{align*}
	Then the natural injections $\bbZ^a\rightarrow\bbZ^{a+b}$ and $\bbZ^{a+b}\rightarrow\bbZ^{a+b+c}$ induce homomorphisms $\varphi\colon P\rightarrow Q$ and $\psi\colon Q\rightarrow S$, respectively.
	(For the latter, we use the fact that $\cM^{\mathrm{gp}}_{\cY,y}\rightarrow\cM^{\mathrm{gp}}_{\cY',y}$ is exact.)
	Note that $P$, $Q$, and $S$ extend to charts of $\cM_{\cZ}$, $\cM_{\cZ'}$, and $\cM_\cY$ on some open neighbourhoods of $\theta(y)$, $\theta'(y)$, and $y$, respectively.
	Since log smoothness is a local condition, we may suppose that $\varphi$, $\psi$, and $\psi\circ\varphi$ give charts of $f$, $\theta'$, and $\theta$, respectively.
	Moreover we may assume that $\cZ=\Spwf A$, $\cZ'=\Spwf B$, $\cY=\Spec C$, and $\cY'=\Spec D$ are affine.
	
	Let $J$ be an ideal of definition of $A$.
	Then the map $\wh A\rightarrow C$ induced by $\theta$ factors through $A/J^k$ for some $k\geq 0$.
	As in the proof of Proposition \ref{prop: completion of smooth}, we may construct finitely generated $A/J^k$-subalgebras $C'\subset C$, containing the image of $\wh A\rightarrow C$, and $D'\subset D$, containing the images of $\wh B\rightarrow D$ and $C'\rightarrow D$, such that the map $C'\rightarrow D'$ induced by $i$ is surjective.
	Take elements $s_1,\ldots,s_n\in S$ generating $S$, and denote their images in $C$ by $s'_1,\ldots,s'_n$.
	Let $C'':=C'[s'_1,\ldots,s'_n]\subset C$ and $D'':=\Im(C''\rightarrow C\rightarrow D)$.
	Then $C''$ and $D''$ are finitely generated over $A/J^k$, and hence wcfg with respect to $(R,I)$.
	Then we have a commutative diagram of log rings (i.e.\, each entry is a monoid homomorphism from a monoid to the multiplicative monoid of a ring)
	\[\xymatrix{
	(S\rightarrow D'')&(S\rightarrow C'')\ar[l]\\
	(Q\rightarrow B)\ar[u]&(P\rightarrow A).\ar[l]\ar[u]
	}\]
	Since $f$ is log smooth, we may apply Proposition \ref{prop: fundamental properties of omega}\eqref{item: omega affine case} to the diagram of weak formal log schemes induced from the above diagram, to obtain a morphism $(\Spec C'',\cM)\rightarrow \cZ'$ where $\cM$ denotes the log structure associated with $S\rightarrow C''$.
	The completion of the composition $\cY\rightarrow(\Spec C'',\cM)\rightarrow \cZ'$ gives a morphism $\cY\rightarrow\wh\cZ'$ fitting into the diagram \eqref{eq: diag completion of log smooth}. 
	This shows that $\widehat{f}$ is log smooth as well.
	
	Next, suppose that a morphism $g\colon \cY\rightarrow\wh\cZ'$ fitting into \eqref{eq: diag completion of log smooth} is given.
	Take elements $q_1,\ldots,q_m\in Q$ generating $Q$, and denote their images under $Q\rightarrow \Gamma(\wh\cZ',\cM_{\wh\cZ'})\rightarrow\Gamma(\cY,\cM_\cY)\rightarrow C$ by $q'_1,\ldots,q'_m$.
	Let $C''':=C''[q'_1,\ldots,q'_m]\subset C$.
	Then $g$ factors through a morphism $h\colon (\Spec C''',\cM)\rightarrow\wh\cZ'$, and induces a morphism $h'\colon (\Spec C''',\cM)\rightarrow\cZ'$.
	Here we denote again by $\cM$ the log structure associated with $S\rightarrow C''\rightarrow C'''$.
	Since $h=\wh{h'}$, we see that such a $g$ is unique if $f$ is log \'{e}tale.
\end{proof}

\begin{proposition}\label{prop: exactification is log etale}
	Let $j\colon Z\hookrightarrow\cZ$ be an immersion of fine weak formal log schemes and $\cZ^{\mathrm{ex}}_Z$ its exactification.
	Consider a commutative diagram
	\[\xymatrix{
	\cY'\ar[r]^-i\ar[d]^-{\theta'}&\cY\ar[d]^-\theta\\
\cZ_Z^{\mathrm{ex}}\ar[r]^-f&\cZ
	}\]
	where $f$ is the canonical morphism induced by adjointness, and $i$ is a homeomorphic exact closed immersion.
	Then there exists a unique morphism $g\colon\cY\rightarrow\cZ_Z^{\mathrm{ex}}$ such that $g\circ i=\theta'$ and $f\circ g=\theta$.
	In particular, $f\colon\cZ_Z^{\mathrm{ex}}\rightarrow\cZ$ is strongly log \'{e}tale. 
\end{proposition}

\begin{corollary}\label{cor: exactification smooth}
	Consider a commutative diagram of fine weak formal log schemes
	\[\xymatrix{
	Z'\ar@{^(->}[r]\ar[d]&\cZ'\ar[d]^-f\\
	Z\ar@{^(->}[r]&\cZ
	}\]
	whose horizontal morphisms are immersions.
	If $f$ is log smooth (resp.\ log \'{e}tale, strongly log smooth, strongly log \'{e}tale), then the induced morphism $\cZ'^{\mathrm{ex}}_{Z'}\rightarrow\cZ_Z^{\mathrm{ex}}$ between the exactifications is also log smooth (resp.\ log \'{e}tale, strongly log smooth, strongly log \'{e}tale).
\end{corollary}

Next we will give sufficient conditions for log smoothness and strong log smoothness.
The following statements are given similarly to the classical case for fine log schemes \cite[Prop.\,3.4]{Ka}.

\begin{lemma}\label{lem: monoid rings}
	For fine monoids $P$ and $Q$, let $\Spwf \wh R[P]^\dagger$ and $\Spwf\wh R[Q]^\dagger$ be the weak formal schemes endowed with the log structures associated to the canonical maps $P\rightarrow\wh R[P]^\dagger$ and $Q\rightarrow\wh R[Q]^\dagger$, respectively.
	Let $\alpha\colon Q\rightarrow P$ be a homomorphism and $f\colon \Spwf\wh R[P]^\dagger\rightarrow\Spwf\wh R[Q]^\dagger$ the morphism induced by $\alpha$.
	Consider the following conditions for $\alpha^{\mathrm{gp}}\colon Q^{\mathrm{gp}}\rightarrow P^{\mathrm{gp}}$:
	\begin{enumerate}
	\item\label{item: monoid ring 1} $\Ker\alpha^{\mathrm{gp}}$ and $(\Coker\,\alpha^{\mathrm{gp}})_{\mathrm{tors}}$ are finite groups whose orders are invertible in $\wh R$.
	\item\label{item: monoid ring 2} $\Ker\alpha^{\mathrm{gp}}$ and $\Coker\,\alpha^{\mathrm{gp}}$ are finite groups whose orders are invertible in $\wh R$.
	\item\label{item: monoid ring 3} $\alpha^{\mathrm{gp}}$ is injective and $\Coker\,\alpha^{\mathrm{gp}}$ is free.
	\item\label{item: monoid ring 4} $\alpha^{\mathrm{gp}}$ is bijective.
	\end{enumerate}
	Then the conditions \eqref{item: monoid ring 1}, \eqref{item: monoid ring 2}, \eqref{item: monoid ring 3}, \eqref{item: monoid ring  4}, respectively, imply that $f$ is log smooth, log \'{e}tale, strongly log smooth, and strongly log \'{e}tale.
\end{lemma}

\begin{proof}
	A commutative diagram as in \eqref{eq: test diagram log} with $\cZ'=\Spwf\wh R[P]^\dagger$ and $\cZ=\Spwf\wh R[Q]^\dagger$ corresponds to a commutative diagram
	\begin{equation}\label{eq: monoid rings diagram}
	\xymatrix{
	\Gamma(\cY',\cM_{\cY'})&\Gamma(\cY,\cM_\cY)\ar[l]\\
	P\ar[u]&Q\ar[u]\ar[l]
	}\end{equation}
	of monoids, and the lifting property of $f$ corresponds to the existence of a homomorphism $P\rightarrow\Gamma(\cY,\cM_\cY)$ fitting into \eqref{eq: monoid rings diagram}.
	Since $\cM_{\cY}\rightarrow\cM_{\cY}^{\mathrm{gp}}\times_{\cM_{\cY'}^{\mathrm{gp}}}\cM_{\cY'}$ is an isomorphism, it suffices to consider the groupification of \eqref{eq: monoid rings diagram} and discuss the existence of $P^{\mathrm{gp}}\rightarrow\Gamma(\cY,\cM_\cY^{\mathrm{gp}})$.
	Such a homomorphism clearly exists if $\alpha^{\mathrm{gp}}$ is injective and $\Coker\,\alpha^{\mathrm{gp}}$ is free, hence \eqref{item: monoid ring 3} implies that $f$ is strongly log smooth.
	Similarly one proves that \eqref{item: monoid ring 4} implies the strong log \'{e}taleness of $f$.
	That \eqref{item: monoid ring 1} and \eqref{item: monoid ring 2} imply log smoothness and log \'{e}taleness, respectively, is proved by the same argument as in \cite[Prop.\ 3.4]{Ka} or \cite[Thm.\ 3.1.8]{Og}.
\end{proof}

\begin{lemma}\label{lem: strict}
	Let $f\colon \cZ'\rightarrow\cZ$ be a morphism of fine weak formal log schemes and suppose that $f^*\cM_\cZ\rightarrow\cM_{\cZ'}$ is an isomorphism.
	Then $f$ is log smooth (resp.\ log \'{e}tale, strongly log smooth, strongly log \'{e}tale) if and only if its underlying morphism of weak formal schemes is smooth (resp.\ \'{e}tale, strongly smooth, strongly \'{e}tale).
\end{lemma}

\begin{proof}
	This follows similarly to the classical case (see \cite[Prop.\ 3.8]{Ka} or \cite[Prop.\ 3.1.6]{Og}).
\end{proof}

\begin{proposition}\label{prop: condition for log smooth}
	Let $f\colon \cZ'\rightarrow\cZ$ be a morphism of fine weak formal log schemes, and assume that a chart $Q_\cZ\rightarrow\cM_\cZ$ is given.
	Suppose there exists locally on $\cZ'$ a chart $(P_{\cZ'}\rightarrow\cM_{\cZ'},Q_{\cZ}\rightarrow\cM_\cZ,Q\xrightarrow{\alpha} P)$ extending the given chart of $\cZ$ satisfying the following conditions:
	\begin{enumerate}
	\item $\alpha^{\mathrm{gp}}\colon Q^{\mathrm{gp}}\rightarrow P^{\mathrm{gp}}$ satisfies the condition \eqref{item: monoid ring 1} (resp.\  \eqref{item: monoid ring 2}, \eqref{item: monoid ring 3}, \eqref{item: monoid ring  4}) in Lemma \ref{lem: monoid rings}, replacing ``invertible over $\wh R$'' by ``invertible over $\cZ'$''. 
	\item The induced morphism $\cZ'\rightarrow\cZ\times_{\Spwf \wh{R}[Q]^\dagger}\Spwf \wh{R}[P]^\dagger$ is smooth (resp.\ \'{e}tale, strongly smooth, strongly \'{e}tale).
	\end{enumerate}
	Then $f$ is log smooth (resp.\ log \'{e}tale, strongly log smooth, strongly log \'{e}tale).
\end{proposition}

\begin{proof}
	This follows by Proposition \ref{prop: easy properties log} \eqref{item: easy property 2} and \eqref{item: easy property 3}, Lemma \ref{lem: monoid rings} and Lemma \ref{lem: strict} similarly to \cite[Thm.\ 3.5]{Ka}.
\end{proof}

Finally, consider the case $(R,I)=(V,\frm)$.
We obtain the following proposition similarly to \ref{prop: lift strong smooth}
\begin{proposition}
	Consider a commutative diagram \eqref{eq: test diagram log} of fine weak formal log schemes with respect to $(V,\frm)$, and let $\cJ$ be an ideal of definition of $\cY$.
	For $k\in\bbN$ we endow $\cY[\frac{\cJ^k}p]^\dagger$ with the log structure given by the pull-back from $\cY$.
	If $f$ is strongly log smooth, then there exists locally on $\cY[\frac{\cJ^k}p]^\dagger$ a morphism $g_k\colon \cY[\frac{\cJ^k}p]^\dagger\rightarrow\cZ'$ such that the diagram
	\[\xymatrix{
	\cY'\times_\cY\cY[\frac{\cJ^k}p]^\dagger\ar@{^(->}[r]\ar[d]&\cY[\frac{\cJ^k}p]^\dagger\ar[d]\ar[ld]^-{g_k}\\
	\cZ'\ar[r]_-f&\cZ
	}\]
	commutes.
	If $f$ is strongly log \'{e}tale, such a $g_k$ exist globally and uniquely.
\end{proposition}
The following is given similarly to Proposition \ref{prop: Omega on generic fiber}:

\begin{proposition}\label{prop: omega compare}
	Let $f\colon\cZ'\rightarrow\cZ$ be a morphism of fine weak formal log schemes with respect to $(V,\frm)$, $\cJ$ an ideal of definition of $\cZ'$, and $k$ a positive integer.
	Endow $\cZ'_k:=\cZ'[\frac{\cJ^k}p]^\dagger$ with the log structure given by pull-back from $\cZ'$.
	Let $\frZ$, $\frZ'$, and $\frZ'_k$ be the dagger spaces associated to $\cZ$, $\cZ'$, and $\cZ'_k$, respectively.
	Then there exists a canonical isomorphism
	\begin{equation}\label{eq: omega on generic fiber}\omega^1_{\cZ'/\cZ,\bbQ}|_{\frZ'_k}\xrightarrow{\cong}\omega^1_{\cZ'_k/\cZ}\otimes_VK.\end{equation}
	Here the right hand side is defined in the same way as Proposition \ref{prop: Omega on generic fiber}.
\end{proposition}

In the situation in Proposition \ref{prop: omega compare}, we have an isomorphism
\[\omega^m_{\cZ'/\cZ,\bbQ}|_{\frZ'_k}=\Bigl(\bigwedge^m_{\cO_{\frZ'}}\omega^1_{\cZ'/\cZ,\bbQ}\Bigr)|_{\frZ'_k}\cong \bigwedge^m_{\cO_{\cZ'_k}\otimes_VK}(\omega^1_{\cZ'_k/\cZ}\otimes_VK)=\Bigl(\bigwedge^m_{\cO_{\cZ'_k}}\omega^1_{\cZ'_k/\cZ}\Bigr)\otimes_VK=\omega^m_{\cZ'_k/\cZ}\otimes_VK,\]
where $\bigwedge^m$ denotes the respective exterior power.
Since the differential $\omega^m_{\cZ'_k/\cZ}\rightarrow\omega^{m+1}_{\cZ'_k/\cZ}$ is $V$-linear, it induces a morphism $\omega^m_{\cZ'/\cZ,\bbQ}|_{\frZ'_k}\rightarrow\omega^{m+1}_{\cZ'/\cZ,\bbQ}|_{\frZ'_k}$.
Varying the positive integer $k$, we obtain a morphism $\omega^m_{\cZ'/\cZ,\bbQ}\rightarrow\omega^{m+1}_{\cZ'/\cZ,\bbQ}$.
This makes $\omega^\bullet_{\cZ'/\cZ,\bbQ}=\{\omega^m_{\cZ'/\cZ,\bbQ}\}_{m\in\bbN}$ into a complex.

%%%%%%%%%%%%%%%%%
%
\section{Log rigid cohomology}\label{Sec: Log rig coh}
%
%%%%%%%%%%%%%%%%%

In this section we discuss the theory of log rigid cohomology using weak formal log schemes.
Let $k$ be a perfect field of characteristic $p>0$.
Denote by $W = W(k)$  the ring of Witt vectors of $k$, by $F: = \mathrm{Frac}(W)$ 
its fraction field and by $\sigma\colon W  \rightarrow  W$ the Frobenius morphism. 

In what follows, we always consider $(W,pW)$ as the base $(R,I)$, so all weak formal schemes are taken with respect to $(W,pW)$.
If a morphism $\cZ'\rightarrow\cZ$ of weak formal (log) schemes with respect to $(W,pW)$ is given, we sometimes say that ``$\cZ'$ is a weak formal (log) scheme over $\cZ$''.
Moreover, for a fine log scheme over $k$ we always assume that the underlying scheme is separated and locally of finite type over $k$.

\begin{definition}
	Let $T$ be a fine log scheme over $k$, $\cT$ be a fine weak formal log scheme over $W$, and $\iota\colon T\hookrightarrow\cT$ be a homeomorphic exact closed immersion.
	Assume that the underlying weak formal scheme of $\cT$ is flat over $W$.
	A {\it log rigid triple} over $\iota\colon T\hookrightarrow\cT$ is a triple $(Y,\cZ,i)$ consisting of a fine log scheme $Y$, a fine weak formal log scheme $\cZ$ which is strongly log smooth over $\cT$, and a homeomorphic exact closed immersion $i\colon Y\hookrightarrow\cZ$ over $\cT$.
	In this situation, we also call $(\cZ,i)$ a {\it log rigid datum} for $Y$ over $\iota\colon T\hookrightarrow\cT$.
	A morphism of log rigid triples $f = (f_T,f_{\cT})\colon(Y',\cZ',i') \rightarrow (Y,\cZ,i)$ is a pair consisting of morphisms $f_T\colon Y' \rightarrow  Y$ over $T$ and $f_\cT\colon\cZ' \rightarrow \cZ$ such that $f_\cT\circ i = i'\circ f_T$.
	A morphism of log rigid data for $Y$ is a morphism of log rigid triples of the form $(\mathrm{id}_Y,f_\cT)$.
	
	Furthermore assume that a lift $\cT \rightarrow \cT$ of the $p$\textsuperscript{th} power Frobenius on $\cT\times_{W^\varnothing}k^\varnothing$ compatible with $\sigma\colon W^\varnothing\rightarrow W^\varnothing$ is given, and denote it again by $\sigma$.
	A {\it log rigid $F$-quadruple} over $(\iota\colon T\hookrightarrow\cT,\sigma)$ is a quadruple $(Y,\cZ,i,\phi)$ consisting of a log rigid triple $(Y,\cZ,i)$ on $\iota\colon T\hookrightarrow\cT$ and a lift $\phi\colon\cZ \rightarrow \cZ$ of the $p$\textsuperscript{th} power Frobenius on $\cZ\times_{W^\varnothing}k^\varnothing$ which is compatible with $\sigma$ on $\cT$.
	In this situation, we also call $(\cZ,i,\phi)$ a {\it log rigid $F$-datum} over $(\iota\colon T\hookrightarrow\cT,\sigma)$ for $Y$.
	A morphism of log rigid $F$-quadruples (respectively log rigid $F$-data) is a morphism of log rigid triples (respectively log rigid data) which is compatible with the Frobenius lifts.
	
	We often omit the expressions ``over $\iota\colon T\hookrightarrow\cT$" and ``over $(\iota\colon T\hookrightarrow\cT,\sigma)$" if there is no ambiguity.
\end{definition}

\begin{proposition}
	Let $Y$ be a fine log scheme over $T$.
	\begin{enumerate}
	\item
	The category of log rigid data for $Y$ and the category of log rigid $F$-data for $Y$ have direct products.
	\item
	If $(\cZ'',i'')=(\cZ,i)\times(\cZ',i')$ in the category of log rigid data for $Y$, then the projections $\cZ''\rightarrow\cZ$ and $\cZ''\rightarrow\cZ'$ are strongly log smooth.
	A similar statement holds for products of log rigid $F$-data.
	\end{enumerate}
\end{proposition}

\begin{proof}
	For two log rigid data $(\cZ,i)$ and $(\cZ',i')$, let $i''\colon Y\hookrightarrow\cZ''$ be the 
canonical  
exactification of the diagonal embedding $Y\hookrightarrow\cZ\times_{\cT}\cZ'$ 
(cf.\,Proposition-Definition \ref{def: exactification}).
	Then $(\cZ'',i'')$ is the direct product of $(\cZ,i)$ and $(\cZ',i')$ in the category of log rigid data for $Y$.
	
	If $\phi$ and $\phi'$ are Frobenius lifts on $\cZ$ and $\cZ'$, then $(\phi,\phi')\colon\cZ\times_{\cT}\cZ' \rightarrow \cZ\times_{\cT}\cZ'$ is compatible with the $p$\textsuperscript{th} power Frobenius on $Y$, and hence it extends to a morphism $\phi''\colon\cZ'' \rightarrow \cZ''$ by functoriality of exactifications.
	Then $(\cZ'',i'',\phi'')$ is the direct product of $(\cZ,i,\phi)$ and $(\cZ',i',\phi')$ in the category of log rigid $F$-data for $Y$.
	
	The strong log smoothness of projections follows from Proposition \ref{prop: easy properties log} \eqref{item: easy property 2}, \eqref{item: easy property 3} and Proposition \ref{prop: exactification is log etale}.
\end{proof}

\begin{definition}\label{def: rig cohomology / complex}
	Let $(Y,\cZ,i)$ be a log rigid triple and $\frZ$ the dagger space associated to $\cZ$.
	We denote by $\omega^\bullet_{\cZ/\cT,\bbQ}$ the complex on $\frZ$ given at the end of the previous section.
	We define log rigid cohomology of $Y$ over $T\hookrightarrow\cT$ with respect to $(\cZ,i)$ to be
		\[R\Gamma_\rig(Y/\cT)_\cZ: = R\Gamma(\frZ,\omega^\bullet_{\cZ/\cT,\bbQ}).\]
Note that since the underlying Grothendieck topological space of a
dagger space is defined to be that of a rigid analytic space, 
we may consider $R\Gamma_\rig(Y/\cT)_\cZ$ as a complex via generalised Godement resolution (c.f.\,\cite[\S3]{CCM}).
\end{definition}

\begin{remark}
	Let $Y$ be a fine log scheme over $T$, $\cZ$ a weak formal log scheme which is strongly log smooth over $\cT$.
	Let  $Y\hookrightarrow\cZ$ be an exact closed immersion (not necessarily homeomorphic) over $\cT$, 
	and $Y\hookrightarrow\cZ'$ its exactification given in Proposition-Definition \ref{def: exactification}.
	Note that in this case, this is just the weak completion of $\cZ$ along $Y$.
	Denote by $\frZ$ and $\frZ'$ the dagger spaces associated to $\cZ$ and $\cZ'$ respectively.
	Then the canonical morphism $\cZ'\rightarrow\cZ$ induces isomorphisms $\frZ'\xrightarrow{\cong}]Y[_\frZ$ by \cite[(0.2.7)]{Ber} and \cite[Thm.\,2.19 (4)]{GK1}.
	Moreover we have $\omega^\bullet_{\cZ/\cT,\bbQ}|_{]Y[_\frZ}\xrightarrow{\cong}\omega^\bullet_{\cZ'/\cT,\bbQ}$ and hence $R\Gamma(]Y[_\frZ,\omega^\bullet_{\cZ/\cT,\bbQ})\xrightarrow{\cong}R\Gamma(\frZ',\omega^\bullet_{\cZ'/\cT,\bbQ})$.
	This shows that our definition of log rigid cohomology coincides with that of Gro\ss{}e-Kl\"{o}nne in \cite[\S 1]{GK3}.
\end{remark}

The following result appeared in the proof of \cite[Lem.\,1.4]{GK3}, but we give a detailed proof for the benefit of the reader.

\begin{proposition}\label{prop: log rigid coh}
	Let $Y$ be a fine log scheme over $T$.
	If a log rigid datum for $Y$ exists, the log rigid cohomology $R\Gamma_\rig(Y/\cT)_\cZ$ is independent of the choice of log rigid datum up to canonical quasi-isomorphisms.
\end{proposition}

\begin{proof}
	For two log rigid data $(\cZ_1,i_1)$ and $(\cZ_2,i_2)$ for $Y$, let $(\cZ',i')$ be their product, and denote by $\frZ_1$, $\frZ_2$, and $\frZ'$ the associated dagger spaces.
	Then by Proposition \ref{prop: strong fibration}, the morphisms of dagger spaces $\frZ'\rightarrow \frZ_l$ ($l=1,2$) induced by the natural projections $\cZ'\rightarrow\cZ_l$ ($l=1,2$) are relative open polydisks.
	Thus as in the proof of \cite[Prop.\,2.2.14]{Sh1} one can use the Poincar\'e lemma to show that the morphisms 
	\begin{equation}\label{independent}
	R\Gamma(\frZ_l,\omega^\bullet_{\cZ_l/\cT,\bbQ})\rightarrow R\Gamma(\frZ',\omega^\bullet_{\cZ'/\cT,\bbQ}),\quad l=1,2
	\end{equation}
	  induced by the projections are quasi-isomorphisms (see also \cite[Lem.\,1.4]{GK3}): 

Fix $l=1$ or $2$.
We can find an open covering $\{\cU_\lambda\}_\lambda$ of $\cZ_l$ with isomorphisms
\[\cU'_\lambda:=\cZ'\times_{\cZ_l}\cU_\lambda\cong\cU_\lambda\times\Spwf W\llbracket s_1,\ldots,s_d\rrbracket\]
such that the composition $\cU_\lambda\hookrightarrow \cZ_l\rightarrow\cT$ factors through an affine open subset $\cT_\lambda\subset\cT$.
Using the spectral sequence associated with this covering, we reduce to show that 
$R\Gamma(\frZ,\omega^\bullet_{\cZ/\cT_\lambda,\bbQ})\rightarrow R\Gamma(\frZ',\omega^\bullet_{\cZ'/\cT_\lambda,\bbQ})$
is a quasi-isomorphism for $\cZ'\cong \cZ\times W\llbracket s_1,\ldots,s_d\rrbracket$ and $\frZ$, $\frZ'$ the associated dagger spaces.

By induction on $d$, we may further suppose that $d=1$.
Moreover, by restricting to an affinoid subspace of $\frZ$, we may assume that $\cT=\Spwf A$ and $\cZ=\Spwf B$ are $p$-adic and affine, and hence $\cZ'$ is of the form $\Spwf B'$ with $B':=B\otimes^\dagger W\llbracket s\rrbracket$. 
Note that, this reduction preserves the local freeness of $\omega^1_{\cZ/\cT,\bbQ}$, but it breaks the log smoothness of $\cZ$ over $\cT$.

Then the left hand side of \eqref{independent} is computed by $\omega^\bullet_{B/A}\otimes\bbQ$ by \cite[Prop.\,3.1]{GK1}.
For $r\geq 1$, set $B'_r:=B\otimes^\dagger W[s,\frac{s^r}p]^\dagger$ and $\frZ'_r:=\Sp (B'_r\otimes\bbQ)$.
Since $\frZ'=\bigcup_r\frZ'_r$, the right hand side of \eqref{independent} is $R\varprojlim_r(\omega^\bullet_{B'_r/A}\otimes\bbQ)$.
Thus, to prove that \eqref{independent} is a quasi-isomorphism, it suffices to show that
\begin{equation}\label{eq: independent 2}\omega^\bullet_{B/A}\otimes\bbQ\rightarrow\omega^\bullet_{B'_r/A}\otimes\bbQ\end{equation}
is a quasi-isomorphism for any $r$.

By Proposition \ref{prop: fundamental properties of omega} \eqref{item: omega exact sequence},\eqref{item: omega product} we have $\omega^1_{B'_r/A}=B'_r\otimes_B\omega^1_{B/A}\oplus B'_rds$, and this shows that
\[\omega^\bullet_{B'_r/A}\otimes\bbQ=\Cone(B'_r\otimes_B\omega^\bullet_{B/A}\otimes\bbQ\xrightarrow{\partial_s} B'_rds \otimes_B\omega^\bullet_{B/A}\otimes\bbQ)[-1]\]
where the map $\partial_s$ is defined by differentiating an element of $B'_r$ with respect to $s$.
Thus to show that \eqref{eq: independent 2} is a quasi-isomorphism, it is enough to prove that for fixed $i\geq 0$ the sequence
\[0\rightarrow\omega^i_{B/A}\otimes\bbQ\rightarrow B'_r\otimes_B\omega^i_{B/A}\otimes\bbQ\xrightarrow{\partial_s}B'_r ds\otimes_B\omega^i_{B/A}\otimes\bbQ\rightarrow 0\]
is exact.
Since $\omega^i_{B/A}\otimes\bbQ$ is locally free over $B\otimes\bbQ$, it suffices to show that
\begin{equation}\label{eq: relative disk}
0\rightarrow B\otimes\bbQ\rightarrow B'_r\otimes\bbQ\xrightarrow{\partial_s} B'_rds\otimes\bbQ\rightarrow 0
\end{equation}
is exact.
Note that $B'_r\otimes\bbQ$ is a subring of $(B\otimes\bbQ)\llbracket s\rrbracket$, and we have
\[\Ker\partial_s=(B'_r\otimes\bbQ)\cap\Ker((B\otimes\bbQ)\llbracket s\rrbracket\rightarrow (B\otimes\bbQ)\llbracket s\rrbracket ds)=B\otimes\bbQ.\]
To show that $\partial_s$ is surjective, take a surjection $W[t_1,\ldots,t_m]^\dagger\rightarrow B$.
Then an element $f\in B'_r\otimes\bbQ$ can be represented as
\[f=\sum_{\alpha=(\alpha_1,\ldots,\alpha_m)\in\bbN^m}\sum_{n\in\bbN}c_{\alpha,n}t_1^{\alpha_1}\cdots t_m^{\alpha_m}s^n\]
with $c_{\alpha,n}\in F$, such that
\[\lvert c_{\alpha,n}\rvert\rho^{\lvert\alpha\rvert+n}p^{-\frac nr}\rightarrow 0\hspace{15pt}(\lvert\alpha\rvert+n\rightarrow\infty)\]
for some real number $\rho>1$, where we denote $\lvert\alpha\rvert:=\alpha_1+\cdots+\alpha_m$.
Take $\rho'\in\bbR$ with $1<\rho'<\rho$.
Then the series
\[g:=\sum_{\alpha=(\alpha_1,\ldots,\alpha_m)\in\bbN^m}\sum_{n\in\bbN}\frac{c_{\alpha,n}}{n+1}t_1^{\alpha_1}\cdots t_m^{\alpha_m}s^{n}\]
defines an element of $B'_r\otimes\bbQ$, too, since we have
\[\left\lvert\frac{c_{\alpha,n}}{n+1}\right\rvert\rho'^{\lvert\alpha\rvert+n}p^{-\frac nr}=\lvert c_{\alpha,n}\rvert\rho^{\lvert\alpha\rvert+n}p^{-\frac nr}\cdot \frac{1}{\lvert n+1\rvert}\left(\frac{\rho'}{\rho}\right)^{n+1}\cdot \left(\frac{\rho'}{\rho}\right)^{\lvert\alpha\rvert-1}\rightarrow 0\hspace{15pt}(\lvert\alpha\rvert+n\rightarrow\infty)\]
where we use that $\frac{1}{\lvert n+1\rvert}(\frac{\rho'}{\rho})^{n+1}$ is bounded.
As $g$ was chosen such that $\partial_s(sg)=f$, we proved that $\partial_s$ is surjective.
\end{proof}

In the case that $\cT$ is affine and admits a chart, the above result allows us to globalise our construction:

\begin{lemma}\label{lem: local embedding}
	Suppose that $\cT$ is affine and admits a chart.
	Then for any fine log scheme $Y$ over $T$, there exists locally on $Y$ a log rigid datum for $Y$ over $\iota\colon T\hookrightarrow\cT$.
\end{lemma}

\begin{proof}
	We arrange the proof of \cite[Prop.\ 2.2.11]{Sh1}.
	Since the statement is local, we may assume that $Y=\Spec C$ is affine and there exists a chart $(P_Y\rightarrow\cM_Y,Q\rightarrow\cM_{\cT},Q\xrightarrow{\alpha}P)$ of the composite $Y\rightarrow T\hookrightarrow\cT$.
	We set $A:=\Gamma(\cT,\cO_{\cT})$ and $B:=\Gamma(T,\cO_T)$.
	Let $f\colon A\rightarrow C$, $\beta\colon P\rightarrow C$ and $\gamma\colon Q\rightarrow A$ the natural maps.
	Take surjections
	\begin{align*}
	\delta\colon B[\bbN^n]\rightarrow C,&&\epsilon\colon \bbN^m\rightarrow P.
	\end{align*}
	
	Let $\cZ':=\Spwf A[\bbN^m\oplus\bbN^n]^\dagger$ and endow it with the log structure associated to the map $\eta\colon\bbN^m\oplus Q\rightarrow A[\bbN^m\oplus\bbN^n]^\dagger$ defined by $\eta(x,q):=x\cdot\gamma(q)$.
	Then the natural injection $Q\rightarrow \bbN^m\oplus Q$ defines a morphism $\cZ'\rightarrow\cT$.
	The underlying weak formal scheme of $\cZ'$ is strongly smooth over $\cT\times_{\Spwf \bbZ_p[Q]^\dagger}\Spwf\bbZ_p[\bbN^m\oplus Q]^\dagger=\Spwf A[\bbN^m]^\dagger$ by Lemma \ref{lem: disc smooth}.
	Thus by Proposition \ref{prop: condition for log smooth} we see that $\cZ'$ is strongly log smooth over $\cT$.
	
	We have a commutative diagram
	\[\xymatrix{
	C&A[\bbN^m\oplus\bbN^n]^\dagger\ar[l]\\
	P\ar[u]^-\beta&\bbN^m\oplus Q\ar[u]^-\eta\ar[l]
	}\]
	where the upper horizontal arrow is defined by $a(x,y)\mapsto f(a)\cdot \beta\circ\epsilon(x)\cdot\delta(y)$ for any $a\in A$, $x\in\bbN^m$, $y\in\bbN^n$, and the lower horizontal arrow is defined by $(x,q)\mapsto \epsilon(x)\cdot \alpha(q)$.
	This induces a closed immersion $Y\hookrightarrow\cZ'$.
	
	Let $\cZ$ be the exactification of $Y\hookrightarrow\cZ'$, then it is strongly log smooth over $\cT$ by Corollary \ref{cor: exactification smooth}, and hence give a log rigid datum for $Y$.
\end{proof}

\begin{definition}\label{def: simplicial rig coh}
	Suppose that $\cT$ is affine and admits a chart.
	For a fine log scheme $Y$ over $T$, we may find an affine cover $\{U_j\rightarrow Y\}_{j\in J}$, such that a log rigid datum $(\cZ_j,i_j)$ for $U_j$ exists. 
	Then 
	$$
	(U_m,\cZ_m, i_m):= \coprod_{\underline{j}\in J^{m+1}} (U_{\underline{j}},\cZ_{\underline{j}}, i_{\underline{j}})
	$$
where for $\underline{j}=(j_0,\ldots,j_{m})$ we set $U_{\underline{j}}:= U_{j_0}\cap\cdots\cap U_{j_{m}}$, 
and let $i_{\underline{j}}:U_{\underline{j}} \hookrightarrow \cZ_{\underline{j}}$ be the exactification of the diagonal embedding
$U_{\underline{j}} \hookrightarrow \cZ_{j_0}\times_{\cT} \cdots\times_{\cT} \cZ_{j_m}$,
provides a simplicial log rigid triple $(U_\bullet,\cZ_\bullet,i_\bullet)$ over $\iota:T\hookrightarrow \cT$ such that $U_\bullet$ is a Zariski hypercovering of $Y$.
We define $R\Gamma_\rig(Y/\cT)$ to be the complex associated to the cosimplicial complex $R\Gamma_\rig(U_\bullet/\cT)_{\cZ_\bullet}$ where we regard $R\Gamma_\rig(U_m/\cT)_{\cZ_m}$ for each $m\geq 0$ as a complex by using Godement resolution.
\end{definition}

\begin{proposition}\label{prop: independence of global log rigid cohomology}
	The log rigid cohomology $R\Gamma_\rig(Y/\cT)$ is in the derived category independent of the choice of the covering $\{U_j\}$ and log rigid data $(\cZ_j,i_j)$.
\end{proposition}

\begin{proof}
	Suppose that another choice of a covering $\{U'_{j'}\}_{j'\in J'}$ of $Y$ and log rigid data $(\cZ'_{j'},i'_{j'})$ for $U'_{j'}$ are given.
	For $m,n\geq 0$ let
	\[(U_{m,n},\cZ_{m,n},i_{m,n}):=\coprod_{(\underline{j},\underline{j'})\in J^{m+1}\times J'^{n+1}}(U_{\underline{j},\underline{j'}},\cZ_{\underline{j},\underline{j'}},i_{\underline{j},\underline{j'}})\]
	where we set $U_{\underline{j},\underline{j'}}:=U_{j_0}\cap\cdots U_{j_m}\cap U'_{j'_0}\cap\cdots\cap U'_{j'_n}$, and let $i_{\underline{j},\underline{j'}}\colon U_{\underline{j},\underline{j'}}\hookrightarrow\cZ_{\underline{j},\underline{j'}}$ be the exactification of the diagonal embedding $U_{\underline{j},\underline{j'}}\hookrightarrow\cZ_{j_0}\times_\cT\cdots\times_\cT\cZ_{j_m}\times_\cT\cZ'_{j'_0}\times_\cT\cdots\times_\cT\cZ'_{j'_n}$.
	Then we get a bisimplicial log rigid triple $(U_{\bullet,\bullet},\cZ_{\bullet,\bullet},i_{\bullet,\bullet})$ with canonical morphisms
	\begin{equation}\label{eq: bicosimplicial}
	s(R\Gamma_\rig(U_\bullet/\cT)_{\cZ_\bullet})\rightarrow s(R\Gamma_\rig(U_{\bullet,\bullet}/\cT)_{\cZ_{\bullet,\bullet}})\leftarrow s(R\Gamma_\rig(U'_\bullet/\cT)_{\cZ'_\bullet})
	\end{equation}
	where $s(-)$ denotes the single complex associated with a cosimplicial or a bicosimplicial complex.
	To see that the first morphism in \eqref{eq: bicosimplicial} is a quasi-isomorphism, it suffices to prove that
	\begin{equation}\label{eq: bicosimplicial 2}
	R\Gamma_\rig(U_m/\cT)_{\cZ_m}\rightarrow s(R\Gamma_\rig(U_{m,\bullet}/\cT)_{\cZ_{m,\bullet}})
	\end{equation}
	is a quasi-isomorphism for any $m$.
	For $\underline{j}\in J^{m+1}$ and $n\geq 0$, let
	\[(U_{\underline{j},n},\cZ_{\underline{j},n},i_{\underline{j},n}):=\coprod_{\underline{j'}\in J'^{n+1}}(U_{\underline{j},\underline{j'}},\cZ_{\underline{j},\underline{j'}},i_{\underline{j},\underline{j'}}).\]
	Then \eqref{eq: bicosimplicial 2} is written as the direct product of morphisms
	\begin{equation}\label{eq: bicosimplicial 3}
	R\Gamma_\rig(U_{\underline{j}}/\cT)_{\cZ_{\underline{j}}}\rightarrow s(R\Gamma_\rig(U_{\underline{j},\bullet}/\cT)_{\cZ_{\underline{j},\bullet}}),
	\end{equation}
	hence it suffices to show that \eqref{eq: bicosimplicial 3} is a quasi-isomorphism.
	For each $n\geq 0$ and $\underline{j'}\in J'^{n+1}$, let $U_{\underline{j},\underline{j'}}\hookrightarrow\cU_{\underline{j},\underline{j'}}$ be the exactification of the composition $U_{\underline{j},\underline{j'}}\hookrightarrow U_{\underline{j}}\hookrightarrow\cZ_{\underline{j}}$, and set $\cU_{\underline{j},n}:=\coprod_{\underline{j'}\in J'^{n+1}}\cU_{\underline{j},\underline{j'}}$.
	Note that $\cU_{\underline{j},\underline{j}'}$ is just the open weak formal log subsecheme of $\cZ_{\underline{j}}$ whose underlying topological set is $U_{\underline{j},\underline{j'}}$.
	Then there are canonical morphisms $\cZ_{\underline{j},n}\rightarrow\cU_{\underline{j},n}$, and \eqref{eq: bicosimplicial 3} factors into the composition
	\[R\Gamma_\rig(U_{\underline{j}}/\cT)_{\cZ_{\underline{j}}}\rightarrow s(R\Gamma_\rig(U_{\underline{j},\bullet}/\cT)_{\cU_{\underline{j},\bullet}})\rightarrow
	s(R\Gamma_\rig(U_{\underline{j},\bullet}/\cT)_{\cZ_{\underline{j},\bullet}}).\]
	The first morphism is a quasi-isomorphism by the cohomological descent for an admissible open covering of a dagger space, and the second morphism is also a quasi-isomorphism by Proposition \ref{prop: log rigid coh}.
	This finishes the proof that the first map of \eqref{eq: bicosimplicial} is a quasi-isomorphism, and similarly the second morphism is also a quasi-isomorphism.
	Thus \eqref{eq: bicosimplicial} gives an isomorphism in the derived category.
	One can see that this construction of an isomorphism satisfies the cocycle condition by passing through the complex associated with a trisimplicial log rigid triple.
\end{proof}

\begin{remark}
	By using Besser's construction as in \cite[\S 2]{EY} with a slight modification, one can define $R\Gamma_\rig(Y/\cT)$ as a canonical complex (not only on the level of cohomology).
\end{remark}

\begin{lemma}\label{lem: product cohomology}
	Let $\frZ$ be a dagger space and $\{\cF_\lambda\}_{\lambda\in\Lambda}$ a family of coherent sheaves on $\frZ$.
	Then for any $n\in\bbZ$ we have
	\[H^n\bigl(\frZ,\prod_{\lambda\in\Lambda}\cF_\lambda\bigr)\cong \prod_{\lambda\in\Lambda}H^n(\frZ,\cF_\lambda),\]
	where the product $\prod_{\lambda\in\Lambda}\cF_\lambda$ is given in the category of sheaves of abelian groups. 
\end{lemma}

\begin{proof}
     Taking into account the vanishing theorem for coherent sheaves on affinoid dagger spaces \cite[Prop.\ 3.1]{GK1}, the statement can be proved similarly as \cite[Lem.\ 3.8]{BHYY} using dagger spaces instead of rigid spaces.
\end{proof}

\begin{proposition}\label{prop: Mayer-Vietoris}
	Let $Y$ be a fine log scheme over $T$ and $\{Y_\lambda\}_{\lambda\in\Lambda}$ be an open covering of $Y$.
	Then there exists a spectral sequence
	\[E_1^{p,q}=\prod_{(\lambda_0,\ldots,\lambda_p)\in\Lambda^{p+1}}H^q_\rig(Y_{\lambda_0}\cap\cdots\cap Y_{\lambda_p}/\cT)\Rightarrow H^{p+q}_\rig(Y/\cT),\]
	where we endow each $Y_{\lambda_0}\cap\cdots\cap Y_{\lambda_p}$ with a log structure by restriction from $Y$.
\end{proposition}

\begin{proof}
Let $\{U_j\rightarrow Y\}_{j\in J}$ be an affine cover with log rigid data $(\cZ_j,i_j)$, and define the log rigid triples $(U_{\underline{j}},\cZ_{\underline{j}},i_{\underline{j}})$ for $m\in\bbN$ and $\underline{j}\in J^{m+1}$ as in Definition \ref{def: simplicial rig coh}.
	For $\underline{j}=(j_0,\ldots,j_m)\in J^{m+1}$ and $\underline{\lambda}=(\lambda_0,\ldots,\lambda_n)\in\Lambda^n$, set $U_{\underline{j},\underline{\lambda}}:=U_{j_0}\cap\cdots\cap U_{j_m}\cap Y_{\lambda_0}\cap\cdots\cap Y_{\lambda_n}$ and let $\cZ_{\underline{j},\underline{\lambda}}$ be the weak formal open log subscheme of $\cZ_{\underline{j}}$ whose underlying open subset is $U_{\underline{j},\underline{\lambda}}$.
	Then $i_{\underline{j}}$ induces a homeomorphic exact closed immersion $i_{\underline{j},\underline{\lambda}}\colon U_{\underline{j},\underline{\lambda}}\hookrightarrow \cZ_{\underline{j},\underline{\lambda}}$.
	For $m,n\in\bbN$ let $(U_{m,n},\cZ_{m,n},i_{m,n}):=\coprod_{\underline{j}\in J^{m+1}}\coprod_{\underline{\lambda}\in\Lambda^{n+1}}(U_{\underline{j},\underline{\lambda}},\cZ_{\underline{j},\underline{\lambda}},i_{\underline{j},\underline{\lambda}})$.
	
	Then the log rigid cohomology $R\Gamma(Y/\cT)$ can be computed as the complex associated to the bicosimplicial complex $R\Gamma_\rig(U_{\bullet,\bullet}/\cT)_{\cZ_{\bullet,\bullet}}$.
	Thus we obtain a spectral sequence 
$$
E_1^{p,q}\Rightarrow H^{p+q}_\rig(Y/\cT)
$$ 
where $E_1^{p,q}$ is the $q$\textsuperscript{th}cohomology group of the complex associated to the cosimplicial complex $R\Gamma_\rig(U_{\bullet,p}/\cT)_{\cZ_{\bullet,p}}$.
	However we have 
	\[E_1^{p,q}\cong H^q_\rig\left(\coprod_{(\lambda_0,\ldots,\lambda_p)\in\Lambda^{p+1}}Y_{\lambda_0}\cap\cdots\cap Y_{\lambda_p}/\cT\right)
	\cong\prod_{(\lambda_0,\ldots,\lambda_p)\in\Lambda^{p+1}}H^q_\rig(Y_{\lambda_0}\cap\cdots\cap Y_{\lambda_p}/\cT),\]
	where the first isomorphism follows from the definition of the log rigid cohomology and the second isomorphism follows from Lemma \ref{lem: product cohomology}.
\end{proof}

%%%%%%%%%%%%%%%%%%%%%%%%%%%%%
%
\section{Rigid Hyodo--Kato cohomology}\label{Sec: Rig HK coh}
%
%%%%%%%%%%%%%%%%%%%%%%%%%%%%%

In this section we define rigid Hyodo--Kato cohomology for any fine log scheme over $k^0$, and study it more deeply in the semistable case.
Continuing from the previous section, we always suppose that any fine log scheme over $k$ is separated and locally of finite type.
From now on we will use the following conventions:

\begin{notation}\label{main notations}
Let $K$ be a $p$-adic field with perfect residue field $k$. Denote by $V$ the ring of integers of $K$ and by $\frm$ the maximal ideal of $V$.  Let $e$ be the ramification index of $K$ over $\bbQ_p$. 
Denote by $W = W(k)$  the ring of Witt vectors of $k$, by $F: = \mathrm{Frac}(W)$ its fraction field and by $\sigma\colon W  \rightarrow  W$ the Frobenius morphism. 

Consider the weak formal scheme $\mathcal{S}: = \Spwf W\llbracket s\rrbracket$ with log structure associated to the map $\bbN \rightarrow  W\llbracket s\rrbracket;\ 1\mapsto s$.
Let $\tau\colon k^0\hookrightarrow\cS$ be the exact closed immersion defined by the ideal $(p,s)$.
Let $V^\sharp$ be $\Spwf V$ endowed with the canonical log structure, that is $\Gamma(\Spwf V,\cM_{V^\sharp}) = V\setminus\{0\}$. 
For any pseudo-wcfg $W$-algebra $A$, let $A^\varnothing$ be the weak formal log scheme whose underlying space is $\Spwf A$ and whose log structure is trivial.
Let furthermore  $\sigma\colon W^\varnothing \rightarrow  W^\varnothing$ be the Frobenius endomorphism, and extend it to $\sigma\colon\cS \rightarrow \cS$ by $s\mapsto s^p$.
\end{notation}

To define the rigid Hyodo--Kato complex, we use and enhance a construction due to Kim and Hain \cite[pp.~1259--1260]{KH}.

\begin{definition}\label{def: HK coh local}
	Let $(Y,\cZ,i,\phi)$ be a log rigid $F$-quadruple over $(k^0\hookrightarrow \cS,\sigma)$.
	We also regard $(Y,\cZ,i,\phi)$ as a log rigid $F$-quadruple over $(k^\varnothing\hookrightarrow W^\varnothing,\sigma)$.
	Let $\mathcal{Y}: = \mathcal{Z}\times_{\mathcal{S}}W^0$ be the fibre of $s = 0$, and denote by $\mathfrak{Z}$ and $\mathfrak{Y}$ the dagger spaces associated to $\cZ$ and $\cY$ respectively.
	We set $\widetilde{\omega}^\bullet_{\cY,\bbQ}: = \omega^\bullet_{\cZ/W^\varnothing,\bbQ}\otimes_{\cO_\frZ}\mathcal{O}_{\mathfrak{Y}}$.
	
	Let $\omega^\bullet_{\cZ/W^\varnothing,\bbQ}[u]$ (respectively $\widetilde{\omega}^\bullet_{\cY,\bbQ}[u]$) be the CDGA generated by $\omega^\bullet_{\cZ/W^\varnothing,\bbQ}$ (respectively $\widetilde{\omega}_{\cY,\bbQ}^\bullet$) and degree zero elements $u^{[i]}$ for $i\geq 0$ with the relations
	\[
		du^{[i + 1]} = -d\log s\cdot u^{[i]} \qquad \text{ and } \qquad u^{[0]} = 1,
	\]
and the multiplication defined by
	\begin{align*}
		u^{[i]}\wedge u^{[j]} = \frac{(i + j)!}{i!j!}u^{[i + j]}.
	\end{align*}
	A Frobenius action on $\omega_{\cZ/W^\varnothing,\bbQ}^\bullet[u]$ (respectively $\widetilde{\omega}_{\cY,\bbQ}^\bullet[u]$) is defined by the Frobenius action on $\omega_{\cZ/W^\varnothing,\bbQ}^\bullet$ (respectively $\widetilde{\omega}_{\cY,\bbQ}^\bullet$) and $\phi(u^{[i]}): = p^iu^{[i]}$.
	The monodromy operator $N$ is the $\cO_\frZ$-linear (respectively $\cO_\frY$-linear) endomorphism on $\omega_{\cZ/W^\varnothing,\bbQ}^\bullet[u]$ (respectively $\widetilde{\omega}_{\cY,\bbQ}^\bullet[u]$) defined by $u^{[i]}\mapsto u^{[i-1]}$.
	We call $\omega^\bullet_{\cZ/W^\varnothing,\bbQ}[u]$ (respectively $\widetilde{\omega}^\bullet_{\cY,\bbQ}[u]$) the {\it Kim--Hain complex} over $\frZ$ (respectively $\frY$).
	
	For $m\geq 0$, we let $\omega^\bullet_{\cZ/W^\varnothing,\bbQ}[u]_m$ (respectively $\widetilde{\omega}_{\cY,\bbQ}^\bullet[u]_m$) be the subcomplex of $\omega^\bullet_{\cZ/W^\varnothing,\bbQ}[u]$ (respectively $\widetilde{\omega}_{\cY,\bbQ}^\bullet[u]$) whose degree $n$ component consists of sections of the form $\sum_{i=0}^m\eta_iu^{[i]}$ with $\eta_i\in\omega^n_{\cZ/W^\varnothing,\bbQ}$ (respectively $\eta_i\in\widetilde{\omega}^n_{\cY,\bbQ}$).
	Note that they form a directed system and we have 
	\begin{align}\label{equ:directed}
	\omega^\bullet_{\cZ/W^\varnothing,\bbQ}[u]=\varinjlim_m\omega^\bullet_{\cZ/W^\varnothing,\bbQ}[u]_m, && \widetilde{\omega}_{\cY,\bbQ}^\bullet[u]=\varinjlim_m\widetilde{\omega}_{\cY,\bbQ}^\bullet[u]_m.
	\end{align}	
	We define {\it completed Kim--Hain complexes} $\omega^\bullet_{\cZ/W^\varnothing,\bbQ}\llbracket u\rrbracket$ and $\widetilde{\omega}^\bullet_{\cY,\bbQ}\llbracket u\rrbracket$ as the inverse limits of quotient algebras
		\begin{align*}
		\omega^\bullet_{\cZ/W^\varnothing,\bbQ}\llbracket u\rrbracket: = \varprojlim_i\omega^\bullet_{\cZ/W^\varnothing,\bbQ}[u]/(u^{[i]}) & &  \text{and} & &  \widetilde{\omega}^\bullet_{\cY,\bbQ}\llbracket u\rrbracket: = \varprojlim_i \widetilde{\omega}^\bullet_{\cY,\bbQ}[u]/(u^{[i]}),
		\end{align*}
	which are also CDGAs with Frobenius and monodromy actions.
\end{definition}

\begin{definition}\label{def: HK coh}
	Let $(Y,\cZ,i,\phi)$ be a log rigid $F$-quadruple over $(k^0\hookrightarrow\cS,\sigma)$ and $\frZ$ the dagger space associated to $\cZ$.
	We define the {\it rigid Hyodo--Kato cohomology} of $Y$ with respect to $(\cZ,i,\phi)$ to be
	\[R\Gamma_\HK(Y)_\cZ:=R\Gamma(\frZ,\omega^\bullet_{\cZ/W^\varnothing,\bbQ}[u]).\]
	The endomorphisms on $\omega^\bullet_{\cZ/W^\varnothing,\bbQ}[u]$ induced by $\phi$ and $N$ respectively define endomorphisms $\varphi$ and $N$ on $R\Gamma_\HK(Y)_\cZ$, which we call the Frobenius operator and the monodromy operator respectively.
	Clearly they satisfy the relation $N\varphi = p\varphi N$.
\end{definition}

\begin{remark}
	Note that this definition differs from the one in \cite{EY}. We will show however, that on the level of cohomology they coincide.
	
	In a previous version of this paper, we had defined the Hyodo--Kato cohomology as the homotopy colimit of the cohomology of $\omega^\bullet_{\cZ/W^\varnothing,\bbQ}[u]_m$.
	It is not clear that the current and previous definitions coincide, because the colimit and cohomology do not commute with each other ingeneral.
	However we will see at Proposition \ref{prop: HK vs colim} that the two definitions in fact coincide in the strictly semistable case.
\end{remark}

\begin{lemma}\label{lem: colim coh}
	Let $(Y,\cZ,i,\phi)$, $\frZ$ and $\frY$ be as in Definition \ref{def: HK coh}.
	Let $\frU\subset\frZ$ be an admissible open subset.
	If $\frU$ is quasi-compact, the natural map
	\[\varinjlim_mH^i(\frU,\omega^\bullet_{\cZ/W^\varnothing,\bbQ}[u]_m)\rightarrow H^i(\frU,\omega^\bullet_{\cZ/W^\varnothing,\bbQ}[u])\]
	is an isomorphism for any $i$.
	If $\frU$ is an affinoid space, the natural morphism
	\[\Gamma(\frU,\omega^\bullet_{\cZ/W^\varnothing,\bbQ}[u])\rightarrow R\Gamma(\frU,\omega^\bullet_{\cZ/W^\varnothing,\bbQ}[u])\]
	is a quasi-isomorphism.
\end{lemma}

\begin{proof}
	The first assertion follows from the equality \eqref{equ:directed} by \cite[\href{https://stacks.math.columbia.edu/tag/0739}{Lem.\,0739}]{stacks}.
	This with \cite[Prop.\,3.1]{GK1} implies the second assertion.
\end{proof}

\begin{proposition}\label{prop: HK indep local}
	Let $Y$ be a fine log scheme over $k^0$.
	If a log rigid $F$-datum over $(k^0\hookrightarrow\cS,\sigma)$ for $Y$ exists, the rigid Hyodo--Kato cohomology $R\Gamma_\HK(Y)_\cZ$ is independent of the choice of log rigid $F$-datum up to canonical quasi-isomorphisms.
\end{proposition}

\begin{proof}
	Taking Lemma \ref{lem: colim coh} into account, the assertion follows by the same proof as Proposition \ref{prop: log rigid coh}.
	Note that the assertion is reduced to showing that the sequence
	\[0\rightarrow B\otimes\bbQ[u]\rightarrow B'_r\otimes\bbQ[u]\xrightarrow{\partial_s} B'_rds\otimes\bbQ[u]\rightarrow 0\]
	is exact, and this follows by \eqref{eq: relative disk}.
\end{proof}

\begin{definition}\label{def: HK coh global}
	According to the above proposition, we may define the rigid Hyodo--Kato cohomology  $R\Gamma_\HK(Y)$ for any fine log scheme $Y$ over $k^0$ by gluing the local construction:
	By Lemma \ref{lem: local embedding} we may find an open covering $\{U_j\rightarrow Y\}_{j\in J}$, such that a log rigid $F$-datum $(\cZ_j,i_j,\phi_j)$ over $(k^0\hookrightarrow \cS,\sigma)$ for $U_j$ exists.
	The same construction as in Definition \ref{def: simplicial rig coh}
provides a simplicial log rigid $F$-quadruple $(U_\bullet,\cZ_\bullet,i_\bullet)$ over $(k^0\hookrightarrow \cS,\sigma)$.
	We define $R\Gamma_\HK(Y)$ to be the complex associated to the cosimplicial complex $R\Gamma_\HK(U_\bullet)_{\cZ_\bullet}$.
\end{definition}

\begin{proposition}\label{prop: HK coh independence}
	The Hyodo--Kato cohomology $R\Gamma_\HK(Y)$ is in the derived category independent of the choice of the covering $\{U_j\}$ and log rigid $F$-data $(\cZ_j,i_j,\phi_j)$.
\end{proposition}

\begin{proof}
	This is proved by similarly as Proposition \ref{prop: independence of global log rigid cohomology} using Proposition \ref{prop: HK indep local}.
\end{proof}

\begin{remark}
	Again using Besser's construction as in \cite[\S 2]{EY} with a slight modification, one can define $R\Gamma_\HK(Y)$ as a canonical complex (not only in the derived category).
\end{remark}

To compare our constructions with more classical ones, it will be of advantage to regard the complexes in questions as double complexes.
The following conventions and statements will be used at key steps of the proof.
To avoid confusion, we point out that all complexes in this paper are cochain complexes. 

\begin{definition}\label{def: double complexes}
	Let $\cA$ be an abelian category which admits countable direct sums and countable direct products.
	Let $C^{\bullet,\bullet}$ be a double complex of objects of $\cA$ with differentials $\partial_1^{i,j}\colon C^{i,j} \rightarrow  C^{i + 1,j}$ and $\partial_2^{i,j}\colon C^{i,j} \rightarrow  C^{i,j + 1}$.
	\begin{enumerate}
	\item We say $C^{\bullet,\bullet}$ is in the \textit{eighth octant} if the implication
	\[C^{i,j}\neq 0\Rightarrow j\leq0, i+j\geq 0\]
	holds. 
	\item The {\it product-total complex} $\Tot^\times C^{\bullet,\bullet}$ of $C^{\bullet,\bullet}$ is the complex whose $m$\textsuperscript{th} component is $\prod_{k\in\bbZ}C^{k,m-k}$ and whose differential is given by the product of $(\partial_1^{i,j},(-1)^i\partial_2^{i,j})\colon C^{i,j} \rightarrow  C^{i + 1,j}\times C^{i,j + 1}$.
	\item The {\it sum-total complex} $\Tot^+ C^{\bullet,\bullet}$ of $C^{\bullet,\bullet}$ is the complex whose $m$\textsuperscript{th} component is $\bigoplus_{k\in\bbZ}C^{k,m-k}$ and whose differential is given by the sum of $\partial_1^{i,j} + (-1)^i\partial_2^{i,j}\colon C^{i,j} \rightarrow  C^{i + 1,j}\oplus C^{i,j + 1}$.
	\end{enumerate}
	Note that there is a natural morphism $\Tot^+C^{\bullet,\bullet} \rightarrow \Tot^\times C^{\bullet,\bullet}$.
	This is an equality if for any $k\in\bbZ$ there are only finitely many $(i,j)$ satisfying $i+j=k$ and $C^{i,j}\neq 0$.
	In that case we identify them and use the notation $\Tot C^{\bullet,\bullet}=\Tot^+C^{\bullet,\bullet}=\Tot^\times C^{\bullet,\bullet}$.
	
	We may also define the sum-total complex and the product-total complex of an $n$-tuple complex for any $n$ in a similar manner.
\end{definition}

\begin{lemma}\label{lem: assembly}
	Let $A$ be a commutative ring.
	Let $B^{\bullet,\bullet} \rightarrow  C^{\bullet,\bullet}$ be a map of double complexes of $A$-modules.
	Suppose that for any integer $k\in\bbZ$ there exists an integer $i_0\in\bbZ$ such that $B^{i,k-i}=C^{i,k-i}=0$ for any $i<i_0$.
	\begin{enumerate}
	\item If $B^{i,\bullet} \rightarrow  C^{i,\bullet}$ is a quasi-isomorphism for any $i$, then $\Tot^\times(B^{\bullet,\bullet}) \rightarrow \Tot^\times (C^{\bullet,\bullet})$  is also a quasi-isomorphism.
	\item If $B^{\bullet,j} \rightarrow  C^{\bullet,j}$ is a quasi-isomorphism for any $j$, then $\Tot^+ (B^{\bullet,\bullet}) \rightarrow \Tot^+ (C^{\bullet,\bullet})$ is also a quasi-isomorphism.
	\end{enumerate}
\end{lemma}

\begin{proof}
	This statement follows immediately from the acyclic assembly lemma (see \cite[Lem.~2.7.3]{We} and the subsequent remark).
\end{proof}

\begin{lemma}\label{lem: ss of double complex}
	Let $A$ be a commutative ring and $C^{\bullet,\bullet}$ be a double complex of $A$-modules.
	Suppose that there exist integers $i_0,j_0\in\bbZ$ such that $C^{i,j}=0$ for any $i\leq i_0$ and $j\geq j_0$.
	Then there exist a weakly convergent spectral sequence
		\[{}^IE_2^{p,q}\Rightarrow H^{p+q}(\mathrm{Tot}^\times(C^{\bullet,\bullet}))\]
	where $^IE_2^{p,q}$ is the $p$\textsuperscript{th} cohomology of
		\[\cdots\rightarrow H^q(C^{p-1,\bullet})\rightarrow H^q(C^{p,\bullet})\rightarrow H^q(C^{p+1,\bullet})\rightarrow\cdots,\]
	and a convergent spectral sequence
		\[{}^{II}E_2^{p,q}\Rightarrow H^{p+q}(\mathrm{Tot}^+(C^{\bullet,\bullet}))\]
	where $^{II}E_2^{p,q}$ is the $p$\textsuperscript{th} cohomology of
		\[\cdots\rightarrow H^q(C^{\bullet,p-1})\rightarrow H^q(C^{\bullet,p})\rightarrow H^q(C^{\bullet,p+1})\rightarrow\cdots.\]
\end{lemma}

\begin{proof}
	This is given in \cite[\S 5.6]{We}.
\end{proof}

\begin{definition}\label{def: eighth octant}
	Let $\frZ$ be a dagger space and $\cC^{\bullet,\bullet}$ be a double complex 
	of sheaves of abelian groups on $\frZ$. 
	For any integer $k\in\bbZ$ set $R\Gamma_\times(\frZ,\cC^{\bullet,\bullet}):=R\Gamma(\frZ,\Tot^\times\cC^{\bullet,\bullet})$.
\end{definition}

For the sake of several computations that we have to carry out in the course of this section, we would like to be able to commute the derived global section functor and the total complex functor in the above definition at least in the relevant situations.

\begin{lemma}\label{lem: product total representation}
Let $\frZ$ be a dagger space and $\cC^{\bullet,\bullet}$ be a double complex in the eighth octant of sheaves of abelian groups on $\frZ$, such that each $\cC^{i,j}$ are coherent sheaves.
	Then $R\Gamma_\times(\frZ,\cC^{\bullet,\bullet})$ is represented by the product total complex of the triple complex $\Gamma(\frZ,\mathrm{Gd}^\bullet\cC^{\bullet,\bullet})$.
\end{lemma}

\begin{proof}
	We define a double complex $C_\times^{\bullet,\bullet}$ 
	by setting for each $n\in\bbZ$
	\[C_\times^{n,\bullet}:=\Gamma(\frZ,\mathrm{Gd}^\bullet\prod_{{i+j}=n}\cC^{i,j}).\]
	We remark that by the hypothesis $C_\times^{\bullet,\bullet}$ is a first quadrant double complex.
	Hence by the remark at the end of Definition \ref{def: double complexes}, $\Tot^\times(C_\times^{\bullet,\bullet})=\Tot(C_\times^{\bullet,\bullet})$.
	Moreover, by definition this complex represents $R\Gamma_\times(\frZ,\cC^{\bullet,\bullet})$.
	
	Then there exist natural maps
	\begin{equation}\label{eq: prod total 1}
	C_\times^{n,\bullet}\rightarrow\Gamma\bigl(\frZ,\prod_{i+j=n}\mathrm{Gd}^\bullet\cC^{i,j}\bigr)=\prod_{i+j=n}\Gamma(\frZ,\mathrm{Gd}^\bullet\cC^{i,j}).
	\end{equation}
	Taking cohomology groups, this induces maps
	\[H^m\bigl(\frZ,\prod_{i+j=n}\cC^{i,j}\bigr)\rightarrow\prod_{i+j=n}H^m(\frZ,\cC^{i,j}),\]
	which are isomorphisms by Lemma \ref{lem: product cohomology}. 
	Consequently, the maps \eqref{eq: prod total 1} are quasi-isomorphisms.
	
	Varying $n$, the maps \eqref{eq: prod total 1} induce a map of first quadrant double complexes.
	The total complex of its codomain coincides with the product total complex of the triple complex $\Gamma(\frZ,\mathrm{Gd}^\bullet\cC^{\bullet,\bullet})$.
	Therefore
	we obtain a map
	\[\Tot C_\times^{\bullet,\bullet}\rightarrow\Tot^\times\Gamma(\frZ,\mathrm{Gd}^\bullet\cC^{\bullet,\bullet}),\]
	which is a quasi-isomorphism by the first part of Lemma \ref{lem: assembly}.
	Since $R\Gamma_\times(\frZ,\cC^{\bullet,\bullet})$ is represented by $\Tot C_\times^{\bullet,\bullet}$, we obtain the assertion.
\end{proof}

We would like to mimic Definition \ref{def: eighth octant} 
for the sum total complex in place of the product total complex. 
However if we do this, the equivalent statement to Lemma \ref{lem: product total representation} does not hold.
Thus we have to modify the definition slightly to obtain $R\Gamma_+(\frZ,\cC^{\bullet,\bullet})$.

\begin{definition}
Let $\frZ$ be a dagger space and $\cC^{\bullet,\bullet}$ be a double complex of sheaves of abelian groups on $\frZ$. 
For any integer $k\in\bbZ$ we denote by $\cC^{\bullet,\bullet\geq k}$ the subcomplex of $\cC^{\bullet,\bullet}$ whose $(i,j)$-component is $\cC^{i,j}$ if $j\geq k$ and $0$ if $j<k$.
	We define $R\Gamma_+(\frZ,\cC^{\bullet,\bullet})$ as the homotopy colimit of the $R\Gamma(\frZ,\Tot^+\,\cC^{\bullet,\bullet\geq k})$, that is
	$$R\Gamma_+(\frZ,\cC^{\bullet,\bullet}):=\hocolim R\Gamma(\frZ,\Tot^+\,\cC^{\bullet,\bullet\geq k}).$$
\end{definition}

\begin{lemma}\label{lem: sum total representation}
Let $\frZ$ be a dagger space and $\cC^{\bullet,\bullet}$ be a double complex in the eighth octant of sheaves of abelian groups on $\frZ$.
	Then $R\Gamma_+(\frZ,\cC^{\bullet,\bullet})$ is represented by the sum total complex of the triple complex $\Gamma(\frZ,\mathrm{Gd}^\bullet\cC^{\bullet,\bullet})$.
\end{lemma}

\begin{proof}
	By the hypothesis, for any $k\in\bbZ$ there exist only finitely many $j\in\bbZ$ satisfying $j\geq k$ and $\cC^{k-j,j}\neq 0$.
	Since the Godement resolution commutes with finite sum, $R\Gamma(\frZ,\Tot\cC^{\bullet,\bullet\geq k})$ is 
	represented by 
	$\Tot\Gamma(\frZ,\mathrm{Gd}^\bullet\cC^{\bullet,\bullet\geq k})$.
	
	Now we observe that 
	$\Tot^+ \Gamma(\frZ,\mathrm{Gd}^\bullet\cC^{\bullet,\bullet})$
	is given as a colimit
	$$
	\Tot^+ \Gamma(\frZ,\mathrm{Gd}^\bullet\cC^{\bullet,\bullet}) =\varinjlim \Tot\Gamma(\frZ,\mathrm{Gd}^\bullet\cC^{\bullet,\bullet\geq k})
	$$
	as complexes in the category of abelian groups.
	This category is in fact a Grothendieck abelian category (or an AB5-category),
	and hence colimits are exact.
	Passing to the derived category, we thus obtain a quasi-isomorphism
	$$
	\Tot^+ \Gamma(\frZ,\mathrm{Gd}^\bullet\cC^{\bullet,\bullet}) \cong \hocolim \Tot \Gamma(\frZ,\mathrm{Gd}^\bullet\cC^{\bullet,\bullet\geq k}).
	$$
	As $\Tot\Gamma(\frZ,\mathrm{Gd}^\bullet\cC^{\bullet,\bullet\geq k})$
	represents $R\Gamma(\frZ,\Tot\cC^{\bullet,\bullet\geq k})$
	by what we have seen above, this means that $\Tot^+ \Gamma(\frZ,\mathrm{Gd}^\bullet\cC^{\bullet,\bullet})$ represents
	$R\Gamma_+(\frZ,\cC^{\bullet,\bullet}):=\hocolim R\Gamma(\frZ,\Tot\cC^{\bullet,\bullet\geq k})$
	as claimed.
\end{proof}

\begin{lemma}\label{lem: acyclic assembly lemma for sheaf}
	Let $\frZ$ be a dagger space. Let $\cB^{\bullet,\bullet}\rightarrow\cC^{\bullet,\bullet}$ be a map of double complexes in the eighth octant of sheaves of abelian groups on $\frZ$.
	Suppose that $\cB^{i,j}$ and $\cC^{i,j}$ are coherent sheaves for each $i$ and $j$.
	\begin{enumerate}
	\item\label{item: sheaf assembly column}
	If $R\Gamma(\frZ,\cB^{i,\bullet})\rightarrow R\Gamma(\frZ,\cC^{i,\bullet})$
	is a quasi-isomorphism for any $i$, then 
	$R\Gamma_\times(\frZ,\cB^{\bullet,\bullet})\rightarrow R\Gamma_\times(\frZ,\cC^{\bullet,\bullet})$ is a quasi-isomorphism.
	\item\label{item: sheaf assembly row}
	If $R\Gamma(\frZ,\cB^{\bullet,j})\rightarrow R\Gamma(\frZ,\cC^{\bullet,j})$ 
	is a quasi-isomorphism for any $j$, then $R\Gamma_+(\frZ,\cB^{\bullet,\bullet})\rightarrow R\Gamma_+(\frZ,\cC^{\bullet,\bullet})$ is a quasi-isomorphism.
	\end{enumerate}
\end{lemma}

\begin{proof}
	To see \eqref{item: sheaf assembly column}, define double complexes of abelian groups $K^{\bullet,\bullet}$ and $L^{\bullet,\bullet}$ by
	\begin{align}\label{eq: double complex K and L}
	K^{i,\bullet}&:=R\Gamma(\frZ,\cB^{i,\bullet})=\Tot\Gamma(\frZ,\mathrm{Gd}^\bullet\cB^{i,\bullet}),\\
	L^{i,\bullet}&:=R\Gamma(\frZ,\cC^{i,\bullet})=\Tot\Gamma(\frZ,\mathrm{Gd}^\bullet\cC^{i,\bullet}).\nonumber
	\end{align}
	Here note that the product total complexes and total complexes coincide with each other, since $\cB^{\bullet,\bullet}$ and $\cC^{\bullet,\bullet}$ are in the eighth octant.
	Then $K^{\bullet,\bullet}$ and $L^{\bullet,\bullet}$ lie in the right half plane and thus satisfy the condition of Lemma \ref{lem: assembly}.
	Moreover, $\Tot^\times K^{\bullet,\bullet}=\Tot^\times\Gamma(\frZ,\mathrm{Gd}^\bullet\cB^{\bullet,\bullet})$ 
	represents $R\Gamma_\times(\frZ,\cB^{\bullet,\bullet})$ by Lemma \ref{lem: product total representation}, and likewise 
	$\Tot^\times L^{\bullet,\bullet}=\Tot^\times\Gamma(\frZ,\mathrm{Gd}^\bullet\cC^{\bullet,\bullet})$ 
	represents $R\Gamma_\times(\frZ,\cC^{\bullet,\bullet})$.
	By hypothesis $K^{i,\bullet}\rightarrow L^{i,\bullet}$ is a quasi-isomorphism for any $i$,
	hence $R\Gamma_\times(\frZ,\cB^{\bullet,\bullet})\rightarrow R\Gamma_\times(\frZ,\cC^{\bullet,\bullet})$ is a quasi-isomorphism by Lemma \ref{lem: assembly}.

	To see \eqref{item: sheaf assembly row}, define double complexes of abelian groups $M^{\bullet,\bullet}$ and $N^{\bullet,\bullet}$ by
	\begin{align}\label{eq: double complex M and N}
	M^{\bullet,j}&:=R\Gamma(\frZ,\cB^{\bullet,j})=\Tot \Gamma(\frZ,\mathrm{Gd}^\bullet\cB^{\bullet,j}), \\
	N^{\bullet,j}&:=R\Gamma(\frZ,\cC^{\bullet,j})=\Tot \Gamma(\frZ,\mathrm{Gd}^\bullet\cC^{\bullet,j}).\nonumber\end{align} 
	Then again $M^{\bullet,\bullet}$ and $N^{\bullet,\bullet}$ satisfy the condition of Lemma \ref{lem: assembly} as they lie in the fourth quadrant.
	Moreover, $\Tot^+ M^{\bullet,\bullet}=\Tot^+\Gamma(\frZ,\mathrm{Gd}^\bullet\cB^{\bullet,\bullet})$ represents $R\Gamma_+(\frZ,\cB^{\bullet,\bullet})$ by Lemma \ref{lem: sum total representation}, and likewise $\Tot^+ N^{\bullet,\bullet}=\Tot^+\Gamma(\frZ,\mathrm{Gd}^\bullet\cC^{\bullet,\bullet})$ represents $R\Gamma_+(\frZ,\cC^{\bullet,\bullet})$.
	 By hypothesis $M^{\bullet,j}\rightarrow N^{\bullet,j}$ is a quasi-isomorphism for each $j$,
	and hence $R\Gamma_+(\frZ,\cB^{\bullet,\bullet})\rightarrow R\Gamma_+(\frZ,\cC^{\bullet,\bullet})$ is a quasi-isomorphism again by Lemma \ref{lem: assembly}.
\end{proof}

\begin{lemma}\label{lem: sheaf spectral sequence}
	Let $\frZ$ be a dagger space and $\cC^{\bullet,\bullet}$ a double complex in the eighth octant of sheaves of abelian groups on $\frZ$.
	Suppose that $\cC^{i,j}$ is a coherent sheaf for each $i$ and $j$.
	Then there exist a weakly convergent spectral sequence
	\[{}^IE_2^{p,q}\Rightarrow H^{p+q}_\times(\frZ,\cC^{\bullet,\bullet})\]
	where ${}^IE_2^{p,q}$ is the $p$\textsuperscript{th} cohomology of
	\[\cdots\rightarrow H^q(\frZ,\cC^{p-1,\bullet})\rightarrow H^q(\frZ,\cC^{p,\bullet})\rightarrow H^q(\frZ,\cC^{p+1,\bullet})\cdots,\]
	and a convergent spectral sequence
	\[{}^{II}E_2^{p,q}\Rightarrow H^{p+q}_+(\frZ,\cC^{\bullet,\bullet})\]
	where ${}^{II}E_2^{p,q}$ is the $p$\textsuperscript{th} cohomology of
	\[\cdots\rightarrow H^q(\frZ,\cC^{\bullet,p-1})\rightarrow H^q(\frZ,\cC^{\bullet,p})\rightarrow H^q(\frZ,\cC^{\bullet,p+1})\cdots,\]
\end{lemma}

\begin{proof}
     Define double complexes 
     $L^{\bullet,\bullet}$ and $N^{\bullet,\bullet}$ 
     as in \eqref{eq: double complex K and L} and \eqref{eq: double complex M and N}, respectively, by
     \begin{align*}
     L^{i,\bullet}&:=R\Gamma(\frZ,\cC^{i,\bullet})=\Tot \Gamma(\frZ,\mathrm{Gd}^\bullet\cC^{i,\bullet}),\\
     N^{\bullet,j}&:=R\Gamma(\frZ,\cC^{\bullet,j})=\Tot \Gamma(\frZ,\mathrm{Gd}^\bullet\cC^{\bullet,j}).
     \end{align*}
    By Lemmas \ref{lem: product total representation} and \ref{lem: sum total representation}, $\Tot^\times L^{\bullet,\bullet}=\Tot^\times\Gamma(\frZ,\mathrm{Gd}^\bullet\cC^{\bullet,\bullet})$ represents $R\Gamma_\times(\frZ,\cC^{\bullet,\bullet})$ and $\Tot^+ N^{\bullet,\bullet}=\Tot^+\Gamma(\frZ,\mathrm{Gd}^\bullet\cC^{\bullet,\bullet})$ represents $R\Gamma_+(\frZ,\cC^{\bullet,\bullet})$.
     Thus one obtains the first spectral sequence by applying the first part of Lemma \ref{lem: ss of double complex} to the double complex $L^{\bullet,\bullet}$,
     and the second spectral sequence by applying the second part of Lemma \ref{lem: ss of double complex} to the double complex $N^{\bullet,\bullet}$.
\end{proof}

The following description of the above Kim--Hain complexes by double complexes is useful for computations.

\begin{lemma}\label{lem: total complex}
	Let $(Y,\cZ,i,\phi)$ be a log rigid $F$-quadruple over $(k^0\hookrightarrow\cS,\sigma)$ and set $\cY:=\cZ\times_{\cS}W^0$.
	Let $\frZ$ and $\frY$ be the dagger spaces associated to $\cZ$ and $\cY$, respectively.
	Let $\cC_\frZ^{\bullet,\bullet}$ be the double complex in the eighth octant of sheaves on $\frZ$ whose $(i,j)$-component is $\cC_\frZ^{i,j}: = \omega_{\cZ/W^\varnothing,\bbQ}^{i + j}u^{[-j]}$ if $i + j\geq 0$ and $j\leq 0$, and $\cC_\frZ^{i,j}: = 0$ otherwise, with differential maps
		\begin{align*}
		&\partial_1^{i,j}\colon \cC_\frZ^{i,j} \rightarrow  \cC_\frZ^{i + 1,j}; & & 
		\alpha u^{[-j]}\mapsto \nabla(\alpha)u^{[-j]}, & & (\alpha\in\omega^{i + j}_{\cZ/W^\varnothing,\bbQ}),\\
		&\partial_2^{i,j}\colon \cC_\frZ^{i,j} \rightarrow \cC_\frZ^{i,j + 1}; & & 
		\alpha u^{[-j]}\mapsto (-1)^{i + 1}d\log s\wedge\alpha u^{[-j-1]}, & & (\alpha\in\omega^{i + j}_{\cZ/W^\varnothing,\bbQ}).
		\end{align*}
	We also define a double complex $\cC^{\bullet,\bullet}_\frY$ in the eighth octant of sheaves on $\frY$ in a similar manner by using $\cC^{i,j}_\frY: = \widetilde{\omega}_{\cY,\bbQ}^{i + j}u^{[-j]}$.
	
	Then we have equalities
	\begin{align*}
	&\omega^\bullet_{\cZ/W^\varnothing,\bbQ}[u] = \Tot^+ \cC_\frZ^{\bullet,\bullet}, &&  \widetilde{\omega}^\bullet_{\cY,\bbQ}[u] = \Tot^+ \cC_\frY^{\bullet,\bullet},\\
	&\omega^\bullet_{\cZ/W^\varnothing,\bbQ}[u]_m = \Tot \cC_\frZ^{\bullet,\bullet\geq -m}, &&   \widetilde{\omega}^\bullet_{\cY,\bbQ}[u]_m = \Tot \cC_\frY^{\bullet,\bullet\geq -m},\\
	&\omega^\bullet_{\cZ/W^\varnothing,\bbQ}\llbracket u\rrbracket = \Tot^\times\cC_\frZ^{\bullet,\bullet},  && \widetilde{\omega}^\bullet_{\cY,\bbQ}\llbracket u\rrbracket = \Tot^\times\cC_\frY^{\bullet,\bullet}.
	\end{align*}
	We also have
	\begin{align}\label{eq: HK coh as double complex}
	&R\Gamma_+(\frZ,\cC^{\bullet,\bullet}_\frZ)=\hocolim_mR\Gamma(\frZ,\omega^\bullet_{\cZ/W^\varnothing,\bbQ}[u]_m),&&R\Gamma(\frZ,\cC_\frZ^{\bullet,-m})=R\Gamma_\rig(Y/W^\varnothing)_{\cZ}[-m],\\
	\nonumber&R\Gamma_+(\frY,\cC^{\bullet,\bullet}_\frY)=\hocolim_mR\Gamma(\frY,\wt{\omega}^\bullet_{\cY,\bbQ}[u]_m),&&R\Gamma(\frY,\cC_\frY^{\bullet,-m})=R\Gamma(\frY,\wt{\omega}^\bullet_{\cY,\bbQ})[-m]
	\end{align}
	for each $m\geq 0$.
\end{lemma}

\begin{proof}
	This follows directly from the definition.
\end{proof}

We will pay particular attention to strictly semistable log schemes over $k^0$.
Here we allow log structures arising not only from degeneration but also from a ``horizontal divisor".

\begin{definition}\label{def: ss quadruple} $\quad$
	\begin{enumerate}
	\item \label{item: ss log scheme}
	Let $(\cT,\sO_\cT,t)$ be one of the triples $(k^0,k,0)$, $(\cS,W\llbracket s\rrbracket,s)$, or $(V^\sharp,V,\pi)$.
		For integers $n\geq 1$ and $m\geq 0$, let $\cT(n,m)\rightarrow \cT$ be the morphism of fine weak formal log schemes induced by the diagram
		\[\xymatrix{
		\bbN^n\oplus\bbN^m\ar[d]_-\alpha & \bbN\ar[d]^-{\alpha_t}\ar[l]_-\beta\\
		\sO_\cT[\bbN^n,\bbN^m]^\dagger/(\alpha\circ\beta(1)-t) & \sO_\cT\ar[l],
		}\]
	where $\alpha_t$ maps $1$ to $t$, $\alpha$ is the natural inclusion,
	and $\beta$ is the composition of the diagonal map $\bbN\rightarrow\bbN^n$ and the canonical injection $\bbN^n\rightarrow\bbN^n\oplus\bbN^m$.
	
	\item For $1\leq r\leq n$ and $0\leq m$, the underlying scheme of $k^0(n,m)$ can be given by
	\[\Spec k[s_1,\ldots,s_n,t_1,\ldots,t_m]/(s_1\cdots s_n).\]
Let $k^0(n,m)_r$ be the exact closed  log subscheme of $k^0(n,m)$ defined by $s_1=\cdots =s_r=0$, i.e.\,with underlying scheme $\Spec k[s_{r+1},\ldots,s_n,t_1,\ldots,t_m]$.
	
	\item Let $\cZ$ be a fine log scheme over $\cT=k^0$ or a fine weak formal log scheme over $\cT=\cS$ or $V^\sharp$.
	 We say that $\cZ$ is {\it strictly semistable} if Zariski locally on $\cZ$ there exists a strict and strongly log smooth morphism $\cZ \rightarrow \cT(n,m)$ over $\cT$.
	
	By the local description above, one can easily see that
	\[\{z\in \cZ\mid \text{$\forall x\in\cN_{\cZ,z}$ $\exists y\in\cN_{\cZ,z}$ s.t.\,$xy$ is contained in the image of $\Gamma(\cT,\cN_{\cT})$}\}\]
	is an open subset of $\cZ$.
	We call its complement equipped with the reduced structure the {\it horizontal divisor} of $\cZ$.
	It is locally defined by $(1,\ldots,1)\in\bbN^m$.
	The horizontal divisor can be empty because we allow the case $m=0$.

	\item A log rigid $F$-quadruple $(Y,\cZ,i,\phi)$ (or a log rigid $F$-datum $(\cZ,i,\phi)$ for $Y$) over $(k^0\hookrightarrow\cS,\sigma)$ is said to be {\it strictly semistable} if Zariski locally there exist a commutative diagram
			\begin{equation}\label{eq: ss quadruple}\xymatrix{
			Y \ar[r]^i \ar[d] & \cZ \ar[d]\\
			k^0(n,m) \ar[r] & \cS(n,m),
			}\end{equation}
		where the vertical morphisms are strict and strongly log smooth.
	\end{enumerate}
\end{definition}

\begin{remark}\label{rem: ss log scheme} $\quad$
\begin{enumerate}
\item\label{item: ss lift} For any strictly semistable log scheme $Y$ over $k^0$, a strictly semistable log rigid $F$-datum $(\cZ,i,\phi)$ for $Y$ over $(k^0\hookrightarrow\cS,\sigma)$ exists locally on $Y$.
Moreover the morphism $\cZ\rightarrow\cS(n,m)$ can be taken to be adic. (See the construction in \cite[\S 5.1]{GK3}.)
\item To establish rigid Hyodo--Kato theory, it is enough to consider the case $D=\emptyset$.
We will use strictly semistable log scheme as above for comparison with crystalline Hyodo--Kato 
theory.
\end{enumerate}
\end{remark}

A standard technique to investigate log rigid cohomologies and the Hyodo--Kato cohomology of a strictly semistable log scheme is to reduce to the study of cohomologies on intersections of irreducible components. 
For such a log scheme with decomposition into irreducible components given by $Y=\bigcup_{i\in\Upsilon}Y_i$, we denote for a non-empty subset $I\subset\Upsilon$ the intersection $Y_I:=\bigcap_{i\in I}Y_i$ and equip it with the pull-back log structure from $Y$.
Locally, an intersection of irreducible components is strict log smooth over $k^0(n,m)_r$ for some $1\leq r\leq n$ and $0\leq m$, and the following construction of a log rigid datum will be applied to it.

\begin{construction}\label{const}
Let $Y$ be a fine log scheme which is affine and strict log smooth over $k^0(n,m)_r$ for some $1\leq r\leq n$ and $0\leq m$.
We may lift the underlying scheme to a $p$-adic weak formal scheme $\cV$ over $W$ with a (strongly) smooth morphism
\begin{equation*}
\cV\rightarrow \Spwf W[s_{r+1},\ldots,s_n,t_1,\ldots,t_m]^\dagger.
\end{equation*}
	Set
\begin{equation*}
\cZ:=\Spwf W\llbracket s_1,\ldots,s_r\rrbracket\times\cV
\end{equation*}
and equip it with the log structure associated to the monoid homomorphism $\bbN^n\oplus \bbN^m\rightarrow \cO_{\cZ}$ sending $1_i\in\bbN^n$ to $s_i$ and $1_j\in\bbN^m$ to $t_j$.
We regard $\cZ$ as a weak formal log scheme over $\cS$ by the association $s\mapsto s_1\cdots s_n$. 
There is a homeomorphic exact closed immersion $i\colon Y\hookrightarrow\cZ$, and we may take Frobenius lifts $\phi$ on $\cZ$. 
	Thus we get a log rigid $F$-datum $(\cZ,i,\phi)$ over $(k^0\hookrightarrow\cS,\sigma)$ for $Y$.
\end{construction}

\begin{proposition}\label{prop: HK vs colim}
	Let $(Y,\cZ,i,\phi)$, $\cY$, $\frZ$, and $\frY$ be as in Lemma \ref{lem: total complex}, and assume that $Y$ is quasi-compact.
	Then the natural morphism
	\begin{align}\label{eq: hocolim compare}
	R\Gamma_+(\frZ,\cC^{\bullet,\bullet}_\frZ)\rightarrow R\Gamma_\HK(Y)_\cZ
	\end{align}
	is a quasi-isomorphism in the following two cases:
	\begin{enumerate}
	\item\label{item: comp hocolim} if $Y$ is strict log smooth over $k^0(n,m)_r$, or
	\item\label{item: ss hocolim}if $Y$ is strictly semistable.
	\end{enumerate}
\end{proposition}

\begin{proof}
	The morphism \eqref{eq: hocolim compare} is immediately induced from Lemma \ref{lem: total complex}.
	We first prove that \eqref{eq: hocolim compare} is a quasi-isomorphism for the case \eqref{item: comp hocolim}.
	For each $m\in\bbN$, the complex $R\Gamma(\frZ,\cC_\frZ^{\bullet,-m})=R\Gamma_\rig(Y/W^\varnothing)_\cZ[-m]$ is independent in the derived category from the log rigid $F$-datum $(\cZ,i,\phi)$ for $Y$, and satisfies the cohomological descent with respect to admissible coverings of $\frZ$.
	Thus, by Lemma \ref{lem: acyclic assembly lemma for sheaf} we see that $R\Gamma_+(\frZ,\cC^{\bullet,\bullet}_\frZ)$ is also independent from $(\cZ,i,\phi)$ and satisfies cohomological descent with respect to finite coverings of $\frZ$.
	(Note that we use the fact that homotopy colimits commute with finite products.)
	
	As a consequence, we may assume that $Y$ is affine and that we have $\cZ=\Spwf W\llbracket s_1,\ldots,s_r\rrbracket\times\cV$ as in Construction \ref{const}, by shrinking $Y$ and changing $\cZ$.
	Let $\frV$ and $\frZ$ be the dagger spaces associated to $\cV$ and $\cZ$, respectively.
	For $\ell_1,\ldots,\ell_r\geq 1$, let $\cZ_{\ell_1,\ldots,\ell_r}:=\Spwf W[s_1,\ldots,s_r,\frac{s_1^{\ell_1}}p,\ldots,\frac{s_r^{\ell_r}}p]^\dagger\times\cV$ and let $\frZ_{\ell_1,\ldots,\ell_r}$ be the associated dagger space.
	
	Next we claim that $R\Gamma(\frZ_{\ell_1,\ldots,\ell_r},\omega^\bullet_{\cZ/W^\varnothing,\bbQ})$ is independent of the choice of tuple $\ell_1,\ldots,\ell_r$  up to canonical quasi-isomrophisms.
	By induction we may assume that $r=1$.
	We write $\cV=\Spwf B$ and $\cZ_{\ell}=\Spwf B_{\ell}$.
	Similarly to the proof of Proposition \ref{prop: log rigid coh}, $R\Gamma(\frZ_{\ell},\omega^\bullet_{\cZ/W^\varnothing,\bbQ})$ is represented by the complex
	\begin{equation}\label{eq: cone}\Cone(B_\ell\otimes_B\omega^\bullet_B\otimes\bbQ\xrightarrow{\partial_s}B_\ell d\log s\otimes_B\omega_B^\bullet\otimes\bbQ)[-1]\end{equation}
	where we set $\omega^\bullet_B:=\Gamma(\cV,\omega^\bullet_{\cV/W^\varnothing,\bbQ})$.
	By the proof of Proposition \ref{prop: log rigid coh} we see that
	\begin{align*}	&\Ker(B_\ell\otimes_B\omega^\bullet_B\otimes\bbQ\xrightarrow{\partial_s}B_\ell d\log s\otimes_B\omega_B^\bullet\otimes\bbQ)=\omega_B^\bullet\otimes\bbQ,\\
	&\Coker(B_\ell\otimes_B\omega^\bullet_B\otimes\bbQ\xrightarrow{\partial_s}B_\ell d\log s\otimes_B\omega_B^\bullet\otimes\bbQ)=B_\ell d\log s\otimes_B\omega_B^\bullet\otimes\bbQ/B_\ell d s\otimes_B\omega_B^\bullet\otimes\bbQ\cong \omega_B^\bullet\otimes\bbQ.
	\end{align*}
	They are independent of $\ell$, hence the cone \eqref{eq: cone} is also independent and we proved the claim.
	
	Then by the claim the natural projection
	\[R\Gamma(\frZ,\omega^\bullet_{\cZ/W^\varnothing,\bbQ})=\holim_\ell R\Gamma(\frZ_{\ell,\ldots,\ell},\omega^\bullet_{\cZ/W^\varnothing,\bbQ})\rightarrow R\Gamma(\frZ_{\ell',\ldots,\ell'},\omega^\bullet_{\cZ/W^\varnothing,\bbQ})\]
	is a quasi-isomorphism for any $\ell'$.
	Thus by the Lemmas \ref{lem: acyclic assembly lemma for sheaf} and \ref{lem: total complex} we see that
	\begin{equation}\label{eq: + shrink}R\Gamma_+(\frZ,\cC^{\bullet,\bullet}_\frZ)\rightarrow R\Gamma_+(\frZ_{\ell',\ldots,\ell'},\cC^{\bullet,\bullet}_\frZ)\end{equation}
	is also a quasi-isomorphism. 
	Moreover, since we have
	\begin{equation}\label{eq: affinoid colim}
	R\Gamma_+(\frZ_{\ell,\ldots,\ell},\cC^{\bullet,\bullet}_\frZ)
	=\hocolim_mR\Gamma(\frZ_{\ell,\ldots,\ell},\omega^\bullet_{\cZ/W^\varnothing,\bbQ}[u]_m)
	\cong R\Gamma(\frZ_{\ell,\ldots,\ell},\omega^\bullet_{\cZ/W^\varnothing,\bbQ}[u])
	\end{equation}
	by Lemma \ref{lem: colim coh}, $R\Gamma(\frZ_{\ell,\ldots,\ell},\omega^\bullet_{\cZ/W^\varnothing,\bbQ}[u])$ is also independent of $\ell$.
	Thus
	\begin{equation}\label{eq: [u] shrink}R\Gamma(\frZ,\omega^\bullet_{\cZ/W^\varnothing,\bbQ}[u])=\holim_\ell R\Gamma(\frZ_{\ell,\ldots,\ell},\omega^\bullet_{\cZ/W^\varnothing,\bbQ}[u])\rightarrow R\Gamma(\frZ_{\ell',\ldots,\ell'},\omega^\bullet_{\cZ/W^\varnothing,\bbQ}[u])\end{equation}
	is also a quasi-isomorphism.
	Now we have a commutative diagram
	\[\xymatrix{
	R\Gamma_+(\frZ,\cC^{\bullet,\bullet}_\frZ)\ar[r]_-\sim^-{\eqref{eq: + shrink}}\ar[d]
	&R\Gamma(\frZ_{\ell',\ldots,\ell'},\omega^\bullet_{\cZ/W^\varnothing,\bbQ})\ar[d]_-\sim^-{\eqref{eq: affinoid colim}}\\
	R\Gamma(\frZ,\omega^\bullet_{\cZ/W^\varnothing,\bbQ}[u])\ar[r]_-\sim^-{\eqref{eq: [u] shrink}}& R\Gamma(\frZ_{\ell',\ldots,\ell'},\omega^\bullet_{\cZ/W^\varnothing,\bbQ}[u]),
	}\]
	hence the left vertical morphism is also a quasi-isomorphism.
	This shows the assertion for \eqref{item: comp hocolim}.
	
	For the case \eqref{item: ss hocolim}, let $Y=\bigcup_{i\in\Upsilon}Y_i$ be the decomposition into irreducible components.
For the moment, fix a non-empty subset $I\subset\Upsilon$.
Set $Y_I:=\bigcap_{i\in I}Y_i$ and equip it with the pull-back log structure from $Y$.
	Let $Y_I\hookrightarrow \cZ_I$ be the exactification of $Y_I\hookrightarrow Y\hookrightarrow \cZ$.
	Then the dagger spaces $\frZ_I$ associated to $\cZ_I$ cover $\frZ$.
	Thus it suffices to show that $R\Gamma_+(\frZ_I,\cC^{\bullet,\bullet}_\frZ)\rightarrow R\Gamma(\frZ_I,\omega^\bullet_{\cZ/W^\varnothing,\bbQ}[u])$ is a quasi-isomorphism.
	But this follows from \eqref{item: comp hocolim}, since locally on $Y_I$ there exists a strict log smooth morphism $Y_I\rightarrow k^0(n,m)_r$ for some $1\leq r\leq n$ and $0\leq m$.
\end{proof}

In Corollary \ref{cor: coh of Y_I} we will describe cohomologies of intersections $Y_I$ of irreducible components of a strictly semistable log scheme $Y$ which will play an important role in our Hyodo--Kato theory.
To prepare the proof, we start with calculations of cohomologies on polydisks.

For a dagger algebra $A$ with a surjection $F[t_1,\ldots,t_n]^\dagger\rightarrow A$, the quotient norm of the Gauss norm defines a topology on $A$.
This is independent of the choice of the surjection.

Let $\frX$ be a Stein space with a Stein covering $\{\frU_i\}$ \cite[2.24]{GK1}.
Let $A_i:=\Gamma(\frU_i,\cO_{\frU_i})$ and $B:=\Gamma(\frX,\cO_\frX)=\varprojlim_iA_i$.
We equip $B$ with the coarsest topology such that all projections $B\rightarrow A_i$ are continuous with respect to the norm topology on $A_i$.
	For a finite $B$-module $M$, we choose a surjection $B^r\rightarrow M$ and equip $M$ with the quotient topology of the product topology.
	Note that these topologies on $B$ and $M$ are independent of the choice of a Stein covering.

\begin{lemma}\label{lem: convex}
	Let $\frX$ be a Stein space and $M$ a finite module over $B:=\Gamma(\frX,\cO_\frX)$.
	Then $B$ and $M$ equipped with the topology as above are Fr\'{e}chet spaces.
\end{lemma}

\begin{proof}
	For $\rho\in\lvert F^\times\rvert\otimes\bbQ$ and an integer $n\geq 0$, let $T_n(\rho)$ be the ring of formal power series $\sum_{\alpha=(\alpha_1,\ldots,\alpha_n)\in\bbN^n}a_\alpha t_1^{\alpha_1}\cdots t_n^{\alpha_n}$ converging in $\lvert t_i\rvert\leq\rho$.
	For $\rho<\rho'$, the natural inclusion $T_n(\rho')\rightarrow T_n(\rho)$ is compact \cite[p.98]{Sch}.
	Then it follows that for a Stein covering $\{\frU_i\}$ of $\frX$ the natural inclusions $\Gamma(\frU_{i+1},\cO_{\frU_{i+1}})\rightarrow\Gamma(\frU_i,\cO_{\frU_i})$ are compact.
	Thus the statement follows from \cite[Cor.\,16.6]{Sch}.
\end{proof}

For any dagger space $\frX$ over $F$, a complex of sheaves $\cF^\bullet$ on $\frX$, and a subspace $\frU\subset\frX$, we set $H^\ast(\frU,\cF^\bullet):=\bigoplus_{j\geq 0}H^j(\frU,\cF^\bullet|_\frU)$ and regard it as a graded $F$-vector space.

For an integer $n\geq 1$, let $\cS^n$ be the $n$-fold product of $\cS$ over $W^\varnothing$ regarded as a weak formal log scheme over $\cS$ via the morphism $\cS^n=\Spwf W\llbracket s_1,\ldots,s_n\rrbracket\rightarrow\Spwf W\llbracket s\rrbracket=\cS$ defined by $s\mapsto s_1\cdots s_n$, and set $\cS^n_0:=\cS^n\times_\cS W^0$.
	
For a real number $\eta>0$, let $\bbD_{<\eta}^n$ and $\bbD^n_{\leq\eta}$ be the $n$-dimensional open and closed dagger polydisks of radius $\eta$, respectively.
Let $\bbD_{<\eta,0}^n\subset\bbD^n_{<\eta}$ and  $\bbD_{\leq\eta,0}^n\subset\bbD^n_{\leq\eta}$ be the closed subspaces defined by the equation $s_1\cdots s_n=0$, where $s_i$ denotes the canonical coordinates of the polydisks.
Note that the dagger space associated to $\cS^n$ and $\cS^n_0$ are $\bbD^n_{<1}$ and $\bbD_{<1,0}^n$, respectively.
We set $F\{\eta^{-1}s_1,\ldots,\eta^{-1}s_n\}:=\Gamma(\bbD^n_{<\eta},\cO_{\bbD^n_{<\eta}})$ and $F[\eta^{-1}s_1,\ldots,\eta^{-1}s_n]^\dagger:=\Gamma(\bbD^n_{\leq\eta},\cO_{\bbD^n_{\leq\eta}})$.

\begin{lemma}\label{lem: polydisk}	
	For $n\geq 1$, let $\bsH^\ast_n=\bigoplus_{j\geq 0}\bsH^j_n$ be the exterior algebra of the $F$-vector space freely generated by symbols $d\log s_1,\ldots, d\log s_n$, and let $\overline{\bsH}_n^\ast=\bigoplus_{j\geq 0}\overline{\bsH}^j_n$ be its quotient algebra defined by the relation $\sum_{i=1}^nd\log s_i=0$.
	\begin{enumerate}
	\item\label{item: polydisk1} For $\eta\leq 1$, there exist canonical isomorphisms displayed as the horizontal maps in the following diagram
	\begin{equation}\label{eq: diag polydisk}
	\xymatrix{
	\bsH^\ast_n\ar[r]^-{\cong}\ar[d]&
	H^\ast(\bbD_{<\eta}^n,\omega^\bullet_{\cS^n/W^\varnothing,\bbQ})\ar[d]\\
	\overline{\bsH}^{\ast}_n\ar[r]^-{\cong}\ar[d]&
	\varinjlim_mH^\ast(\bbD_{<\eta}^n,\omega^\bullet_{\cS^n/W^\varnothing,\bbQ}[u]_m)\ar[d]\\
	F\{\eta^{-n}s\}\otimes_F\overline{\bsH}_n^{\ast}\ar[r]^-{\cong}\ar[d]&
	H^\ast(\bbD_{<\eta}^n,\omega^\bullet_{\cS^n/\cS,\bbQ})\ar[d]\\
	\overline{\bsH}^{\ast}_n\ar[r]^-{\cong}&
	H^\ast(\bbD_{<\eta,0}^n,\omega^\bullet_{\cS^n_0/W^0,\bbQ}),}
	\end{equation}
	where the left vertical maps are induced by the canonical projection $\bsH^\ast_n\rightarrow\overline{\bsH}_n^\ast$, the canonical injection $\overline{\bsH}_n^\ast\rightarrow F\{s\}\otimes_F\overline{\bsH}_n^\ast$, and the map $F\{s\}\otimes_F\overline{\bsH}_n^\ast\rightarrow\overline{\bsH}_n^\ast$ defined by $s\mapsto 0$, respectively.
	Moreover this diagram is commutative, and compatible with the restriction maps with respect to the immersion $\bbD^n_{<\eta}\hookrightarrow\bbD^n_{<\eta'}$ for $\eta<\eta'\leq 1$.
	\item\label{item: polydisk2} The differentials on $\Gamma(\bbD^n_{<\eta},\omega^\bullet_{\cS^n/W^\varnothing,\bbQ})$, $\Gamma(\bbD^n_{<\eta},\omega^\bullet_{\cS^n/\cS,\bbQ})$, and $\Gamma(\bbD^n_{<\eta,0},\omega^\bullet_{\cS^n_0/W^0,\bbQ})$ are strict and have closed images.
	The quotient topologies on $H^i(\bbD_{<\eta}^n,\omega^\bullet_{\cS^n/W^\varnothing,\bbQ})$ and $H^i(\bbD^n_{<\eta,0},\omega^\bullet_{\cS_0^n/W^0,\bbQ})$ induced from the topologies on the kernels of the differentials coincide with the topologies defined as finite $F$-modules.
	Similarly, the quotient topology on $H^i(\bbD^n_{<\eta},\omega^\bullet_{\cS^n/\cS,\bbQ})$ coincides with the topology defined as a finite $F\{\eta^{-n}s\}$-module.
	\end{enumerate}
\end{lemma}

\begin{proof}
	We first note that $\bbD^n_{<\eta}$ and $\bbD^n_{<\eta,0}$ are Stein spaces, hence the cohomology groups appearing in the diagram \eqref{eq: diag polydisk} can be computed by the global sections.
	The existence and the commutativity of the maps in \eqref{eq: diag polydisk} are obvious.
	We first prove that the third map of \eqref{eq: diag polydisk} is an isomorphism, by constructing its inverse.
	
	Noting that $d\log s_1=-\sum_{i=2}^nd\log s_i$ in $\omega^1_{\cS^n/\cS,\bbQ}$, any element $\theta\in\Gamma(\bbD_{<\eta}^n,\omega^j_{\cS^n/\cS,\bbQ})$ can be uniquely written as
	\begin{equation}\label{eq: theta}
	\theta=\sum_{\substack{P\subset\{2,\ldots,n\}\\ \lvert P\rvert=j}}\sum_{\alpha\in\bbN^n}c_{\alpha,P}\underline{s}^\alpha d\log\underline{s}_P
	\end{equation}
	for some $c_{\alpha,P}\in F$, where we set $\underline{s}^\alpha:=s_1^{\alpha_1}s_2^{\alpha_2}\cdots s_n^{\alpha_n}$ and $d\log\underline{s}_P:=d\log s_{i_1}\wedge\cdots \wedge d\log s_{i_j}$ for $P=\{i_1,\ldots,i_j\}\subset\{2,\ldots,n\}$.
	For any $\ell=1,\ldots,n$ we set
	\[\omega^j_\ell:=\left\{\theta\in\Gamma(\bbD_{<\eta}^n,\omega^j_{\cS^n/\cS,\bbQ})\middle| \text{$c_{\alpha,P}=0$ if $\ell<\max P$ or $\alpha_r\neq\alpha_1$ for some $r>\ell$}\right\}.\]
	Namely, an element $\theta\in\omega^j_\ell$ can be written as
	\[\theta=\sum_{\substack{P\subset\{2,\ldots,\ell\}\\ \lvert P\rvert =j}}\sum_{\substack{\alpha\in\bbN^n\\ \alpha_r=\alpha_1\ \forall r>\ell}}c_{\alpha,P}\underline{s}^\alpha d\log\underline{s}_P.\]
	
	For $\theta\in\omega^j_\ell$ satisfying $d\theta=0$, put
	\begin{equation}\label{eq: eta definition}\eta:=\sum_{\substack{P\subset\{2,\ldots,\ell\}\\ \lvert P\rvert=j,\, P\ni\ell}}\sum_{\substack{\alpha\in\bbN^n\\ \alpha_r=\alpha_1\ \forall r>\ell\\ \alpha_\ell\neq\alpha_1}}\frac{(-1)^{\lvert P\rvert-1}}{\alpha_\ell-\alpha_1}c_{\alpha,P}\underline{s}^\alpha d\log\underline{s}_{P\setminus\{\ell\}},
	\end{equation}
	which defines an element of $\Gamma(\bbD_{<\eta}^n,\omega^{j-1}_{\cS^n/\cS,\bbQ})$.
	Indeed, since $\theta$ is an element of $\Gamma(\bbD^n_{<\eta},\omega^j_{\cS^n/\cS,\bbQ})$, we have
	\begin{equation}\label{eq: theta converge}
	\lvert c_{\alpha,P}\rvert\rho^{\lvert\alpha\rvert}\rightarrow 0\hspace{15pt}(\lvert\alpha\rvert\rightarrow\infty)
	\end{equation}
	for any $\rho<\eta$.
	
	Fix any $\rho_0$ with $\rho_0<\eta$, and take $\rho$ such that $\rho_0<\rho<\eta$.
	Since the set $\{\lvert a\rvert^{-1}(\rho_0\rho^{-1})^a\mid a\in\bbN\}$ is bounded and we have $\lvert\alpha\rvert\geq\max\{\alpha_\ell-\alpha_1,\alpha_1-\alpha_\ell\}$, it follows that the set
	\[\{\lvert\alpha_\ell-\alpha_1\rvert^{-1}(\rho_0\rho^{-1})^{\lvert\alpha\rvert}\mid\text{$\alpha=(\alpha_1,\ldots,\alpha_n)\in\bbN^n$ with $\alpha_\ell\neq\alpha_1$}\}\]
	is also bounded.
	This together with \eqref{eq: theta converge} implies that
	\begin{align*}
	\left\lvert\frac{(-1)^{\lvert P\cap\{2,\ldots,\ell-1\}\rvert}c_{\alpha,P}}{\alpha_\ell-\alpha_1}\right\rvert \rho_0^{\lvert\alpha\rvert}&=\lvert c_{\alpha,P}\rvert\rho^{\lvert\alpha\rvert}\cdot\lvert\alpha_\ell-\alpha_1\rvert^{-1}(\rho_0\rho^{-1})^{\lvert\alpha\rvert}\rightarrow 0\\
	&\hspace{110pt}(\alpha_\ell\neq\alpha_1,\hspace{10pt}\lvert\alpha\rvert\rightarrow\infty)
	\end{align*}
	for any $P$.
	Consequently \eqref{eq: eta definition} converges as an element of $\Gamma(\bbD_{<\eta}^n,\omega^{j-1}_{\cS^n/\cS,\bbQ})$.
	
	Noting that $d\underline{s}^\alpha=\underline{s}^\alpha\sum_{i=2}^n(\alpha_i-\alpha_1)d\log s_i$, we have
	\[d\theta=\sum_{\substack{P\subset\{2,\ldots,\ell\}\\ \lvert P\rvert=j}}\sum_{\substack{\alpha\in\bbN^n\\ \alpha_r=\alpha_1\ \forall r>\ell}}\sum_{i=2}^\ell c_{\alpha,P}(\alpha_i-\alpha_1)d\log s_i\wedge d\log\underline{s}_P.\]
	 Thus the equation $d\theta=0$ is equivalent to that
	 \begin{equation}\label{eq: theta is cocycle}\sum_{i\in P}(-1)^{\lvert Q\cap\{2,\ldots,i-1\}\rvert}c_{\alpha,Q\setminus\{i\}}(\alpha_i-\alpha_1)=0
	 \end{equation}
	 for every $Q\subset\{2,\ldots,\ell\}$ with $\lvert Q\rvert=j+1$ and $\alpha\in\bbN^n$.
	 Thus we have
	 \begin{align*}
	 d\eta&=\sum_{\substack{P\subset\{2,\ldots,\ell\}\\ \lvert P\rvert=j\\ P\ni\ell}}\sum_{\substack{\alpha\in\bbN^n\\ \alpha_r=\alpha_1\ \forall r>\ell\\ \alpha_\ell\neq\alpha_1}}c_{\alpha,P}\underline{s}^\alpha d\log\underline{s}_P
	 +\sum_{\substack{P\subset\{2,\ldots,\ell\}\\ \lvert P\rvert=j\\ P\reflectbox{$\notin$}\ell}}\sum_{\substack{\alpha\in\bbN^n\\ \alpha_r=\alpha_1\ \forall r>\ell\\ \alpha_\ell\neq\alpha_1}}\sum_{i\in P}(-1)^{\lvert P\cap\{i,\ldots,\ell\}\rvert +1}\frac{\alpha_i-\alpha_1}{\alpha_\ell-\alpha_1}c_{\alpha,P\cup\{\ell\}\setminus\{i\}}\underline{s}^\alpha d\log\underline{s}_P\\
	 &=\sum_{\substack{P\subset\{2,\ldots,\ell\}\\ \lvert P\rvert=j\\ P\ni\ell}}\sum_{\substack{\alpha\in\bbN^n\\ \alpha_r=\alpha_1\ \forall r>\ell\\ \alpha_\ell\neq\alpha_1}}c_{\alpha,P}\underline{s}^\alpha d\log\underline{s}_P
	 +\sum_{\substack{P\subset\{2,\ldots,\ell\}\\ \lvert P\rvert=j\\ P\reflectbox{$\notin$}\ell}}\sum_{\substack{\alpha\in\bbN^n\\ \alpha_r=\alpha_1\ \forall r>\ell\\ \alpha_\ell\neq\alpha_1}}c_{\alpha,P}\underline{s}^\alpha d\log\underline{s}_P\\
	 &=\sum_{\substack{P\subset\{2,\ldots,\ell\}\\ \lvert P\rvert=j}}\sum_{\substack{\alpha\in\bbN^n\\ \alpha_r=\alpha_1\ \forall r>\ell\\ \alpha_\ell\neq\alpha_1}}c_{\alpha,P}\underline{s}^\alpha d\log\underline{s}_P,
	 \end{align*}
	 where we used \eqref{eq: theta is cocycle} with $Q=P\cup\{\ell\}$ for the second equality.
	Now we have
	\[\theta-d\eta=\sum_{\substack{P\subset\{2,\ldots,\ell\}\\ \lvert P\rvert=j}}\sum_{\substack{\alpha\in\bbN^n\\ \alpha_r=\alpha_1\ \forall r\geq\ell}}c_{\alpha,P}\underline{s}^\alpha d\log\underline{s}_P,\]
	and this belongs to $\omega^{j-1}_{\ell-1}d\log s_\ell+\omega^j_{\ell-1}$.
	Repeating this construction, we see that the cocycle $\theta\in\Gamma(\bbD_{<\eta}^n,\omega^j_{\cS^n/\cS,\bbQ})$ is cohomologous to the cocycle  
	\begin{equation}\label{eq: theta'}\theta'=\sum_{\substack{P\subset\{2,\ldots,n\}\\ \lvert P\rvert=j}}\sum_{m\in\bbN}c_{m,P}s_1^m\cdots s_n^md\log\underline{s}_P
	\end{equation}
	where $c_{m,P}=c_{(m,\ldots,m),P}$.
	Since $\theta'$ belongs to $\Gamma(\bbD_{<\eta}^n,\omega^j_{\cS^n/\cS,\bbQ})$, for any $P$ and $\rho<\eta$ we have $\lvert c_{m,P}\rvert\rho^{mn}\rightarrow 0$ when $m\rightarrow \infty$.
	This is equivalent to the statement that $\sum_{m\in\bbN}c_{m,P}s^m$ belongs to $F\{\eta^{-n}s\}$.
	Thus the association
	\[\theta=\sum_{\substack{P\subset\{2,\ldots,n\}\\ \lvert P\rvert=j}}\sum_{\alpha\in\bbN^n}c_{\alpha,P}\underline{s}^\alpha d\log\underline{s}_P\mapsto \sum_{\substack{P\subset\{2,\ldots,n\}\\ \lvert P\rvert=j}}\left(\sum_{m\in\bbN}c_{m,P}s^m\right)\otimes d\log\underline{s}_P\]
	defines a map $H^\ast(\bbD^n_{<\eta},\omega^\bullet_{\cS^n/\cS,\bbQ})\rightarrow F\{\eta^{-n}s\}\otimes_F\overline{\bsH}_n^\ast$, which is clearly the inverse of the third map in \eqref{eq: diag polydisk}.	
	
	That the first and the fourth maps in \eqref{eq: diag polydisk} are isomorphisms is proved similarly.
	In particular, the fourth isomorphism is essentially given by Gro\ss e-Kl\"{o}nne in the proof of \cite[Lem.\ 4.8]{GK3}.
	
	To see that the second map in \eqref{eq: diag polydisk} is an isomorphism, let $\cC^{\bullet,\bullet}_{\bbD^n_{<1}}$ be the double complex of sheaves on the open unit polydisk defined as in Lemma \ref{lem: total complex} for the situation $\cZ=\cS^n$.
	Lemma \ref{lem: sheaf spectral sequence} applied to $\cC^{\bullet,\bullet}|_{\bbD^n_{< \eta}}$ induces a spectral sequence
	\[E_2^{p,q}\Rightarrow\varinjlim_m H^{p+q}(\bbD^n_{<\eta},\omega^\bullet_{\cS^n/W^\varnothing,\bbQ}[u]_m)\]
	where $E_2^{p,q}$ is the $(p+q)$\textsuperscript{th} cohomology group of the complex
	\begin{align*}
	M_q^\bullet&:=\left[H^0(\bbD^n_{<\eta},\omega^\bullet_{\cS^n/W^\varnothing,\bbQ})\xrightarrow{\wedge d\log s}H^1(\bbD^n_{<\eta},\omega^\bullet_{\cS^n/W^\varnothing,\bbQ})\xrightarrow{\wedge d\log s}\cdots\xrightarrow{\wedge d\log s}H^q(\bbD^n_{<\eta},\omega^\bullet_{\cS^n/W^\varnothing,\bbQ})\right]\\
	&\cong\left[\bsH_n^0\xrightarrow{\wedge d\log s}\bsH_n^1\xrightarrow{\wedge d\log s}\cdots\xrightarrow{\wedge d\log s}\bsH_n^q\right].
	\end{align*}
	Thus we see that $E_2^{p,q}$ is $\overline{\bsH}_n^q$ if $p=0$ and zero if $p\neq 0$.
	Consequently the spectral sequence degenerates and we obtain $\overline{\bsH}_n^q\cong \varinjlim_mH^q(\bbD^n_{<\eta},\omega^\bullet_{\cS^n/W^\varnothing,\bbQ}[u]_m)$.
	This finishes the proof of \eqref{item: polydisk1}.
	
	We prove \eqref{item: polydisk2} for $\omega^\bullet_{\cS^n/\cS,\bbQ}$.
	We denote the $i$\textsuperscript{th} differential map on $\Gamma(\bbD_{<\eta}^n,\omega^\bullet_{\cS^n/\cS,\bbQ})$ by $d^i$ and the surjection
	\[\Ker d^i\rightarrow H^i(\bbD_{<\eta}^n,\omega^\bullet_{\cS^n/\cS,\bbQ})\cong F\{\eta^{-n}s\}\otimes_F\overline{\bsH}^i_n\cong\bigoplus_{\substack{P\subset\{2,\ldots,n\}\\ \lvert P\rvert=i}}F\{\eta^{-n}s\}d\log \underline{s}_P\] 
	by $q$.
	The subspace topology on $\Ker d^i$ is the convex topology defined by the open lattices of the form
	\[L_{\rho,a}:=\Biggl\{\delta=\sum_{\substack{P\subset\{2,\ldots,n\}\\ \lvert P\rvert=i}}\sum_{\alpha\in\bbN^n}c_{\alpha,P}\underline{s}^\alpha d\log\underline{s}_P\in\Gamma(\bbD_{<\eta}^n,\omega^i_{\cS^n/\cS,\bbQ})\mid \lvert c_{\alpha,P}\rvert\rho^{\lvert\alpha\rvert}\leq a,\ d\delta=0\Biggr\}\]
	for real numbers $\eta$ and $a$ with $0<\rho<\eta$ and $0<a$.
	Thus the quotient topology on $H^i(\bbD^n_{<\eta},\omega^\bullet_{\cS^n/\cS,\bbQ})$ is defined by the open lattices
	\[q(L_{\rho,a})=\Biggl\{\sum_{\substack{P\subset\{2,\ldots,n\}\\ \lvert P\rvert=i}}\sum_{m\in\bbN}c_{m,P}s^m d\log\underline{s}_P\in\bigoplus_{\substack{P\subset\{2,\ldots,n\}\\ \lvert P\rvert=i}}F\{\eta^{-n}s\}d\log \underline{s}_P\mid \lvert c_{\alpha,P}\rvert\rho^{mn}\leq a\Biggr\}.\]
	But the topology defines as a finite $F\{\eta^{-n}s\}$-module is also defined by these lattices, hence the two topologies coincide with each other.
	
	As this topology is separated, it follows that $\Im d^{i-1}$ is closed subset of $\Gamma(\bbD^n_{<\eta},\omega^i_{\cS^n/\cS,\bbQ})$.
	By Lemma \ref{lem: convex}, $\Gamma(\bbD^n_{<\eta},\omega^i_{\cS^n/\cS,\bbQ})$ and hence $\Im d^{i-1}$ are Fr\'{e}chet spaces.
	Thus by the open mapping theorem \cite[Prop.\,8.6]{Sch} we see that $d^{i-1}\colon\Gamma(\bbD^n_{<\eta},\omega^{i-1}_{\cS^n/\cS,\bbQ})\rightarrow \Im d^{i-1}$ is an open map, namely $d^{i-1}$ is strict.
	The statements for $\omega^{\bullet}_{\cS^n/W^\varnothing,\bbQ}$ and $\omega^{\bullet}_{\cS^n_0/W_0,\bbQ}$ can be proved similarly.
\end{proof}

\begin{lemma}\label{lem: ML}
	Let $I$ be a directed set, $\bbF$ a field, and $\{M_i\}_{i\in I}=(\{M_i\}_{i\in I},\{f_{ij}\colon M_j\rightarrow M_i\}_{i\leq j})$ an inverse system of $\bbF$-vector spaces.
	\begin{enumerate}
	\item\label{item: ML1}If $M_i$ is finite dimensional for each $i\in I$, then $\{M_i\}_{i\in I}$ is a Mittag-Leffler system.
	\item\label{item: ML2}If $M:=\varprojlim_i M_i$ is finite dimensional, then the natural map $L\otimes_\bbF M\rightarrow\varprojlim_i(L\otimes_\bbF M_i)$ is an isomorphism for any $\bbF$-vector space $L$.
	\end{enumerate}
\end{lemma}

\begin{proof}
	\eqref{item: ML1} This is a well-known fact. Indeed, for each $i\in I$, $\{f_{ij}(M_j)\}_{j\geq i}$ is a decreasing system of subspaces of $M_i$, hence it eventually stabilizes if $M_i$ is finite dimensional.
	
	\eqref{item: ML2} For $i\in I$, let $\pi_i\colon M\rightarrow M_i$ be the natural projection.
	Then $\{\Ker\pi_i\}_{i\in I}$ is a decreasing system of subspaces of $M$, hence it eventually stabilizes.
	Namely we have $\Ker \pi_{j_0}=\bigcap_{i\in I} \Ker\pi_i$ for some $j_0\in I$, which however must be zero by definition of the inverse limit.
	Thus $\pi_{j_0}$ is injective.
	
	 Let $(e_\lambda)_{\lambda\in\Lambda}$ be a base of $L$.
	Then any element $x$ of $\varprojlim_i(L\otimes_\bbF M_i)$ can be uniquely written as a system of elements $\sum_\lambda e_\lambda\otimes y_{\lambda,i}\in L\otimes_\bbF M_i$ with $f_{ij}(y_{\lambda,j})=y_{\lambda,i}$.
	Let $y_\lambda:=(y_{\lambda,i})_i\in M$.
	Since the set $\{\lambda\in \Lambda\mid y_{\lambda,j_0}\neq 0\}$ is finite, the injectivity of $\pi_{j_0}$ shows that the set $\{\lambda\in \Lambda\mid y_{\lambda}\neq 0\}$ is also finite.
	Now the association $x\mapsto\sum_\lambda e_\lambda\otimes y_\lambda$ gives the inverse of $L\otimes_\bbF M\rightarrow\varprojlim_i(L\otimes_\bbF M_i)$.
\end{proof}

\begin{proposition}\label{prop: Kunneth 1}
	Let $n\geq 1$ be an integer, $\eta$ a real number with $0<\eta\leq 1$, and $\cX$ a weak formal scheme equipped with the trivial log structure.
	Suppose that the following conditions for the dagger space $\frX$ associated to $\cX$ hold.
	\begin{enumerate}\renewcommand{\labelenumi}{(\roman{enumi})}
	\item $\frX$ is separated, smooth, and admits an affinoid open covering indexed by a countable set,
	\item $H^m(\frX,\Omega^\bullet_{\frX/F})$ is finite dimensional for any $m$.
	\end{enumerate}
	We regard $\cS^n\times_{W^\varnothing}\cX$ as a weak formal log scheme over $\cS$ via the composite of the projection $\cS^n\times_{W^\varnothing}\cX\rightarrow \cS^n$ and the morphism $\cS^n\rightarrow \cS$ defined as in Lemma \ref{lem: polydisk}.
	
	Then there exist canonical isomorphisms displayed as the horizontal maps in the following diagram
	\begin{equation}\label{eq: diag Kunneth}
	\xymatrix{
	\bsH_n^\ast\otimes_FH^\ast(\frX,\Omega^\bullet_{\frX/F})\ar[r]^-{\cong}\ar[d]&
	H^\ast(\bbD^n_{<\eta}\times\frX,\omega^\bullet_{\cS^n\times\cX/W^\varnothing,\bbQ})\ar[d]\\
	\overline{\bsH}_n^{\ast}\otimes_FH^\ast(\frX,\Omega^\bullet_{\frX/F})\ar[r]^-{\cong}\ar[d]&
	H^\ast(\bbD^n_{<\eta}\times\frX,\omega^\bullet_{\cS^n\times\cX/W^\varnothing,\bbQ}[u])\ar[d]\\
	F\{\eta^{-n}s\}\otimes_F\overline{\bsH}_n^{\ast}\otimes_F H^\ast(\frX,\Omega^\bullet_{\frX/F})\ar[r]^-{\cong}\ar[d]&
	H^\ast(\bbD^n_{<\eta}\times\frX,\omega^\bullet_{\cS^n\times\cX/\cS,\bbQ})\ar[d]\\
	\overline{\bsH}^{\ast}_n\otimes_FH^\ast(\frX,\Omega^\bullet_{\frX/F})\ar[r]^-{\cong}&
	H^\ast(\bbD_{<\eta,0}^n\times\frX,\omega^\bullet_{\cS^n_0\times\cX/W^0,\bbQ}),}
	\end{equation}
	where the left vertical maps are defined similarly to Lemma \ref{lem: polydisk}.
	This diagram is commutative, and compatible with the restriction maps with respect to the immersion $\bbD^n_{\eta}\times\frX\hookrightarrow\bbD_{<\eta'}^n\times\frX$ for $\eta<\eta'\leq 1$.
	
	Moreover, similar statements hold for $\eta<1$ even if we replace $\bbD^n_{<\eta}$, $\bbD^n_{<\eta,0}$, and $F\{\eta^{-n}s\}$ by $\bbD^n_{\leq\eta}$, $\bbD^n_{\leq\eta,0}$, and $F[\eta^{-n}s]^\dagger$, respectively.
\end{proposition}

\begin{proof}	
	We will first give a decomposition
	\begin{equation}\label{eq: Kunneth closed}
	H^i(\bbD^n_{\leq\eta}\times\frX,\omega^\bullet_{\cS^n\times\cX/\cS,\bbQ})\cong \bigoplus_{j\in\bbZ}F[\eta^{-n}s]^\dagger\otimes_F\overline{\bsH}_n^j\otimes_FH^{i-j}(\frX,\Omega^\bullet_{\frX/F})
	\end{equation}
	for $\eta<1$, assuming that $\frX$ is a smooth affinoid space.
	Take an isomorphism $\Gamma(\frX,\cO_\frX)\cong F[t_1,\ldots,t_m]^\dagger/(g_1,\ldots,g_r)$.
	Then there exists a real number $\lambda_0\in\lvert F^\times\rvert\otimes\bbQ$ with $1<\lambda_0<\eta^{-1}$, such that all $g_i$ belong to the subring $F[\lambda_0^{-1}t_1,\ldots,\lambda_0^{-1}t_m]^\dagger\subset F[t_1,\ldots,t_m]^\dagger$. 
	For any $\lambda\in\lvert F^\times\rvert\otimes\bbQ$ with $1<\lambda\leq\lambda_0$, consider the affinoid dagger algebra
	\begin{align*}
	&A_\lambda:=F[\lambda^{-1}t_1,\ldots,\lambda^{-1}t_m]^\dagger/(g_1,\ldots,g_r).
	\end{align*}
	Let $\frX_{\leq\lambda}:=\Sp A_\lambda$ and $\frX_{<\lambda}:=\bigcup_{\lambda'<\lambda}\frX_{\leq\lambda'}$, which is a Stein space by construction.
	We take a weak formal model $\cX_{\leq\lambda_0}$ of $\frX_{\leq\lambda_0}$ and consider the cohomology groups $H^i(\bbD^n_{<\eta\lambda}\times\frX_{<\lambda},\omega^\bullet_{\cS^n\times\cX_{\leq\lambda_0}/\cS,\bbQ})$.
	Then we have canonical isomorphisms
	\begin{align*}
	H^i(\bbD^n_{<\eta\lambda}\times\frX_{<\lambda},\omega^\bullet_{\cS^n\times\cX_{\leq\lambda_0}/\cS,\bbQ})
	&\cong\bigoplus_{j\in\bbZ}H^j(\bbD_{<\eta\lambda}^n\times,\omega^\bullet_{\cS^n/\cS,\bbQ})\otimes_FH^{i-j}(\frX_{<\lambda},\Omega^\bullet_{\frX_{<\lambda}/F})\\
	&\cong \bigoplus_{j\in\bbZ}F\{(\eta\lambda)^{-n}s\}\otimes_F\overline{\bsH}_n^j\otimes_FH^{i-j}(\frX_{<\lambda},\Omega^\bullet_{\frX_{<\lambda}/F}).
	\end{align*}
	Indeed, the first isomorphism is proved by the same argument as the proof of \cite[Thm.\,4.12]{GK1} using Lemma \ref{lem: polydisk} \eqref{item: polydisk2}.
	The second isomorphism follows by Lemma \ref{lem: polydisk} \eqref{item: polydisk1}.
	Thus we obtain
	\begin{align*}
	H^i(\bbD^n_{\leq\eta}\times\frX,\omega^\bullet_{\cS^n\times\cX/\cS,\bbQ})
	&\cong \varinjlim_{\lambda\rightarrow 1}H^i(\bbD^n_{<\eta\lambda}\times\frX_{<\lambda},\omega^\bullet_{\cS^n\times\cX_{\leq\lambda_0}/\cS,\bbQ})\\
	&\cong \varinjlim_{\lambda\rightarrow 1}\bigoplus_{j\in\bbZ}F\{(\eta\lambda)^{-1}s\}\otimes_F\overline{\bsH}_n^j\otimes_FH^{i-j}(\frX_{<\lambda},\Omega^\bullet_{\frX_{<\lambda}/F})\\
	&\cong \bigoplus_{j\in\bbZ}F[\eta^{-n}s]^\dagger\otimes_F\overline{\bsH}_n^j\otimes_FH^{i-j}(\frX,\Omega^\bullet_{\frX/F}).
	\end{align*}
	
	When $\frX$ is quasi-compact, separated, and smooth but not necessarily an affinoid, take a finite admissible covering $\{\frX_\lambda\}_\lambda$ by affinoid open subsets.
	 Then we may obtain the decomposition \eqref{eq: Kunneth closed} for $\frX$ from those for the $\frX_\lambda$'s by using the spectral sequence associated to the open covering $\{\bbD_{\leq\eta}^n\times\frX_\lambda\}_\lambda$ of $\bbD_{\leq\eta}^n\times\frX$.
	 
	 For a general $\frX$ which satisfies the conditions (i) and (ii), we may by condition (i) take an admissible open covering $\{\frU_m\}_{m\in\bbN}$ with $\frU_m\subset\frU_{m+1}$ such that each $\frU_m$ is quasi-compact, separated, and smooth.
	  Since $H^i(\frU_m,\Omega^\bullet_{\frU_m/F})$ is finite dimensional for each $i$ and $m$ \cite[Cor.\,3.5]{GK2}, the inverse system $\{H^i(\frU_m,\Omega^\bullet_{\frU_m/F})\}_m$ for fixed $i$ is Mittag-Leffler by Lemma \ref{lem: ML} \eqref{item: ML1}.
	  Thus the system $\{F[\eta^{-n}s]^\dagger\otimes_F\overline{\bsH}^j_n\otimes_FH^i(\frU_m,\Omega^\bullet_{\frU_m/F})\}_m$ for fixed $i$, $j$ and $n$ is also Mittag-Leffler.
	 
	 Thus we have $H^i(\frX,\Omega^\bullet_{\frX/F})\cong\varprojlim_mH^i(\frU_m,\Omega^\bullet_{\frU_m/F})$ and
	 \begin{align*}
	 R^l\varprojlim_mH^i(\bbD_{\leq\eta}^n\times\frU_m,\omega^\bullet_{\cS^n\times\cX/\cS,\bbQ})&\cong R^l\varprojlim_m\bigoplus_{j\in\bbZ}F[\eta^{-n}s]^\dagger\otimes_F\overline{\bsH}^j_n\otimes_FH^{i-j}(\frU_m,\Omega^\bullet_{\frU_m/F})\\
	 &\cong
	 \begin{cases}
	 \displaystyle\bigoplus_{j\in\bbZ}F[\eta^{-n}s]^\dagger\otimes_F\overline{\bsH}^j_n\otimes_FH^{i-j}(\frX,\Omega^\bullet_{\frX/F})&(l=0)\\
	 0&(l\neq 0),
	 \end{cases}\end{align*}
	 where we use the condition (ii) and Lemma \ref{lem: ML} \eqref{item: ML2} for the second isomorphism for $l=0$.
	 This together with the short exact sequence
	 \begin{align*}0\rightarrow R^1\varprojlim_mH^{i-1}(\bbD_{\leq\eta}^n\times\frU_m,\omega^\bullet_{\cS^n\times\cX/\cS,\bbQ})
	 &\rightarrow H^i(\bbD_{\leq\eta}^n\times\frX,\omega^\bullet_{\cS^n\times\cX/\cS,\bbQ})\\
	&\rightarrow \varprojlim_mH^i(\bbD_{\leq\eta}^n\times\frU_m,\omega^\bullet_{\cS^n\times\cX/\cS,\bbQ})\rightarrow 0
	\end{align*}
	of \cite[\href{https://stacks.math.columbia.edu/tag/0CQE}{Lem.\,0CQE}]{stacks} gives an isomorphism \eqref{eq: Kunneth closed} as desired.
	
	Next we provide the decomposition of $H^i(\bbD^n_{<\eta}\times\frX,\omega^\bullet_{\cS^n\times\cX/\cS,\bbQ})$ for $\eta\leq 1$.
	We first note that
	\begin{equation}\label{eq: Stein vanishing}
	R^l\varprojlim_{\rho<\eta} F[\rho^{-n}s]^\dagger=H^l(\bbD^1_{<\eta^n},\cO_{\bbD^1_{<\eta^n}})=\begin{cases}F\{\eta^{-n}s\}^\dagger &(l=0)\\ 0&(l\neq 0)\end{cases}
	\end{equation}
	by \cite[Prop.\,3.1]{GK1}.
	Then we have isomorphisms
	\begin{align*}
	R^l\varprojlim_{\rho<\eta} H^i(\bbD^n_{\leq\rho}\times\frX,\omega^\bullet_{\cS^n\times\cX/\cS,\bbQ})
	&\cong R^l\varprojlim_{\rho<\eta}\bigoplus_{j\in\bbZ}F[\eta^{-n}s]^\dagger\otimes_F\overline{\bsH}_n^j\otimes_FH^{i-j}(\frX,\Omega^\bullet_{\frX/F})\\
	&\cong\begin{cases}
	\displaystyle\bigoplus_{j\in\bbZ}F\{\eta^{-n}s\}\otimes_F\overline{\bsH}_n^j\otimes_FH^{i-j}(\frX,\Omega^\bullet_{\frX/F})&(l=0)\\
	0&(l\neq 0),\end{cases}
	\end{align*}
	where the first isomorphism is \eqref{eq: Kunneth closed} and the second isomorphism follows from \eqref{eq: Stein vanishing} with the fact that $\overline{\bsH}_n^j\otimes_FH^{i-j}(\frX,\Omega^\bullet_{\frX/F})$ is finite dimensional.
	Thus, again by the short exact sequence \cite[\href{https://stacks.math.columbia.edu/tag/0CQE}{Lem.\,0CQE}]{stacks}, we obtain the isomorphism
	\begin{equation}\label{eq: Kunneth open}
	H^i(\bbD^n_{<\eta}\times\frX,\omega^\bullet_{\cS^n\times\cX/\cS,\bbQ})\cong \bigoplus_{j\in\bbZ}F\{\eta^{-n}s\}\otimes_F\overline{\bsH}_n^i\otimes_FH^{i-j}(\frX,\Omega^\bullet_{\frX/F}).
	\end{equation}
	
	The decompositions of the cohomlogy groups of $\omega^\bullet_{\cS^n\times\cX/W^\varnothing,\bbQ}$ and $\omega^\bullet_{\cS^n_0\times\cX/W^0,\bbQ}$ are given by the same method.
	
	By the first isomorphism in \eqref{eq: diag Kunneth}, $H^i(\bbD^n_{\leq\eta}\times\frX,\omega^\bullet_{\cS^n\times\cX/W^\varnothing,\bbQ}[u]_m)$ is independent of $\eta$.
	Then we see that the morphisms
	\begin{align*}
	&\varinjlim_mH^i(\bbD^n_{<\eta}\times\frX,\omega^\bullet_{\cS^n\times\cX/W^\varnothing,\bbQ}[u]_m)\rightarrow H^i(\bbD^n_{<\eta}\times\frX,\omega^\bullet_{\cS^n\times\cX/W^\varnothing,\bbQ}[u]),\\
	&\varinjlim_mH^i(\bbD^n_{\leq\eta}\times\frX,\omega^\bullet_{\cS^n\times\cX/W^\varnothing,\bbQ}[u]_m)\rightarrow H^i(\bbD^n_{\leq\eta}\times\frX,\omega^\bullet_{\cS^n\times\cX/W^\varnothing,\bbQ}[u])
	\end{align*}
	are quasi-isomorphisms by the same arguments as used to prove Proposition \ref{prop: HK vs colim}.
	Thus the second isomorphism in \eqref{eq: diag Kunneth} is proved by the same arguments as Lemma \ref{lem: polydisk}.
		
	Finally, when $\frX$ is a smooth affinoid space, the isomorphisms constructed above can be described explicitly.
	For example, the first isomorphism in \eqref{eq: diag Kunneth} is given by sending
	\[d\log s_{k_1}\wedge\cdots\wedge d\log s_{k_j}\otimes [\xi]\mapsto [d\log s_{k_1}\wedge\cdots\wedge d\log s_{k_j}\wedge \xi]\]
	where $\xi$ is any cocycle $\xi$ of $\Gamma(\frX,\Omega^\bullet_{\frX/F})$ and $[-]$ indicates the cohomology class of a cocycle.
	With these descriptions one can verify the commutativity of the diagram \eqref{eq: diag Kunneth} for the case of smooth affinoid, and the general case   can be deduced by the arguments as above.
\end{proof}

\begin{proposition}\label{prop: coh of Y_I}
	Let $Y$ be a strictly semistable log scheme over $k^0$ with horizontal divisor $D$, and let $Y=\bigcup_{i\in\Upsilon}Y_i$ be the decomposition into irreducible components.
	For a nonempty subset $I\subset\Upsilon$, set $Y_I := \bigcap_{i\in I}Y_i$ and $U_I^\heartsuit:=Y_I\setminus (D\cup \bigcup_{j\in\Upsilon\setminus I}Y_j)$ and equip them with the pull-back log structures from $Y$.
	Then the natural morphisms
	\begin{align*}
	&R\Gamma_\rig(Y_I/\cT)\rightarrow R\Gamma_\rig(U_I^\heartsuit/\cT)\hspace{10pt}(\text{$\cT=W^\varnothing$, $W^0$, or $\cS$}),\\
	&R\Gamma_\HK(Y_I)\rightarrow R\Gamma_\HK(U_I^\heartsuit)&
	\end{align*}
	are quasi-isomorphisms.
\end{proposition}

\begin{proof}
We focus on the proof for the log rigid cohomology over $\cS$ which is more subtle.
The proofs for the log rigid cohomology over $W^\varnothing$ and $W^0$ are similar, and then the proof for the Hyodo--Kato cohomology follows by Lemma \ref{lem: acyclic assembly lemma for sheaf} \eqref{item: sheaf assembly row} and Proposition \ref{prop: HK vs colim}.

By Proposition \ref{prop: Mayer-Vietoris} we may work locally.
Thus we may assume that $Y_I$ is strictly log smooth over $k^0(n,m)_r$ for some $1\leq r\leq n$ and $0\leq m$, and we may take a homeomorphic exact closed immersion $i\colon Y_I\hookrightarrow \cZ_I=\Spwf W\llbracket s_1,\ldots,s_r\rrbracket\times\cV_I$ where $\cV_I$ and $\cZ_I$ are constructed as in Construction \ref{const}, respectively.
Let $\cW_I\subset\cV_I$ and $\cU_I\subset\cZ_I$ be the open subsets given by inverting $s_{r+1},\ldots,s_n,t_1,\ldots,t_m$.
Let $\frZ_I$, $\frV_I$, $\frU_I$ and $\frW_I$ be the dagger spaces associated to $\cZ_I$, $\cV_I$, $\cU_I$ and $\cW_I$, respectively.
Then we have
\begin{align*}
\frZ_I=\bbD_{<1}^r\times\frV_I,&&\frU_I=\bbD_{<1}^r\times\frW_I.
\end{align*}
For $\eta<1$, let
\begin{align*}
\frZ_{I,\eta}:=\bbD^r_{\leq\eta}\times\frV_I,&&\frU_{I,\eta}:=\bbD^r_{\leq\eta}\times\frW_I.
\end{align*}

Because of the short exact sequences
\begin{align*}
&0\rightarrow R^1\varprojlim_{\eta\rightarrow 1}H^{i-1}(\frZ_{I,\eta},\omega^\bullet_{\cZ_I/\cS,\bbQ})\rightarrow H^i_\rig(Y_I/\cS)\rightarrow 
\varprojlim_{\eta\rightarrow 1}H^{i}(\frZ_{I,\eta},\omega^\bullet_{\cZ_I/\cS,\bbQ})
\rightarrow 0,\\
&0\rightarrow R^1\varprojlim_{\eta\rightarrow 1}H^{i-1}(\frU_{I,\eta},\omega^\bullet_{\cZ_I/\cS,\bbQ})\rightarrow H^i_\rig(U_I^\heartsuit/\cS)\rightarrow\varprojlim_{\eta\rightarrow 1}H^i(\frU_{I,\eta},\omega^\bullet_{\cZ_I/\cS,\bbQ})\rightarrow 0,
\end{align*}
it suffices to show that the map
\begin{equation}\label{eq: eta isom}
	H^i(\frZ_{I,\eta},\omega^\bullet_{\cZ_I/\cS,\bbQ})\rightarrow H^i(\frU_{I,\eta},\omega^\bullet_{\cZ_I/\cS,\bbQ})
\end{equation}
is an isomorphism for any $i\in\bbZ$ and $\eta<1$.

Let $\frW'_I$ and $\frW''_I$ be the open subspaces of $\frV_I$ given by inverting $s_{r+1}$ and $t_1$, respectively, and set
\begin{align*}
	\frU'_{I,\eta}:=\bbD^r_{\leq\eta}\times\frW'_I,
	&&\frU''_{I,\eta}:=\bbD^r_{\leq\eta}\times\frW''_I.
\end{align*}
Then the map \eqref{eq: eta isom} can be written as the composition of maps of the following two types:
\begin{align}
\label{eq: vertical isom}H^i(\frZ_{I,\eta},\omega^\bullet_{\cZ_I/\cS,\bbQ})\rightarrow H^i(\frU'_{I,\eta},\omega^\bullet_{\cZ_I/\cS,\bbQ}),\\
\label{eq: horizontal isom}H^i(\frZ_{I,\eta},\omega^\bullet_{\cZ_I/\cS,\bbQ})\rightarrow H^i(\frU''_{I,\eta},\omega^\bullet_{\cZ_I/\cS,\bbQ}).
\end{align}
Hence it suffices to prove that these two types of maps are isomorphisms.
We will only prove that \eqref{eq: vertical isom} is an isomorphism, following the methods of \cite{BC} and \cite[Thm.\ 2.4.4]{Sh1}.
That \eqref{eq: horizontal isom} is an isomorphism is proved in the same way.

Since \eqref{eq: vertical isom} is clearly an isomorphism when $s_{r+1}$ is invertible in $\cZ_I$, we may assume that $s_{r+1}$ is not invertible in $\cZ_I$.
Let $\cV_I^\circ$ be the weak completion of $\cV_I$ along the closed weak formal subscheme $\cX:=\cV_{I\cup\{r+1\}}$ defined by the ideal $(s_{r+1})$.
Let $\frV_I^\circ$ and $\frX$ be the dagger spaces associated to $\cV_I^\circ$ and $\cX$, respectively.
Then similar to above we have $\frV_I^\circ\cong\bbD^1_{<1}\times\frX$.
Let
\begin{align*}
	\frZ_{I,\eta}^\circ:=\bbD^r_{\leq\eta}\times\frV_I^\circ\cong \bbD^r_{\leq\eta}\times\bbD^1_{<1}\times\frX.
\end{align*}

For $\delta=p^{-\frac ab}$ with $a,b\in\bbN$, let
\begin{align*}
\frV_I^{(\delta)}:=\Sp\left(\Gamma(\frV_I,\cO_{\frV_I})\left[\frac{p^a}{s_{r+1}^b}\right]^\dagger\right),&&\frZ_{I,\eta}^{(\delta)}:=\bbD^r_{\leq\eta}\times\frV_I^{(\delta)}.
\end{align*}
In other words, $\frZ_{I,\eta}^{(\delta)}$ is the subset of $\frZ_{I,\eta}$ consisting of the points $z$ satisfying $\lvert s_{r+1}(z)\rvert\geq\delta$.
Then $\{\frZ_{I,\eta}^\circ,\frZ_{I,\eta}^{(\delta)}\}$ is an admissible covering of $\frZ_{I,\eta}$, hence we have a long exact sequence
\begin{align}
	\label{eq: MV seq}\cdots\rightarrow H^i(\frZ_{I,\eta},\omega^\bullet_{\cZ_I/\cS,\bbQ})&\rightarrow H^i(\frZ_{I,\eta}^{(\delta)},\omega^\bullet_{\cZ_I/\cS,\bbQ})\oplus H^i(\frZ_{I,\eta}^\circ,\omega^\bullet_{\cZ_I/\cS,\bbQ})\\
	\nonumber&\rightarrow H^i(\frZ_{I,\eta}^{(\delta)}\cap\frZ_{I,\eta}^\circ,\omega^\bullet_{\cZ_I/\cS,\bbQ})\rightarrow H^{i+1}(\frZ_{I,\eta},\omega^\bullet_{\cZ_I/\cS,\bbQ})\rightarrow\cdots.
\end{align}

We will prove later that the map
\begin{equation}\label{eq: BC isom}
H^i(\frZ_{I,\eta}^\circ,\omega^\bullet_{\cZ_I/\cS,\bbQ})\rightarrow H^i(\frZ_{I,\eta}^{(\delta)}\cap\frZ_{I,\eta}^\circ,\omega^\bullet_{\cZ_I/\cS,\bbQ})
\end{equation}
is an isomorphism for any $i$.
If this is true, the exactness of \eqref{eq: MV seq} implies that the map
\begin{equation}\label{eq: delta independent}
	H^i(\frZ_{I,\eta},\omega^\bullet_{\cZ_I/\cS,\bbQ})\rightarrow H^i(\frZ_{I,\eta}^{(\delta)},\omega^\bullet_{\cZ_I/\cS,\bbQ})
\end{equation}
is an isomorphism.

Noting that $\frU'_{I,\eta}$ and $\frZ_{I,\eta}^{(\delta)}$ are affinoid spaces, the cohomology groups of $\omega^\bullet_{\cZ_I/\cS,\bbQ}$ on them can be computed by global sections.
Since $\Gamma(\frU'_{I,\eta},\cO_{\frU_{I,\eta}})=\varinjlim_{\delta\rightarrow 1}\Gamma(\frZ_{I,\eta}^{(\delta)},\cO_{\frZ_{I,\eta}^{(\delta)}})$, we have
	\begin{equation*}\label{eq: direct limit cohomology}
	H^i(\frU'_{I,\eta},\omega^\bullet_{\cZ_I/\cS,\bbQ})=\varinjlim_{\delta\rightarrow 1}H^i(\frZ_{I,\eta}^{(\delta)},\omega^\bullet_{\cZ_I/\cS,\bbQ}).
	\end{equation*}
This together with \eqref{eq: delta independent} shows that \eqref{eq: vertical isom} and hence \eqref{eq: eta isom} are isomorphisms, as claimed.

It remains to show that \eqref{eq: BC isom} is an isomorphism.
For $\rho\in\bbR$ with $0<\rho<1$, let
\[\frZ_{I,\eta,\rho}:=\bbD^r_{\leq\eta}\times\bbD^1_{\leq\rho}\times\frX.\]
Then it suffices to show that the map
\[H^i(\frZ_{I,\eta,\rho},\omega^\bullet_{\cZ_I/\cS,\bbQ})\rightarrow H^i(\frZ_{I,\eta}^{(\delta)}\cap\frZ_{I,\eta,\rho},\omega^\bullet_{\cZ_I/\cS,\bbQ})\]
is an isomorphism for any $i\in\bbZ$ and $\eta,\delta,\rho\in\bbR$ with $0<\eta<1$ and $0<\delta<\rho<1$, by the same argument as when we showed \eqref{eq: eta isom} is an isomorphism.
For this, we will construct a homotopy inverse of the natural map
\[\varrho\colon \Gamma(\frZ_{I,\eta,\rho},\omega^\bullet_{\cZ_I/\cS,\bbQ})\rightarrow\Gamma(\frZ_{I,\eta}^{(\delta)}\cap\frZ_{I,\eta,\rho},\omega^\bullet_{\cZ_I/\cS,\bbQ}).\]
Note that we have
\[\frZ_{I,\eta}^{(\delta)}\cap\frZ_{I,\eta,\rho}=\bbD_{\leq\eta}^r\times\bbD^1_{[\delta,\rho]}\times\frX\]
where $\bbD^1_{[\delta,\rho]}$ denotes the closed annulus of radius $[\delta,\rho]$. 
For any subset $P=\{ i_1,\ldots,i_a\}\subset\{ 2,\ldots,r\}$ we let
\[d\log\underline{s}_P:=d\log s_{i_1}\wedge d\log s_{i_2}\wedge \cdots\wedge d\log s_{i_a}.\]
We may suppose that $\omega^1_{\cX/W^\varnothing,\bbQ}$ is free.
Choose a basis $\tau_1,\ldots,\tau_q$ of $\omega^1_{\cX/W^\varnothing,\bbQ}$, and set $\underline{\tau}_Q:=\tau_{i_1}\wedge\cdots\wedge\tau_{i_\nu}$ for each subset $Q=\{i_1,\ldots,i_\nu\}\subset\{1,\ldots,q\}$.
Take a surjection
	\begin{equation}\label{eq: Y representation}
	W[y_1,\ldots,y_h]^\dagger\rightarrow \Gamma(\cX,\cO_{\cX}).
	\end{equation}
Since $d\log s_1=-\sum_{j=2}^nd\log s_j$, any element $\theta\in\Gamma(\frZ_{I,\eta}^{(\delta)}\cap\frZ_{I,\eta,\rho},\omega^i_{\cZ_I/\cS,\bbQ})$ can be represented as
	\begin{align}\label{eq: theta expansion}
	\theta=&\sum_{P\subset\{2,\ldots,r\}}\sum_{\substack{Q\subset\{1,\ldots,q\}\\ \lvert Q\rvert=i-\lvert P\rvert}}\sum_{\alpha\in\bbN^r}\sum_{\beta\in\bbZ}\sum_{\gamma\in\bbN^h}a_{\alpha,\beta,\gamma,P,Q}\underline{s}^\alpha s_{r+1}^\beta\underline{y}^\gamma d\log\underline{s}_P\wedge \underline{\tau}_Q\\
	\nonumber&+\sum_{P\subset\{2,\ldots,r\}}\sum_{\substack{Q\subset\{1,\ldots,q\}\\ \lvert Q\rvert=i-\lvert P\rvert-1}}\sum_{\alpha\in\bbN^r}\sum_{\beta\in\bbZ}\sum_{\gamma\in\bbN^h}b_{\alpha,\beta,\gamma,P,Q}\underline{s}^\alpha s_{r+1}^\beta\underline{y}^\gamma d\log\underline{s}_P\wedge d\log s_{r+1}\wedge \underline{\tau}_Q,
	\end{align}
with $\underline{s}^\alpha:=s_1^{\alpha_1}\cdots s_r^{\alpha_r}$ and $\underline{y}^\gamma:=y_1^{\gamma_1}\cdots y_h^{\gamma_h}$, such that there exists $\lambda>1$ satisfying
\begin{align}
	\label{eq: convergence outside}&\lvert a_{\alpha,\beta,\gamma,P,Q}\rvert \eta^{\lvert\alpha\rvert}\rho^\beta\lambda^{\lvert\alpha\rvert+\beta+\lvert\gamma\rvert}\rightarrow 0,
	&&\lvert b_{\alpha,\beta,\gamma,P,Q}\rvert \eta^{\lvert\alpha\rvert}\rho^\beta\lambda^{\lvert\alpha\rvert+\beta+\lvert\gamma\rvert}\rightarrow 0&&
	(\lvert\alpha\rvert+\beta+\lvert\gamma\rvert\rightarrow \infty),\\
	\label{eq: convergence inside}&\lvert a_{\alpha,\beta,\gamma,P,Q}\rvert \eta^{\lvert\alpha\rvert}\delta^\beta\lambda^{\lvert\alpha\rvert-\beta+\lvert\gamma\rvert}\rightarrow 0,
	&&\lvert b_{\alpha,\beta,\gamma,P,Q}\rvert \eta^{\lvert\alpha\rvert}\delta^\beta\lambda^{\lvert\alpha\rvert-\beta+\lvert\gamma\rvert}\rightarrow 0
	&&(\lvert\alpha\rvert-\beta+\lvert\gamma\rvert\rightarrow \infty).
\end{align}
Let
	\begin{align*}
	\sigma_{\alpha,\beta,P}:=\sum_{\substack{Q\subset\{1,\ldots,q\}\\ \lvert Q\rvert=i-\lvert P\rvert}}\sum_{\gamma\in\bbN^h}a_{\alpha,\beta,\gamma,P,Q}\underline{y}^\gamma\underline{\tau}_Q,
	&&\xi_{\alpha,\beta,P}:=\sum_{\substack{Q\subset\{1,\ldots,q\}\\ \lvert Q\rvert=i-\lvert P\rvert-1}}\sum_{\gamma\in\bbN^h}b_{\alpha,\beta,\gamma,P,Q}\underline{y}^\gamma\underline{\tau}_Q,
	\end{align*}
which are well-defined as elements of $\Gamma(\frX,\omega^{i-\lvert P\rvert}_{\cX/W^\varnothing,\bbQ})$ and $\Gamma(\frX,\omega^{i-\lvert P\rvert-1}_{\cX/W^\varnothing,\bbQ})$, respectively.
Thus we may write as
\begin{equation}\label{eq: theta}
\theta=\sum_{\alpha,\beta,P}\underline{s}^\alpha s_{r+1}^\beta d\log\underline{s}_P\wedge \sigma_{\alpha,\beta,P}+\sum_{\alpha,\beta,P}\underline{s}^\alpha s_{r+1}^\beta d\log\underline{s}_P\wedge d\log s_{r+1}\wedge\xi_{\alpha,\beta,P}
\end{equation}
where $\alpha$, $\beta$, and $P$ run through the same range as in \eqref{eq: theta expansion}, and the differential forms $\sigma_{\alpha,\beta,P}$ and $\xi_{\alpha,\beta,P}$ are determined uniquely from $\theta$.

Now we define
\begin{align}
	&\cR(\theta):=\sum_{\substack{\alpha,\beta,P\\ \beta\geq 0}}\underline{s}^\alpha s_{r+1}^\beta d\log\underline{s}_P\wedge \sigma_{\alpha,\beta,P}+\sum_{\substack{\alpha,\beta,P\\ \beta\geq 0}}\underline{s}^\alpha s_{r+1}^\beta d\log\underline{s}_P\wedge d\log s_{r+1}\wedge\xi_{\alpha,\beta,P},\\
	\label{eq: definition of H}&\cH(\theta):=\sum_{\substack{\alpha,\beta,P\\ \beta<0}}\frac{(-1)^{\lvert P\rvert}}{\beta-\alpha_1}\underline{s}^\alpha s_{r+1}^\beta d\log\underline{s}_P\wedge\xi_{\alpha,\beta,P}.
\end{align}
It is clear that $\cR$ defines a map of complexes $\cR\colon\Gamma(\frZ_{I,\eta}^{(\delta)}\cap\frZ_{I,\eta,\rho},\omega^\bullet_{\cZ_I/\cS,\bbQ})\rightarrow\Gamma(\frZ_{I,\eta,\rho},\omega^\bullet_{\cZ_I/\cS,\bbQ})$.

Moreover one can see that $\cH$ is well-defined as an endomorphism of degree $-1$ on $\Gamma(\frZ_{I,\eta}^{(\delta)}\cap\frZ_{I,\eta,\rho},\omega^\bullet_{\cZ_I/\cS,\bbQ})$ by Lemma \ref{lem: H is well defined} stated below.
By explicit computation  $d\circ\cH+\cH\circ d=1-\varrho\circ\cR$, in other words $\cR$ is a homotopy inverse of $\varrho$, as desired.
\end{proof}

\begin{lemma}\label{lem: H is well defined}
	The infinite sum $\cH(\theta)$ in \eqref{eq: definition of H} is a well-defined element of $\Gamma(\frZ_{I,\eta}^{(\delta)}\cap\frZ_{I,\eta,\rho},\omega^{i-1}_{\cZ_I/\cS,\bbQ})$.
\end{lemma}

\begin{proof}
	We first note that
	\begin{equation}\label{eq: H theta}\cH(\theta)=\sum_{P\subset\{2,\ldots,r\}}\sum_{\substack{Q\subset\{1,\ldots,q\}\\ \lvert Q\rvert=i-\lvert P\rvert-1}}\sum_{\alpha\in\bbN^r}\sum_{\beta<0}\sum_{\gamma\in\bbN}\frac{(-1)^{\lvert P\rvert}}{\beta-\alpha_1} b_{\alpha,\beta,\gamma,P,Q}\underline{s}^\alpha s_{r+1}^\beta \underline{y}^\gamma d\log\underline{s}_P\wedge d\log s_{r+1}\wedge\underline{\tau}_Q.
	\end{equation}
	Take real numbers 
	$1<\lambda'<\lambda$ such that $\lambda$ satisfies \eqref{eq: convergence outside} and \eqref{eq: convergence inside}.
	For $\alpha$, $\beta$, $\gamma$, $P$ and $Q$ as in the range of the sum in \eqref{eq: H theta}, we have $\alpha_1-\beta>0$ and hence
	\begin{align}
	\label{eq: convergence inside 2}
	&\left\lvert \frac{(-1)^{\lvert P\rvert}}{\beta-\alpha_1}b_{\alpha,\beta,\gamma,P,Q}\right\rvert \eta^{\lvert\alpha\rvert}\delta^\beta\lambda'^{\lvert\alpha\rvert-\beta+\lvert\gamma\rvert}\\
	\nonumber&=\lvert b_{\alpha,\beta,\gamma,P,Q}\rvert \eta^{\lvert\alpha\rvert}\delta^\beta\lambda^{\lvert\alpha\rvert-\beta+\lvert\gamma\rvert}\cdot \left(\frac{\lambda'}\lambda\right)^{\alpha_2+\cdots+\alpha_r+\lvert\gamma\rvert}\cdot \frac{1}{\lvert\alpha_1-\beta\rvert}\left(\frac{\lambda'}\lambda\right)^{\alpha_1-\beta}\\
	\nonumber&\xrightarrow[\lvert\alpha\rvert-\beta+\lvert\gamma\rvert\rightarrow\infty]{} 0,
	\end{align}
	noting that the set 
	\[\left\{\frac{1}{\lvert \nu\rvert}\left(\frac{\lambda'}\lambda\right)^\nu\mid \nu\in\bbN\right\}\]
	is bounded.
	Moreover we have
	\begin{align}
	\label{eq: convergence outside 2}&\left\lvert \frac{(-1)^{\lvert P\rvert}}{\beta-\alpha_1}b_{\alpha,\beta,\gamma,P,Q}\right\rvert \eta^{\lvert\alpha\rvert}\rho^\beta\lambda'^{\lvert\alpha\rvert+\beta+\lvert\gamma\rvert}\\
	\nonumber&=\lvert b_{\alpha,\beta,\gamma,P,Q}\rvert\eta^{\lvert\alpha\rvert}\delta^\beta\lambda^{\lvert\alpha\rvert-\beta+\lvert\gamma\rvert}\cdot \left(\frac{\lambda'^2\rho}{\delta}\right)^\beta\left(\frac{\lambda'}\lambda\right)^{\alpha_2+\cdots+\alpha_r+\lvert\gamma\rvert}\cdot \frac{1}{\lvert\alpha_1-\beta\rvert}\left(\frac{\lambda'}\lambda\right)^{\alpha_1-\beta}\\
	\nonumber&\xrightarrow[\lvert\alpha\rvert+\beta+\lvert\gamma\rvert\rightarrow\infty]{} 0,
	\end{align}
	since $\lvert\alpha\rvert-\beta+\lvert\gamma\rvert\rightarrow\infty$ when $\lvert\alpha\rvert+\beta+\lvert\gamma\rvert\rightarrow\infty$.
	
	The convergence of \eqref{eq: convergence inside 2} and \eqref{eq: convergence outside 2} show that the series
	\[\sum_{\alpha\in\bbN^r}\sum_{\beta<0}\sum_{\gamma\in\bbN}\frac{(-1)^{\lvert P\rvert}}{\beta-\alpha_1} b_{\alpha,\beta,\gamma,P,Q}\underline{s}^\alpha s_{r+1}^\beta \underline{y}^\gamma \]
	convergents on $\bbD^r_{\leq\eta\lambda'}\times\bbD^1_{[\delta\lambda'^{-1},\rho\lambda']}\times\bbD^h_{\leq\lambda'}$ for any $P$ and $Q$, and hence define an element of $\Gamma(\frZ^{(\delta)}_{I,\eta}\cap\frZ_{I,\eta,\rho},\cO_{\frZ_I})$.
	Thus $\cH(\theta)$ is well-defined as an element of $\Gamma(\frZ_{I,\eta}^{(\delta)}\cap\frZ_{I,\eta,\rho},\omega^{i-1}_{\cZ_I/\cS,\bbQ})$.
\end{proof}

\begin{corollary}\label{cor: coh of Y_I}
	Let $Y$ be a strictly semistable log scheme over $k^0$ with horizontal divisor $D$, and let $Y=\bigcup_{i\in\Upsilon}Y_i$ be the decomposition into irreducible components.
	For a non-empty subset $I\subset\Upsilon$, we define $Y_I$ and $U_I^\heartsuit$ as in Proposition \ref{prop: coh of Y_I}.
	Let $U_I^{\heartsuit,\varnothing}$ be the scheme $U_I^\heartsuit$ endowed with the trivial log structure.
	Let $\bsH^\ast_I=\bigoplus_{j\geq 0}\bsH_I^j$ be the exterior algebra of the $F$-vector space freely generated by symbols $d\log x_i$ for $i\in I$, and let $\overline{\bsH}_I^\ast=\bigoplus_{j\geq 0}\overline{\bsH}_I^j$ be its quotient algebra defined by the relation $\sum_{i\in I}d\log x_i=0$.
	
	We suppose that
	\begin{enumerate}\renewcommand{\labelenumi}{(\roman{enumi})}
	\item $Y$ admits an affine open covering indexed by a countable set,
	\item $H^m_\rig(U_I^{\heartsuit,\varnothing}/W^\varnothing)$ is finite dimensional for any $I$ and $m$,
	\item a strict log smooth morphism $Y\rightarrow k^0(n,m)$ exists globally.
	\end{enumerate}
	Then there exists the following commutative diagram
	\begin{equation}\label{eq: diag Y and U}\xymatrix{
	H^\ast_\rig(U_I^{\heartsuit,\varnothing}/W^\varnothing)\otimes_F \bsH_I^\ast\ar[r]^-\cong\ar[d]
	&H^\ast_\rig(U_I^\heartsuit/W^\varnothing)\ar[d]
	&H^\ast_\rig(Y_I/W^\varnothing)\ar[d]\ar[l]_-\cong\\
	H^\ast_\rig(U_I^{\heartsuit,\varnothing}/W^\varnothing)\otimes_F \overline{\bsH}_I^\ast \ar[r]^-{\cong}\ar[d]
	&H^\ast_\HK(U_I^\heartsuit)\ar[d]
	&H^\ast_\HK(Y_I)\ar[d]\ar[l]_-\cong\\
	H^\ast_\rig(U_I^{\heartsuit,\varnothing}/W^\varnothing)\otimes_F \overline{\bsH}_I^\ast\otimes_F F\{s\}\ar[r]^-{\cong}\ar[d]
	&H^\ast_\rig(U_I^\heartsuit/\cS)\ar[d]
	&H^\ast_\rig(Y_I/\cS)\ar[d]\ar[l]_-\cong\\
	H^\ast_\rig(U_I^{\heartsuit,\varnothing}/W^\varnothing)\otimes_F \overline{\bsH}_I^\ast\ar[r]^-{\cong}
	&H^\ast_\rig(U_I^\heartsuit/W^0)
	&H^\ast_\rig(Y_I/W^0)\ar[l]_-\cong
	}\end{equation}
	where all horizontal maps are isomorphisms.
\end{corollary}
\begin{proof}
	The isomorphisms on the right are given by Proposition \ref{prop: coh of Y_I}.
	We will prove the existence of the isomorphisms on the left.
	For $i=1,\ldots,n$, we denote by $1_i\in\bbN^n$ the element whose $i$\textsuperscript{th} component is one and all other components are zero. 
	Denote by $s_i$ the image of $1_i$ under the composition $\bbN^n\rightarrow \Gamma(k^0(n,m),\cO_{k^0(n,m)})\rightarrow\Gamma(Y,\cO_Y)$.
	By permutation we may assume that $s_1,\ldots,s_r$ correspond to elements of $I$ and $s_{r+1},\ldots,s_n$ are invertible in $U_I^\heartsuit$.
	Set $x_i:=s_i$ for $i=1,\ldots,r-1$ and $x_r:=s_rs_{r+1}\cdots s_n$. 
	Then the map $\bbN^r\rightarrow\Gamma(U_I^\heartsuit,\cO_{U_I^\heartsuit});\ 1_i\mapsto x_i$ gives a chart of the log structure of $U_I^\heartsuit$.
	
	When $U_I^{\heartsuit,\varnothing}$ is affine, we can take a $p$-adic smooth weak formal scheme $\cU$ over $W$ which lifts $U_I^{\heartsuit,\varnothing}$.
	We endow $\cU$ with the trivial log structure.
	Then the de Rham cohomology of the dagger space associated to $\cU$ computes $R\Gamma_\rig(U^{\heartsuit,\varnothing}/W^\varnothing)$.
	Moreover, the cohomologies of $U_I^\heartsuit$ appearing in the diagram \eqref{eq: diag Y and U} are computed as the cohomologies on the dagger space associated to $\cS^r\times\cU$.
	Thus by Proposition \ref{prop: Kunneth 1} we obtain the left horizontal isomorphisms in \eqref{eq: diag Y and U}.
	
	For the case that $U_I^{\heartsuit,\varnothing}$ is not necessarily affine, we may give the left horizontal isomorphisms in \eqref{eq: diag Y and U} by the same arguments as Proposition \ref{prop: Kunneth 1}.
\end{proof}

\begin{remark}
The isomorphisms in the bottom row in the diagram \eqref{eq: diag Y and U} for the case that the horizontal divisor is empty has been proved by Gro\ss e-Kl\"{o}nne \cite[Lem.\ 4.4]{GK3}.
\end{remark}

Let $Y$ be a strictly semistable log scheme over $k^0$ with horizontal divisor $D$. 
Then $U:=Y\backslash D$ is strictly semistable with empty horizontal divisor.
The pair $(U,Y)$ is a strictly semistable $k^0$-log scheme with boundary in the sense of \cite[Def.~2.43]{EY}.
Intuitively, the different rigid cohohomology theories for $U$ and $Y$ should coincide. 
Indeed, we have the following.

\begin{proposition}\label{prop: horizontal divisor}
Let $Y$ and $U$ be as above.
The natural morphisms
\begin{align*}
 R\Gamma_{\rig}(Y/V^\sharp)& \rightarrow R\Gamma_{\rig}(U/V^\sharp)  \\
 R\Gamma_{\rig}(Y/W^0) &\rightarrow R\Gamma_{\rig}(U/W^0)  \\
 R\Gamma_{\rig}(Y/W^\varnothing) &\rightarrow R\Gamma_{\rig}(U/W^\varnothing) \\
 R\Gamma_{\rig}(Y/\cS) &\rightarrow R\Gamma_{\rig}(U/\cS)\\
 R\Gamma_{\HK}(Y) &\rightarrow R\Gamma_{\HK}(U) 
\end{align*}
are quasi-isomorphisms.
\end{proposition}

\begin{proof}
Similarly to \cite[Lem.~4.4]{GK3} overconvergence is crucial for the statement to be true. 
The first two quasi-isomorphisms were already proved in \cite[Cor.~3.4]{EY}. 

To prove the third quasi-isomorphism, we may assume that there exists a 
log rigid datum $(\cZ,i)$ for $Y$ over $k^\varnothing\hookrightarrow W^\varnothing$ due to Proposition \ref{prop: Mayer-Vietoris}.
By Propoistion \ref{prop: coh of Y_I}, we obtain quasi-isomorphisms
	\begin{equation}\label{eq: Y_I and U_I}
	R\Gamma_\rig(Y_I/W^\varnothing)\xrightarrow{\cong} R\Gamma_\rig(U_I/W^\varnothing)\xrightarrow{\cong} R\Gamma_\rig(U^\heartsuit_I/W^\varnothing).
	\end{equation}
We denote by $\frZ$ the dagger space associated to $\cZ$.
Let $\frU$, $\frU_I$, and $\frZ_I$ be the dagger spaces associated to the  
exactifications of $\cZ$ along $U$, $U_I$, and $Y_I$, respectively.
Then the $\frU_I$'s and $\frZ_I$'s give admissible coverings of $\frU$ and $\frZ$, respectively.
Moreover, the restriction of $\omega^\bullet_{\cZ/W^\varnothing,\bbQ}$ to $\frZ_I$ coincides with $\omega^\bullet_{\cZ_I/W^\varnothing,\bbQ}$ by Propositions \ref{prop: fundamental properties of omega} and \ref{prop: exactification is log etale}. 
Thus $R\Gamma_\rig(Y/W^\varnothing)\rightarrow R\Gamma_\rig(U/W^\varnothing)$ is a quasi-isomorphism by \eqref{eq: Y_I and U_I} and cohomological descent for admissible coverings of dagger spaces.
The fourth and fifth quasi-isomorphisms also follow by the same arguments.
\end{proof}

\begin{remark}
	To be precise, the log rigid cohomology $R\Gamma_\rig(Y/V^\sharp)$ of a log scheme $Y$ over $k^0$ depends on the choice of a unifomrizer $\pi\in V$, because the inclusion $k^0\hookrightarrow V^\sharp$ depends on that choice.
	However it does not play an important role here.
\end{remark}

We will now define the monodromy operator on the log rigid cohomology over $\cS$ and $W^0$.
Let $(Y,\cZ,i)$ be a log rigid triple over $k^0\hookrightarrow\cS$.
By Proposition \ref{prop: fundamental properties of omega} we have an exact sequence of locally free sheaves
\[0\rightarrow \cO_\cZ d\log s\rightarrow \omega^1_{\cZ/W^\varnothing}\rightarrow\omega^1_{\cZ/\cS}\rightarrow 0,\]
which locally splits.
Thus we have a short exact sequence
\begin{equation}\label{eq: ses S}
0\rightarrow \omega^\bullet_{\cZ/\cS,\bbQ}[-1]\xrightarrow{\wedge d\log s}\omega^\bullet_{\cZ /W^\varnothing,\bbQ}\rightarrow\omega^\bullet_{\cZ /\cS,\bbQ}\rightarrow 0.
\end{equation}
If we set $\cY:=\cZ\times_\cS W^0$, then we have $\omega^1_{\cZ/\cS}\otimes_{\cO_\cZ}\cO_\cY=\omega^1_{\cY/W^0}$ again by Proposition \ref{prop: fundamental properties of omega}.
Thus\eqref{eq: ses S} induces a short exact sequence
\begin{equation}\label{eq: ses W} 0\rightarrow \omega^\bullet_{\cY/W^0,\bbQ}[-1]\xrightarrow{\wedge d\log s}\wt\omega^\bullet_{\cY,\bbQ}\rightarrow\omega^\bullet_{\cY/W^0,\bbQ}\rightarrow 0.
\end{equation}

\begin{definition}
Let $Y$ be a fine log scheme over $k^0$ and consider a simplicial system of log rigid triples $(U_\bullet,\cZ_\bullet,i_\bullet)$ over $k^0\hookrightarrow\cS$ as in Definition \ref{def: HK coh global}.
Then the exactness of \eqref{eq: ses S} for $\cZ_m$  and \eqref{eq: ses W} for $\cY_m:=\cZ_m\times_\cS W^0$ induce distinguished triangles
\begin{align*}
&R\Gamma_\rig(Y/\cS)[-1]\rightarrow R\Gamma_\rig(Y/W^\varnothing)\rightarrow R\Gamma_\rig(Y/\cS)\rightarrow R\Gamma_\rig(Y/\cS),\\
&R\Gamma_\rig(Y/W^0)[-1]\rightarrow s(R\Gamma(\frY_\bullet,\wt\omega_{\cY,\bbQ}^\bullet))\rightarrow R\Gamma_\rig(Y/W^0)\rightarrow R\Gamma_\rig(Y/W^0)
\end{align*}
where $\frY_m$ denotes the dagger space associated to $\cY_m$ for any $m\in\bbN$.
The connecting homomorphisms respectively induce endomorphisms in the derived category on $R\Gamma_\rig(Y/\cS)$ and $R\Gamma_\rig(Y/W^0)$, 
which we call the \textit{monodromy operators}.
\end{definition}

Next we give relations between the log rigid cohomology, the Hyodo--Kato cohomology, and the cohomologies of some variants of the Kim--Hain complex.

\begin{proposition}\label{prop: HK and rig} 
	Let $(Y,\cZ,i,\phi)$ be a  log rigid $F$-quadruple on $(k^0\hookrightarrow\cS,\sigma)$, and suppose that $Y$ is strictly semistable.
	Let $\cY: = \cZ\times_\cS W^0$ and denote its associated dagger space by $\frY$.
	Let $\tau\colon\frY\hookrightarrow\frZ$ be the canonical closed immersion.
	\begin{enumerate}
	\item\label{item: HK rig 1} The natural map $\omega^\bullet_{\cZ/W^\varnothing,\bbQ} \rightarrow  \tau_*\widetilde{\omega}^\bullet_{\cY,\bbQ}$
		induces a quasi-isomorphism
			\[R\Gamma_\rig(Y/W^\varnothing) = R\Gamma(\frZ,\omega^\bullet_{\cZ/W^\varnothing,\bbQ}) \xrightarrow{\cong} R\Gamma(\frY,\widetilde{\omega}^\bullet_{\cY,\bbQ}),\]
		which is compatible with Frobenius operators.
	\item\label{item: HK rig 2} We have canonical quasi-isomorphisms
			\[R\Gamma_\HK(Y)\xleftarrow{\cong} \hocolim_mR\Gamma(\frZ,\omega^\bullet_{\cZ/ W^\varnothing,\bbQ}[u]_m) \xrightarrow{\cong} \hocolim_mR\Gamma(\frY,\widetilde{\omega}^\bullet_{\cY,\bbQ}[u]_m),\]
		which are compatible with Frobenius and monodromy operators.
		Here the second quasi-isomorphism is induced by tnatural map $\omega^\bullet_{\cZ/W^\varnothing,\bbQ}[u] \rightarrow  \tau_*\widetilde{\omega}^\bullet_{\cY,\bbQ}[u]$.
	\item\label{item: HK rig 3} The maps $\omega^\bullet_{\cZ/W^\varnothing,\bbQ}\llbracket u\rrbracket \rightarrow \omega^\bullet_{\cZ/\cS,\bbQ}$ and $\widetilde{\omega}_{\cY,\bbQ}^\bullet\llbracket u\rrbracket \rightarrow \omega^\bullet_{\cY/W^0,\bbQ}$ defined by $u^{[i]}\mapsto 0$ for $i>0$ 
		induce quasi-isomorphisms
			\begin{eqnarray*}
			 & & R\Gamma(\frZ,\omega^\bullet_{\cZ/W^\varnothing,\bbQ}\llbracket u\rrbracket) \xrightarrow{\cong}R\Gamma(\frZ,\omega^\bullet_{\cZ/\cS,\bbQ}) = R\Gamma_\rig(Y/\cS),\\
			 & & R\Gamma(\frY,\widetilde{\omega}^\bullet_{\cY,\bbQ}\llbracket u\rrbracket) \xrightarrow{\cong}R\Gamma(\mathfrak{Y},\omega^\bullet_{\cY/W^0,\bbQ}) = R\Gamma_\rig(Y/W^0),
			\end{eqnarray*}
		which are compatible with Frobenius and monodromy operators.
	\item\label{item: HK rig 4} 
	We have canonical quasi-isomorphisms
			\[R\Gamma_\HK(Y) \cong
			\hocolim_mR\Gamma(\frY,\widetilde{\omega}^\bullet_{\cY,\bbQ}[u]_m) \xrightarrow{\cong}R\Gamma(\frY,\widetilde{\omega}^\bullet_{\cY,\bbQ}\llbracket u\rrbracket)
			\xrightarrow{\cong}R\Gamma_\rig(Y/W^0),\]
		which are compatible with Frobenius and monodromy operators.
	\end{enumerate}
\end{proposition}

\begin{proof}
	\eqref{item: HK rig 1} The compatibility with Frobenius operators is obvious. To prove the quasi-isomorphism, it suffices to prove that the cohomology groups of $s\cdot \omega^\bullet_{\cZ/W^\varnothing,\bbQ}=\Ker (\omega^\bullet_{\cZ/W^\varnothing,\bbQ}\rightarrow\tau_*\widetilde{\omega}^\bullet_{\cY,\bbQ})$ vanish.
	Since we may work locally on $\frZ$, it suffices to show the assertion with $Y$ replaced by $Y_I$ for the same reason as in the proof of Proposition \ref{prop: HK vs colim} \eqref{item: ss hocolim}. 
	Moreover, that $R\Gamma(\frZ,s\cdot\omega^\bullet_{\cZ/W^\varnothing,\bbQ})$ with Frobenius action is independent of the choice of the log rigid $F$-datum $(\cZ,\phi)$ can be proved by the same argument as Proposition \ref{prop: log rigid coh}.
	Thus, similarly to the proof of Proposition \ref{prop: coh of Y_I}, we may assume that we have a homeomorphic exact closed immersion $i\colon Y_I\hookrightarrow\cZ_I=\Spwf W\llbracket s_1,\ldots,s_r\rrbracket\times\cV_I$ where $\cV_I$ and $\cZ_I$ are constructed as in Construction \ref{const}, respectively.
	Let $\frZ_I$ and $\frV_I$ be the dagger spaces associated to $\cZ_I$ and $\cV_I$, respectively.
	Recall that we have $\frZ_I=\frV_I\times\bbD^r_{<1}$, and let $\frZ_{I,\eta}=\frV_I\times\bbD^r_{\leq\eta}$ for $\eta<1$.
Since the $\frZ_I$'s form an admissible covering of $\frZ$ and the $\frZ_{I,\eta}$'s form an admissible covering of $\frZ_I$, it suffices to show that $H^n(\frZ_{I,\eta},s\cdot\omega^\bullet_{\cZ_I/W^\varnothing,\bbQ})=0$.
	
	Let $A:=\Gamma(\frV_I,\cO_{\frV_I})$ and $B:=\Gamma(\frZ_{I,\eta},\cO_{\frZ_{I,\eta}})$.
	Note that $B$ is an $A$-subalgebra of $A\llbracket s_1,\ldots,s_r\rrbracket$, and we have
	\begin{equation}\label{eq: decomp omega}
	\Gamma(\frZ_{I,\eta},\omega^1_{\cZ_I/W^\varnothing,\bbQ})=\bigoplus_{1\leq a\leq n}Bd\log s_a\oplus \bigoplus_{1\leq b\leq m}Bd\log t_b \oplus\bigoplus_{1\leq c\leq\ell}Bd\log x_c.
	\end{equation}
	Let $\cV'_I$ be the weak formal log scheme whose underlying weak formal scheme is $\cV_I$ and whose log structure is defined by $s_{r+1},\ldots,s_n,t_1,\ldots,t_m$.
	
	For a non-empty subset $J\subset\{1,\ldots,r\}$, we define $d\log\underline{s}_J:=d\log s_{j_1}\wedge\cdots\wedge d\log s_{j_r}$ where we write $J=\{j_1,\ldots,j_r\}$ with $j_1<\cdots<j_r$.
	We set $d\log \underline{s}_\emptyset:=0$.
	For $J\subset\{1,\ldots,r\}$ and $\mu\in\{1,\ldots,r\}\setminus J$, the composition
	\[Bd\log\underline{s}_J\hookrightarrow\Gamma(\frZ_{I,\eta},\omega^{\lvert J\rvert}_{\cZ_I/W^\varnothing,\bbQ})\xrightarrow{d}\Gamma(\frZ_{I,\eta},\omega^{\lvert J\rvert+1}_{\cZ_I/W^\varnothing,\bbQ})\rightarrow Bd\log\underline{s}_{J\cup\{\mu\}}\]
	induces an $A$-linear map $\partial_\mu\colon Bs_1\cdots s_rd\log\underline{s}_J\rightarrow Bs_1\cdots s_rd\log\underline{s}_{J\cup\{\mu\}}$.
	For $(i_1,\ldots,i_r)\in\bbZ^r$, we let
	\[J(i_1,\ldots,i_r):=\begin{cases}\{j\in\{1,\ldots,r\}\mid i_j=1\}&(\text{if $i_\nu\in\{0,1\}$ for any $\nu=1,\ldots,r$})\\
	\emptyset&(\text{otherwise}).
	\end{cases}\]

	Then $R\Gamma(\frZ_{I,\eta},s\cdot\omega^\bullet_{\cZ_I/W^\varnothing,\bbQ})=\Gamma(\frZ_{I,\eta},s\cdot\omega^\bullet_{\cZ_I/W^\varnothing,\bbQ})$ is given as the total complex of the $(r+1)$-ple complex $\Gamma^{\bullet,\ldots,\bullet}$ whose $(i_0,i_1,\ldots,i_r)$-component is
	\[\Gamma^{i_1,\ldots,i_{r+1}}:=\Gamma(\frV_I,s_{r+1}\cdots s_n\cdot\omega^{i_{r+1}}_{\cV'_I/W^\varnothing,\bbQ})\otimes_ABs_1\cdots s_rd\log \underline{s}_{J(i_1,\ldots,i_r)},\]
	The $\mu$\textsuperscript{th} differential map $\Gamma^{i_1,\ldots,0,\ldots,i_{r+1}}\rightarrow \Gamma^{i_1,\ldots,1,\ldots,i_{r+1}}$ for $\mu=1,\ldots,r$ is induced by the $A$-linear map $\partial_\mu$.
	The $(r+1)$\textsuperscript{st} differential map $\Gamma^{i_1,\ldots,i_{r+1}}\rightarrow\Gamma^{i_1,\ldots,i_{r+1}+1}$ is defined by
	\[\alpha\otimes s_1^{\nu_1}\cdots s_r^{\nu_r}d\log\underline{s}_{J(i_1,\ldots,i_r)}\mapsto d\alpha\otimes s_1^{\nu_1}\cdots s_r^{\nu_r}d\log\underline{s}_{J(i_1,\ldots,i_r)} \]
	for $\alpha\in\Gamma(\frV_I,s_{r+1}\cdots s_n\cdot\omega^{i_{r+1}}_{\cV'_I/W^\varnothing,\bbQ})$ and integers $\nu_1,\ldots,\nu_r\geq 1$.
	
	The proof of the exactness of \eqref{eq: relative disk} shows that $\partial_1\colon Bs_1\cdots s_rd\log\underline{s}_J\rightarrow Bs_1\cdots s_rd\log\underline{s}_{J\cup\{1\}}$ is an isomorphism for any $J\subset \{2,\ldots,r\}$.
	This implies that $\Gamma^{0,i_2,\ldots,i_{r+1}}\rightarrow\Gamma^{1,i_2,\ldots,i_{r+1}}$ is an isomorphism, and hence that $R\Gamma(\frZ_{I,\eta},s\cdot\omega^\bullet_{\cZ_I/W^\varnothing,\bbQ})$ is acyclic.
	This finishes the proof of \eqref{item: HK rig 1}.
	
	\eqref{item: HK rig 2} Recall that we have
	\begin{align*}
	R\Gamma_+(\frZ,\cC^{\bullet,\bullet}_\frZ)=\hocolim_mR\Gamma(\frZ,\omega^\bullet_{\cZ/W^\varnothing,\bbQ}[u]_m)&&\text{and}&&R\Gamma(\frZ,\cC^{\bullet,-m}_\frZ)=R\Gamma(\frZ,\omega^\bullet_{\cZ/W^\varnothing,\bbQ})[-m],
	\end{align*}
	and the former is moreover quasi-isomorphic to $R\Gamma_\HK(Y)$ by Proposition \ref{prop: HK vs colim}.
	Similarly, we also have
	\begin{align*}
	R\Gamma_+(\frZ,\tau_\ast\cC_\frY^{\bullet,\bullet})=\hocolim_mR\Gamma(\frY,\wt\omega^\bullet_{\cY,\bbQ}[u]_m)&&\text{and}&&R\Gamma(\frZ,\tau_\ast\cC^{\bullet,-m}_\frY)=R\Gamma(\frY,\wt\omega^\bullet_{\cY,\bbQ})[-m].
	\end{align*}
	Therefore we get the assertion by \eqref{item: HK rig 1}  and Lemma \ref{lem: acyclic assembly lemma for sheaf}\eqref{item: sheaf assembly row}.
	The compatibilities with the Frobenius and monodromy operators are obvious.
	
	\eqref{item: HK rig 3}  From the exact sequences \eqref{eq: ses S} and \eqref{eq: ses W}
	we obtain exact sequences
		\begin{eqnarray*}
		 & & 0 \rightarrow \omega^0_{\cZ/W^\varnothing,\bbQ}u^{[i]} \xrightarrow{d\log s\wedge}\omega^1_{\cZ/W^\varnothing,\bbQ}u^{[i-1]} \xrightarrow{d\log s\wedge}\cdots \xrightarrow{d\log s\wedge}\omega^i_{\cZ/W^\varnothing,\bbQ}u^{[0]} \rightarrow \omega^i_{\cZ/\cS,\bbQ} \rightarrow  0,\\
		 & & 0 \rightarrow \widetilde{\omega}^0_{\cY,\bbQ}u^{[i]} \xrightarrow{d\log s\wedge}\widetilde{\omega}^1_{\cY,\bbQ}u^{[i-1]} \xrightarrow{d\log s\wedge}\cdots \xrightarrow{d\log s\wedge}\widetilde{\omega}^i_{\cY,\bbQ}u^{[0]} \rightarrow \omega^i_{\cY/W^0,\bbQ} \rightarrow  0
		\end{eqnarray*}
	for any $i\geq 0$.
	In other words, we have quasi-isomorphisms $\cC_\frZ^{i,\bullet} \xrightarrow{\cong}\omega^i_{\cZ/\cS,\bbQ}$ and $\cC_\frY^{i,\bullet} \xrightarrow{\cong}\omega^i_{\cY/W^0,\bbQ}$.
	In particular $R\Gamma(\frZ,\cC_\frZ^{i,\bullet})\rightarrow R\Gamma(\frZ,\omega^i_{\cZ/\cS,\bbQ})$ and $R\Gamma(\frY,\cC_\frY^{i,\bullet})\rightarrow R\Gamma(\frY,\omega^i_{\cY/W^0,\bbQ})$ are quasi-isomorphisms.
	Since we have
	\begin{align*}
	R\Gamma_\times(\frZ,\cC_\frZ^{\bullet,\bullet})=R\Gamma(\frZ,\omega^\bullet_{\cZ/W^\varnothing,\bbQ}\llbracket u\rrbracket)
	&&\text{and}&&
	R\Gamma_\times(\frY,\cC_\frY^{\bullet,\bullet})=R\Gamma(\frY,\wt\omega_{\cY,\bbQ}^\bullet\llbracket u\rrbracket),
	\end{align*}
	we obtain the statement by Lemma \ref{lem: acyclic assembly lemma for sheaf} \eqref{item: sheaf assembly column}.
	
	The compatibility with Frobenius operators is obvious.
	To see the compatibility with the monodromy, consider a morphism $\omega^\bullet_{\cZ/W^\varnothing,\bbQ}\llbracket u\rrbracket[-1]\rightarrow\omega^\bullet_{\cZ/W^\varnothing,\bbQ}$ defined by $u^{[0]}\mapsto d\log s$ and $u^{[i]}\mapsto 0$ for $i>0$.
	Then the maps
	$\omega^m_{\cZ/W^\varnothing,\bbQ}\llbracket u\rrbracket\oplus\omega^m_{\cZ/W^\varnothing,\bbQ}\rightarrow\omega^m_{\cZ/W^\varnothing,\bbQ}\llbracket u\rrbracket$ for $m\geq 0$ defined by
	$(\sum_{i\geq 0}\eta_iu^{[i]},\xi)\mapsto \xi+\sum_{i\geq 0}\eta_iu^{[i+1]}$ for $\eta_i,\xi\in\omega^m_{\cZ/W^\varnothing,\bbQ}\llbracket u\rrbracket$
	induce an isomorphism
	\[\Cone(\omega^\bullet_{\cZ/W^\varnothing,\bbQ}\llbracket u\rrbracket[-1]\rightarrow\omega^\bullet_{\cZ/W^\varnothing,\bbQ})\xrightarrow{\cong}\omega^\bullet_{\cZ/W^\varnothing,\bbQ}\llbracket u\rrbracket.\]
	Therefore we have a morphism of distinguished triangles
	\[\xymatrix{
	\omega^\bullet_{\cZ/W^\varnothing,\bbQ}\llbracket u\rrbracket[-1]\ar[r]\ar[d]&\omega^\bullet_{\cZ/W^\varnothing,\bbQ}\ar[r]\ar[d]&\omega^\bullet_{\cZ/W^\varnothing,\bbQ}\llbracket u\rrbracket\ar[d]\ar[r]&\\
	\omega^\bullet_{\cZ/\cS,\bbQ}[-1]\ar[r]&\omega^\bullet_{\cZ/W^\varnothing,\bbQ}\ar[r]&\omega^\bullet_{\cZ/\cS,\bbQ}\ar[r]&.
	}\]
	The connecting homomorphism induced by the upper triangle is computed by $N\colon u^{[i]}\mapsto u^{[i-1]}$, and that induced by the lower triangle defines the monodromy operator on $R\Gamma_\rig(Y/\cS)$.
	Thus we obtain the desired compatibility.

	\eqref{item: HK rig 4} The first and third quasi-isomorphisms and their compatibility with the Frobenius and monodromy operators are given by \eqref{item: HK rig 2} and \eqref{item: HK rig 3}.
	Moreover it is clear that the second morphism is compatible with the Frobenius and monodromy operators.
	Thus it suffices to show that the composition $R\Gamma_\HK(Y)\rightarrow R\Gamma_\rig(Y/W^0)$ is a quasi-isomorphism.
	This can be reduced to showing that $R\Gamma_\HK(Y_I)\rightarrow R\Gamma_\rig(Y_I/W^0)$ is a quasi-isomorphism for each $I$, similarly to the proof of Proposition \ref{prop: horizontal divisor}.
	But this immediately follows from Corollary \ref{cor: coh of Y_I}.
\end{proof}

\begin{remark}\label{rem: HK coh} $\quad$
	\begin{enumerate}
	\item\label{item: HK coh of Y_I} As in Proposition \ref{prop: coh of Y_I}, let $Y_I$ be an intersection of irreducible components of $Y$ endowed with the pull-back log structure.
		Let $\cZ_I$ be the exactification of $\cZ$ along $Y_I$ and set $\cY_I:=\cZ_I\times_{\cS}W^0$.
		Denote by $\frZ_I$ and $\frY_I$ the dagger spaces associated to $\cZ_I$ and $\cY_I$, respectively.
		Then the statements of Proposition \ref{prop: HK and rig} remains valid if we replace $Y$, $\frZ$, $\frY$ by $Y_I$, $\frZ_I$, $\frY_I$, respectively.
		Indeed, \eqref{item: HK rig 1} and \eqref{item: HK rig 4} are proved by reducing to this case, and the proofs of \eqref{item: HK rig 2} and \eqref{item: HK rig 3} work for this case exactly in the same way.
	\item In the original paper of Kim and Hain \cite{KH}, they introduced a de Rham--Witt counterpart of $\widetilde{\omega}^\bullet_{\cY,\bbQ}[u]$ and stated that its cohomology computes the log crystalline cohomology over $W^0$.
		However the proof has a gap, and works if one replaces their Kim--Hain complex by the completed version.
		The comparison isomorphism $R\Gamma_\rig(Y/W^0)\cong R\Gamma_\cris(Y/W^0)$ and the exact sequence \eqref{eq: ses W} together imply the comparison isomorphisms of $R\Gamma(\frY,\wt\omega^\bullet_{\cY,\bbQ})$ and $\hocolim_mR\Gamma(\frY,\wt\omega^\bullet_{\cY,\bbQ}[u]_m)$ with their respective de Rham--Witt counterparts. 
		Then Proposition \ref{prop: HK and rig} \eqref{item: HK rig 4} fills the gap of \cite{KH}.
	\item It is worth to point out that the map $\omega^\bullet_{\cZ/W^\varnothing,\bbQ}[u] \rightarrow  \omega^\bullet_{\cZ/W^\varnothing,\bbQ}\llbracket u\rrbracket$ is not a quasi-isomorphism.
		For example, for any $k\geq 0$, $(k^is^ku^{[i]})_i$ is a $0$-cocycle in $\omega^\bullet_{\cZ/W^\varnothing,\bbQ}\llbracket u\rrbracket$ and not coming from $\omega^\bullet_{\cZ/W^\varnothing,\bbQ}[u]$ if $k>0$.
		Indeed for $Y = k^0$ and $\cZ = \cS$ we have $H^0_\HK(k^0) = F$ and $H^0_\rig(k^0/\cS) = F\{s\}$, the ring of formal power series converging on $\lvert s\rvert<1$.
	\end{enumerate}
\end{remark}

The morphisms in Proposition \ref{prop: HK and rig} naturally glue and we obtain the following:

\begin{corollary}\label{cor: natural diagram}
Let $Y$ be strictly semistable.
Then there is a commutative diagram of natural morphisms
	\begin{equation}\label{eq: diag HK S}
	\xymatrix{R\Gamma_{\HK}(Y) \ar[r] \ar[d]^\sim & R\Gamma_{\rig}(Y/\cS)\ar[dl]_{j_0^\ast} \\
	R\Gamma_{\rig}(Y/W^0) &}
	\end{equation}
compatible with Frobenius operators and monodromy operators, such that the left vertical map is a quasi-isomorphism.
Moreover the upper horizontal map induces a quasi-isomorphism
\[R\Gamma_\HK(Y)\otimes_FF\{s\}\xrightarrow{\cong}R\Gamma_\rig(Y/\cS).\]
If $Y$ is quasi-compact, then $H^i_\rig(Y/W^0)$, $H^i_\rig(Y/W^\varnothing)$, and $H^i_\HK(Y)$ are finite dimensional $F$-vector spaces, and $H^i_\rig(Y/\cS)$ is a finite free $F\{s\}$-module for any $i\in\bbZ$.
\end{corollary}

\begin{proof}
The diagram, including the fact that the vertical morphism is a quasi-isomorphism, constitutes a summary of the gloabl version of Proposition \ref{prop: HK and rig} and thus follows immediately.

Next we will show that the horizontal morphism induces a quasi-isomorphism after tensoring with $F\{s\}$.
Working locally, we may assume that there exists a strictly semistable log rigid datum $(\cZ,i)$ for $Y$ over $k^0\hookrightarrow \cS$.
As before, let $Y_I$ be an intersection of irreducible components of $Y$ and $\cZ_I$ the exactification of $\cZ$ along $Y_I$.
Denote by $\frZ$ and $\frZ_I$ the dagger spaces associated to $\cZ$ and $\cZ_I$.
Since the $\frZ_I$'s cover $\frZ$, it suffices to show that
\[R\Gamma_\HK(Y_I)\otimes_FF\{s\}\xrightarrow{\cong}R\Gamma_\rig(Y_I/\cS)\]
is a quasi-isomorphism.
There exists a commutative diagram \eqref{eq: diag HK S} with $Y$ replaced by $Y_I$, and $R\Gamma_\HK(Y_I)\rightarrow R\Gamma_\rig(Y_I/W^0)$ is also a quasi-isomorphism as mentioned in Remark \ref{rem: HK coh}\eqref{item: HK coh of Y_I}.
Thus the claim follows from the computation of log rigid cohomology groups in Corollary \ref{cor: coh of Y_I}.

It remains to prove the finiteness of cohomology groups for a quasi-compact $Y$.
To prove the finiteness of $H^i_\rig(Y/W^\varnothing)$, we may suppose by Proposition \ref{prop: Mayer-Vietoris} that a strict smooth morphism $Y\rightarrow k^0(n,m)$ and a log rigid datum $(\cZ,i)$ for $Y$ over $\iota\colon k^\varnothing\hookrightarrow W^\varnothing$ exist.
Moreover, since the dagger space associated to $\cZ$ is covered by those associated to the exactifications along irreducible components, it suffices to show the finiteness of $H^i_\rig(Y_I/W^\varnothing)$, similarly to above.
By the finiteness of the non-logarithmic rigid cohomology \cite[Cor.\,3.8]{GK2}, all conditions (i)--(iii) of Corollary \ref{cor: coh of Y_I} are satisfied, and we deduce the finiteness of the cohomology by the corollary.
The finiteness of $H^i_\rig(Y/W^0)$ and $H^i_\HK(Y)$ are proved similarly, and it follows that $H^i_\rig(Y/\cS)$ is finite free over $F\{s\}$ by the above quasi-isomorphism.
(Note that the finiteness of $H^i_\rig(Y/W^0)$ has been proved in \cite[Thm.\,5.3]{GK3}.)
\end{proof}

For a uniformiser $\pi\in V$, let $j_\pi\colon V^\sharp\hookrightarrow\cS$ be the exact closed immersion defined by $s\mapsto\pi$.

\begin{lemma}\label{lem: base change from S}
	Let $Y$ be a strictly semistable log scheme over $k^0$ and let $\pi\in K$ be a uniformiser.
	We suppose that
	\begin{enumerate}\renewcommand{\labelenumi}{(\roman{enumi})}
	\item $Y$ admits an affine open covering indexed by a countable set,
	\item $H_\HK^m(Y)$ is finite dimensional for any $m$.
	\end{enumerate}
	Then the natural morphisms
		\begin{align*}
		j_0^*\colon R\Gamma_\rig(Y/\cS)\rightarrow R\Gamma_\rig(Y/W^0) &&\text{and}&&
		j_\pi^*\colon R\Gamma_\rig(Y/\cS)\rightarrow R\Gamma_\rig(Y/V^\sharp)
		\end{align*}
	induce isomorphisms
		\begin{align*}
		H_\rig^k(Y/\cS)\otimes_{F\{s\}}F \xrightarrow{\cong} H_\rig^k(Y/W^0) &&\text{and}&&
		H_\rig^k(Y/\cS)\otimes_{F\{s\}}K \xrightarrow{\cong} H_\rig^k(Y/V^\sharp)
		\end{align*}
	for any $k$.
\end{lemma}	

\begin{proof}
	We first assume that $Y$ is quasi-compact and let $\{U_\lambda\}_{\lambda\in\Lambda}$ be a finite covering of $Y$ by affine open subsets with strict log smooth morphisms $U_\lambda\rightarrow k^0(n_\lambda,m_\lambda)$.
	By Proposition \ref{prop: Mayer-Vietoris} there are spectral sequences
	\begin{align}
	\label{eq: MVss W}&E_1^{p,q}=\bigoplus_{(\lambda_0,\ldots,\lambda_p)\in\Lambda^p}H_\rig^q(U_{\lambda_0}\cap\cdots\cap U_{\lambda_p}/W^0)\Rightarrow H^{p+q}_\rig(Y/W^0),\\
	\label{eq: MVss V}&E_1^{p,q}=\bigoplus_{(\lambda_0,\ldots,\lambda_p)\in\Lambda^p}H_\rig^q(U_{\lambda_0}\cap\cdots\cap U_{\lambda_p}/V^\sharp)\Rightarrow H^{p+q}_\rig(Y/V^\sharp).
	\end{align}
	Moreover, by similar arguments as in Proposition \ref{prop: Mayer-Vietoris} or by combining \eqref{eq: MVss W} with the isomorphisms $H^i_\HK(Y)\cong H^i_\rig(Y/W^0)$,  we have a spectral sequence
	\begin{equation}\label{eq: ss HK}
	E_1^{p,q}=\bigoplus_{(\lambda_0,\ldots,\lambda_p)\in\Lambda^p}H_\HK^q(U_{\lambda_0}\cap\cdots\cap U_{\lambda_p})\Rightarrow H^{p+q}_\HK(Y).
	\end{equation}
	Let $\bbF$ be one of $F$ or $K$.
	By Corollary \ref{cor: natural diagram}, we have $H^i_\HK(Y)\otimes_F\bbF=H^i_\rig(Y/\cS)\otimes_{F\{s\}}\bbF$.
	Thus, by tensoring \eqref{eq: ss HK} with $\bbF$ we obtain a spectral sequence
	\[E_1^{p,q}=\bigoplus_{(\lambda_0,\ldots,\lambda_p)\in\Lambda^p}H_\rig^q(U_{\lambda_0}\cap\cdots\cap U_{\lambda_p}/\cS)\otimes_{F\{s\}}\bbF\Rightarrow H^{p+q}_\rig(Y/\cS)\otimes_{F\{s\}}F.\]
	Using this with \eqref{eq: MVss W} and \eqref{eq: MVss V}, we may assume that $Y$ is affine and there exists a strict log smooth morphism $Y\rightarrow k^0(n,m)$.
	In addition, by similar arguments we may replace $Y$ by an intersection of irreducible components $Y_I$ .
	Then the assertion follows by Corollary \ref{cor: coh of Y_I}.
	
	We next discuss the general case.
	By assumption (i) we may take an open covering $\{U'_i\}_{i\in\bbN}$ with $U'_i\subset U'_{i+1}$ such that each $U_i$ is quasi-compact.
	Then we have canonical isomorphisms
	\begin{align*}
	H^m_\rig(Y/\cS)\otimes_{F\{s\}}\bbF
	&\cong H^m_\HK(Y)\otimes_F\bbF
	\cong(\varprojlim_iH^m_\HK(U'_i))\otimes_F\bbF
	\cong\varprojlim_i(H^m_\HK(U'_i)\otimes_F\bbF)\\
	&\cong\varprojlim_i(H^m_\rig(U'_i/\cS)\otimes_{F\{s\}}\bbF)\\
	&\cong\begin{cases}
	\varprojlim_iH^m_\rig(U'_i/W^0)\cong H^m_\rig(Y/W^0)&(\bbF=F)\\
	\varprojlim_iH^m_\rig(U'_i/V^\sharp)\cong H^m_\rig(Y/V^\sharp)&(\bbF=K),
	\end{cases}
	\end{align*} 
	where the first and fourth isomorphisms are given by Corollary \ref{cor: natural diagram}, the second and final isomorphisms are given by Lemma \ref{lem: ML} \eqref{item: ML1}, the third isomorphism is given by the assumption (ii) and Lemma \ref{lem: ML} \eqref{item: ML2}, and the fifth isomorphism is proved in the former half of the proof.
\end{proof}

%%%%%%%%%%%%%%%%%%
%
\section{Rigid Hyodo--Kato map}\label{Sec: Rig HK map}
%
%%%%%%%%%%%%%%%%%%
We finally are ready to discuss the rigid Hyodo--Kato map.
From now on, we always suppose that any fine log scheme over $k$ is separated and locally of finite type as in previous sections, and in addition admits an affine open covering indexed by a countable set.
Moreover, when we consider a weak formal log scheme $\cX$ adic over $V^\sharp$, we always suppose that its special fiber satisfies the above conditions.
In addition, we assume that $k$ is algebraic over $\bbF_p$.

For the choice of a uniformiser $\pi\in K$, let $j_\pi\colon V^\sharp\hookrightarrow \cS$ be the exact closed immersion defined by $s\mapsto\pi$, and let $i_\pi\colon k^0\hookrightarrow V^\sharp$ be the unique exact closed immersion such that $\tau = j_\pi\circ i_\pi$.
Let $\boldsymbol{\mu}\subset W^\times$ be the image of $k^\times$ under the Teichm\"uller map.
Then there is a decomposition $V^\times = \boldsymbol{\mu}(1 + \frm)$.

\begin{definition}
	Let $\log\colon V^\times \rightarrow  K$ be the $p$-adic logarithm function defined by
		\begin{eqnarray*}
		\log(v): = -\sum_{n\geq 1}\frac{(1-v)^n}{n} & & \text{for }v\in(1 + \frm),\\
		\log(u): = 0 & & \text{for }u\in\boldsymbol{\mu}.
		\end{eqnarray*}
	A {\it branch of the $p$-adic logarithm} on $K$ is a group homomorphism from $K^\times$ to (the additive group of) $K$ whose restriction to $V^\times$ coincides with $\log$ as above.
	
	For $q\in\frm\setminus\{0\}$, let $\log_q\colon K^\times \rightarrow  K$ be the unique branch of the $p$-adic logarithm which satisfies $\log_q(q) = 0$.
	More precisely, for any uniformiser $\pi$ the element $q$ can be written as $q=\pi^m v$, for some $m\geq 1$ and $v\in V^\times$. 
	Then $\log_q$ is defined by $\log_q(\pi): = -m^{-1}\log(v)$.
\end{definition}

\begin{remark}
\begin{enumerate}
\item 
One could use the same definition of the $p$-adic logarithm in the case that $k$ is not algebraic over $\bbF_p$. 
However, in that case an element $u\in\boldsymbol{\mu}$ is not necessarily a root of unity, 
so there is no intrinsic reason to have $\log(u)=0$. 
That is why we restrict ourselves to the case where $k$ is algebraic over $\bbF_p$ in this and the subsequent section.
\item Note that any branch of the $p$-adic logarithm can be written in the form $\log_q$ for some $q\in\frm\setminus\{0\}$.
	Indeed, if $L\colon K^\times\rightarrow K$ is a branch of the $p$-adic logarithm, $\exp(-p^k L(p))$ is well defined as an element of $1+\frm$ for a sufficiently large integer $k$.
	Then by setting $q:=p^{p^k}\exp(-p^kL(p))$ one can see that $L=\log_q$.
\end{enumerate}
\end{remark}

\begin{definition}
	Let $\cX$ be a weak formal log scheme which is log smooth and adic over $V^\sharp$,  $\pi\in V$ be a uniformiser, and $Y: =  \cX \times_{V^\sharp,i_\pi}k^0$.
	\begin{enumerate}
	\item We define the rigid Hyodo--Kato cohomology of $\cX$ with respect to $\pi$ to be
			\[R\Gamma_\HK(\cX,\pi): = R\Gamma_\HK(Y)\]
		as a complex, and endow it with the Frobenius operator $\varphi$ and the normalized monodromy operator $\mathbf{N}: = e^{-1}N$.
		Here $\varphi$ and $N$ are as in Definition \ref{def: HK coh}. 
		Note that  $Y \hookrightarrow \cX$ is induced by the exact closed immersion of log schemes $i_\pi:k^0\hookrightarrow V^\sharp$. 
		This means it depends on the choice of a uniformiser $\pi$, and hence so does the complex defined here, which we indicate by the notation.
		For simplicity, we often write $R\Gamma_\HK(\cX,\pi)_K: = R\Gamma_\HK(\cX,\pi)\otimes_FK$.
	\item Let $\frX$ be the dagger space associated to $\cX$.
		For $q\in\frm\setminus\{0\}$, we define the {\it rigid Hyodo--Kato map}
			\[\Psi_{\pi,q}\colon R\Gamma_\HK(\cX,\pi) \rightarrow R\Gamma_\dR(\cX):= R\Gamma(\frX,\omega^\bullet_{\cX/V^\sharp,\bbQ})\]
		with respect to $\pi$ and $\log_q$ as follows:
		
		In case there exists a log rigid $F$-datum $(\cZ,i,\phi)$ over $(k^0\hookrightarrow\cS,\sigma)$ for $Y: = \cX\times_{V^\sharp,i_\pi}k^0$, set $\cX': = \cZ\times_{\cS,j_\pi}V^\sharp$ and let $\frX'$ be the associated dagger space.
		Then we define $\Psi_{\pi,q}$ as the composite
			\[R\Gamma_\HK(\cX,\pi) = R\Gamma(\frZ,\omega^\bullet_{\cZ/W^\varnothing,\bbQ}[u]) \xrightarrow{\psi_{\pi,q}}R\Gamma(\frX',\omega^\bullet_{\cX'/V^\sharp,\bbQ}) \xrightarrow[\cong]{\theta} R\Gamma(\frX,\omega^\bullet_{\cX/V^\sharp,\bbQ}),\]
		where $\psi_{\pi,q}$ is defined by the natural maps $\omega^\bullet_{\cZ/W^\varnothing,\bbQ} \rightarrow  \omega^\bullet_{\cZ/\cS,\bbQ} \rightarrow \omega^\bullet_{\cX'/V^\sharp,\bbQ}$ and by
			\[\psi_{\pi,q}(u^{[i]}): = \frac{(-\log_q(\pi))^i}{i!},\]
		and $\theta$ is the canonical quasi-isomorphism given by Proposition \ref{prop: log rigid coh}.
		In the general case, we may extend the above construction by taking a simplicial log rigid $F$-quadruple.
		
		We often write
			\[\Psi_{\pi,q,K}: = \Psi_{\pi,q}\otimes 1\colon R\Gamma_\HK(\cX,\pi)_K \rightarrow  R\Gamma_\dR(\cX).\]
	\end{enumerate}
\end{definition}

\begin{remark}
Recall that $R\Gamma_\HK(\cX,\pi)=R\Gamma_\HK(Y)$ is defined as the complex associated to the cosimplicial complex $R\Gamma(\frZ_\bullet,\omega^\bullet_{\cZ_\bullet/W^\varnothing,\bbQ}[u])$ by using a simplicial log rigid $F$-quadruple $(U_\bullet,\cZ_\bullet,i_\bullet,\phi_\bullet)$ over $(k^0\hookrightarrow\cS,\sigma)$ such that $U_\bullet\rightarrow Y$ is a Zariski hypercovering.
With this setting, $R\Gamma_\HK(\cX,\pi)_K$ is associated to the cosimplicial complex $R\Gamma(\frZ_\bullet,\omega^\bullet_{\cZ_\bullet/W^\varnothing,\bbQ}[u]\otimes_FK).$

As any section of $\omega^\bullet_{\cZ_m/W^\varnothing,\bbQ}[u]$ vanishes under a power of the monodromy operator $N$, it follows that for any $f\in K\llbracket T\rrbracket$, the endomorphism $f(N)$ on $\omega^\bullet_{\cZ_m/W^\varnothing,\bbQ}[u]\otimes_FK$ is well-defined.
Thus $f(N)$ induces an endomorphism on $R\Gamma_\HK(\cX,\pi)_K$.
\end{remark}

The following proposition explains the relation between different choices of 
$q$ for the rigid Hyodo--Kato map. 

\begin{proposition}\label{prop: choice of pi} 
For a uniformiser $\pi\in V$ and elements $q,q'\in\frm\setminus\{0\}$, we have
			\begin{equation}\label{eq: choice of log 2}
			\Psi_{\pi,q,K}=\Psi_{\pi,q',K}\circ\exp\left(-\frac{\log_q(q')}{\mathrm{ord}_p(q')}\cdot \mathbf{N}\right),
			\end{equation}
		and in particular	
			\begin{equation}\label{eq: choice of log}
			\Psi_{\pi,q,K} = \Psi_{\pi,\pi,K}\circ\exp(-e\log_q(\pi)\cdot\mathbf{N}).
			\end{equation}
\end{proposition}

\begin{proof}
	The equation \eqref{eq: choice of log 2} immediately follows from \eqref{eq: choice of log}.
	Note that we have $\Psi_{\pi,\pi}(u^{[i]}) = 0$ for $i>1$, since $\log_\pi(\pi) = 0$.
	The equation \eqref{eq: choice of log} follows from
		\[\Psi_{\pi,\pi}\circ\exp(-e\log_q(\pi)\cdot \mathbf{N})(u^{[i]}) = \Psi_{\pi,\pi} \left(\sum_{j = 0}^i\frac{(-\log_q(\pi))^j}{j!}u^{[i-j]}\right) = \frac{(-\log_q(\pi))^i}{i!} = \Psi_{\pi,q}(u^{[i]}).\]
\end{proof}

We study now the effects of extending the base field and different choices of uniformisers.  
For this consider the following situation: 
let $K'/K$ be a finite extension of ramification index $\ell=\frac{e'}{e}$. 
Let $V'$ be its ring of integers with maximal ideal $\frm'$ and  residue field $k'$. Denote $W'=W(k')$ with Frobenius $\sigma'$ and fraction field $F'$. 
Let $\boldsymbol{\mu}'\subset {W'}^\times$ be the image of ${k'}^\times$ under the Teichmüller map. 
Again there is a decomposition ${V'}^\times=\boldsymbol{\mu}'(1+\frm')$. 
Thus for a uniformiser $\pi'\in V'$, there uniquely exist
$\alpha\in 1+\frm'$ and $w\in \boldsymbol{\mu}'$ such that $\pi=\alpha w{\pi'}^\ell$. 
Let $k^0$, ${V'}^\sharp$,  $\cS'$, be the weak formal log schemes analogous to before. 
There is a commutative diagram
$$
\xymatrix{
{k'}^0 \ar[d]^{\overline{\rho}_{\pi,\pi'}} \ar[r]^{i_{\pi'}} & {V'}^\sharp \ar[d]^{\rho_{V'/V}} \\
k^0 \ar[r]^{i_\pi} & V^\sharp 
}
$$
where the right vertical morphism is induced by the natural inclusion $V\hookrightarrow V'$, and the left vertical morphism is defined by the natural inclusion $k\hookrightarrow k'$ on the level of rings and by $s\mapsto \overline{w}{s'}^\ell$ on the level of log structures (where $\overline{w}$ denotes the reduction of $w$ modulo $\pi'$ and $s$ and $s'$ are the canonical generators of the log structures of $k^0$ and ${k'}^0$ respectively.)
Consider a morphism
$f:\cX'\rightarrow \cX$ over $\rho_{V'/V}$ of fine weak formal log schemes which are log smooth and adic over $V'^\sharp$ and $V^\sharp$ respectively.
Set $Y'_{\pi'}:=\cX'\times_{{V'}^\sharp,i_{\pi'}} {k'}^0$ and 
$Y_{\pi}:=\cX\times_{{V}^\sharp,i_{\pi}} {k}^0$. 
Then $f$ indues a morphism $\overline{f}_{\pi,\pi'}: Y'_{\pi'}\rightarrow Y_\pi$. 
which is compatible with $\overline{\rho}_{\pi,\pi'}$.

\begin{proposition}\label{prop: base change and HK map}
In the situation above, we have the following statements:
\begin{enumerate}
\item
The morphism $f$ induces canonical morphisms
\begin{align*}
&f^\ast_{\HK}\colon R\Gamma_{\HK}(\cX,\pi) \rightarrow R\Gamma_{\HK}(\cX',\pi')\\
&f^\ast_{\dR}\colon R\Gamma_{\dR}(\cX)\rightarrow R\Gamma_{\dR}(\cX')
\end{align*}
such that $f^\ast_{\HK}$ is compatible with Frobenius and monodromy operators, and
such that for any $q\in\frm\setminus\{0\}$
\begin{equation}\label{eq: base change and HK map}
f_\dR^*\circ\Psi_{\pi,q} = \Psi_{\pi',q}\circ f_\HK^*.
\end{equation}
\item
Moreover, there is a commutative diagram
\begin{equation}\label{formula base change}
\xymatrix{
R\Gamma_{\HK}(\cX,\pi)_{K'} \ar[d]_{f^\ast_{\HK,K'}} \ar[rrrr]^{\Psi_{\pi,\pi,K'}} &&&& R\Gamma_{\dR}(\cX)_{K'} \ar[d]^{f_{\dR,K'}^\ast}\\
R\Gamma_{\HK}(\cX',\pi')_{K'} \ar[rr]^{\exp(e\log(\alpha)\cdot\mathbf{N})}&& R\Gamma_{\HK}(\cX',\pi')_{K'} \ar[rr]^{\Psi_{\pi',\pi',K'}} && R\Gamma_{\dR}(\cX')
.}
\end{equation}

\item
In addition, we have the following transitivity statement: 
Let $K_3\supset K_2\supset K_1$ be finite extensions with rings of integers $V_3\supset V_2\supset V_1$.
Let $\cX_3\xrightarrow{g}\cX_2\xrightarrow{f}\cX_1$ be morphisms of fine weak formal log schemes such that $\cX_l$ is log smooth and adic over $V_l^\sharp$ for each $l=1,2,3$.
Then for any choice of uniformisers of $V_1$, $V_2$, and $V_3$, 
we have $(g\circ f)_{\HK}^\ast=g_{\HK}^\ast\circ f^\ast_{\HK}$ 
and $(g\circ f)_{\dR}^\ast=g_{\dR}^\ast\circ f^\ast_{\dR}$.
\end{enumerate}
\end{proposition}

\begin{proof}
The map $f_\dR^*$ is the natural morphism 
$f_\dR^*\colon R\Gamma(\frX,\omega^\bullet_{\cX/V^\sharp,\bbQ}) \rightarrow  R\Gamma(\frX',\omega^\bullet_{\cX'/V'^\sharp,\bbQ})$
induced by $f$, where $\frX$ and $\frX'$ are the dagger spaces associated to $\cX$ and $\cX'$, respectively.
 
Next we discuss the construction of $f_\HK^\ast$.
We denote the canonical coordinate of $\cS'$ by $s'$. 
Note that there is a commutative diagram
$$
\xymatrix{
{k'}^0\ar[d]^{\overline{\rho}_{\pi,\pi'}} \ar@{^{(}->}[r] & \cS' \ar[d]^{\rho_{\pi,\pi'}}\\
k^0 \ar@{^{(}->}[r] & \cS
}
$$
where the right vertical map is defined by $s\mapsto w{s'}^\ell$.
We may assume that 
there exist log rigid $F$-data $(\cZ,i,\phi)$ for $Y_\pi$ over $(k^0\hookrightarrow\cS,\sigma)$ and $(\cZ',i',\phi')$ for $Y'_{\pi'}$ over $(k'^0\hookrightarrow\cS',\sigma')$.
(In the general case, we work with simplicial objects as before, 
which we decided to omit for the sake of a more straight forward exposition.) 
Then $Y'_{\pi'}\xrightarrow{\overline{f}_{\pi,\pi'}} Y_\pi \hookrightarrow \cZ$ and $Y'_{\pi'}\hookrightarrow \cZ'$ together induce a morphism
$Y'\hookrightarrow \cZ\times_{\cS}\cZ'$, where $\cZ'\rightarrow\cS$ is induced through $\rho_{\pi,\pi'}$. 
Let $\cZ''$ be its exactification. 
As usual, denote by $\frZ$, $\frZ'$ and $\frZ''$ the respective associated dagger spaces. 
Then the natural projections induce morphisms
$$
R\Gamma(\frZ,\omega^\bullet_{\cZ/W^\varnothing,\bbQ}) \rightarrow
R\Gamma(\frZ'',\omega^\bullet_{\cZ''/{W'}^\varnothing,\bbQ}) \xleftarrow{\sim}
R\Gamma(\frZ',\omega^\bullet_{\cZ'/{W'}^\varnothing,\bbQ}) 
$$
where the right one is a quasi-isomorphism by Proposition \ref{prop: log rigid coh}. 
We define CDGAs $\omega^\bullet_{\cZ''/W'^\varnothing,\bbQ}[u']$ and $\omega^\bullet_{\cZ'/W'^\varnothing,\bbQ}[u']$ similarly to Definition \ref{def: HK coh local} but we set $du'^{[i+1]}:=-d\log s'\cdot u'^{[i]}$.
Then the above morphisms extend to
\begin{align*}
R\Gamma(\frZ,\omega^\bullet_{\cZ/W^\varnothing,\bbQ}[u])&\xrightarrow {\beta}R\Gamma(\frZ'', \omega^\bullet_{\cZ''/{W'}^\varnothing,\bbQ}[u'])
\xleftarrow{\sim}R\Gamma(\frZ',\omega^\bullet_{\cZ'/{W'}^\varnothing,\bbQ}[u'])
\end{align*}
where $\beta$ is defined by $\beta(u^{[i]}):=\ell^i{u'}^{[i]}$ and the second morphism is a quasi-isomorphism.
Here we note that the compatibility of $\beta$ with the differential maps follows by
\[d\circ\beta(u^{[i]})=-\ell^id\log s'\cdot u'^{[i-1]}=-\ell^{i-1}d\log s\cdot u^{[i-1]}=\beta\circ d(u^{[i]}).\]
This gives a morphism in the derived category
\[f^\ast_\HK\colon R\Gamma_\HK(\cX,\pi)\rightarrow R\Gamma_\HK(\cX',\pi').\]
Clearly, $f^\ast_{\HK}$ commutes with the Frobenius operators. 
Since the normalized monodromy $\mathbf{N}$ on 
$R\Gamma_\HK(\cX,\pi)$ and $R\Gamma_\HK(\cX',\pi')$ are defined to be 
$e^{-1}N$ and $e'^{-1}N=(e\ell)^{-1}N$, respectively, we moreover see that $f_\HK^*$ commutes with the monodromy operators by
\[
\mathbf{N}\circ f_\HK^*(u^{[i]}) = \mathbf{N}(\ell^iu'^{[i]}) = (e\ell)^{-1}\ell^iu'^{[i-1]} = e^{-1}\ell^{i-1}u'^{[i-1]} = f_\HK^*(e^{-1}u^{[i-1]}) =  f_\HK^*\circ\mathbf{N}(u^{[i]}).
\]

To see the compatibility of $f_{\HK}^\ast$ and $f_{\dR}^\ast$ via the Hyodo--Kato map, let $\widetilde{\cS}$ be the exactification of ${k'}^0\hookrightarrow \cS\times_{W^\varnothing}\cS'$
(making use of the compositions ${k'}^0\xrightarrow{\overline{\rho}_{\pi,\pi'}} k^0\hookrightarrow \cS$ and $\cS'\rightarrow {W'}^\varnothing \rightarrow W^\varnothing$). 
Then there exists an exact closed immersion 
$j_{\pi,\pi'}:{V'}^\sharp\hookrightarrow \widetilde{\cS}$ such that the diagram
$$
\xymatrix{
 & &\cS'\\
{k'}^0 \ar[r]^-{i_{\pi'}} \ar[d]_-{\overline{\rho}_{\pi,\pi'}} \ar@/^1pc/[rru] & {V'}^\sharp  \ar[r]^-{j_{\pi,\pi'}} \ar[ru]^-{j_{\pi'}} \ar[d]_-{\rho_{V'/V}}& \widetilde{\cS} \ar[u] \ar[d]\\
k^0 \ar[r]^-{i_\pi} \ar@/_1pc/[rr] & V^\sharp \ar[r]^-{j_\pi} & \cS
}
$$
Let moreover be $\widetilde{\cZ}$ the exactification of $Y'_{\pi'}\hookrightarrow \cZ\times_{W^\varnothing}\cZ'$, $\widetilde{\cX}:=\widetilde{\cZ}\times_{\widetilde{\cS},j_{\pi,\pi'}}{V'}^\sharp$ with associated dagger spaces $\widetilde{\frZ}$ and $\widetilde{\frX}$. 
Noting that $\omega^\bullet_{\wt\cZ/W^\varnothing,\bbQ}=\omega^\bullet_{\wt\cZ/W'^\varnothing,\bbQ}$, we may define a quasi-isomorphism
$$
\gamma\colon R\Gamma(\widetilde{\frZ},\omega^\bullet_{\widetilde{\cZ}/W^\varnothing, \bbQ}[u]) \xrightarrow{\sim} R\Gamma(\widetilde{\frZ},\omega^\bullet_{\widetilde{\cZ}/{W'}^\varnothing, \bbQ}[u'])
$$
where we define CDGAs $\omega^\bullet_{\widetilde{\cZ}/W^\varnothing, \bbQ}[u]$ and $\omega^\bullet_{\widetilde{\cZ}/W'^\varnothing, \bbQ}[u']$ similarly to above,
by setting
$$
\gamma(u^{[i]}):= 
\exp\Bigl(\log\bigl(\frac{w{s'}^\ell}{s}\bigr)\cdot\frac{N}{\ell}\Bigr) (\ell^i{u'}^{[i]}) =
\sum_{j=0}^i\frac{\bigl(\log\bigl(\frac{w{s'}^\ell}{s}\bigr)\bigr)^j}{j!} \ell^{i-j}{u'}^{[i-j]}.
$$
Note that 
$$
\log\bigl(\frac{w{s'}^\ell}{s}\bigr) = 
\sum_{m=1}^\infty\frac{(-1)^{m-1}}{m}\bigl(\frac{w{s'}^\ell}{s}-1\bigr)^m
$$
is well defined as global functions on $\widetilde{\cZ}$ and $\widetilde{\cS}$, 
and the compatibility of $\gamma$ with the differentials is confirmed by the following calculation
\begin{align*}
d\circ\gamma(u^{[i]})&=\frac{N}{\ell}d\log\bigl(\frac{ws'^\ell}{s}\bigr)\exp\Bigl(\log\bigl(\frac{ws'^\ell}{s}\bigr)\cdot \frac{N}\ell\Bigr)(\ell^i u'^{[i]})-\exp\Bigl(\log\bigl(\frac{ws'^\ell}{s}\bigr)\cdot\frac N\ell\Bigr)(\ell^id\log s'\cdot u'^{[i-1]})\\
&=\exp\Bigl(\log\bigl(\frac{ws'^\ell}{s}\bigr)\cdot \frac{N}\ell\Bigr)\left(\left(d\log\bigl(\frac{ws'^\ell}{s}\bigr)-\ell d\log s'\right)\ell^{i-1}u'^{[i-1]}\right)\\
&=-\exp\Bigl(\log\bigl(\frac{ws'^\ell}{s}\bigr)\cdot \frac{N}\ell\Bigr)\left(\ell^{i-1}d\log s\cdot u'^{[i-1]}\right)=\gamma\circ d(u^{[i]}).
\end{align*}
Moreover we see that $\gamma$ is compatible with the Frobenius operators and the normalized monodromy operators by similar computation as above.

The morphism $\gamma$ fits into two diagrams, firstly,
$$\xymatrix{
R\Gamma(\frZ, \omega^\bullet_{\cZ/W^\varnothing}[u]) \ar[r] \ar[d] &  R\Gamma(\widetilde{\frZ},\omega^\bullet_{\widetilde{\cZ}/W^\varnothing, \bbQ}[u]) \ar[d]_\gamma^\sim \\
R\Gamma(\frZ'',\omega^\bullet_{\cZ''/{W'}^\varnothing}[u']) & 
R\Gamma(\widetilde{\frZ},\omega^\bullet_{\widetilde{\cZ}/{W'}^\varnothing, \bbQ}[u']) \ar[l]_\sim}$$
which commutes because on $\frZ''$ we have
$$\log\bigl(\frac{w{s'}^\ell}{s}\bigr)=\log\bigl(\frac{w{s'}^\ell}{\rho_{\pi,\pi'}(s)}\bigr)=\log(1)=0,$$
and secondly a triangle
\begin{equation}\label{eq: diagram base change HK}
\xymatrix{
 R\Gamma(\widetilde{\frZ},\omega^\bullet_{\widetilde{\cZ}/W^\varnothing, \bbQ}[u]) \ar[d]_-\gamma^-\sim \ar[dr]^-\psi & \\
 R\Gamma(\widetilde{\frZ},\omega^\bullet_{\widetilde{\cZ}/{W'}^\varnothing, \bbQ}[u']) \ar[r]^-{\psi'}&
R\Gamma(\widetilde{\frX},\omega^\bullet_{\widetilde{\cX}/{V'}^\sharp,\bbQ})}
\end{equation}
where $\psi$ and $\psi'$ are defined by
\begin{align*}
u^{[i]} \mapsto \frac{(-\log_q(\pi))^i}{i!} &&\text{and}&& {u'}^{[i]} \mapsto \frac{(-\log_q(\pi'))^i}{i!},
\end{align*}
This triangle commutes by the following calculation
\begin{align*}
\psi'\circ\gamma(u^{[i]}) &=
\psi'\left(\sum_{j=0}^i\frac{\bigl(\log\bigl(\frac{w{s'}^\ell}{s}\bigr)\bigr)^j}{j!} \ell^{i-j}{u'}^{[i-j]}\right)
=\sum_{j=0}^i\frac{\bigl(\log\bigl(\frac{w{\pi'}^\ell}{\pi}\bigr)\bigr)^j\ell^{i-j}(-\log_q(\pi'))^{i-j}}{j!(i-j)!}\\
&=\sum_{j=0}^i \frac{\binom{i}{j} (-\log(\alpha))^j(-\log_q(\pi'^\ell))^{i-j}}{i!}= \frac{(-\log_q(\alpha{\pi'}^\ell))^i}{i!} = \frac{(-\log_q(\pi))^i}{i!} = \psi(u^{[i]}).
\end{align*}
where we used in the last line that $\log(w)=0$ since $w\in\boldsymbol{\mu}'$. 
Recall that we have
\begin{align*}
&R\Gamma_{\HK}(\cX,\pi)= R\Gamma(\frZ,\omega^\bullet_{\cZ/W^\varnothing,\bbQ}[u]),&&
R\Gamma_{\dR}(\cX)=R\Gamma(\frX,\omega^\bullet_{\cX/{V}^\sharp,\bbQ}),\\
&R\Gamma_{\HK}(\cX',\pi')=R\Gamma(\frZ',\omega^\bullet_{\cZ'/{W'}^\varnothing,\bbQ}[u']),&&
R\Gamma_{\dR}(\cX') = R\Gamma(\frX',\omega^\bullet_{\cX'/{V'}^\sharp,\bbQ}).
\end{align*}
We also let
\begin{align*}
&R\Gamma_{\HK}(\cX'',\pi'):= R\Gamma(\frZ'',\omega^\bullet_{\cZ''/{W'}^\varnothing,\bbQ}[u']) \\
&R\Gamma_{\HK}(\widetilde{\cX},\pi,\pi'):= 
R\Gamma(\widetilde{\frZ},\omega^\bullet_{\widetilde{\cZ}/W^\varnothing, \bbQ}[u]) \overset{\gamma}{\cong} R\Gamma(\widetilde{\frZ},\omega^\bullet_{\widetilde{\cZ}/{W'}^\varnothing, \bbQ}[u'])\\
&R\Gamma_{\dR}(\widetilde{\cX}):=
R\Gamma(\widetilde{\frX},\omega^\bullet_{\widetilde{\cX}/{V'}^\sharp,\bbQ}).
\end{align*}
Then we obtain a commutative diagram
$$\xymatrix{
R\Gamma_{\HK}(\cX,\pi) \ar[dr] \ar[rr]^{\Psi_{\pi,q}} \ar[d] \ar@/_3pc/[dd]_{f^\ast_{\HK}} && R\Gamma_{\dR}(\cX/V^\sharp) \ar[d] \ar@/^3pc/[dd]^{f^\ast_{\dR}}\\
R\Gamma_{\HK}(\cX'',\pi') & R\Gamma_{\HK}(\widetilde{\cX},\pi,\pi') \ar[l]_\sim \ar[r] & R\Gamma_{\dR}(\widetilde{\cX}/{V'}^\sharp)\\
R\Gamma_{\HK}(\cX',\pi') \ar[u]^\sim \ar[ur]^\sim \ar[rr]_{\Psi_{\pi,q}}&& 
R\Gamma_{\dR}(\cX'/{V'}^\sharp) \ar[u]^\sim
}$$
where the commutativity of the upper trapezoid follows from that of \eqref{eq: diagram base change HK}.
This shows the desired compatibility of $f^\ast_{\HK}$ and $f^\ast_{\dR}$.

With $\pi=\alpha w{\pi'}^\ell$ and $e\ell=e'$, the commutativity of the diagram \eqref{formula base change} follows from the following calculation
\begin{align*}
f^\ast_{\dR,K'}\circ \Psi_{\pi,\pi,K'}&= 
\Psi_{\pi',\pi,K'}\circ f_{\HK,K'}^\ast =\Psi_{\pi',\pi',K'} \circ\exp(-e'\log_{\pi}(\pi')\cdot\mathbf{N}) \circ f_{\HK,K'}^\ast\\
&=\Psi_{\pi',\pi',K'} \circ\exp(e\log(\alpha)\cdot \mathbf{N}) \circ f_{\HK,K'}^\ast\\
&=\Psi_{\pi',\pi',K'} \circ f_{\HK,K'}^\ast\circ\exp(e\log(\alpha)\cdot \mathbf{N})
\end{align*}
where we used \eqref{eq: choice of log} and \eqref{eq: base change and HK map}.

For the last statement, suppose that morphisms $\cX_3\xrightarrow{g}\cX_2\xrightarrow{f}\cX_1$ as in the statement are given.
For $l=1,2,3$, let $\pi_l$ a uniformiser of $V_l$, $k_l$ the residue field of $V_l$, and $W_l=W(k_l)$ the ring of Witt vectors.
Let $i_{\pi_l}\colon k_l^0\hookrightarrow V^\sharp_l$ be the exact closed immersion sending $\pi_l$ to the canonical generator of the log structure of $k_l^0$.
Let $\cS_l:=\Spwf W_l\llbracket s_l\rrbracket$ with the log structure defined by $s_l$.
Assume again that there exist log rigid $F$-data $(\cZ_l,i_l,\phi_l)$ for $Y_{l}:=\cX_l\times_{V^\sharp_l,i_{\pi_l}}k^0_l$.
We define morphisms $\overline{f}_{\pi_1,\pi_2}\colon Y_1\rightarrow Y_2$ and $\overline{g}_{\pi_2,\pi_3}\colon Y_2\rightarrow Y_3$ similarly to above.
For $1\leq l<n\leq 3$, let $\cZ_{l,n}$ be the exactification of the diagonal embedding $Y_n\hookrightarrow \cZ_l\times_{\cS_l}\cZ_n$.
Let $\breve\cZ$ be the exactification of the diagonal embedding $Y_3\hookrightarrow\cZ_1\times_{\cS_1}\cZ_2\times_{\cS_2}\cZ_3$.
We denote by $\frZ_{l,n}$ and $\breve\frZ$ the dagger spaces associated to $\cZ_{l,n}$ and $\breve\cZ$, respectively.
Let
\begin{align*}
&R\Gamma_\HK(\cX_{l,n},\pi_n):=R\Gamma(\frZ_{l,n},\omega^\bullet_{\cZ_{l,n}/W_n^\varnothing,\bbQ}[u_n]),&
&R\Gamma_\HK(\breve\cX,\pi_3):=R\Gamma(\breve\frZ,\omega^\bullet_{\breve\cZ/W_3^\varnothing,\bbQ}[u_3]),
\end{align*}
where we define CDGAs $\omega^\bullet_{\cZ_{l,n}/W_n^\varnothing,\bbQ}[u_l]$ and $\omega^\bullet_{\breve\cZ/W_3^\varnothing,\bbQ}[u_3]$ by $du_n^{[i+1]}:=-d\log s_n\cdot u_n^{[i]}$ for $n=1,2,3$.
Then we have a diagram
\[\xymatrix{
&&R\Gamma_\HK(\cX_{1,3},\pi_3)\ar[d]^-\sim&&\\
&&R\Gamma_\HK(\breve\cX,\pi_3)&&\\
R\Gamma_\HK(\cX_1,\pi_1)\ar[r]\ar[rruu]
&R\Gamma_\HK(\cX_{1,2},\pi_2)\ar[ru]
&R\Gamma_\HK(\cX_2,\pi_2)\ar[l]^-\sim\ar[r]
&R\Gamma_\HK(\cX_{2,3},\pi_3)\ar[lu]^-\sim
&R\Gamma_\HK(\cX_3,\pi_3)\ar[l]^-\sim\ar[lluu]^-\sim
}\]
where the morphisms of type $R\Gamma_\HK(-,\pi_l)\rightarrow R\Gamma_\HK(-,\pi_n)$ are defined by $u_l^{[i]}\mapsto \ell_{l,n}^iu_n^{[i]}$ where $\ell_{l,n}$ denotes the ramification index of $K_n$ over $K_l$.
Since $\ell_{1,3}=\ell_{1,2}\cdot\ell_{2,3}$ the above diagram commutes and we proved $(g\circ f)_{\HK}^\ast=g_{\HK}^\ast\circ f^\ast_{\HK}$.
The general case can be treated similarly by simplicial construction, i.e., replacing $\cZ_l$, $\cZ_{l,n}$, and $\breve\cZ$ by simplicial, bisimiplicial, and trisimplicial weak formal log schemes.
\end{proof}

\begin{corollary}
Let $\cX$ be a fine weak formal log scheme which is log smooth and adic over $V^\sharp$.
For any uniformisers $\pi$ and $\pi'$  of $V$, there exists a canonical isomorphism in the derived category of $(\varphi,N)$-modules
\[\rho_{\pi,\pi'}\colon R\Gamma_\HK(\cX,\pi)\xrightarrow{\cong}R\Gamma_\HK(\cX,\pi')\]
such that $\rho_{\pi,\pi}=\id$ and $\rho_{\pi',\pi''}\circ\rho_{\pi,\pi'}=\rho_{\pi,\pi''}$.
Identifying $R\Gamma_\HK(\cX,\pi)$ for all uniformisers $\pi$ via these isomorphisms, we have
\begin{equation*}
	\Psi_{\pi,\pi,K} = \Psi_{\pi',\pi',K}\circ\exp(e\log(v)\cdot \mathbf{N}),
\end{equation*}
where $v:=\pi(\pi')^{-1}\in V^\times$ and
\begin{equation*}
	\Psi_{\pi,q} = \Psi_{\pi',q}
\end{equation*}
for any $q\in\frm\setminus\{0\}$.
\end{corollary}

\begin{proof}
	The isomorphism $\rho_{\pi,\pi'}$ is given as $f_\HK^\ast$ in Proposition \ref{prop: base change and HK map} for the case $f=\id_\cX$.
	The only assertion not yet proven is that $\rho_{\pi,\pi}=\id$, but this follows immediately from the construction in the proof of Proposition \ref{prop: base change and HK map}.
\end{proof}

For a fixed $q\in\frm\setminus\{0\}$, we may by the above corollary identify $R\Gamma_\HK(\cX,\pi)$ and $\Psi_{\pi,q}$ for all choices of $\pi$.
For simplicity, we often use the notation $R\Gamma_\HK(\cX)$ and $\Psi_q$ for them.

\begin{corollary}
	Let $\cX$ be a strictly semistable weak formal log scheme quasi-compact and adic over $V^\sharp$.
	For any finite extension $K'/K$ with ring of integers $V'$, there exists a strictly semistable weak formal log scheme $\cX'$ over $V'^\sharp$ with a morphism $f\colon \cX' \rightarrow \cX$ over $V^\sharp$ which induces an isomorphism $\frX' \xrightarrow{\cong}\frX\times_KK'$.
	Moreover for any $q\in\frm\setminus\{0\}$ we have a commutative diagram
		\[\xymatrix{
		R\Gamma_\HK(\cX) \ar[r]^{f_\HK^*} \ar[d]^{\Psi_q} & R\Gamma_\HK(\cX') \ar[d]^{\Psi_q}\\
		R\Gamma_\dR(\cX) \ar[r]^{f_\dR^*} & R\Gamma_\dR(\cX'),
		}\]
	whose morphisms are all quasi-isomorphisms after tensoring with appropriate fields.
\end{corollary}

\begin{proof}
   The desired weak formal log scheme $\cX'$ can be obtained from the base change $\cX\times V'$ by a finite sequence of blow-ups.
   The process is analogous to the classical case explained in \cite[Prop.\,2.2]{Ha}, and to the formal case explained locally in \cite[Prop.~1.10]{HL} which can be globalised as stated in \cite[Rem.\,1.11.1]{HL}.
	The second statement follows immediately from the last proposition. 
\end{proof}

A crucial property of the original Hyodo--Kato map is that it provides a quasi-isomorphism after extension by $K$. 
In the strictly semistable case, the same holds for the rigid analogue.

\begin{proposition}\label{prop: HK is qis}
	Let $\cX$ be a strictly semistable weak formal log scheme adic over $V^\sharp$.
	For any uniformiser $\pi$ and $q\in\frm\setminus\{0\}$, the map $\Psi_{\pi,q,K}$ is a quasi-isomorphism.
\end{proposition}

\begin{proof}
	Working locally, we may suppsose that $\cX$ is affine and that there exists a strict log smooth morphism $\cX\rightarrow V^\sharp(n,m)$.
	Then $H^m_\HK(\cX)$ is finite dimensional for any $m$, by Corollary \ref{cor: coh of Y_I}.
	Moreover, by Proposition \ref{prop: choice of pi} (1), we may assume that $q=\pi$.
	In this case, $\Psi_{\pi,\pi}$ factors through the natural morphism $R\Gamma_\HK(Y)\rightarrow R\Gamma_\rig(Y/\cS)$.
	Thus the assertion follows from Corollary \ref{cor: natural diagram} and Lemma \ref{lem: base change from S}.
\end{proof}

Lastly, we note that in the smooth case our construction recovers the base change of (non-logarithmic) rigid cohomology.

\begin{proposition}
	Let $\cX^\varnothing$ be a smooth weak formal scheme adic over $V$ with the special fibre $Y^\varnothing$, and let $\cX$ be the weak formal log scheme whose underlying weak formal scheme is $\cX^\varnothing$ and whose log structure is the pull-back of that of $V^\sharp$.
	Then the monodromy operator on $H^k_\HK(\cX)$ is trivial, and there exists a canonical isomorphism $R\Gamma_\HK(\cX)\cong R\Gamma_\rig(Y^\varnothing/W^\varnothing)$  compatible with Frobenius such that the diagram
		\[\xymatrix{
		R\Gamma_\rig(Y^\varnothing/W^\varnothing)\ar[d]^\cong\ar[r] &R\Gamma_\rig(Y^\varnothing/V^\varnothing)\ar[d]^\cong\\
		R\Gamma_\HK(\cX)\ar[r]^{\Psi_q} & R\Gamma_\dR(\cX)
		}\]
	for any $q\in\frm\setminus\{0\}$ commutes.
	Here the upper horizontal map is the base change morphism of (non-logarithmic) rigid cohomology.
\end{proposition}

\begin{proof}
	Let $\pi\in V$ be a uniformiser and let $Y := \cX\times_{V^\sharp,i_\pi}k^0$.
	By taking an affine open covering of $Y$, we may construct a simplicial (log) scheme $U_\bullet$ and a simplicial weak formal scheme $\cU_\bullet$ over $W$ such that the $\cU_m$ are smooth over $W$ and endowed with Frobenius lifts $\phi'_m$ which commute with each other.
	Consider the weak formal scheme $\cZ_m := \cU_m\times_{\Spwf W}\cS$  and endow it with the pull-back log structure of $\cS$.
	Let $\phi_m$ be the endomorphism on $\cZ_m$ defined by $\phi'_m$ on $\cU_m$ and $\sigma$ on $\cS$.
	We endow $\cU_m$ with the pull-back log structure of $\cZ_m$ via the zero section.
	Let $i_m\colon U_m\hookrightarrow\cZ_m$ be the embedding given by the zero section.
	Then $(U_\bullet,\cU_\bullet,i_\bullet,\phi_\bullet)$ is a simplicial log rigid $F$-quadruple.
	We use this to compute the cohomology groups
		\begin{equation}\label{eq: sm1}
		H^k_\rig(Y/W^0)=H^k(\frU_\bullet,\omega^\bullet_{\cU_\bullet/W^0,\bbQ})\cong H^k(\frU_\bullet,\Omega^\bullet_{\frU_\bullet})= H^k_\rig(Y^\varnothing/W^\varnothing)
		\end{equation}
	and
		\begin{equation}\label{eq: sm2}
		H^k_\rig(Y/W^\varnothing)\cong H^k_\rig(Y^\varnothing/W^\varnothing)\oplus (H^{k-1}_\rig(Y^\varnothing/W^\varnothing)\otimes_FFd\log s)
		\end{equation}
	by the K\"{u}nneth formula.
	By Proposition \ref{prop: HK and rig}, the isomorphism in \eqref{eq: sm1} gives a quasi-isomorphism $R\Gamma_\HK(\cX)\cong R\Gamma_\rig(Y^\varnothing/W^\varnothing)$.
	Together with \eqref{eq: sm2} this implies that $H^k_\rig(Y/W^\varnothing)\rightarrow H^k_\HK(Y)$ is surjective.
	Thus any element of $H^k_\HK(Y)$ can be represented by a \v{C}ech cocycle of $\omega^\bullet_{\cZ_\bullet/W^\varnothing,\bbQ}$, which maps to zero by the monodromy operator.
	
	Now we consider the following diagram:
		\[\xymatrix{
		R\Gamma_\rig(Y^\varnothing/W^\varnothing)\ar[d]_\cong \ar[rd]\ar[rr]& & R\Gamma_\rig(Y^\varnothing/V^\varnothing)\ar[dd]^\cong\\
		R\Gamma_\rig(Y/W^0) & R\Gamma_\rig(Y/W^\varnothing)\ar[ld]\ar[l]\ar[rd]\ar@{}[ru]^<<<<<<{(*)} & \\
		R\Gamma_\HK(Y)\ar[u]^\cong\ar[rr]_{\Psi_{\pi,q}} & & R\Gamma_\rig(Y/V^\sharp)	
		}\]
	The triangles except for $(*)$ commute naturally.
	Therefore it suffices to show that $(*)$ also commutes.
	
	Indeed, on the one hand $R\Gamma_\rig(Y^\varnothing/W^\varnothing)\rightarrow R\Gamma_\rig(Y/W^\varnothing)\rightarrow R\Gamma_\rig(Y/V^\sharp)$ comes from the composition of natural projections $\cX_\bullet := \cZ_\bullet\times_\cS V^\sharp\rightarrow \cZ_\bullet\rightarrow \cU^\varnothing_\bullet$, and on the other hand $R\Gamma_\rig(Y^\varnothing/W^\varnothing)\rightarrow R\Gamma_\rig(Y^\varnothing/V^\varnothing)\rightarrow R\Gamma_\rig(Y/V^\sharp)$ comes from the composition of $\cX_\bullet\rightarrow\cX^\varnothing_\bullet \rightarrow\cU^\varnothing_\bullet$, where the second morphism is the projection given by $\cX^\varnothing_\bullet= (\cU^\varnothing_\bullet\times_{W^\varnothing}\cS^\varnothing) \times_{\cS^\varnothing}V^\varnothing= \cU^\varnothing_\bullet\times_{W^\varnothing}V^\varnothing$.
	Thus  they coincide with each other.
\end{proof}

%%%%%%
%
\section{Example: Tate curves}\label{Sec: Tate curve}
%
%%%%%%

In this section, we compute the rigid Hyodo--Kato cohomology and the rigid Hyodo--Kato map of a Tate curve.
The computations will also demonstrate the advantage of weak formal schemes which are not necessarily adic over the base.
It allows us to vary a Tate curve over the open unit disk, which is not possible over the closed unit disk.
We continue to use the conventions established in \ref{main notations}.

Note that the $p$-adic Hodge structure of a Tate curve has been computed as the Di\'{e}udonne module associated to the dual of a Tate module.
Thus the computations in this section don't present a new result, but give an explicit description in terms of \v{C}ech cocycles.
In continuation of the previous section, we assume that $k$ is algebraic over $\bbF_p$.

For any $n\in\bbZ$, let $\cZ_n:=\Spwf W\llbracket s\rrbracket[v_n,w_n]^\dagger/(v_nw_n-s)$ and endow it with the log structure associated to the map 
	\[\bbN^2\rightarrow W\llbracket s\rrbracket[v_n,w_n]^\dagger/(v_nw_n-s);\ (1,0)\mapsto v_n, (0,1)\mapsto w_n.\]
If we set $t:=s^{n-1}v_n=\frac{s^n}{w_n}$, we may write 
	\[\cZ_n=\Spwf W\llbracket s\rrbracket[v_n,w_n]^\dagger/(v_nw_n-s)=\Spwf W\llbracket s\rrbracket[\frac{t}{s^{n-1}},\frac{s^n}{t}]^\dagger.\]

Let
	\begin{eqnarray*}
	\cV_n&:=&\Spwf W\llbracket s\rrbracket[v_n,v_n^{-1}]^\dagger=\Spwf  W\llbracket s\rrbracket [\frac{t}{s^{n-1}},\frac{s^{n-1}}{t}]^\dagger,\\
	\cW_n&:=&\Spwf W\llbracket s\rrbracket[w_n,w_n^{-1}]^\dagger=\Spwf W\llbracket s\rrbracket[\frac{s^n}{t},\frac{t}{s^n}]^\dagger
	\end{eqnarray*}
be open subsets of $\cZ_n$.

For $r\geq 1$, we glue $\cZ_1,\ldots,\cZ_r$ via the natural isomorphisms $\cV_{n+1}\cong\cW_n$ and the isomorphism $\cV_1\cong\cW_r,\ v_1\mapsto w_r^{-1}$.
Denote by $\cZ^{(r)}$ the resulting weak formal log scheme over $\cS$.
Let $i^{(r)}\colon Y^{(r)}\hookrightarrow\cZ^{(r)}$ be the exact closed immersion defined by the ideal $(p,s)$.
For $n$ varying, the endomorphisms $\phi_n\colon \cZ_n\rightarrow\cZ_n$ given by the Frobenius action on $W$ and $v_n\mapsto v_n^p$, $w_n\mapsto w_n^p$ naturally glue to form an endomorphism $\phi^{(r)}\colon \cZ^{(r)}\rightarrow\cZ^{(r)}$.
Note that $Y^{(r)}$ is strictly semistable if $r\geq 2$ and a nodal curve if $r=1$.
In particular, if $r\geq 2$, $(Y^{(r)},\cZ^{(r)},i^{(r)},\phi^{(r)})$ is a strictly semistable log rigid $F$-quadruple over $(k^0\hookrightarrow\cS,\sigma)$.
We denote by $\frZ^{(r)}$, $\frZ_n$,  $\frW_n$ and $\frS$ the dagger spaces associated to $\cZ^{(r)}$, $\cZ_n$, $\cW_n$ and $\cS$, respectively.

We note that for each $a\in\overline{\bbQ}_p$ with $\lvert a\rvert<1$, the fibre at $s=a$ in $\mathfrak{Z}^{(r)}$ is the Tate curve over $F(a)$ with period $a^r$.
Namely, there is a natural isomorphism of dagger spaces
	\[\mathfrak{Z}^{(r)}\times_{\frS,s\mapsto a}\Sp F(a)\cong F(a)^\times/a^{r\bbZ}\]
which sends $v_1$ to the canonical parameter $t$ of $F(a)^\times/a^{r\bbZ}$. 

To describe Hyodo--Kato cohomology and Hyodo--Kato map, we define a kind of ordered \v{C}ech complex $\check{C}_\HK^\bullet$ to be the complex
	\begin{eqnarray*}
	&&\prod_{n=1}^r\bigoplus_{i=0}^\infty\Gamma(\frZ_n,\cO_{\frZ_n} u^{[i]})\xrightarrow{D_1} \prod_{n=1}^r\bigoplus_{i=0}^\infty\Gamma(\frZ_n,\omega^1_{\cZ_n/W^\varnothing,\bbQ}u^{[i]})\oplus\prod_{n=1}^r\bigoplus_{i=0}^\infty \Gamma(\frW_n,\cO_{\frW_n} u^{[i]})\\
	&&\xrightarrow{D_2} \prod_{n=1}^r\bigoplus_{i=0}^\infty\Gamma(\frZ_n,\omega^2_{\cZ_n/W^\varnothing,\bbQ}u^{[i]})\oplus\prod_{n=1}^r\bigoplus_{i=0}^\infty \Gamma(\frW_n,\omega^1_{\cW_n/W^\varnothing,\bbQ} u^{[i]})\xrightarrow{D_3} \prod_{n=1}^r\bigoplus_{i=0}^\infty\Gamma(\frW_n,\omega^2_{\cW_n/W^\varnothing,\bbQ}u^{[i]})
	\end{eqnarray*}
with differentials given by
	\begin{align*}
	D_1\alpha:=(d\alpha,\partial\alpha),
	&&D_2(\alpha,\beta):=(d\alpha,d\beta-\partial\alpha),
	&&D_3(\alpha,\beta):=d\beta+\partial\alpha,
	\end{align*}
where
	\begin{eqnarray*}
	d\colon \prod_{i=1}^r\bigoplus_{i=0}^\infty\Gamma(\frZ_n,\omega^k_{\cZ_n/W^\varnothing,\bbQ}u^{[i]})&\rightarrow& \prod_{n=1}^r\bigoplus_{i=0}^\infty\Gamma(\frZ_n,\omega^{k+1}_{\cZ_n/W^\varnothing,\bbQ}u^{[i]})\\
	(\text{respectively }d\colon \prod_{i=1}^r\bigoplus_{i=0}^\infty\Gamma(\frW_n,\omega^k_{\cW_n/W^\varnothing,\bbQ}u^{[i]})&\rightarrow& \prod_{n=1}^r\bigoplus_{i=0}^\infty\Gamma(\frW_n,\omega^{k+1}_{\cW_n/W^\varnothing,\bbQ}u^{[i]})\\
	(\eta_1,\ldots,\eta_r)&\mapsto&(d\eta_1,\ldots,d\eta_r))
	\end{eqnarray*}
and
	\begin{eqnarray*}
	\partial\colon \prod_{i=1}^r\bigoplus_{i=0}^\infty\Gamma(\frZ_n,\omega^k_{\cZ_n/W^\varnothing,\bbQ}u^{[i]})&\rightarrow& \prod_{n=1}^r\bigoplus_{i=0}^\infty\Gamma(\frW_n,\omega^k_{\cW_n/W^\varnothing,\bbQ}u^{[i]})\\
	(\eta_1,\ldots,\eta_r)&\mapsto&(\eta_2-\eta_1,\ldots,\eta_r-\eta_{r-1},\eta_r-\eta_1)
	\end{eqnarray*}
denote the differential of $\omega^\bullet_{\cZ^{(r)}/W^\varnothing,\bbQ}$ and the \v{C}ech differential respectively.
Frobenius and monodromy operators on $\check{C}_\HK^\bullet$ are defined by the same manner as the Hyodo--Kato cohomology.

Let $K$ be a totally ramified extension of $F$ with ring of integers $V$, and let $\pi\in K$ be a uniformiser. 
Let $\cX:=\cZ^{(r)}\times_{\cS,s\mapsto\pi} V^\sharp$ and $\frX$ the associated dagger space, which is the Tate curve over $K$ with period $\pi^r$.
Then $R\Gamma_\dR(\cX)=R\Gamma(\frX,\omega^\bullet_{\cX/V^\sharp,\bbQ})= R\Gamma(\frX,\Omega^\bullet_\frX)$, and the Hodge filtration on the de Rham cohomology is induced by the stupid filtration $\mathrm{Fil}^i\Omega^\bullet_\frX:=\Omega_\frX^{\bullet\geq i}$.
The de Rham cohomology is computed by the ordered \v{C}ech complex $\check{C}_\dR^\bullet$ 
	\begin{eqnarray*}
	&&\prod_{n=1}^r\bigoplus_{i=0}^\infty\Gamma(\frX_n,\cO_\frX)\xrightarrow{D_1} \prod_{n=1}^r\bigoplus_{i=0}^\infty\Gamma(\frX_n,\Omega^1_{\frX})\oplus\prod_{n=1}^r\bigoplus_{i=0}^\infty \Gamma(\frX_n\cap\frW_n,\cO_\frX)\\
	&&\xrightarrow{D_2} \prod_{n=1}^r\bigoplus_{i=0}^\infty\Gamma(\frX_n,\Omega^2_{\frX})\oplus\prod_{n=1}^r\bigoplus_{i=0}^\infty \Gamma(\frX_n\cap\frW_n,\Omega^1_{\frX})\xrightarrow{D_3} \prod_{n=1}^r\bigoplus_{i=0}^\infty\Gamma(\frX_n\cap\frW_n,\Omega^2_{\frX}),
	\end{eqnarray*}
where the differentials are defined in a similar way as those of $C_\HK^\bullet$.

There is a commutative diagram
\[\xymatrix{
H^i(\check{C}_\HK^\bullet)\ar[r]^-{\check{\Psi}_{\pi,q}}\ar[d]&H^i(\check{C}_\dR^\bullet)\ar[d]^-\sim\\
H^i_\HK(\cX,\pi)\ar[r]^-{\Psi_{\pi,q}}&H^i_\dR(\cX)
}\]
where $\check{\Psi}_{\pi,q}$ is defined by the same rule as $\Psi_{\pi,q}$, i.e., $s\mapsto\pi$ and $u^{[i]}\mapsto\frac{(-\log_q\pi)^i}{i!}$, the left vertical map is compatible with Frobenius and monodromy operators, and the right vertical map is an isomorphism.
We obtain now the following description of the rigid Hyodo--Kato and de Rham cohomology groups and the rigid Hyodo--Kato map.

\begin{proposition}
	\begin{enumerate}
	\item We have $H^1_\HK(\cX)=H^1_\HK(Y^{(r)})\cong Fe_1^\HK\oplus Fe_2^\HK$ with $\varphi(e_1^\HK)=e_1^\HK$, $N(e_1^\HK)=0$, $\varphi(e_2^\HK)=pe_2^\HK$, and $N(e_2^\HK)=re_1^\HK$, where the classes $e_1^\HK$ and $e_2^\HK$ are the images of the classes of $H^1(\check{C}_\HK^\bullet)$ represented by the cocycles
			\begin{eqnarray*}
			(0,\ldots,0,1)&\in&\prod_{n=1}^r\Gamma(\frW_n,\cO_{\frW_n})\subset\check{C}^1_\HK,\\
			(d\log w_1,\ldots,d\log w_r)+(-u^{[1]},\ldots,-u^{[1]},u^{[1]})&\in&\prod_{n=1}^r\Gamma(\frZ_n,\omega^1_{\cZ_n/W^\varnothing,\bbQ})\oplus\prod_{n=1}^r\Gamma(\frW_n,\cO_{\frW_n}u^{[1]})\subset\check{C}^1_\HK
		\end{eqnarray*}
		respectively.
		Note that the normalized monodromy is given by $\mathbf{N}=e^{-1}N$ where $e$ is the ramification index of $K$ over $\bbQ_p$.
	\item We have $H^1_\dR(\cX)\cong Ke_1^\dR\oplus Ke_2^\dR$ with the Hodge filtration
			\[\mathrm{Fil}^i H^1_\dR(\cX)=\begin{cases}
			Ke_1^\dR\oplus Ke_2^\dR&\text{if }i\leq 0,\\
			Ke_2^\dR&\text{if }i=1,\\
			0&\text{if }i\geq 2,
			\end{cases}\]
		where the classes $e_1^\dR$ and $e_2^\dR$ are represented by the cocycles
			\begin{eqnarray*}
			(0,\ldots,0,1)&\in&\prod_{n=1}^r\Gamma(\frX_n\cap\frW_n,\mathcal{O}_{\frX})\subset\check{C}^1_\dR,\\
			(d\log w_1,\ldots,d\log w_r)&\in&\prod_{n=1}^r\Gamma(\frX_n,\Omega^1_{\frX})\subset\check{C}^1_\dR
			\end{eqnarray*}
		respectively.
	\item For $q\in\frm\setminus\{0\}$, the rigid Hyodo--Kato map $\Psi_q=\Psi_{\pi,q}\colon H^1_\HK(\cX)\rightarrow H^1_\dR(\cX)$ is given by
			\begin{align*}
			\Psi_q(e_1^\HK)=e_1^\dR&&\text{and}&&\Psi_q(e_2^\HK)=e_2^\dR-r\log_q(\pi)e_1^\dR.
			\end{align*}
	\end{enumerate}
\end{proposition}

\begin{proof}
The description of the de Rham cohomology is well-known.
The computation of the images of $e_1^\HK$ and $e_2^\HK$ by the Frobenius operator, the monodromy operator, and the Hyodo--Kato map is straightforward.
Then we see that $e_1^\HK$ and $e_2^\HK$ gives a base of $H^1_{\HK}(\cX)$, since $\Psi_{q,K}$ is an isomorphism.
\end{proof}

\begin{remark}
	We may also compute the zeroth and second cohomology groups of $\cX$.
	The rigid Hyodo--Kato map in both cases is independent of the choice of $q$, and the resulting filtered $(\varphi,N)$-modules are $K(0)$ and $K(-1)$, respectively.
\end{remark}

%%%%%%
%%%
\section{Comparison with classical crystalline Hyodo--Kato theory}\label{Sec: comparison}
%%%
%%%%%

The goal of this section is to compare the rigid Hyodo--Kato theory introduced in this paper with classical crystalline Hyodo--Kato theory in the case of a proper semistable log scheme over $V^\sharp$ possibly with a horizontal divisor. Here by abuse of notation we also denote by $V^\sharp$ the scheme $\Spec V$ endowed with the canonical log structure.
We continue to use the conventions established in \ref{main notations}.
In addition,  a natural number $n\geq 1$ as an index denotes the reduction modulo $p^n$.

Our main references are \cite{HK}, where Hyodo--Kato theory was introduced, and  \cite{Fa,Ts}. 
Compare also \cite{Na} where some of the technical problems were solved.

In the following, we will use a definition of log crystalline cohomology which is a slight generalisation of the definitions in \cite{Sh2} and as such will be very useful for comparison with log convergent cohomology. 
For a non-logarithmic version compare \cite[\S 7]{BO}.

Let $T\hookrightarrow\cT$ be an exact closed immersion (not necessarily homeomorphic) of a fine log scheme $T$ over $\bbF_p$ into a noetherian fine formal log scheme $\cT$ over $\bbZ_p$.
Assume that $\cT=\Spf A$ is affine.
Let $\cI\subset A$ be the ideal of $T$ in $\cT$ and take an ideal of definition $\cJ$ of $\cT$.
Let $A_{PD}$ be the $\cJ$-adic completion of the $PD$-envelope of $A$ with respect to $\cI$ over $\bbZ_p$ (which is enodwed with the canonical $PD$-structure on $p\bbZ_p$).
We denote by $\gamma$ the $PD$-structure on $A_{PD}$.
Set $\cT_{PD}:=\Spf A_{PD}$ equipped with the pull-back log structure of that on $\cT$.
Denote by $\cT_{PD,n}$ the closed log subscheme of $\cT_{PD}$ defined by the ideal $\cJ^nA_{PD}$.
For a fine log scheme $Y$ over $T$ the log crystalline site $(Y/\cT_{PD,n})_{\cris}$ is defined as usual. 
To define its ``limit'' we follow definitions in \cite[Def.\ 1.2 and 1.4]{Sh2}, which is the logarithmic version of  \cite[7.17 Def.]{BO}.

\begin{definition}\label{def: crystalline site}
Let $Y$ be a fine log scheme over $T$.
	\begin{enumerate}
	\item
	 The log crystalline site $(Y/\cT_{PD})_\cris$ has as objects quadruples $(U,L,i,\delta)$ where $U$ is a strictly \'{e}tale fine log scheme  over $Y$, $L$ is a fine log scheme over $\cT_{PD,n}$ for some $n$, $i\colon U\hookrightarrow L$ is an exact closed immersion over $\cT_{PD}$, and $\delta$ is a $PD$-structure on the ideal of $U$ in $L$ which is compatible with $\gamma$.
	Morphisms in $(Y/\cT_{PD})_\cris$ are defined in the usual way.
	Covering families are induced by the \'{e}tale topology on $L$.
	
	\item  A sheaf $\sF$ on $(Y/\cT_{PD})_\cris$ is equivalent to a datum of sheaves (with respect to the \'etale topology) $\sF_L$ on $L$ for each  $(U,L,i,\delta)\in(Y/\cT_{PD})_\cris$ and appropriate transition maps. 
	It is clear how to define the crystalline structure sheaf $\sO_{Y/\cT_{PD}}$.
		An $\sO_{Y/\cT_{PD}}$-module $\sE$ is a \textit{crystal} on $Y$ over $\cT_{PD}$ if for any $f\colon (U',L',i',\delta')\rightarrow(U,L,i,\delta)$ in $(Y/\cT_{PD})_\cris$ the natural homomorphism $f^*\sE_L\rightarrow \sE_{L'}$ is an isomorphism.
		\end{enumerate}
\end{definition}

\begin{remark}
	Let $\cT_n$ be the closed subscheme of $\cT$ defined by $\cJ^n$.
	We observe that $\cT_{PD,n}$ coincides with the log $PD$-envelope of $T\hookrightarrow\cT_n$ over $\bbZ_p$ if $n$ is large enough.

	Namely, suppose that $\cT=\Spf A$ and that the ideal of definition $\cJ\subset A$ is generated by $f_1,\ldots,f_m\in A$.
	Let $n$ be an integer so that $\cJ^n\subset\cI$.
	Let $A'$ and $B_n$ be the $PD$-envelopes of $A/\cI$ in $A$ and $A/\cJ^n$, respectively.
	Let $A_{PD}$ be the $\cJ$-adic completion of $A'$.
	Then by \cite[\href{https://stacks.math.columbia.edu/tag/07HB}{Lem.\,07HB}]{stacks} $\cK_n:=\Ker(A'\rightarrow B_n)$ is generated by $(f_1^{n_1}\cdots f_m^{n_m})^k/k!$ for all $k\geq 1$ and $(n_i)_i\in\bbN^m$ with $\sum_in_i=n$.
	Clearly we have $\cJ^nA'\subset \cK_n$.
	If $k=1$, it is clear that  $(f_1^{n_1}\cdots f_m^{n_m})^k/k!$ belongs to $\cJ^nA$.
	Let $k\geq 2$ and $(n_i)\in\bbN^m$ with $\sum_i n_i=n$.
	If $n\geq m+1$, we have $n_i\geq 2$ for some $i$.
	Therefore, $n_i(k-1)\geq k$ and hence $f_i^{n_i(k-1)}/k!\in A'$.
	Thus
		\[(f_1^{n_1}\cdots f_m^{n_m})^k/k!=(f_1^{n_1}\cdots f_m^{n_m})\cdot(f_1^{n_1}\cdots f_m^{n_m})^{k-1}/k!\in\cJ^nA'.\]
	Consequently we have $\cJ^nA'=\cK_n$ and hence $A_{PD}/\cJ^nA_{PD}=A'/\cJ^nA'=B_n$ if $n\geq m+1$.
	This shows the remark.
\end{remark}

\begin{definition}
Let $Y$ be a fine log scheme over $T$.
The crystalline cohomology of $Y$ over $\cT_{PD}$ is the cohomology of the crystalline structure sheaf $\sO_{Y/\cT_{PD}}$ on the crystalline site $(Y/\cT_{PD})_{\cris}$ and similarly for the finite versions. 
Analogous to the non-logarithmic version in \cite[7.19 Prop.]{BO} we have
	\begin{equation}\label{equ: holim}
	R\Gamma_{\cris}(Y/\cT_{PD}) \cong \holim R\Gamma_{\cris}(Y/\cT_{PD,n}).
	\end{equation}
\end{definition}

We will consider crystalline sites for the following $PD$-bases obtained from Definition \ref{def: crystalline site}.

\begin{enumerate}
\item Consider the exact closed immersion $(T\hookrightarrow \cT)=(k^0\rightarrow W^0)$, with the ideals $(p)=\cJ=\cI$.
Because $W$ is already $p$-adically complete and the ideal  $(p)$ has divided powers we have $\cT_{PD}=W^0$ as well.
As mentioned before, the ring $W$ is endowed with the Witt vector Frobenius. 

\item Consider the  exact closed immersion $(T\hookrightarrow \cT)= (V_1^\sharp \hookrightarrow V^\sharp)$ of log schemes induced by reduction modulo $p$ with ideals $\cI=\cJ=(p)$. 
For the same reasons as above $\cT_{PD}$ is again $V^\sharp$.

\item 
Consider the exact closed immersion of fine formal log schemes 
$(T\hookrightarrow \cT) = (V_1^\sharp \hookrightarrow \cS)$  
given by the ideal $\cI=(p,f)$ where $f=s^e+ \sum_{i=0}^{e-1}a_is^i\in W[s]$  is the Eisenstein polynomial of a uniformiser $\pi$ of $V$, and $e$ the ramification index of $V$. 
Note that this immersion depends on $\pi$.
By the formalism described at the beginning of the section, 
we obtain a ring $R_{\PD}$   by completing with respect to the ideal $\cJ=\cI=(p,f)$ the $\PD$-envelope of $(V_1^\sharp \hookrightarrow \cS)$
for the ideal $\cI=(p,f)$. 
Since $(p,s)^e\subset (p,f)\subset (p,s)$, it is clear that $R_{\PD}$ is also $(p,s)$-adically complete. 
Its  $\PD$-ideal is given as the closure of $(p, \frac{f^n}{n!}\;\vert\;  n\geq 1)  =( p,\frac{s^{en}}{n!}\;\vert\; n\geq 1)$.
There is a lifting of Frobenius $\sigma$ to $R_{\PD}$ induced by $s\mapsto s^p$ extending the canonical Frobenius on $W$ which is a $\PD$-morphism.
Denote by $\cS_{\PD}$  the formal scheme $\Spf R_{\PD}$ with the log structure generated by $s$, and by  $\cS_{\PD,n}$ its reduction modulo $p^n$.
\end{enumerate}

\begin{remark}
Classically, one would rather start with the exact closed immersion of fine formal log schemes 
$(T\hookrightarrow \cT) = (V_1^\sharp \hookrightarrow \Spf\widehat{W[s]})$, 
where $\widehat{W[s]}$ denotes the $p$-adic completion of $W[s]$,  
given by the ideal $\cI=(p,f)$ where $f=s^e+ \sum_{i=0}^{e-1}a_is^i\in W[s]$  is the Eisenstein polynomial of a uniformiser $\pi$ of $V$ 
(compare \cite[Lem.\,5.2]{HK} or \cite[p.\,1721]{NN}). 
The next lemma shows however, that both approaches lead to identical 
PD-envelopes.
We chose the above construction to stress the relation with 
the rigid approach where $\cS$ is used instead of $\Spf\widehat{W[s]}$. 
\end{remark}

\begin{lemma}
Consider  the exact closed immersion of fine formal log schemes 
$(T\hookrightarrow \cT)=(V_1^\sharp \hookrightarrow \Spf \widehat{W[s]})$ 
given by the ideal $\cI=(p,f)$ 
as in the remark above. 
Let $R'_{\PD}$ be the $p$-adic completion of its $\PD$-envelope over $\bbZ_p$ with the canonical PD-structure. 
The formal scheme $\cS'_{\PD} = \Spf R'_{\PD}$ endowed with the log structure generated by $s$ 
coincides with the log formal scheme $\cS_{\PD}$. 
In particular, $\cS_{\PD}$ is $p$-adic.
\end{lemma}

\begin{proof}
We first show that the ring $R'_{\PD}$ is complete with respect to the 
$(p,s)$-adic topology.
Indeed, the ring $R'_{\PD}$ is the $W$-subalgebra of $F\llbracket s\rrbracket$ given by
$$
R'_{\PD}=\left\{ \sum_{i=0}^\infty a_i\frac{s^i}{\lfloor \frac{i}{e}\rfloor !}\;\big\vert\, a_i\in W(k),\lim_{i\rightarrow \infty} a_i=0\right\}.
$$
Since it contains divided powers, 
we have 
$$
s^{ep}=\frac{(p-1)! s^{ep}}{p!} \cdot p \in (p)=p R'_{\PD},
$$
where $e$ is the ramification index of $V$, 
and therefore $(p)^{ep} \subset (p,s)^{ep}\subset (p)$. 
It follows that $R'_{\PD}$ is $(p,s)$-complete 
if and only if it is $p$-complete, which it is by construction. 
Since $(p,s)^e\subset (p,f)\subset (p,s)$, 
the ring $R'_{\PD}$ is evidently $(p,f)$-complete, too. 

Next we observe that for any power of $(p,f)$, 
the ring $R'_{\PD}/(p,f)^n$ has divided powers, 
thus its spectrum is the PD-envelope of 
$(V_1 \hookrightarrow \Spec(R'_{\PD}/(p,f)^n))$. 
By definition of $R'_{\PD}$ this coincides with the PD-envelope of
$(V_1 \hookrightarrow \Spec(\widehat{W[s]}/(p,f)^n))$. 
But since $W\llbracket s\rrbracket /(p,f)^n = \widehat{W[s]}/(p,f)^n$, 
this in turn coincides with the PD-envelope of $(V_1 \hookrightarrow \Spec W\llbracket s\rrbracket/(p,f)^n)$ which we denote by $\cS_{n,\PD}= \Spec(R_{n,\PD})$. 
Putting everything together, we obtain 
$$
R_{\PD} = \varprojlim_n R_{n,\PD} = \varprojlim_n(R'_{\PD}/(p,f)^n) = R'_{\PD}, 
$$
where the first equality holds by definition and the last equality by the $(p,f)$-completeness of $R'_{\PD}$. 
This shows the claim. 
\end{proof}

Let $\pi$ be a uniformiser of $V$. There are exact closed embeddings
	$$
	W^0 \xrightarrow{j_0} \cS_{\PD}  \xleftarrow{j_\pi} V^\sharp 
	$$
via $s \mapsto  0$ and $s \mapsto  \pi$.
There is no ambiguity in denoting $\tau:=j_0\circ i_0$, where $i_0:k^0\hookrightarrow W^0$ is the canonical embedding.
Let $i_\pi: k^0\hookrightarrow V^\sharp$ be the unique morphism such that $\tau=j_\pi\circ i_\pi$.  
These morphisms fit into a commutative diagram of formal log schemes
\begin{equation}\label{diag: comm}
\xymatrix{
&k^0 \ar[dl]_{i_0} \ar[dr]^{i_\pi} \ar[d]^\tau&\\
W^0  \ar[r]_{j_0}& \cS_{\PD}& V^\sharp  \ar[l]^{j_\pi}.
}\end{equation}

\begin{definition}
For a fine proper log smooth log scheme  $X$ over $V^\sharp$ 
such that  $Y :=  X\times_{V^\sharp,i_\pi}k^0$ is of Cartier type  over $k^0$, we consider the log crystalline complexes
	\begin{align*}
	 R\Gamma_{\cris}(X/V^\sharp) &: =   \holim  R\Gamma_{\cris}(X_1/V_n^\sharp),\\
	 R\Gamma_{\cris}(X/\cS_{PD},\pi) &: =  \holim R\Gamma_{\cris}(X_1/\cS_{PD,n},\pi),\\	
	 R\Gamma_{\HK}^\cris(X,\pi) &: =    R\Gamma_{\cris}(Y/W^0): =  \holim  R\Gamma_{\cris}(Y/W_n^0).
	\end{align*}
The second complex is defined via the morphism $V^\sharp_1=\cS_{PD,1}\hookrightarrow \cS_{PD}$
which is the reason it depends on $\pi$. 
As for the last complex, it depends on $\pi$ because the immersion $Y\hookrightarrow X$ induced by the exact closed immersion of log schemes $i_\pi:k^0\hookrightarrow V^\sharp$ depends on $\pi$ by definition.

The Frobenius action $\varphi$ on  
$R\Gamma_{\cris}(X/\cS_{PD},\pi)$ (respectively $R\Gamma_{\HK}^\cris(X,\pi)$)  is induced by the absolute Frobenius on $X_1$ (respectively $Y$) and the Frobenius $\sigma$ on 
$\cS_{PD}$ (respectively  on $W$). 
The Frobenius action is invertible on $R\Gamma_{\HK}^\cris(X,\pi)_\bbQ$. 
\end{definition}

\begin{remark}
In the situation above, for two choices of uniformisers $\pi$ and $\pi'$ of $V$, let $Y :=  X\times_{V^\sharp,i_\pi}k^0$ and $Y':=  X\times_{V^\sharp,i_{\pi'}}k^0$.
If the $k^0$-log scheme $Y$ is of Cartier type over $k^0$, then the same is true for $Y'$. 
Thus for simplicity we say in this case that $X$ is of Cartier type. 
\end{remark}

\begin{proposition-definition}\label{prop-def: cris sections}
Let $X$ be a proper fine log smooth log scheme  over $V^\sharp$ 
which is of Cartier type.
Consider the morphisms
\begin{equation}\label{eq:crisHK}
	R\Gamma_{\HK}^\cris(X,\pi) \xleftarrow{j_0^\ast}  R\Gamma_{\cris}(X/{\cS_{PD}},\pi) \xrightarrow{j_\pi^\ast} R\Gamma_{\cris}(X/V^\sharp)
\end{equation}
induced by the morphisms of log schemes $j_0$ and $j_\pi$. 
The map $j_0^\ast: R\Gamma_{\cris}(X/{\cS_{PD}},\pi)_\bbQ \rightarrow   R\Gamma_{\HK}^\cris(X,\pi)_\bbQ$ admits in the derived category a unique functorial $F$-linear section $s_\pi\colon R\Gamma_{\HK}^\cris(X)_{\bbQ} \rightarrow   R\Gamma_{\cris}(X/{\cS_{PD}},\pi)_{\bbQ}$ which commutes with the Frobenius. 
We set
	$$
	\Psi_\pi^\cris: = j_\pi^\ast\circ s_\pi\colon  R\Gamma_{\HK}^\cris(X,\pi)_{\bbQ} \rightarrow   R\Gamma_{\cris}( X/V^\sharp)_{\bbQ}.
	$$
It induces a $K$-linear functorial quasi-isomorphism 
$$
\Psi_{\pi,K}^\cris :=  \Psi_{\pi}^\cris\otimes 1\colon   R\Gamma_{\HK}^\cris(X,\pi)\otimes_{W(k)}K  \rightarrow   R\Gamma_{\cris}( X/V^\sharp)_{\bbQ}.
$$
\end{proposition-definition}

\begin{proof}
The existence of such a section $s_\pi$ follows from \cite[Lem.\,5.2]{HK} (cf.\,\cite[Prop.~4.4.6]{Ts}). 
It was shown in \cite[Thm.\,5.1]{HK} that the composition $\Psi_\pi^\cris: = j_\pi^\ast\circ s_\pi$ induces a $K$-linear quasi-isomorphism $\Psi_{\pi,K}^\cris$ after tensoring with $K$. 

It remains to address the uniqueness. 
According to \cite[Lem.\,4.4.10]{Ts} 
the cohomology groups $H^n_{\cris}(X/\cS_{\PD})_{\bbQ}$, $n\in\bbN$, are finitely generated free $\bbQ\otimes R_{\PD}$-modules, 
and thus by \cite[Lem.\,4.4.11]{Ts} 
the section  $s_\pi$ are unique on the level of cohomology groups. 
In the derived category, we are dealing with
bounded complexes of $F$-vector spaces. 
But the category of $F$-vector spaces clearly is a semisimple abelian category 
and therefore its derived category is also semisimple abelian and 
in particular equivalent to the category of graded $F$-vector spaces. 
Therefore uniqueness on the level of cohomology groups imply uniqueness 
in the derived category as desired.
\end{proof}

\begin{remark}
According to \cite[(12)]{NN} there exists a canonical quasi-isomorphism
	$$
	\gamma\colon  R\Gamma_{\dR}(X_K)  \xrightarrow{\sim}  R\Gamma_{\cris}(X/V^\sharp)_\bbQ,
	$$
where  the left hand side is the de~Rham cohomology of $X_K$ with the Hodge filtration. 
The composition in the derived category 
	$$
	\Psi_\pi^\cris: = \gamma^{-1}\circ j_\pi^\ast\circ s_\pi\colon  R\Gamma_{\HK}^\cris(X,\pi)_{\bbQ}  \rightarrow   R\Gamma_{\dR}(X_K)
	$$
is classically called the Hyodo--Kato morphism.
\end{remark}

\begin{remark}
For a uniformiser $\pi$ and $Y:= X\times_{V^\sharp,i_\pi}k^0$ we can also consider the complexes $R\Gamma_{\cris}(Y/\cS_{PD})$ defined via $\tau: k^0\rightarrow \cS_{PD}$ and $R\Gamma_{\cris}(Y/V^\sharp)$ defined via $i_\pi: k_0 \rightarrow V^\sharp$. 
In the first case, one can show by a similar argument as for the Hyodo--Kato complex that this definition is up to canonical quasi-isomorphism independent of the choice of $\pi$. 
Moreover,  the natural morphism $\kappa_\pi\colon R\Gamma_{\cris}(X/\cS_{PD},\pi) \rightarrow R\Gamma_{\cris}(Y/\cS_{PD}) $ is a quasi-isomorphism. 
This does not hold in general over $V^\sharp$. 

In light of the above we have a commutative diagram
	$$
	\xymatrix{R\Gamma_{\HK}^\cris(X,\pi)_{\bbQ} \ar@<-.5ex>[rd]_{s_\pi}  & 
	R\Gamma_{\cris}(Y/\cS_{PD})_\bbQ  \ar[l]_{j_0^\ast} \ar[r]^{j_\pi^\ast} & 
	R\Gamma_{\cris}(Y/V^\sharp)_\bbQ \\
	      	&   
	R\Gamma_{\cris}(X/{\cS_{PD}},\pi)_{\bbQ} \ar[u]^\sim_{\kappa_\pi} \ar[r]^{j_{\pi}^\ast}  \ar@<-.5ex>[lu]_{j_0^\ast} & 
	R\Gamma_{\cris}(X/V^\sharp)_\bbQ \ar[u]_{\kappa_\pi}} 
	$$
where the right vertical map is not necessarily a quasi-isomorphism.
It highlights the fact that the classical crystalline Hyodo--Kato map depends on the choice of a uniformiser. 
\end{remark}

The comparison between the rigid Hyodo--Kato map and the crystalline Hyodo--Kato map passes through log convergent cohomology.

For a weak formal log scheme $\cZ$, the completion of the structure sheaf with respect to an ideal of definition defines a formal log scheme $\widehat{\cZ}$.
When the structure sheaf of $\cZ$ is already complete (e.g.\ $\cT=W^0,V^\sharp,\cS$), we often identify $\cZ$ and $\wh\cZ$, because they indeed coincide with each other as ringed spaces.

Let $\iota\colon T\hookrightarrow\cT$ be a homeomorphic exact closed immersion of a fine log scheme over $k$ into a fine weak formal log scheme over $W$.
For a fine log scheme $Y$ over $T$, the log convergent cohomology $R\Gamma_\conv(Y/\wh\cT)$ is defined in the same manner as the log rigid cohomology by using formal schemes and rigid analytic spaces instead of weak formal schemes and dagger spaces, respectively.
The completion induces a canonical morphism
\begin{equation}\label{eq: rigconv}
R\Gamma_\rig(Y/\cT)\rightarrow R\Gamma_\conv(Y/\wh\cT).
\end{equation}

Let $\wh\cT_\PD$ be the completion of the PD-envelope of $\wh\cT$ with respect to the ideal of $T$ in $\wh\cT$.

\begin{lemma}
	There exists a canonical morphism
	\begin{equation}\label{eq: convcris}
	R\Gamma_\conv(Y/\wh\cT)\rightarrow R\Gamma_\cris(Y/\wh\cT_\PD).
	\end{equation}
\end{lemma}

\begin{proof}
Let $\{Y_\lambda\}_{\lambda\in\Lambda}$ be an open covering of $Y$ such that there exist closed immersions $Y_\lambda\hookrightarrow\cZ_\lambda$ admitting charts into formal schemes which are adic and log smooth over $\wh\cT$.
For $i\geq 0$ let $Y^{(i)}$ and $\cZ^{(i)}$ be the $(i+1)$-fold fiber product of $\coprod_{\lambda\in\Lambda}Y_\lambda$ over $Y$ and $\coprod_{\lambda\in\Lambda}\cZ_\lambda$ over $\cT$, respectively.
For $n\geq 1$, let $\cZ^{(i)}_n$ be the exact closed log subscheme of $\cZ^{(i)}$ defined by the ideal $\cI^n\cO_{\cZ^{(i)}}$, where $\cI$ is the ideal of $T$ in $\wh\cT$.
Note that $\cZ^{(i)}_n$ is a log scheme since $\cZ^{(i)}$ is adic over $\cT$.

Choose a factorization $Y^{(i)}\hookrightarrow \cX^{(i)}\rightarrow\cZ^{(i)}$ into an exact closed immersion $Y^{(i)}\hookrightarrow X^{(i)}$ and an adic log \'{e}tale morphism $\cX^{(i)}\rightarrow\cZ^{(i)}$.
Then the log PD-envelope $D_n^{(i)}$ of  $Y^{(i)}\hookrightarrow \cZ_n^{(i)}$ is given as the (non-logarithmic) PD-envelope of $Y^{(i)}\hookrightarrow \cX_n^{(i)}:=\cX^{(i)}\times_{\cZ^{(i)}}\cZ^{(i)}_n$ equipped with the pull-back log structure, and is independent of the choice of $\cX^{(i)}$.
Note that the 
exactification $Y^{(i)}\hookrightarrow\cZ^{(i),\mathrm{ex}}$ of $Y^{(i)}\hookrightarrow\mathcal{Z}^{(i)}$ is given as the completion of $\cX^{(i)}$ along $Y^{(i)}$.
Let $\cJ\subset\cO_{\cZ^{(i),\mathrm{ex}}}$ be the ideal of $Y^{(i)}$ in $\cZ^{(i),\mathrm{ex}}$.
For $k\geq 1$, let $\cZ^{(i),\mathrm{ex}}[\frac{\cJ^k}{p}]^\wedge$ the $p$-adic formal scheme by applying the construction of Definition \ref{def: gen fib} (but using completion instead of weak completion), and endow it with the pull-back log structure.
	
We claim that the natural morphism $D_n^{(i)}\rightarrow\cX^{(i)}$ factors through the natural morphism $\cZ^{(i),\mathrm{ex}}[\frac{\cJ^k}{p}]^\wedge\rightarrow\cX^{(i)}$ for $k$ large enough.
Indeed, suppose the ideal of $Y^{(i)}\hookrightarrow\cX^{(i)}$ is generated by $f_1,\ldots,f_m\in\cO_{\cX^{(i)}}$.
	Then $\cO_{\cZ^{(i),\mathrm{ex}}[\frac{\cJ^k}{p}]^\dagger}$ is given by adding to $\cO_{\cX^{(i)}}$ sections of the form $f_1^{k_1}\cdots f_m^{k_m}/p$ for $(k_i)\in\bbN^m$ with $k_1+\cdots+k_m=k$.
	If $k\geq p(m-1)+1$, one of $k_1,\ldots,k_m$ must satisfy $k_j\geq p$, hence $f_1^{k_1}\cdots f_m^{k_m}/p$ can be divided by $f_j^p/p!$.
	Since $\cO_{D^{(i)}_n}$ is given by adding to $\cO_{\cX^{(i)}_n}$ sections of the form $f_j^r/r!$ for $1\leq j\leq m$ and $r\geq 0$, we obtain a morphism
		\begin{equation}\label{eq: DT2}
		D^{(i)}_n\rightarrow\cZ^{(i),\mathrm{ex}}\left[\frac{\cJ^k}{p}\right]^\wedge
		\end{equation}
	as desired.
	Since $R\Gamma_\cris(Y/\widehat{\cT}_{\PD,n})$ is computed by the log de Rham complex of $D^{(\bullet)}_n$ over $\widehat{\cT}_{\PD,n}$ (\cite[Thm.\,6.4]{Ka}), \eqref{eq: DT2} induces the desired morphism \eqref{eq: convcris}.
\end{proof}

Contrary to the rigid definition given in this paper, the crystalline Hyodo--Kato map does not offer the possibility to choose a branch of logarithm.
It is therefore only logical that we may compare the rigid and crystalline Hyodo--Kato morphism only for a specific choice of the $p$-adic logarithm.

For the comparison between the rigid and the crystalline Hyodo--Kato maps, 
we restrict ourselves to the case of a proper strictly semistable log scheme over $V^\sharp$ with a horizontal divisor.

\begin{definition}
\begin{enumerate}
\item For a uniformiser $\pi\in V$ and integers $n\geq 1$ and $m\geq 0$, let $V^\sharp(n,m)\rightarrow V^\sharp$ be the morphism of fine log schemes induced by the diagram
		\[\xymatrix{
		\bbN^n\oplus\bbN^m\ar[d]_-\alpha & \bbN\ar[d]^-{\alpha_{\pi}}\ar[l]_-\beta\\
		V[\bbN^n,\bbN^m] / (\alpha\circ \beta (1)-\pi) & V\ar[l],
		}\]
where $\alpha_{\pi}$ is defined by $\alpha_\pi(1):=\pi$, $\alpha$ is the natural inclusion, and $\beta$ is the composition of the diagonal map $\bbN\rightarrow\bbN^n$ and the canonical injection $\bbN^n\rightarrow\bbN^n\oplus\bbN^m$.
	\item A $V^\sharp$-log scheme $X$ is called {\it strictly semistable}, if Zariski locally on $X$ there exists a strict log smooth morphism $X \rightarrow V^\sharp(n,m)$ over $V^\sharp$.
	Note that this condition is independent of the choice of a uniformizer $\pi$.
	By the local description above,
	\[\{x\in X\mid \text{$\forall a\in\cN_{X,x}$ $\exists b\in\cN_{X,x}$ s.t.\ $ab$ is contained in the image of $\cN_{V^\sharp,f(x)}$}\},\]
	where $f$ denotes the structure morphism $X\rightarrow V^\sharp$, is an open subset of $X$.
	We call its complement equipped with the reduced structure the \textit{horizontal divisor} of $X$.
	The horizontal divisor can be empty because we allow the case $m=0$.
\end{enumerate}
\end{definition}

Let $X$ be a strictly semistable log scheme over $V^\sharp$ with horizontal divisor $D$. 
The complement  $U:=X\backslash D$ is a strictly semistable log scheme over $V^\sharp$ with empty horizontal divisor. 
Sometimes, the pair $(U,X)$ is called a strictly semistable $V^\sharp$-log scheme with boundary as in \cite[Def.~4.4]{EY}.
Denote by $\cX$ and $\cU$ the weak completions of $X$ and $U$, respectively.

\begin{proposition}\label{prop: rig-cris}
In the situation described above, assume in addition that $X$ is proper over $V$. 
For any choice of uniformiser $\pi$ there is a commutative diagram
	\begin{equation}\label{eq: rig cris diag}
	\xymatrix{R\Gamma_\HK(\cU,\pi)  \ar[r]^{\Psi_{\pi,\pi}} &  R\Gamma_\dR(\cU) \\
	R\Gamma_\HK(\cX,\pi)   \ar[u]^\sim\ar[d]^\sim \ar[r]^{\Psi_{\pi,\pi}}  &  R\Gamma_\dR(\cX) \ar[d]^\sim \ar[u]^\sim\\
	R\Gamma_\HK^\cris(X,\pi)_{\bbQ} \ar[r]^{\Psi_\pi^\cris} & R\Gamma_\cris(X/V^\sharp)_{\bbQ}}
	\end{equation}
compatible with Frobenius operators.
\end{proposition}
	
\begin{proof}
The upper square of \eqref{eq: rig cris diag} follows from Proposition \ref{prop: horizontal divisor}.
To give the bottom square, let $Y:=\cX\times_{V^\sharp,j_\pi}k^0$.
Consider the diagram
	$$
	\xymatrix{
	R\Gamma_{\HK}(Y) \ar[d]^\sim \ar[r]^{\psi} & R\Gamma_{\rig}(Y/\cS) \ar[r]^{j^\ast_\pi}\ar[ld]_-{j_0^{\ast}}& R\Gamma_{\rig}(Y/V^\sharp)  \\
	R\Gamma_{\rig}(Y/W^0)\ar[d]^\sim &R\Gamma_{\rig}(X_1/\cS)\ar[u]_-\sim\ar[d] \ar[r]^{j^\ast_\pi} \ar[l]_{j_0^\ast} & R\Gamma_{\rig}(X_1/V^\sharp)\ar[u]_-\sim\ar[d]^\sim \\
	R\Gamma_{\cris}(Y/W^0)_{\bbQ} \ar@{=}[u] \ar@<-0.5mm>[r]_{s_\pi}& R\Gamma_{\cris}(X/\cS_{\PD},\pi)_{\bbQ} \ar[r]^-{j^\ast_\pi} \ar@<-0.5mm>[l]_{j_0^\ast} & R\Gamma_{\cris}(X/V^\sharp)_\bbQ
	}$$
Note that, since $Y\hookrightarrow X_1$ is a homeomorphic exact closed immersion, their log rigid cohomologies coincide by definition.
The middle vertical morphisms from rigid to crystalline pass through log convergent cohomology as constructed in \eqref{eq: rigconv} and \eqref{eq: convcris}.
Note that $R\Gamma_\rig(Y/W^0)\rightarrow R\Gamma_\cris(Y/W^0)_\bbQ$ is a quasi-isomorphism by \cite[Lem.\,5.3]{EY} and \cite[Thm.\,2.36]{Sh2}.

Since the map $\psi$ is given by $u^{[i]}\mapsto 0$ for $i>0$, the composition of the upper horizontal maps gives $\Psi_{\pi,\pi}\colon R\Gamma_{\HK}(Y) \rightarrow R\Gamma_{\rig}(Y/V^\sharp)$.
The composition of the bottom maps give $\Psi_\pi^\cris: R\Gamma_{\HK}^{\cris}(X,\pi)_{\bbQ} \rightarrow R\Gamma_{\cris}(X/V^\sharp)_{\bbQ}$ by construction. 

The compositions
	\begin{align*}
	&R\Gamma_{\cris}(Y/W^0)_\bbQ\xleftarrow{\cong}R\Gamma_{\HK}(Y)\xrightarrow{\psi}R\Gamma_{\rig}(Y/\cS)\xleftarrow{\cong}R\Gamma_{\rig}(X_1/\cS)\rightarrow R\Gamma_{\cris}(X/\cS_{\PD},\pi)_{\bbQ},\\
	&R\Gamma_{\cris}(Y/W^0)_\bbQ\xrightarrow{s_\pi}R\Gamma_{\cris}(X/\cS_{\PD},\pi)_{\bbQ}\rightarrow R\Gamma_{\cris}(X/\cS_{\PD},\pi)_{\bbQ}
	\end{align*}
are both  sections of $j_0^{\ast}\colon R\Gamma_{\cris}(X/\cS_{\PD},\pi)_{\bbQ}\rightarrow R\Gamma_{\cris}(Y/W^0)_\bbQ$ and compatible with the Frobenius actions.
Hence the statement follows by uniqueness of such a section as given in Proposition-Definition \ref{prop-def: cris sections}.
\end{proof}

To finish this section, we compare our construction with Gro\ss{}e-Kl\"onne's construction of a rigid Hyodo--Kato morphism in \cite[\S~3]{GK3}. 
His construction for strictly semistable log schemes over $k^0$ was extended to allow a horizontal divisor in \cite[\S~3.1]{EY}. 
The construction uses an auxiliary simplicial log scheme with boundary over $\cS':=(\Spwf W[s]^\dagger,1\mapsto s)$. 

\begin{definition}
Let $Y$ be a strictly semistable log scheme over $k^0$ (not necessarily proper) with a horizontal divisor $D$, 
and set $U=Y\backslash D$.
Denote by $(V_\bullet^U,P_\bullet^U)$ and $(V^Y_\bullet , P^Y_\bullet)$ the simplicial log schemes with boundary associated to $U$ and $Y$, respectively, as described in \cite[\S~3.1]{EY}. 

Then there is a commutative diagram given in \cite[\S~3.1]{EY}
	$$
	\xymatrix{R\Gamma_\rig(U/W^0)  & R\Gamma_{\rig}((V^U_\bullet,P^U_\bullet)/\cS') \ar[l]_{j^\ast_0}^\sim \ar[r]^{j_\pi^\ast} & R\Gamma_\rig(U/V^\sharp)\\
	R\Gamma_\rig(Y/W^0)  \ar[u]^\sim & R\Gamma_{\rig}((V^Y_\bullet,P^Y_\bullet)/\cS') \ar[l]_{j^\ast_0}^\sim \ar[r]^{j_\pi^\ast} \ar[u]^\sim& R\Gamma_\rig(Y/V^\sharp) \ar[u]^\sim}
	$$
where the maps $j_0^\ast$ and $j_\pi^\ast$ are again induced by $s\mapsto 0$ and $s\mapsto \pi$, respectively.
By \cite[Lem.~3.3, Cor.~3.4]{EY} the vertical maps are quasi-isomorphisms. 	
For the upper row, it was shown in \cite[Thm.~3.1]{GK3} that the left horizontal map $j^\ast_0$ is a quasi-isomorphism
and that the composition $\Psi_\pi^{\GK}:= j_\pi^\ast\circ j_0^{\ast,-1}$ becomes a quasi-isomorphism which depends on $\pi$ after tensoring with $K$. 
It follows from \cite[Lem.~3.3, Cor.~3.4]{EY} that the same is true for the lower row.  
Moreover, it is functorial by \cite[Prop.~3.5]{EY}.
In both cases we refer to the map $\Psi_\pi^{\GK}$ as Gro\ss{}e-Kl\"onne's Hyodo--Kato map.
\end{definition}

\begin{corollary}
Let $X$ be a proper  strictly semistable $V^\sharp$-log scheme 
with a horizontal divisor $D$.
Let $U_X:=X\backslash D$. 
Denote by $\cU$  the weak completion of $U_X$.
For a choice of uniformiser $\pi$ of $V$, 
let $U:=\cU\times_{V^\sharp,j_\pi}k^0$.
There is a commutative diagram
	$$
	\xymatrix{R\Gamma_\HK(\cU,\pi) \ar[d]^\sim \ar[r]^{\Psi_{\pi,\pi}} &  R\Gamma_\dR(\cU) \ar@{=}[d] \\
	R\Gamma_\rig(U/W^0)  \ar[r]^{\Psi_\pi^{\GK}} & R\Gamma_\rig(U/V^\sharp)}
	$$
compatible with Frobenius operators.
\end{corollary}

\begin{proof}
Note that the left vertical map is the quasi-isomorphism from Corollary \ref{cor: natural diagram}. 
As we have seen in Proposition \ref{prop: rig-cris} our rigid Hyodo--Kato map $\Psi_{\pi,\pi}$ for $\cX$ is compatible with the crystalline one $\Psi_\pi^\cris$ for $X$. 
Hence it suffices to show that there is a commutative diagram
	$$
	\xymatrix{R\Gamma_\rig(U/W^0)  \ar[d]^\sim \ar[r]^{\Psi_\pi^{\GK}} &  R\Gamma_\rig(U/V^\sharp) \ar[d]^\sim \\
	R\Gamma_\HK^\cris(X,\pi)_{\bbQ} \ar[r]^{\Psi_\pi^\cris} & R\Gamma_\cris(X/V^\sharp)_{\bbQ}}
	$$
Indeed, with  $Y:=\cX \times_{V^\sharp,j_\pi}k^0$  for the weak completion $\cX$ of $X$, there is, as discussed in \cite[\S~5.2]{EY},  a commutative diagram
	$$
	\xymatrix{  	& R\Gamma_{\rig}((V^U_\bullet,P^U_\bullet)/\cS') \ar[ld]_{j^\ast_0}^\sim \ar[rd]^{j_\pi^\ast} & \\
	R\Gamma_\rig(U/W^0)  & R\Gamma_{\rig}((V^Y_\bullet, P^Y_\bullet)/\cS') \ar[ld]_{j^\ast_0}^\sim \ar[dr]^{j_\pi^\ast} \ar[u]^\sim \ar[d] &   R\Gamma_\rig(U/V^\sharp)\\
	R\Gamma_\rig(Y/W^0)  \ar[u]^\sim \ar[d]_\sim  & R\Gamma_{\rig}(Y/\cS) \ar[l]_{j^\ast_0} \ar[r]^{j_\pi^\ast} \ar[d] &     R\Gamma_\rig(Y/V^\sharp) \ar[u]^\sim \ar[d]_\sim \\
	R\Gamma_\HK^\cris(X,\pi)_{\bbQ}  \ar@<-.5ex>[r]_-{s_\pi}  & R\Gamma_{\cris}(X/{\cS_{PD}},\pi)_{\bbQ}   \ar@<-.5ex>[l]_-{j_0^\ast} \ar[r]^-{j_\pi^\ast} & R\Gamma_{\cris}(X/V^\sharp)_\bbQ  }
	$$
where the composition of the top maps give $\Psi_\pi^{\GK}: R\Gamma_\rig(U/W^0) \rightarrow R\Gamma_\rig(U/V^\sharp)$ and the composition of the bottom maps give $\Psi_\pi^\cris \colon R\Gamma_\HK^\cris(X,\pi)_{\bbQ}  \rightarrow R\Gamma_{\cris}(X/V^\sharp)_\bbQ$.
Note that we have $R\Gamma_\rig(Y/\cS)=R\Gamma_\rig(Y/\cS')$ since $Y$ is defined over $k^0$.
It is clear from the diagram that the composition 
\begin{align}\label{equ: composition}
R\Gamma_\HK^\cris(X,\pi)_{\bbQ}
\xleftarrow{\sim} R\Gamma_\rig(Y/W^0)
\xrightarrow{\sim} R\Gamma_\rig(U/W^0)
\xleftarrow[j_0^\ast]{\sim} R\Gamma_{\rig}((V^U_\bullet,P^U_\bullet)/\cS')\xleftarrow{\sim}\qquad\qquad\\
\xleftarrow{\sim} R\Gamma_{\rig}((V^Y_\bullet, P^Y_\bullet)/\cS')
\rightarrow R\Gamma_{\rig}(Y/\cS)
\rightarrow  R\Gamma_{\cris}(X/{\cS'_{PD}},\pi)_{\bbQ}\nonumber
\end{align}
is a section of $ R\Gamma_{\cris}(X/{\cS_{PD}},\pi)_{\bbQ} \xrightarrow{j_0^\ast} R\Gamma_\HK^\cris(X,\pi)_{\bbQ}$. 
Note that a priori it doesn't coincide with with $R\Gamma_\HK^\cris(X,\pi)_{\bbQ} \xrightarrow{s_\pi}  R\Gamma_{\cris}(X/{\cS'_{PD}},\pi)_{\bbQ}$. 
However it commutes with Frobenius:
Indeed, we have already discussed that $R\Gamma_{\cris}(X/{\cS_{PD}},\pi)_{\bbQ}$, $R\Gamma_{\rig}(Y/\cS)$, $R\Gamma_\rig(U/W^0)$, $R\Gamma_\rig(Y/W^0)$, and $R\Gamma_\HK^\cris(X,\pi)_{\bbQ}$ all have a Frobenius morphism which is naturally compatible with the morphisms between them. 
Furthermore, because of \cite[Lem.\,5.7]{EY} it is possible to endow $R\Gamma_{\rig}((V^U_\bullet,P^U_\bullet)/\cS')$ and $R\Gamma_{\rig}((V^Y_\bullet, P^Y_\bullet)/\cS') $  with a Frobenius morphism by means of the construction explained in the last paragraph of \cite[\S\,2.3.4]{EY}.
By functoriality properties of the axiomatisation of rigid complexes developed in \cite[\S\,2]{EY} these Frobenius morphisms are clearly compatible with the morphisms in the above diagram and consequently that composition (\ref{equ: composition}) commutes with Frobenius as desired. 
Now we may invoke again the uniqueness of Proposition-Definition \ref{prop-def: cris sections} and conclude that the composition (\ref{equ: composition}) coincides with the section $s_\pi$. 
It follows that the outer square of the above diagram is commutative, 
and in particular that 
$\Psi_\pi^{\GK}$ and $\Psi_\pi^\cris$ are compatible.
\end{proof}


\begin{thebibliography}{BHYY}
\bibitem[AJP]{AJP2005}L.\,Alonso Tarr\'{i}o, A.\,Jeremias, and M.\,P\'erez Rodr\'{i}guez, {\it Infinitesimal local study of formal schemes}, preprint, arXiv:math/0504256v2.

\bibitem[AJP2]{AJP}L.\,Alonso Tarr\'{i}o, A.\,Jeremias, and M.\,P\'erez Rodr\'{i}guez, {\it Infinitesimal lifting and Jacobi criterion for smoothness on formal schemes}, Comm.\,Algebra {\textbf{35}} no.4 (2007), 1341--1367.

\bibitem[BC]{BC}F.\,Baldassarri and B.\,Chiarellotto, {\it Algebraic versus rigid cohomology with logarithmic coefficients}, in: Barsotti Symposium in Algebraic Geometry, ed.\,by V.\,Cristante, W.\,Messing, Perspectives in Mth.\,\textbf{15}, Academic Press (1994). 

\bibitem[BHYY]{BHYY}K.\ Bannai, K.\ Hagihara, K.\ Yamada, and S.\ Yamamoto, {\it $p$-adic polylogarithms and $p$-adic Hecke $L$-functions for totally real fields}, preprint.

\bibitem[Bei]{Bei}A.\,Beilinson, {\it On the crystalline period map}, Cambridge J.\,of Math.\,\textbf{1}, no.\,1 (2013), 1--51.

\bibitem[Ber]{Ber}P.\,Berthelot, {\it Cohomologie rigide et cohomologie rigide \`{a} support propre, Part 1}, Pr\'{e}publication IRMAR 96-03, available at https://perso.univ-rennes1.fr/pierre.berthelot.

\bibitem[BO]{BO}P.\,Berthelot and A.\,Ogus, {\it Notes on crystalline cohomology}, Math.\ Notes \textbf{21}, Princeton University Press, (1978).

\bibitem[Bos]{Bos}S.\ Bosch, {\it A rigid analytic version of M.\ Artin's theorem on analytic equations}, Math.\ Ann.\ \textbf{255}, 395--404, (1981).

\bibitem[CCM]{CCM}
B. Chiarellotto, A. Ciccioni, and N. Mazzari, 
{\it Cycle classes and the syntomic regulator}, 
Algebra Number Theory \textbf{7}, no.\,3, 533--566, (2013).

\bibitem[EY]{EY}V.\,Ertl and K.\,Yamada, {\it Comparison between rigid and crystalline syntomic cohomology for strictly semistable log schemes with boundary}, Rend.\,Semin.\,Mat.\,Univ.\,Padova, vol.\,145, pp.\,213--291, (2021).

\bibitem[EY2]{EY2}V.\,Ertl and K.\,Yamada, {\it Poincaré duality for rigid analytic Hyodo--Kato theory}, preprint, arXiv:2009.09160.

\bibitem[Fa]{Fa}G.\,Faltings, {\it Integral crystalline cohomology over very ramified valuation rings}, Journal of the AMS \textbf{12}, no.\,1 (1999), 117--144.

\bibitem[Fa2]{Fa2}G,\,Faltings, {\it Almost \'{e}tale extensions}, In: Cohomologies $p$-adiques et applications arithm\'{e}tiques II, Ast\'{e}risque \textbf{279} (2002), 185--270.

\bibitem[Fo]{Fo}J.-M.\,Fontaine, {\it Le corps des p\'eriodes $p$-adiques}, Ast{\'e}risque \textbf{223}, Soc.\,Math.\,de France, (1994), 59--111.

\bibitem[GK1]{GK1}
E.\,Gro\ss e-Kl\"{o}nne, {\it Rigid analytic spaces with overconvergent structure sheaf}, 
J.\,Reine Angew.\,Math.\,\textbf{519} (2000), 73--95.

\bibitem[GK2]{GK2}
E.\,Gro\ss e-Kl\"{o}nne, {\it Finiteness of de Rham cohomology in rigid analysis}, 
Duke Math.\,J.\,\textbf{113},no.\,1 (2002), 57--91.

\bibitem[GK3]{GK3}
E.\,Gro\ss e-Kl\"{o}nne, {\it Frobenius and monodromy operators in rigid analysis, and Drinfel'd's symmetric space}, J.\,Algebraic Geom.\,{\textbf{14}} (2005), 391--437.

\bibitem[Ha]{Ha}
U.\,Hartl,
{\it Semi-stablility and base change},
Arch.\,Math.\,77 (2001), 215--221.

\bibitem[HL]{HL}U.\,Hartl and W.\,L\"{u}tkebohmert, {\it On rigid-analytic Picard varieties}, J.\,reine.\,angew.\,Math.\,\textbf{528} (2000), 101--148.

\bibitem[HK]{HK}O.\,Hyodo and K.\,Kato, {\it Semi-stable reduction and crystalline cohomology with logarithmic poles}, Ast{\'e}risque \textbf{223}, Soc.\,Math.\,de France, (1994), 221--268.

\bibitem[Ja]{Ja}U.\,Jannsen, {\it On the $\ell$-adic cohomology of varieties over number fields and its Galois cohomology}, in Galois groups over $\bbQ$, Springer-Verlag (1989), 315--360.

\bibitem[Ka]{Ka}K.\,Kato, {\it Logarithmic structures of Fontaine-Illusie}, in Algebraic Analysis, Geometry, and Number Theory, J-I.\,Igusa ed., 1988, Johns Hopkins University,  191--224.

\bibitem[KH]{KH} M.\,Kim and R.\,Hain, {\it A de Rham--Witt approach to crystalline rational homotopy theory}, Compositio Mathematica \textbf{140}, no.\,5 (2004), 1245--1276.

\bibitem[LM]{LM}
A.\,Langer and A.\,Muralidharan, 
{\it An analogue of Raynaud’s theorem: weak formal schemes and dagger spaces}, M\"unster J.\,Math.\,\textbf{6}, no.\,1 (2013), 271--294.

\bibitem[LNS]{LNS}
J.\ Lipman, S.\ Nayak, and P.\ Sastry,
{\it Pseudofunctorial behavior of Cousin complexes on formal schemes}, Variance and Duality for Cousin complexes on formal schemes, Contemp.\ Math. 375, Providence, RI: Amer.\ Math.\ Soc. (2005), 3--133.

\bibitem[Ma]{Ma}H.\,Matsumura, {\it Commutative ring theory}, Cambridge Studies in Advanced Mathematics, \textbf{8}, Cambridge University Press, Cambridge, (1986).

\bibitem[Me]{Me}D.\,Meredith, {\it Weak formal schemes}, Nagoya Math.\,J. {\textbf{45}} (1972), 1--38.

\bibitem[MW]{MW}P.\,Monsky and G.\,Washnitzer, {\it Formal cohomology: I}, Annals of Mathematics, Second Series {\textbf{88}}, no.\,2 (1968), 181--217.

\bibitem[MW2]{MW2}P.\,Monsky, {\it Formal cohomology: II}, Annals of Mathematics, Second Series {\textbf{88}}, no.\,2 (1968), 218--238.

\bibitem[Na]{Na}Y.\,Nakkajima, {\it $p$-adic weight spectral sequence of log varieties}, J.\,Math.\,Sci.\,Univ.\,Tokyo \textbf{12} (2005), 513--661.

\bibitem[NN]{NN}J.\,Nekov\'a\v{r} and W.K.\,Nizio\l{},  {\em Syntomic cohomology and $p$-adic regulators for varieties over $p$-adic fields}, Algebra Number Theory \textbf{ 10}, no.\,8 (2016), 1695--1790.

\bibitem[Ni]{Ni}W.\,Nizio\l, {\it Semi-stable conjecture via K-theory}, Duke Math.\,J.\,\textbf{141}, no.1 (2008), 151--179.

\bibitem[Og]{Og}A.\ Ogus, {\it Lectures on logarithmic algebraic geometry}, Cambridge Studies in Advanced Mathematics \textbf{178}, Cambridge University Press, Cambridge, (2018).

\bibitem[Sch]{Sch}P.\,Schneider, {\it Nonarchimedean functional analysis}, Springer Monogr.\,Math., Springer-Verlag, Berlin, (2002).

\bibitem[Sh1]{Sh1}A.\,Shiho, {\it Crystalline fundamental groups II -- Log convergent cohomology and rigid cohomology}, J.\,Math.\,Sci.\,Univ.\,Tokyo {\textbf 9} (2002), 1--163.

\bibitem[Sh2]{Sh2}A.\,Shiho, {\it Relative log convergent cohomology and relative rigid cohomology I}, Preprint, available at \textsc{arXiv:0707.1742v2}, (2008).

\bibitem[SP]{stacks}
The Stacks project authors,
\textit{The Stacks project}.
Available at \url{https://stacks.math.columbia.edu}, (2020).

\bibitem[Ts]{Ts}T.\,Tsuji, {\it $p$-adic \'etale cohomology and crystalline cohomology in the semi-stable reduction case}, Invent.\,Math.\,\textbf{137}, no.~2 (1999), 233--411.

\bibitem[vdP]{vdP}M.\ van der Put, {\it The cohomology of Monsky and Washnitzer}, M\'{e}mories de la S.\ M.\ F.\ $2^e$ s\'{e}rie, \textbf{23}, (1986), 33--59.

\bibitem[We]{We}C.A.\,Weibel, {\it An introduction to homological algebra}, Cambridge Studies in Adv.\,Math.\,\textbf{38}, Cambridge University Press, Cambridge UK, 1994.

\bibitem[Ya]{Ya}K.\,Yamada, 
{\it Hyodo--Kato theory with syntomic coefficients},
preprint, arXiv:2005.05694.


\end{thebibliography}
\end{document}